\documentclass[10pt]{article}
   \usepackage{amsmath,amsfonts,amsthm,amssymb,amscd,enumerate, url,slashed}

   \usepackage[pdftex]{graphicx,color}
   \usepackage{hyperref}
\theoremstyle{plain}
\newtheorem{thm}{Theorem}[section]
\newtheorem{prop}[thm]{Proposition}
\newtheorem{lemma}[thm]{Lemma}
\newtheorem{cor}[thm]{Corollary}

\theoremstyle{definition}
\newtheorem{defn}[thm]{Definition}
\newtheorem{ex}[thm]{Example}
\newtheorem{rmk}[thm]{Remark}
\newtheorem{nota}[thm]{Notation}

\newcounter{commentCounter}
\newcommand{\curl}{\ensuremath{\operatorname{curl}}}

\newcommand{\dive}{\ensuremath{\operatorname{div}}}

\newcommand{\dirac}{\ensuremath{\slashed D}}
\newcommand{\mdirac}{\ensuremath{\breve{\dirac}}}
\newcommand{\mdiracstar}{\ensuremath{\mdirac{}^*}}

\newcommand{\mlap}{\ensuremath{\breve{\Delta}}}

\newcommand{\vol}{\ensuremath{\mathsf{vol}}}
\newcommand{\hol}{\ensuremath{\mathrm{Hol}_{\gp} (M)}}
\newcommand{\real}{\ensuremath{\mathrm{Re}}}
\newcommand{\imag}{\ensuremath{\mathrm{Im}}}
\newcommand{\G}{\ensuremath{\mathrm{G}_2}}
\newcommand{\SP}{\ensuremath{\mathrm{Spin}(7)}}
\newcommand{\SUtw}{\ensuremath{\mathrm{SU}(2)}}
\newcommand{\SUth}{\ensuremath{\mathrm{SU}(3)}}
\newcommand{\SOs}{\ensuremath{\mathrm{SO}(7)}}
\newcommand{\Gs}{\ensuremath{\mathrm{G}_2}{-structure}}

\newcommand{\R}{\ensuremath{\mathbb R}}
\newcommand{\C}{\ensuremath{\mathbb C}}

\newcommand{\Z}{\ensuremath{\mathbb Z}}
\newcommand{\PR}{\ensuremath{\mathbb P}}

\newcommand{\ph}{\ensuremath{\varphi}}
\newcommand{\ps}{\ensuremath{\psi}}
\newcommand{\st}{\ensuremath{\ast}}
\newcommand{\hk}{\mathbin{\! \hbox{\vrule height0.3pt width5pt depth 0.2pt \vrule height5pt width0.4pt depth 0.2pt}}}

\newcommand{\ddr}{\ensuremath{\frac{\del}{\del\:\!\! r}}}

\newcommand{\dx}[1]{\ensuremath{d\:\!\! x^{#1}}}
\newcommand{\nab}[1]{\ensuremath{\nabla_{\! \! #1 \,}}}
\newcommand{\del}{\ensuremath{\partial}}

\newcommand{\ddt}{\ensuremath{\frac{d}{d t}}}

\newcommand{\z}{\ensuremath{\mathbf 0}}

\newcommand{\Fm}{\ensuremath{F_{\scriptscriptstyle M}}}
\newcommand{\Fs}{\ensuremath{F_{\scriptscriptstyle \Sigma}}}
\newcommand{\Fc}{\ensuremath{F_{\scriptscriptstyle C}}}
\newcommand{\Xs}{\ensuremath{X_{\scriptscriptstyle \Sigma}}}
\newcommand{\Xc}{\ensuremath{X_{\scriptscriptstyle C}}}
\newcommand{\Pc}{\ensuremath{P_{\scriptscriptstyle C}}}
\newcommand{\Pci}{\ensuremath{P_{\scriptscriptstyle C_i}}}
\newcommand{\dc}{\ensuremath{d_{\scriptscriptstyle C}}}
\newcommand{\dci}{\ensuremath{d_{\scriptscriptstyle C_i}}}
\newcommand{\ds}{\ensuremath{d_{\scriptscriptstyle \Sigma}}}
\newcommand{\dsc}{\ensuremath{d^*_{\scriptscriptstyle C}}}
\newcommand{\dss}{\ensuremath{d^*_{\scriptscriptstyle \Sigma}}}
\newcommand{\dssi}{\ensuremath{d^*_{\scriptscriptstyle \Sigma_i}}}
\newcommand{\curlc}{\ensuremath{\curl_{\scriptscriptstyle \! C}}}
\newcommand{\nabc}{\ensuremath{\nabla_{\scriptscriptstyle \! C}}}
\newcommand{\nabs}{\ensuremath{\nabla_{\scriptscriptstyle \! \Sigma}}}
\newcommand{\stc}{\ensuremath{\st_{\scriptscriptstyle C}}}
\newcommand{\sts}{\ensuremath{\st_{\scriptscriptstyle \Sigma}}}

\newcommand{\lapc}{\ensuremath{\Delta_{\scriptscriptstyle C}}}
\newcommand{\laps}{\ensuremath{\Delta_{\scriptscriptstyle \Sigma}}}
\newcommand{\gc}{\ensuremath{g_{\scriptscriptstyle C}}}
\newcommand{\gs}{\ensuremath{g_{\scriptscriptstyle \Sigma}}}
\newcommand{\volc}{\ensuremath{\mathsf{vol}_{\scriptscriptstyle C}}}
\newcommand{\vols}{\ensuremath{\mathsf{vol}_{\scriptscriptstyle \Sigma}}}

\newcommand{\phc}{\ensuremath{\ph_{\scriptscriptstyle C}}}
\newcommand{\psc}{\ensuremath{\ps_{\scriptscriptstyle C}}}
\newcommand{\gci}{\ensuremath{g_{\scriptscriptstyle C_i}}}
\newcommand{\phci}{\ensuremath{\ph_{\scriptscriptstyle C_i}}}
\newcommand{\psci}{\ensuremath{\ps_{\scriptscriptstyle C_i}}}

\newcommand{\nabci}{\ensuremath{\nabla_{\scriptscriptstyle \! C_i}}}
\newcommand{\lapci}{\ensuremath{\Delta_{\scriptscriptstyle C_i}}}
\newcommand{\dsci}{\ensuremath{d^*_{\scriptscriptstyle C_i}}}
\newcommand{\gm}{\ensuremath{g_{\scriptscriptstyle M}}}
\newcommand{\phm}{\ensuremath{\ph_{\scriptscriptstyle M}}}

\newcommand{\phn}{\ensuremath{\ph_{\scriptscriptstyle N}}}
\newcommand{\psn}{\ensuremath{\ps_{\scriptscriptstyle N}}}
\newcommand{\stm}{\ensuremath{\st_{\scriptscriptstyle M}}}
\newcommand{\nabm}{\ensuremath{\nabla_{\scriptscriptstyle \! M}}}
\newcommand{\volm}{\ensuremath{\mathsf{vol}_{\scriptscriptstyle M}}}
\newcommand{\lapm}{\ensuremath{\Delta_{\scriptscriptstyle M}}}
\newcommand{\dsm}{\ensuremath{d^*_{\scriptscriptstyle M}}}

\newcommand{\gp}{\ensuremath{g_{\ph}}}
\newcommand{\stp}{\ensuremath{\st_{\ph}}}
\newcommand{\Qp}{\ensuremath{Q_{\ph}}}
\newcommand{\Lp}{\ensuremath{L_{\ph}}}

\newcommand{\e}{\ensuremath{\varepsilon}}
\newcommand{\im}{\ensuremath{\operatorname{im}}}
\newcommand{\coker}{\ensuremath{\operatorname{coker}}}
\newcommand{\Ann}{\ensuremath{\operatorname{Ann}}}
\newcommand{\spi}{\ensuremath{\slashed S}}
\newcommand{\diracm}{\ensuremath{\dirac_{\scriptscriptstyle M}}}
\newcommand{\diracc}{\ensuremath{\dirac_{\scriptscriptstyle C}}}
\newcommand{\mdiracm}{\ensuremath{\mdirac_{{\scriptscriptstyle M}}}}
\newcommand{\mdiracc}{\ensuremath{\mdirac_{{\scriptscriptstyle C}}}}
\newcommand{\mdiracci}{\ensuremath{\mdirac_{{\scriptscriptstyle C_i}}}}

\newcommand{\mlapc}{\ensuremath{\mlap_{{\scriptscriptstyle C}}}}
\newcommand{\mlaps}{\ensuremath{\mlap_{{\scriptscriptstyle \Sigma}}}}
\newcommand{\ind}{\ensuremath{\operatorname{ind}}}
\newcommand{\rest}[2]{\ensuremath{ {\left. {#1} \right|}_{{#2}}}}
\newcommand{\DF}{\ensuremath{D F |_0}}
\newcommand{\DFhat}{\ensuremath{D \widehat F |_0}}
\newcommand{\os}{\ensuremath{\omega}}
\newcommand{\Os}{\ensuremath{\Omega}}
\newcommand{\Js}{\ensuremath{J}}
\newcommand{\vdim}{\ensuremath{\mathrm{v}\text{-}\! \dim}}
\newcommand{\dd}{\ensuremath {\mathfrak D}}
\newcommand{\ddc}{\ensuremath {\mathfrak D}_{\scriptscriptstyle C}}
\newcommand{\ddci}{\ensuremath {\mathfrak D}_{\scriptscriptstyle C_i}}
\newcommand{\ddm}{\ensuremath {\mathfrak D}_{\scriptscriptstyle M}}
\DeclareMathOperator\id{\mathrm{Id}}

\begin{document}

\title{Deformation theory of \texorpdfstring{$\mathbf{\G}$}{G2} conifolds}

\author{Spiro Karigiannis\footnote{A very small portion of this research was completed while the first author was supported by a Marie Curie Fellowship of the European Commission under contract number MIF1-CT-2006-039113. The contents of this work reflect only the authors' views and not the views of the European Commission.} \\ {\it Department of Pure Mathematics, University of Waterloo} \\ \tt{karigiannis@uwaterloo.ca} \and Jason D. Lotay\footnote{The second author was supported by an EPSRC Career Acceleration Fellowship EP/H003584/1 and EPSRC grant EP/K010980/1.} \\ {\it Department of Mathematics, University College London} \\ \tt{j.lotay@ucl.ac.uk} }

\maketitle

\begin{abstract}
We consider the deformation theory of \emph{asymptotically conical} (AC) and of \emph{conically singular} (CS) $\G$~manifolds. In the AC case, we show that if the rate of convergence $\nu$ to the cone at infinity is generic in a precise sense and lies in the interval $(-4, 0)$, then the moduli space is smooth and we compute its dimension in terms of topological and analytic data. For generic rates $\nu < -4$ in the AC case, and for generic positive rates of convergence to the cones at the singular points in the CS case, the deformation theory is in general obstructed. We describe the obstruction spaces explicitly in terms of the spectrum of the Laplacian on the link of the cones on the ends, and compute the virtual dimension of the moduli space.

We also present many applications of these results, including: the uniqueness of the Bryant--Salamon AC $\G$~manifolds via local rigidity and the cohomogeneity one property of AC $\G$~manifolds asymptotic to homogeneous cones; 
the smoothness of the CS moduli space if the singularities are modeled on particular $\G$~cones; and the proof of existence of a ``good gauge'' needed for desingularization of CS $\G$~manifolds. Finally, we discuss some open problems.
\end{abstract}

\tableofcontents

\section{Introduction} \label{introsec}

In this paper we study the deformation theory of certain $\G$~manifolds that are modeled on cones, which we call \emph{conifolds}. Specifically, we consider  \emph{asymptotically conical} (AC) $\G$~manifolds, which are noncompact manifolds of $\G$ holonomy with one end that is asymptotic (in an appropriate sense) to a $\G$~cone at infinity. We also consider \emph{conically singular} (CS) $\G$~manifolds. These are compact topological spaces such that after removing a finite set of points $\{x_1, \ldots, x_n\}$ they are noncompact manifolds of $\G$ holonomy with $n$ ends that are asymptotic (in an appropriate sense) to $n$ possibly distinct $\G$~cones at their vertices. The precise definitions will be given in Section~\ref{conemanifoldssec}.

This paper is a sequel to~\cite{Kdesings}, in which the first author studied the \emph{desingularization} of CS $\G$~manifolds by gluing in AC $\G$~manifolds. We shall therefore adopt the notation and conventions from~\cite{Kdesings}, but we shall review and restate the important definitions and results from~\cite{Kdesings} that are needed to keep the present paper as self-contained as possible.

The deformation theory of CS or AC manifolds in the context of special holonomy and calibrated geometry has been studied by Joyce in~\cite{JSL2} for CS special Lagrangian submanifolds, by Marshall~\cite{M} and Pacini~\cite{Pacini} for AC special Lagrangian submanifolds of $\C^m$, and by the second author~\cite{Lth, L1, L2, L3} for coassociative AC and CS submanifolds and associative AC submanifolds of $\G$~manifolds. Nordstr\"om~\cite{Nord} considered the deformation theory of \emph{asymptotically cylindrical} $\G$~manifolds. Finally, Pacini~\cite{Pac1, Pac2, Pac3} has a series of papers on the analysis of special Lagrangian conifolds, allowing for a mixture of both AC and CS ends.

It is interesting to compare the study of moduli spaces of conifolds which are \emph{submanifolds}, such as special Lagrangian, coassociative, or associative, to the moduli spaces of conifolds which are the ambient special holonomy manifolds themselves. The results in both cases are similar in spirit, but there are some notable exceptions in the details. One key point is the issue of gauge-fixing: in the submanifold case this is solved in a trivial way by essentially considering deformations defined using normal vector fields, but in the manifold case one has to work much harder leading to numerous analytic difficulties. Another key point is that in the submanifold case one can easily identify certain deformations (in the AC case) and obstructions (in the CS case) in a simple concrete way by using the fact that on the ends the submanifold may be viewed as a graph over a cone in Euclidean space, whereas in the manifold setting no such easy interpretation is usually possible.

\subsubsection*{Main results and applications}

The main theorem we prove in this paper is the following. The notation and terminology used in this theorem is defined in Section~\ref{conemanifoldssec} and Section~\ref{conifoldmodulidefnsec}.

\noindent {\bf Main Theorem} (Theorem~\ref{mainthm})
\emph{
Let $(M, \ph)$ be a $\G$~conifold, asymptotic to given $\G$~cones on the ends, at some rate $\nu$. Let $\mathcal M_{\nu}$ be the moduli space of all torsion-free $\G$~structures on $M$, asymptotic to the same cones on the ends, at the same rate $\nu$, modulo the appropriate notion of equivalence that preserves these conditions. Then for generic $\nu$ (in a sense made precise later), we have
\begin{itemize}
\item In the AC case, if $\nu \in (-4, 0)$, the space $\mathcal M_{\nu}$ is a \emph{smooth manifold} whose dimension consists of topological and analytic contributions, given precisely in Corollary~\ref{vdimfinalcor}.
\item In the AC case if $\nu < -4$, or in the CS case for any $\nu > 0$, the space $\mathcal M_{\nu}$ is in general only a \emph{topological space}, and the deformation theory may be obstructed. The virtual dimension of $\mathcal M_{\nu}$ again consists of topological and analytic contributions, given precisely in Corollary~\ref{vdimfinalcor}.
\end{itemize}
}
Here the ``appropriate notion of equivalence'' is by the action of diffeomorphisms which are asymptotic to the identity on the ends. This means that we consider only deformations of $\G$~conifolds that \emph{fix} the $\G$~cones on the ends. Note that the link of a $\G$~cone is a compact strictly \emph{nearly K\"ahler} $6$-manifold, also known as a \emph{Gray manifold}. The infinitesimal deformations of Gray manifolds were considered by Moroianu--Nagy--Semmelmann~\cite{MNS}, and the deformations were recently shown by Foscolo~\cite{F} to be obstructed in general. Indeed, there are at present only six known examples of simply-connected Gray manifolds, including the round $S^6$ and two inhomogeneous examples found recently by Foscolo--Haskins~\cite{FH}. Non-simply connected locally homogeneous examples were also found recently by Cort\'es--V\'asquez~\cite{CV}. The known simply-connected Gray manifolds are diffeomorphic to $S^6$, $\C \PR^3$, $\mathrm{SU}(3) / T^2$, and $S^3 \times S^3$, and the homogeneous structures on the latter three spaces are described in more detail in Section~\ref{g2conessec}.

\subsubsection*{Applications and corollaries}

Perhaps even more interesting than our main theorem are its several applications, which are stated precisely in Section~\ref{applicationssec}. In particular, we use Theorem~\ref{mainthm} and its ingredients to establish the following corollaries:

\begin{itemize} \setlength\itemsep{-1mm}
\item The AC moduli space, when it is smooth, is always at least $1$-dimensional because it contains deformations that are asymptotic to rescaling of the $\G$~cone at infinity. As a corollary of this observation and our moduli space dimension formula, we prove the local rigidity of the Bryant--Salamon AC $\G$~manifolds, not just as AC manifolds of rate $\nu +\e$, where $\nu$ is $-3$ or $-4$ in these cases, but in fact as AC manifolds of rate $\lambda$ all the way to any $\lambda < 0$.
\item A consequence of our main theorem is that when $\nu \in (-4, -3 + \e)$, the moduli space $\mathcal M_{\nu}$ of AC $\G$~manifolds, which is smooth in this case, is determined purely by the topology of the underlying $\G$~manifold $M$. Moreover, we show that for $\nu = -3 + \e$ this moduli space $\mathcal M_{-3 + \e}$ can be naturally \emph{immersed} into the vector space $H^3 (M, \R) \times H^4 (M, \R)$.
\item We prove that an AC $\G$~manifold that is asymptotic to a \emph{homogeneous} $\G$~cone must be of cohomogeneity one. Combining this with work of Cleyton--Swann~\cite{ClSw} and Brandhuber~\cite{Brandhuber} establishes that the Bryant--Salamon manifolds $\Lambda^2_-(S^4)$, $\Lambda^2_-(\C \PR^2)$, and $\spi (S^3)$ are the unique AC $\G$~manifolds asymptotic to the cones over the homogeneous Gray manifolds $\C \PR^3$, $\mathrm{SU}(3)/T^2$, and $S^3 \times S^3$, respectively.
\item We argue that CS $\G$~manifolds with particular types of conical singularities, including those modeled on the $\G$~cones over $\C \PR^3$ or $S^3 \times S^3$, have unobstructed deformations and thus admit a smooth moduli space $\mathcal M_{\nu}$.
\item We explicitly compare the dimensions of the moduli space $\mathcal M_{\nu}$ of CS $\G$~manifolds with one singularity (when it is smooth) to the moduli space of the compact smooth $\G$~manifolds obtained by the desingularization construction in~\cite{Kdesings}. This observation gives evidence of the likelihood that CS $\G$~manifolds are the dominant contributor to the ``boundary'' of the moduli space of compact smooth $\G$~manifolds.
\item We prove that an appropriate \emph{gauge-fixing condition} on AC $\G$~manifolds, which is required for the desingularization theorem in~\cite{Kdesings}, can in fact always be achieved.
\end{itemize}

To prove our main theorem and describe the deformation theory of $\G$~conifolds, we follow in spirit the approach of Joyce~\cite{J4}, who considered the deformation theory of \emph{compact} $\G$~manifolds. However, almost all of the steps in his proof use compactness in a crucial way, so we need to make nontrivial extensions to establish our results in the noncompact setting of conifolds.

One technical issue is that we use weighted Banach spaces of sections that decay at some rate $\lambda$ on the ends of the manifold, but these Banach spaces are \emph{not} always subspaces of $L^2$. Indeed, the rate at which the transition occurs between being in $L^2$ or not, specifically $\lambda = - \frac{7}{2}$, lies \emph{precisely between} the rates $-4$ and $-3$ that together encompass all known examples of AC $\G$~manifolds. As a result, we need to delicately work right on the edge of where certain analytic results actually hold, and in several cases we need to find subtle ways to enable us to extend the range where these results hold. This is in contrast, for example, to the asymptotically cylindrical case where one can always essentially work with $L^2$ sections. Another issue is that in the non-$L^2$ case, we are forced to use the \emph{Dirac operator} on $\G$~manifolds to prove our ``slice theorem''. This is similar to Nordstr\"om~\cite{Nord}. Finally, the noncompactness of the manifolds (and thus the nonavailability of classical Hodge theory) makes it more natural to consider the Fredholm theory of the operator $d + d^*$ rather than the Laplacian $\Delta$ to study the moduli space.  Some of the issues highlighted here are purely analytic technical problems but others are geometrically relevant.

\begin{rmk} \label{joycermk}
There are at present no known examples of CS~$\G$~manifolds. The first author and Dominic Joyce have a new construction~\cite{JK} of \emph{smooth, compact} $\G$~manifolds that may be generalizable to produce the first examples of CS~$\G$~manifolds. In these examples the singular cones would all be cones over the nearly K\"ahler $\C \PR^3$. This possible generalization is currently being investigated by the authors of the present paper. The authors are also aware of a proposed construction of CS~$\G$~manifolds due to Foscolo, Haskins, and Nordstr\"om, where the cones at the singularities would have link $S^3\times S^3$.
\end{rmk}

\subsubsection*{Organization of the paper}

We now discuss the organization of our paper. Section~\ref{prelimsec} reviews some aspects of $\G$~geometry that we require, including the spinor bundle and the Dirac operator for $\G$~manifolds. More details about $\G$~structures can be found in Bryant~\cite{Br1} and Joyce~\cite{J4}. Section~\ref{conemanifoldssec} is a review of some of the material from~\cite{Kdesings} about $\G$~conifolds. In Section~\ref{analsec}, we begin with a brief review and summary of the relevant results that we need from the Lockhart--McOwen theory of weighted Sobolev spaces on conifolds, including some Hodge-theoretic results in this setting. We then discuss in great detail many analytic results about $\G$~conifolds. In particular, this includes: a special index-change theorem; topological results about $\G$~conifolds; parallel tensors on $\G$~conifolds; various results concerning our gauge-fixing condition on moduli spaces of conifolds; and some material on analytic aspects of the Dirac operator on $\G$~cones that we require. In Section~\ref{conifoldmodulisec} we consider the deformation theory of $\G$~conifolds, and prove our main theorem in four steps. Finally, in Section~\ref{applicationssec} we present many applications of our results, as described above, and discuss some open problems.

\subsubsection*{Conventions}

We use single vertical bars $| \cdot |$ or angle braces $\langle \cdot, \cdot \rangle$ for a \emph{pointwise} inner product on sections of some vector bundle, and we use double vertical bars $|| \cdot ||$ or angle braces $\langle \langle \cdot, \cdot, \rangle \rangle$ for a global ($L^2$) inner product on sections. Since all of our manifolds are Riemannian, we often use the metric $g$ to identify vector fields and $1$-forms. This will always be clear by the context.

There are two sign conventions in $\G$~geometry. The convention we choose is the one used in Bryant--Salamon~\cite{BS} and in Harvey--Lawson~\cite{HL}, but differs from the convention used in Bryant~\cite{Br1} or Joyce~\cite{J4}. A detailed discussion of sign conventions and orientations in $\G$~geometry can be found in the first author's note~\cite{K2}.

\subsubsection*{Acknowledgements}

The authors would like to thank Benoit Charbonneau, Dominic Joyce, Johannes Nordstr\"om, Uwe Semmelmann, and Nico Spronk for useful discussions. The authors are also extremely grateful to the anonymous referees who made numerous useful suggestions that greatly improved our paper. In particular, one referee made extremely useful suggestions for improving the proofs of Propositions~\ref{excludeextensionprop} and~\ref{gaugefixingpromiseprop.2}, provided the idea for Lemma~\ref{BochnerKillinglemma} and its applications, and suggested some ideas for Section~\ref{cohomogeneitysec}.

\section{Preliminaries on \texorpdfstring{$\mathbf{\G}$}{G2} manifolds} \label{prelimsec}

\subsection{\texorpdfstring{$\mathbf{\G}$}{G2} structures} \label{G2structuressec}

A $\G$~structure on a smooth $7$-manifold $M$ is a smooth $3$-form $\ph$ satisfying a certain ``nondegeneracy'' condition. Various approaches to describing this nondegeneracy condition can be found, for example, in~\cite{Br1, J4, Kflows} but we will not explicitly need these. A $\G$~structure $\ph$ determines a Riemannian metric $g_{\ph}$ and an orientation $\vol_{\ph}$ in a nonlinear way. Thus $\ph$ determines a Hodge star operator $\st_{\ph}$, and we let $\ps = \st_{\ph} \ph$ denote the dual $4$-form. When a $\G$~structure exists, there is an open subbundle of the bundle of $3$-forms whose space of sections, denoted $\Omega^3_+$, consists of nondegenerate $3$-forms, also called \emph{positive} or \emph{stable} $3$-forms.

\begin{defn} \label{g2manifolddefn}
A $\G$~manifold is a connected manifold with a $\G$~structure $(M, \ph)$ such that $\ph$ is \emph{parallel} with respect to the Levi-Civita connection $\nab{}$ determined by $g_{\ph}$. That is, $\nab{g_{\ph}}\ph = 0$. Such a $\G$~structure is also called \emph{torsion-free}. In this case the Riemannian
holonomy $\hol$ of $(M, g_{\ph})$ is contained in the group $\G \subseteq \SOs$.
\end{defn}
\begin{rmk} \label{g2manifoldrmk}
A $\G$~manifold is always \emph{Ricci-flat}, and a $\G$~structure $\ph$ is torsion-free if and only if it is both closed and coclosed: $d\ph = 0$ and $d \ps = 0$.
\end{rmk}

On a manifold with $\G$~structure, there is a decomposition of the bundle $\Lambda^k T^* M$ of $k$-forms determined by irreducible representations of $\G$. The space $\Omega^3$ of $3$-forms decomposes as
\begin{equation} \label{lambda3eq}
\Omega^3 = \Omega^3_1 \oplus \Omega^3_7 \oplus \Omega^3_{27} .
\end{equation}
Similarly we have a decomposition of the space $\Omega^2$ as
\begin{equation} \label{lambda2eq}
\Omega^2 = \Omega^2_7 \oplus \Omega^2_{14},
\end{equation}
and isomorphic splittings of $\Omega^4$ and $\Omega^5$ given by the Hodge star of the above decompositions: $\Omega^k_l = \st_{\ph} (\Omega^{7-k}_l)$. The space  $\Omega^k_l$ consists of sections of a bundle with fibre dimension $l$ and these decompositions of $\Omega^k$ are orthogonal with respect to the metric $\gp$. The explicit descriptions of these spaces that we will need are as follows:
\begin{align}
\label{o27eq} \Omega^2_7 & \, = \, \{ \st (\alpha \wedge \ps) ;\, \alpha \in \Omega^1 \} \, = \, \{ \beta \in \Omega^2; \, \st ( \ph \wedge \beta ) = -2 \beta \}, \\ \label{o214eq} \Omega^2_{14} & \, = \, \{ \beta \in \Omega^2; \, \beta \wedge \ps = 0 \} \, = \, \{ \beta \in \Omega^2; \, \st ( \ph \wedge \beta) = \beta \}, \\ \label{o31eq} \Omega^3_1 & \, = \, \{ f \ph; \, f \in \Omega^0 \}, \\
\label{o37eq} \Omega^3_7 & \, = \, \{ \st( \alpha \wedge \ph); \, \alpha \in \Omega^1 \}, \\
\label{o327eq} \Omega^3_{27} & \, = \, \{ \eta \in \Omega^3; \, \eta \wedge \ph = 0 \text{ and } \eta \wedge \ps = 0 \}.
\end{align}
Moreover, the space $\Omega^3_{27}$ is isomorphic to the space of sections of $S^2_0 (T^*M)$, the \emph{traceless symmetric $2$-tensors} on $M$, where the correspondence is given explicitly as
\begin{equation} \label{o327detailedeq}
\begin{aligned}
\eta = \frac{1}{6} \eta_{ijk} \dx{i} \wedge \dx{j} \wedge \dx{k} \in \Omega^3_{27} \, & \longleftrightarrow \, h_{ab} \dx{a} \dx{b} \in C^{\infty}(S^2_0 (T^*M)), \\
\text{ where } \, \eta_{ijk} \, & \, = \, h_{ip} g^{pq} \ph_{qjk} + h_{jp} g^{pq} \ph_{iqk} + h_{kp} g^{pq} \ph_{ijq}.
\end{aligned}
\end{equation}

\begin{rmk} \label{torsionrmk}
One can thus decompose $d\ph = \pi_1(d\ph) + \pi_7(d\ph) + \pi_{27}(d\ph)$ and $d\ps = \pi_7(d\ps) + \pi_{14}(d\ps)$ for any $\G$~structure. It is a nontrivial but well known fact that $\pi_7(d\ph)$ vanishes if and only if $\pi_7(d\ps)$ vanishes. See, for example,~\cite{Kflows} for a direct verification of this fact. In particular, the implication of this that we will require is that for a \emph{closed} $\G$~structure, we have $d\ps \in \Omega^5_{14}$.
\end{rmk}

\begin{rmk} \label{laplaciansrmk}
When the $\G$~structure is torsion-free, the given decompositions of the spaces of forms are preserved by the Hodge Laplacian $\Delta = d d^* + d^* d$. The essential aspect of this fact that we will need is the following. Suppose $f$ is any function and $X$ is any $1$-form on a $\G$~manifold $M$. Then
\begin{align} \label{temp27eq1}
\Delta (f \ph) & = (\Delta f) \ph, & \Delta (f \ps) & = (\Delta f) \ps, & \\
\label{temp27eq2} \Delta (X \wedge \ph) & = (\Delta X) \wedge \ph, & \Delta (X \wedge \ps) & = (\Delta X) \wedge \ps.
\end{align}
The identities in~\eqref{temp27eq1} can be proved using just the fact that $\ph$ and $\ps$ are parallel, while the identities in~\eqref{temp27eq2} also require the fact that $\G$~manifolds have vanishing Ricci curvature.
\end{rmk}

We end this section with a discussion of the nonlinear map $\Theta : \Omega^3_+ \to \Omega^4$ which associates to any $\G$~structure $\ph$, the dual $4$-form $\ps = \Theta(\ph) = \st_{\ph} \ph$ with respect to the metric $g_{\ph}$ and orientation associated to $\ph$. One result which will be crucial is the following. This is Proposition 10.3.5 in Joyce~\cite{J4}, adapted to suit our present purposes.

\begin{lemma} \label{quadlemma}
Suppose that $\ph$ is a torsion-free $\G$~structure with induced metric $\gp$, and dual $4$-form $\ps = \st_{\ph}  \ph$. Let $\eta$ be a $3$-form which has sufficiently small $C^0$ norm with respect to $\gp$, so that $\ph + \eta$ is still nondegenerate. Then we have
\begin{equation} \label{quadeq}
\Theta(\ph + \eta) \, = \, \ps + \stp \left( \frac{4}{3} \pi_1 (\eta) + \pi_7 (\eta) - \pi_{27} (\eta) \right) + \Qp(\eta),
\end{equation}
where $\pi_k$ is the projection onto the subspace $\Omega^3_k$ with respect to the $\G$~structure $\ph$. The nonlinear map $\Qp : \Omega^3 \to \Omega^4$ satisfies
\begin{equation} \label{quadeq2}
\Qp(0) = 0, \qquad \qquad | \Qp(\eta) | \leq C |\eta|^2, \qquad \qquad |\nab{} \Qp (\eta)| \leq C |\eta| |\nab{} \eta|,
\end{equation}
for some $C > 0$, where the norms and the covariant derivatives are taken with respect to $\gp$.
\end{lemma}

We will denote the second term on the right hand side of~\eqref{quadeq}, which is the term \emph{linear} in $\eta$, by $\Lp (\eta)$. That is,
\begin{equation} \label{Lpeq}
\Lp (\eta) \, = \, \stp \left( \frac{4}{3} \pi_1 (\eta) + \pi_7 (\eta) - \pi_{27} (\eta) \right).
\end{equation}
The map $\Lp : \Omega^3 \to \Omega^4$ is the linearization of the nonlinear map $\Theta$ at $\ph$, and is therefore a key ingredient for understanding the infinitesimal deformations of torsion-free $\G$~structures.

Suppose that $\ph$ is a torsion-free $\G$~structure, so that in particular it is coclosed: $d\ps = 0$. Take the exterior derivative of~\eqref{quadeq} to obtain:
\begin{equation} \label{dThetaeq0}
d(\Theta(\ph + \eta)) \, = \, d (\Lp(\eta)) + d(\Qp(\eta))
\end{equation}
and hence
\begin{equation} \label{dThetaeq}
\stp d(\Theta(\ph + \eta)) \, = \, - d^* \stp (\Lp(\eta)) - d^* \stp (\Qp(\eta)).
\end{equation}
We will use~\eqref{dThetaeq} in Section~\ref{onetoonesec} when we establish a one-to-one correspondence between torsion-free ``gauge-fixed'' $\G$~structures and solutions to a nonlinear partial differential equation.

\subsection{The spinor bundle and the Dirac operator on \texorpdfstring{$\mathbf{\G}$}{G2} manifolds} \label{spinorsec}

A $\G$~structure on a manifold $M$ induces a spin structure, and therefore $M$ admits an associated \emph{Dirac operator} $\dirac$ on its spinor bundle $\spi(M)$. When the $\G$~structure is torsion-free this Dirac operator squares to the Hodge Laplacian, after identifying spinors with forms. These facts are explained in detail in the first author's note~\cite{K2}. Here we only review the facts that are needed in the present paper. The $\G$~structure $\ph$ is always understood to be torsion-free in this section. Also, we will make repeated use of the identities relating the interior product, the wedge product, and the star operator for $\G$~structures, which can be found in~\cite[Lemma 2.23]{K1}. (Note that since we are using the opposite orientation convention from~\cite{K1}, equation (2.13) in that paper should have a factor of $-2$ instead of $+2$.)

\begin{defn} \label{curldefn}
We define the \emph{curl} of a vector field $X$ to be the vector field $\curl X$ given by
\begin{equation} \label{curleq}
\curl X \, = \, \st (d X \wedge \ps)
\end{equation}
In other words, up to $\G$-equivariant isomorphisms, the vector field $\curl X$ is the projection onto the $\Omega^2_7$ component of the $2$-form $d X$. It is easy to check that in local coordinates we have
\begin{equation} \label{curleq2}
(\curl X)_k \, = \, g^{pi} g^{qj} (\nabla _p X_q) \ph_{ijk}.
\end{equation}
\end{defn}

\begin{lemma} \label{curllemma}
Consider the vector field $X$ as a $1$-form using the metric. Then $dX \in \Omega^2 = \Omega^2_7 \oplus \Omega^2_{14}$. The $\Omega^2_7$ component of $dX$ is given by
\begin{equation} \label{curllemmaeq}
\pi_7 (dX) \, = \, \frac{1}{3} (\curl X) \hk \ph \, = \, \frac{1}{3} \st \bigl( (\curl X) \wedge \ps \bigr).
\end{equation}
\end{lemma}
\begin{proof}
We know that $\pi_7 (dX) = W \hk \ph$ for some vector field $W$. We compute
\begin{align*}
\curl X \, & = \, \st (dX \wedge \ps) \, = \, \st ( \pi_7(dX) \wedge \ps) \\
& = \, \st ( (W \hk \ph) \wedge \ps) \, = \, \st ( 3 \st W) \, = \ 3 W,
\end{align*}
as claimed.
\end{proof}

\begin{rmk} \label{curlrmk}
We will have several occasions to use relations between gradient, curl, and divergence. Recall that we always identify $1$-forms with their metric dual vector fields. The facts that will be needed are
\begin{align*}
d^* (\curl Y) \, & = \, 0 \quad \text{for any vector field $Y$}, \\ 
\curl (df) \, & = \, 0 \quad \text{for any function $f$}, \\ 
\curl (\curl Y) \, & = \, - d d^* Y + \Delta Y \, = \, d^* d Y \quad \text{for any vector field $Y$}. 
\end{align*}
These facts are all proved in~\cite{K2}.
\end{rmk}

There is a natural identification of the spinor bundle $\spi(M)$, a rank $8$ real vector bundle, with the bundle $\R \oplus TM$ whose sections lie in $\Omega^0_1 \oplus \Omega^1_7$.
\begin{defn} \label{diracdefn}
The \emph{Dirac operator} $\dirac$ is a first order differential operator from $\spi(M)$ to $\spi(M)$ defined as follows. Let $s = (f, X)$ be a section of $\spi(M)$. Then
\begin{equation} \label{diraceq}
\dirac (f, X) \, = \,  ( d^* X \, , \, df + \curl X ).
\end{equation}
The Dirac operator is \emph{formally self-adjoint}: $\dirac^* = \dirac$.
\end{defn}

We now relate the \emph{Dirac Laplacian} $\dirac^* \! \dirac = \dirac^2$ to the Hodge Laplacian $\Delta$.
\begin{prop} \label{twolapsprop}
Under the identification of the spinor bundle $\spi (M)$ with the bundle $\Omega^0_1 \oplus \Omega^1_7$, the Dirac Laplacian $\dirac^2$ and the Hodge Laplacian $\Delta$ are equal:
\begin{equation} \label{twolapseq}
\dirac^2 (f, X) \, = \, (\Delta f, \Delta X).
\end{equation}
\end{prop}
\begin{proof}
Proposition~\ref{twolapsprop} is proved in~\cite{K2}, using the facts from Remark~\ref{curlrmk}.\end{proof}

For our present purposes, we actually require a slight \emph{modification} of the Dirac operator as follows. The spinor bundle $\spi (M)$ is isomorphic to $\Omega^0_1 \oplus \Omega^1_7$ and hence, via a $\G$-equivariant isomorphism, it is also isomorphic to $\Omega^3_1 \oplus \Omega^3_7$. Now consider the map
\begin{equation} \label{modifieddiracdefneq}
\begin{aligned}
\mdirac \, : \, \Omega^0_1 \oplus \Omega^1_7 & \to \Omega^3_1 \oplus \Omega^3_7 \\
\quad (f, X) & \mapsto \frac{1}{2} \st (df \wedge \ph) + \pi_{1 + 7} (d (X \hk \ph))
\end{aligned}
\end{equation}
where $\pi_{1 + 7}$ denotes orthogonal projection onto $\Omega^3_1 \oplus \Omega^3_7$. This is a first order linear differential operator. Using a particular $\G$-equivariant isomorphism that identifies the codomain of $\mdirac$ with $\Omega^0_1 \oplus \Omega^1_7$, we can compare this operator $\mdirac$ with the usual Dirac operator $\dirac$ from Definition~\ref{diracdefn}. Before we can explicitly describe this identification, we need a preliminary lemma that will be useful on multiple occasions.

\begin{lemma} \label{dXhkphlemma}
Let $X$ be a vector field, and consider the form $X \hk \ph \in \Omega^2_7$. Then
\begin{equation}  \label{dXhkpheq}
\pi_1 \bigl( d (X \hk \ph) \bigr) \, = \, - \frac{3}{7} (d^* X) \ph, \qquad \pi_7 \bigl( d (X \hk \ph) \bigr) \, = \, \frac{1}{2} \st \bigl( (\curl X) \wedge \ph \bigr).
\end{equation}
\end{lemma}
\begin{proof}
We have
\begin{equation*}
\pi_1 ( d(X \hk \ph)) \, = \, h \ph \quad \text{for some $h \in \Omega^0_1$.}
\end{equation*}
Using the fact that $\Omega^3_7 \oplus \Omega^3_{27}$ lies in the kernel of wedge product with $\ps$, we compute
\begin{equation*}
d ( (X \hk \ph) \wedge \ps ) \, = \, d(X \hk \ph) \wedge \ps \, = \, \pi_1 ( d(X \hk \ph)) \wedge \ps \, = \, h \ph \wedge \ps \, = \, 7 h \vol.
\end{equation*}
Hence, we find that
\begin{equation*}
d ( 3 \st X ) \, = \, d ( (X \hk \ph) \wedge \ps ) \, = \, 7 h \vol,
\end{equation*}
and thus $h = \frac{3}{7} \st d (\st X) = -\frac{3}{7} d^* X$. Similarly, we have
\begin{equation*}
\pi_7 ( d(X \hk \ph)) \, = \, \st (Y \wedge \ph) \quad \text{for some $Y \in \Omega^1_7$.}
\end{equation*}
Using the fact that $\Omega^3_1 \oplus \Omega^3_{27}$ lies in the kernel of wedge product with $\ph$, we compute
\begin{equation*}
d ( (X \hk \ph) \wedge \ph ) \, = \, d(X \hk \ph) \wedge \ph \, = \, \pi_7 ( d(X \hk \ph)) \wedge \ph \, = \, \st (Y \wedge \ph) \wedge \ph \, = \, -4 \st Y.
\end{equation*}
Hence, we find that
\begin{equation*}
-4 \st Y \, = \, d ( (X \hk \ph) \wedge \ph ) \, = \, d ( -2 \st ( X \hk \ph ) ) \, = \, -2 d (X \wedge \ps) \, = \, -2 (dX) \wedge \ps,
\end{equation*}
and thus $Y = \frac{1}{2} \st ((dX) \wedge \ps) = \frac{1}{2} \curl X$.
\end{proof}

\begin{prop} \label{modifieddiracprop}
The ``modified Dirac operator'' $\mdirac$ of equation~\eqref{modifieddiracdefneq}, when considered as a linear operator on $\Omega^0_1 \oplus \Omega^1_7$ via the $\G$-equivariant isomorphism
\begin{equation} \label{modifieddiracisomorphismeq}
\begin{aligned}
\Omega^0_1 \oplus \Omega^1_7 & \, \cong \, \Omega^3_1 \oplus \Omega^3_7 \\
(f, X) & \, \leftrightarrow \, \left( - \frac{3}{7} f \ph, \, \frac{1}{2} \st (X \wedge \ph) \right)
\end{aligned}
\end{equation}
is the usual Dirac operator
\begin{equation} \label{modifieddiracdefneq2}
\dirac \, : \, (f, X) \mapsto \left( d^* X, \, df + \curl X \right).
\end{equation}
Hence the operator $\mdirac : \Omega^0_1 \oplus \Omega^1_7 \to \Omega^3_1 \oplus \Omega^3_7$ is essentially the same as $\dirac : \Omega^0_1 \oplus \Omega^1_7 \to \Omega^0_1 \oplus \Omega^1_7$ and is in particular elliptic.
\end{prop}
\begin{proof}
Using Lemma~\ref{dXhkphlemma} and equation~\eqref{modifieddiracdefneq}, we have
\begin{align*}
\mdirac (f, X) \, & = \, \frac{1}{2} \st (df \wedge \ph) - \frac{3}{7} (d^* X) \ph + \frac{1}{2} \st (\curl X \wedge \ph) \\
& = \, \left(- \frac{3}{7} (d^* X) \ph, \frac{1}{2} \st \bigl( ( df + \curl X ) \wedge \ph \bigr) \right) \, \in \,  \Omega^3_1 \oplus \Omega^3_7,
\end{align*}
which is what we wanted to show.
\end{proof}

\begin{cor} \label{mdirackernelcor}
Suppose that $s = (f , X)$ lies in the kernel of $\mdirac$ or $\mdiracstar$. Then $\Delta f = 0$ and $\Delta X = 0$.
\end{cor}
\begin{proof}
This is immediate from Proposition~\ref{modifieddiracprop}, $\dirac^* = \dirac$ and $\dirac^2(f, X) = (\Delta f, \Delta X)$.
\end{proof}

\begin{cor} \label{do27cor}
Let $\mu = X \hk \ph = \st (X \wedge \ps) \in \Omega^2_7$. Then $\pi_1(d\mu) = 0$ if and only if $d^* X = 0$; and $\pi_7(d\mu) = 0$ if and only if $\curl X = 0$.
\end{cor}
\begin{proof}
In Lemma~\ref{dXhkphlemma}, we showed $\pi_1 (d \mu) = - \frac{3}{7} d^* X$ and $\pi_7 (d \mu) = \frac{1}{2} \st (\curl X \wedge \ph)$. The result follows since wedge product with $\ph$ is injective on $1$-forms.
\end{proof}

Notice that Corollary~\ref{do27cor} demonstrates a relationship between symmetries of $\ph$ and the kernels of the Dirac operators. Explicitly, if $\mathcal{L}_X \ph = d (X \hk \ph) = 0$ then $d^*X = 0$ and $\curl X = 0$, and thus $\dirac (0,X) = \mdirac (0,X) = 0$. Of course, if $\mathcal{L}_X \ph = 0$ then $X$ is a Killing vector field (that is, $\mathcal{L}_X g_{\ph} = 0$), but the converse is not necessarily true. We can now see precisely which Killing fields preserve $\ph$, which will be useful in studying the kernels of $\dirac$ and $\mdirac$. Recall that for any Killing field $X$, we always have $d^* X = 0$, which can be seen by taking the trace of $(\mathcal L_X g)_{ij} = \nabla_i X_j + \nabla_j X_i$.

\begin{prop}\label{Killing.prop}
A vector field $X$ on $(M,\ph)$ satisfies $\mathcal{L}_X \ph = 0$ if and only if $\mathcal{L}_X g_{\ph} = 0$ and $\curl X = 0$.
\end{prop}

\begin{proof}
From Corollary~\ref{do27cor} we know that $\pi_{1+7} (\mathcal{L}_X \ph) = \pi_{1+7} d (X \hk \ph) = 0$ if and only if $d^*X = 0$ and $\curl X = 0$, so it only remains to consider $\pi_{27} (\mathcal{L}_X \ph)$. Recall that $\Omega^3_1\oplus\Omega^3_{27}\cong C^{\infty}(S^2T^*M)$ using the map~\eqref{o327detailedeq}. Under this identification, $\pi_{1+27}(\mathcal{L}_X\ph) = \frac{1}{2} \mathcal{L}_X g_{\ph}$. This is explicitly derived in~\cite[Equation (4.7)]{Kflows}. See also~\cite[Lemma 9.3]{LotayWei}. The result now follows.
\end{proof}

\subsection{Some identities for \texorpdfstring{$2$}{2}-forms and \texorpdfstring{$3$}{3}-forms on \texorpdfstring{$\G$}{G2}~manifolds} \label{formsidentitiessec}

In this section we collect some results related to the decompositions of $2$-forms and $3$-forms in the torsion-free case. In some of the proofs in this section we use the local coordinate identities for $\G$~structures that can be found in~\cite{Kflows}. In particular we repeatedly use identities for contractions of $\ph$ with itself given in the appendix of~\cite{Kflows}.

Recall that an element $\eta \in \Omega^3_{27}$ corresponds uniquely to a symmetric traceless $2$-tensor $h$ on $M$.
\begin{defn} \label{divedefn}
The \emph{divergence} $\dive h$ of a symmetric $2$-tensor $h$ is the $1$-form given in local coordinates by
\begin{equation} \label{divedefneq}
(\dive h)_k \, = \, g^{pq} \nabla_p h_{qk}.
\end{equation}
This operation is formally the same as taking the divergence of a vector field to obtain a function, except that we still have a free index, so the resulting object is a $1$-form. It is also formally the same in local coordinates as $- d^* \beta$ when $\beta$ is a $2$-form.
\end{defn}

\begin{prop} \label{special3formprop}
Let $\ph$ be a torsion-free $\G$~structure and let $\zeta \in \Omega^3$ be written in the form $\zeta = f \ph + \st (X \wedge \ph) + \eta$ in terms of the decomposition $\Omega^3 = \Omega^3_1 \oplus \Omega^3_7 \oplus \Omega^3_{27}$, where $\eta \in \Omega^3_{27}$ corresponds to a traceless symmetric $2$-tensor $h$ by~\eqref{o327detailedeq}. Then we have
\begin{align} \label{pi1dzetaeq}
\pi_1 (d \zeta) \, & = \, \frac{4}{7} (d^* X) \ps, & & \\ \label{pi7dzetaeq}
\pi_7 (d \zeta) \, & = \, Y \wedge \ph, & & \text{ where $Y = df - \frac{1}{2} \curl X - \frac{1}{2} \dive h$}, \\ \label{pi7dstarzetaeq}
\pi_7 (d^* \zeta) \, & = \, \st (W \wedge \ph), & & \text{ where $W = - df + \frac{2}{3} \curl X - \frac{2}{3} \dive h$}.
\end{align}
\end{prop}
\begin{proof}
Note that $\st \zeta = f \ps + X \wedge \ph + \st \eta$. Since $\ph$ is torsion-free, we thus find that
\begin{equation} \label{special3formproptempeq1}
\begin{aligned}
d \zeta \, & = \, df \wedge \ph + d \st (X \wedge \ph) + d \eta, \\
d^* \zeta \, & = \, - \st d \st \zeta \, = \, - \st ( df \wedge \ps)  - \st ( d X \wedge \ph) + d^* \eta.
\end{aligned}
\end{equation}

We know that $\pi_1 (d \zeta) = \lambda \ps$ for some function $\lambda$. Thus we compute
\begin{align*}
7 \lambda \, & = \, \langle \lambda \ps, \ps \rangle \, = \, \langle \pi_1 (d \zeta), \ps \rangle \\
& = \, \langle d \zeta, \ps \rangle \, = \, \langle df \wedge \ph + d \st (X \wedge \ph) + d \eta, \ps \rangle.
\end{align*}
The first term on the right hand side vanishes because $df \wedge \ph$ is of type $\Omega^4_7$. The last term vanishes because $\langle d \eta , \ps \rangle \vol = d \eta \wedge \ph = d (\eta \wedge \ph) = 0$ by~\eqref{o327eq}. Thus we have
\begin{equation*}
7 \lambda \, = \, \langle d \st (X \wedge \ph), \ps \rangle \, = \, \langle \st d \st (X \wedge \ph), \ph \rangle \, = \, \langle d^* (X \wedge \ph), \ph \rangle
\end{equation*}
as $\st$ is an isometry and $d^* = \st d \st$ on $4$-forms. Computing in local coordinates we find
\begin{align*}
7 \lambda \, & = \, \frac{1}{6} ( d^* (X \wedge \ph) )_{ijk} \ph_{abc} g^{ia} g^{jb} g^{kc} \\
& = \, - \frac{1}{6} g^{pq} \nabla_p (X \wedge \ph)_{qijk} \ph_{abc} g^{ia} g^{jb} g^{kc} \\
& = \, - \frac{1}{6} g^{pq} \nabla_p (X_q \ph_{ijk} - X_i \ph_{qjk} - X_j \ph_{iqk} - X_k \ph_{ijq}) \ph_{abc} g^{ia} g^{jb} g^{kc} \\
& = \, -\frac{1}{6} g^{pq} ( (\nabla_p X_q) \ph_{ijk} - 3 (\nabla_p X_i) \ph_{qjk} ) \ph_{abc} g^{ia} g^{jb} g^{kc} \\
& = \, -\frac{1}{6} g^{pq} ( 42 \nabla_p X_q - 3 (\nabla_p X_i) g^{ia} (6 g_{qa}) ) \, = \, -\frac{1}{6}(42 - 18) g^{pq} \nabla_p X_q \, = \, 4 d^* X,
\end{align*}
which establishes~\eqref{pi1dzetaeq}.

To derive~\eqref{pi7dzetaeq} and~\eqref{pi7dstarzetaeq}, we will need to contract $\eta \in \Omega^3_{27}$ with $\ph$ on two indices. A short computation using~\eqref{o327detailedeq} gives
\begin{equation} \label{special27eq1}
h_{ia} = \frac{1}{4} \ph_{ijk} \eta_{abc} g^{jb} g^{kc}.
\end{equation}
We have $\pi_7 (d \zeta) = Y \wedge \ph$ for some $1$-form $Y$. Let $Z$ be an arbitrary $1$-form. Computing as before and using identities from~\cite[Lemma 2.2.2]{K1}, we find
\begin{align} \nonumber
4 \langle Y, Z \rangle \, & = \, \langle Y \wedge \ph, Z \wedge \ph \rangle \, = \, \langle \pi_7 (d \zeta), Z \wedge \ph \rangle \, = \langle d \zeta, Z \wedge \ph \rangle \\ \nonumber
& = \, \langle df \wedge \ph + d \st (X \wedge \ph) + d \eta, Z \wedge \ph \rangle \\ \label{pi7dzetatempeq1}
& = \, 4 \langle df, Z \rangle + \langle d \st (X \wedge \ph) , Z \wedge \ph \rangle + \langle d \eta, Z \wedge \ph \rangle.
\end{align}
We can compute the second term on the right hand side of~\eqref{pi7dzetatempeq1} in local coordinates as follows:
\begin{align*}
\langle d \st (X \wedge \ph) , Z \wedge \ph \rangle \, & = \, \langle \st d \st (X \wedge \ph) , \st (Z \wedge \ph) \rangle \\
& = \, \langle d^* (X \wedge \ph), - Z \hk \ps \rangle \, = \, - \frac{1}{6} ( d^*(X \wedge \ph) )_{ijk} (Z \hk \ps)_{abc} g^{ia} g^{jb} g^{kc} \\
& = \, \frac{1}{6} g^{pq} \nabla_p (X_q \ph_{ijk} - X_i \ph_{qjk} - X_j \ph_{iqk} - X_k \ph_{ijq}) Z^m \ps_{mabc} g^{ia} g^{jb} g^{kc} \\
& = \, \frac{1}{6} Z^m g^{pq} ( (\nabla_p X_q) \ph_{ijk} - 3 (\nabla_p X_i) \ph_{qjk} ) \psi_{mabc} g^{ia} g^{jb} g^{kc} \\
& = \,\frac{1}{6} Z^m ( 0 - 3 (\nabla_p X_i) (-4 \ph_{qma}) ) g^{pq} g^{ia} \, = \, -2 Z^m (\nabla_p X_i) \ph_{qam} g^{pq} g^{ia} \\
& \, = \, - 2 Z^m (\curl X)_m \, = \, -2 \langle \curl X, Z \rangle,
\end{align*}
where we have used~\eqref{curleq2} in the last line. Substituting this result into~\eqref{pi7dzetatempeq1} gives
\begin{equation} \label{pi7dzetatempeq2}
4 \langle Y, Z \rangle \, = \, 4 \langle df, Z \rangle - 2 \langle \curl X , Z \rangle + \langle d \eta, Z \wedge \ph \rangle.
\end{equation}
Again, we compute the last term above in local coordinates. We obtain
\begin{align*}
\langle d \eta, Z \wedge \ph \rangle \, & = \, \frac{1}{24} (d \eta)_{ijkl} (Z \wedge \ph)_{abcd} g^{ia} g^{jb} g^{kc} g^{ld} \\
& = \, \frac{1}{24} ( \nabla_i \eta_{jkl} - \nabla_j \eta_{ikl} - \nabla_k \eta_{jil} - \nabla_l \eta_{jki} ) (Z \wedge \ph)_{abcd} g^{ia} g^{jb} g^{kc} g^{ld} \\
& = \, \frac{4}{24} (\nabla_i \eta_{jkl}) (Z_a \ph_{bcd} - Z_b \ph_{acd} - Z_c \ph_{bad} - Z_d \ph_{bca} ) g^{ia} g^{jb} g^{kc} g^{ld} \\
& = \, \frac{1}{6} g^{ia} \bigl( \nabla_i ( \eta_{jkl} \ph_{bcd} g^{jb} g^{kc} g^{ld} ) Z_a - 3 \nabla_i (\eta_{jkl} \ph_{acd} g^{kc} g^{ld} ) g^{jb} Z_b \bigr).
\end{align*}
Now, using~\eqref{special27eq1} twice and the fact that $h$ is traceless and symmetric, we find
\begin{align*}
\langle d \eta, Z \wedge \ph \rangle \, & = \, \frac{1}{6} g^{ia} \bigl( \nabla_i ( 4 h_{jb} g^{jb} ) Z_a - 3 \nabla_i (4 h_{ja}) g^{jb} Z_b \bigr) \\
& = \, \frac{1}{6} g^{ia} (0 - 12 g^{jb} (\nabla_i h_{ja}) Z_b) \, = \, - 2 (\dive h)_j Z_b g^{jb} \, = \, - 2 \langle \dive h, Z \rangle.
\end{align*}
Since $Z$ is arbitrary, substituting the above expression into~\eqref{pi7dzetatempeq2} establishes~\eqref{pi7dzetaeq}.

Next, from~\eqref{special3formproptempeq1}, the decompositions~\eqref{o27eq} and~\eqref{o214eq}, and equation~\eqref{curllemmaeq}, we find
\begin{align} \nonumber
d^* \zeta \, & = \, - \st ( df \wedge \ps) + 2 \pi_7 (dX) - \pi_{14} (dX) + d^* \eta \\ \label{pi7dstarzetatempeq1}
& = \, - \st ( df \wedge \ps) + \frac{2}{3} \st \bigl( (\curl X) \wedge \ps \bigr) - \pi_{14} (dX) + d^* \eta
\end{align}
We have $\pi_7 (d^* \zeta) = \st (W \wedge \ps)$ for some $1$-form $W$. Let $Z$ be an arbitrary $1$-form. Computing again with identities from~\cite[Lemma 2.2.2]{K1}, using~\eqref{pi7dstarzetatempeq1} we find
\begin{align} \nonumber
3 \langle W, Z \rangle \, & = \, \langle \st (W \wedge \ps), \st (Z \wedge \ps) \rangle \, = \, \langle \pi_7 (d^* \zeta), \st (Z \wedge \ps) \rangle \, = \langle d^* \zeta, \st (Z \wedge \ps) \rangle \\ \nonumber
& = \, \langle - \st ( df \wedge \ps) + \frac{2}{3} \st \bigl( (\curl X) \wedge \ps \bigr) - \pi_{14} (dX) + d^* \eta
, \st (Z \wedge \ps) \rangle \\ \label{pi7dstarzetatempeq2}
& = \, -3 \langle df, Z \rangle + 2 \langle \curl X, Z \rangle + 0 + \langle d^* \eta, \st (Z \wedge \ps) \rangle.
\end{align}
As before, we compute the last term above in local coordinates, using~\eqref{special27eq1}. We obtain
\begin{align*}
\langle d^* \eta, \st (Z \wedge \ps) \rangle \, & = \, \langle d^* \eta, Z \hk \ph \rangle \, = \, \frac{1}{2} (d^* \eta)_{ij} (Z \hk \ph)_{ab} g^{ia} g^{jb} \\
& = \, -\frac{1}{2} g^{pq} (\nabla_p \eta_{qij}) Z^m \ph_{mab} g^{ia} g^{jb} \, = \, -\frac{1}{2} g^{pq} \bigl( \nabla_p (\eta_{qij} \ph_{mab} g^{ia} g^{jb} ) \bigr) Z^m \\ 
& = \, - \frac{1}{2} g^{pq} \bigl( \nabla_p (4 h_{qm}) \bigr) Z^m \, = \, -2 \langle \dive h, Z \rangle.
\end{align*}
Since $Z$ is arbitrary, substituting the above expression into~\eqref{pi7dstarzetatempeq2} establishes~\eqref{pi7dstarzetaeq}.
\end{proof}

\begin{rmk} \label{special27bryantcor}
If $\eta \in \Omega^3_{27}$, then (since $f = 0$ and $X = 0$), we conclude that $\pi_7 (d \eta) = 0$ if and only if $\pi_7 (d^* \eta) = 0$, because in this case both conditions are equivalent to $\dive h = 0$ by Proposition~\ref{special3formprop}. This fact was justified in~\cite{Br1} using representation theory.
\end{rmk}

\begin{cor} \label{special27cor}
Let $\ph$ be a torsion-free $\G$~structure and consider $\zeta = f \ph + \st (X \wedge \ph) + \eta$ as in Proposition~\ref{special3formprop}. If $d \zeta = 0$, then $d^* X = 0$ and $\pi_7 (d^* \zeta) = ( - \frac{7}{3} df + \frac{4}{3} \curl X ) \hk \ph$.
\end{cor}
\begin{proof}
This follows immediately from~\eqref{pi1dzetaeq},~\eqref{pi7dzetaeq}, and~\eqref{pi7dstarzetaeq}, by solving for $\dive h$.
\end{proof}

\begin{cor} \label{special27cor2}
Let $\ph$ be a torsion-free $\G$~structure and consider $\zeta = f \ph + \st (X \wedge \ph) + \eta$ as in Proposition~\ref{special3formprop}. If $d \zeta = 0$ and $\pi_7 (d^* \zeta) = 0$, then $\Delta f = 0$ and $\Delta X = 0$.
\end{cor}
\begin{proof}
From Corollary~\ref{special27cor} we have $d^* X = 0$ and $df = \frac{4}{7} \curl X$. Recall the relations between $d^*$, $\curl$, and $\Delta$ given in Remark~\ref{curlrmk}. Taking $d^*$ of both sides of $df = \frac{4}{7} \curl X$ gives $\Delta f = 0$, and taking curl of both sides and using $d^* X = 0$ gives $\Delta X = 0$.
\end{proof}

\begin{cor} \label{special27cor3}
Let $\ph$ be a torsion-free $\G$~structure and consider $\zeta = f \ph + \st (X \wedge \ph) + \eta$ as in Proposition~\ref{special3formprop}. Suppose that $d \zeta = 0$. If $df = 0$ and $\curl X = 0$, then $\pi_7 (d^* \zeta) = 0$.
\end{cor}
\begin{proof}
This follows from Proposition~\ref{special3formprop}, noting that the hypotheses imply that $\dive h = 0$.
\end{proof}

\begin{cor} \label{mdiracusefulcor}
Let $(h, Y) \in \Omega^0_1 \oplus \Omega^1_7$. Consider the modified Dirac operator $\mdirac$ of equation~\eqref{modifieddiracdefneq}. Then we have
\begin{equation} \label{mdiracusefuleq}
\pi_1 d \bigl( \mdirac(h,Y) \bigr) \, = \, \frac{2}{7} (\Delta h)\ps.
\end{equation}
\end{cor}
\begin{proof}
From the proof of Propostion~\ref{modifieddiracprop}, we have
\begin{equation*}
\mdirac (h, Y) \, = \, - \frac{3}{7}(d^* Y) \ph + \frac{1}{2} \st \bigl( (dh + \curl Y) \wedge \ph \bigr).
\end{equation*}
Also, using Proposition~\ref{special3formprop}, if $\gamma = \bigl(f \ph, \st (X \wedge \ph) \bigr) \in \Omega^3_1 \oplus \Omega^3_7$, then $\pi_1 (d \gamma) = \frac{4}{7} (d^*X) \ps$. Hence, taking $X = \frac{1}{2} (dh + \curl Y)$ and using  $d^* \curl(Y) = 0$, we deduce that $\pi_1 d \bigl( \mdirac(h,Y) \bigr)  = \frac{2}{7} (\Delta h) \ps$.
\end{proof}

To motivate the next proposition, consider the following situation. Let $f_t$ be a one parameter family of diffeomorphisms generated by a vector field $X$ on $M$. Then $\ph_t = f_t^* \ph$ is torsion-free for all $t$ for which $f_t$ is defined. Thus in particular, the dual $4$-form $\ps_t = \st_{\ph_t} \ph_t$ is closed. Since $\rest{\ddt}{t=0} \ph_t = \mathcal L_X \ph = d( X \hk \ph)$, differentiating the equation $d \ps_t  = 0$ at $t = 0$ and using~\eqref{dThetaeq0} shows that $d \bigl( \Lp d( X \hk \ph) \bigr) = 0$. Using the identities we have derived, we can actually give a direct proof of this result, which is instructive. It says that infinitesimal diffeomorphisms satisfy the linearized torsion-free equations.
\begin{prop} \label{infinitesimaldiffeosprop}
Let $\ph$ be a torsion-free $\G$~structure and define $\eta = d (X \hk \ph)$ for some vector field $X$. Then $d ( \Lp \eta) = 0$.
\end{prop}
\begin{proof}
From equations~\eqref{Lpeq} and~\eqref{dXhkpheq}, we find
\begin{align*}
\Lp \eta \, & = \, \frac{4}{3} \st \pi_1 \eta + \st \pi_7 \eta - \st \pi_{27} \eta \\
& = \, \frac{7}{3} \st \pi_1 \eta + 2 \st \pi_7 \eta - \st \eta \\
& = - (d^* X) \ps + (\curl X) \wedge \ph - \st \eta.
\end{align*}
Hence we have
\begin{equation} \label{infinitesimaldiffeotempeq}
d (\Lp \eta) \, = \, - (d d^* X) \wedge \ps + (d \curl X) \wedge \ph - d \st \eta.
\end{equation}
Using~\eqref{temp27eq2} and~\eqref{curleq}, the third term above can be rewritten as
\begin{align*}
- d \st \eta \, & = \, - d \st d (X \hk \ph) \, = \, - d \st d \st (X \wedge \ps) \\
& = \, d d^* (X \wedge \ps) \, = \, (\Delta - d^* d ) (X \wedge \ps) \\
& = \, (\Delta X) \wedge \ps - d^* \bigl( (dX) \wedge \ps \bigr) \\
& = \, (\Delta X) \wedge \ps - d^* (\st \curl X) \, = \, (\Delta X) \wedge \ps - \st (d \curl X).
\end{align*}
Substituting the above into~\eqref{infinitesimaldiffeotempeq} and using~\eqref{o27eq} and~\eqref{o214eq}, we obtain
\begin{align*}
d (\Lp \eta) \, & = \, - (d d^* X) \wedge \ps + \bigl( -2 \st \pi_7 (d \curl X) + \st \pi_{14} (d \curl X) \bigr) \\
& \qquad {} + (\Delta X) \wedge \ps - \bigl( \st \pi_7 (d \curl X) + \st \pi_{14} (d \curl X) \bigr) \\
& = \, (d^* d X) \wedge \ps - 3 \st \pi_7 (d \curl X).
\end{align*}
Applying equation~\eqref{curllemmaeq} to the last equation (for the vector field $\curl X$), we obtain
\begin{equation*}
d (\Lp \eta) \, = \, (d^* d X) \wedge \ps - (\curl \curl X) \wedge \ps.
\end{equation*}
The right hand side above vanishes by Remark~\ref{curlrmk}, completing the proof.
\end{proof}

The final result in this section concerns the operator $\pi_7 d^* d$ from $\Omega^2_7$ to itself and will be used many times in the rest of the paper.

\begin{prop} \label{specialellipticprop}
Let $\ph$ be a torsion-free $\G$~structure, and consider the operator $\pi_7 d^* d : \Omega^2_7 \to \Omega^2_7$. Under the identification $\Omega^2_7 \cong \Omega^1$, $\pi_7 d^* d$ corresponds to the operator $\mlap$, where $\mlap X = d d^* X + \frac{2}{3} d^* d X$, and is therefore \emph{elliptic}.
\end{prop}
\begin{proof}
Let $X \hk \ph \in \Omega^2_7$ for $X \in \Omega^1$, where as usual we use the metric $g$ to identify vector fields and $1$-forms. Then $\mlap X = Y$, where $Y \hk \ph = \pi_7 d^* d (X \hk \ph)$. We now compute in local coordinates:
\begin{align*}
(X \hk \ph)_{jk} \, & = \, X^m \ph_{mjk}, \\
\bigl( d (X \hk \ph) \bigr)_{ijk} \, & = \, (\nabla_i X^m) \ph_{mjk} + (\nabla_j X^m) \ph_{mki} + (\nabla_k X^m) \ph_{mij}, \\
\bigl( d^* d (X \hk \ph) \bigr)_{jk} \, & = \, - g^{pi} \nabla_p \bigl( d (X \hk \ph) \bigr)_{ijk} \\ & = \, - g^{pi} (\nabla_p \nabla_i X^m) \ph_{mjk} - g^{pi} (\nabla_p \nabla_j X^m) \ph_{mki} - g^{pi} (\nabla_p \nabla_k X^m) \ph_{mij} \\
& = \, (Y \hk \ph)_{jk} + \mu_{jk},
\end{align*}
for some $\mu \in \Omega^2_{14}$. Contracting both sides of the last equation above with $\ph$ on two indices, we find
\begin{align*}
6 \, Y_l \, & = \, \Bigl( - g^{pi} (\nabla_p \nabla_i X^m) \ph_{mjk} - g^{pi} (\nabla_p \nabla_j X^m) \ph_{mki} - g^{pi} (\nabla_p \nabla_k X^m) \ph_{mij} \Bigr) \ph_{lab} g^{ja} g^{kb} \\
& = \, - 6 \, g^{pi} (\nabla_p \nabla_i X_l) + 2 \, g^{pi} (\nabla_p \nabla_j X^m) (g_{ml} g_{ia} - g_{ma} g_{il} - \ps_{mila})g^{ja} \\
& = \, 6 \, (\Delta X)_l + 2 \, g^{pi} (\nabla_p \nabla_i X_l) - 2 \, (\nabla_l \nabla_j X^j) + 2 \, (\nabla_p \nabla_j X_n) g^{pi} g^{ja} g^{nm} \ps_{iaml} \\
& = \, 6 \, (\Delta X)_l - 2 (\Delta X)_l + 2 (d d^*X)_l + (\nabla_p \nabla_j X_n - \nabla_j \nabla_p X_n) g^{pi} g^{ja} g^{nm} \ps_{iaml}.
\end{align*}
Using the Ricci identities, the last term becomes $- R_{pjnc} X^c g^{pi} g^{ja} g^{nm} \ps_{iaml} = - 2 \, R_{mlnc} X^c g^{nm} = 0$, where we have used the fact that the Riemann tensor is in $\Omega^2_{14}$ with respect to its first and last pair of indices, and that the Ricci curvature vanishes. Thus we conclude that
\begin{equation*}
6 \, Y \, = \, 4 \Delta X + 2 d d^* X \, = \, 4 d^* d X + 6 d d^* X,  
\end{equation*}
which is what we wanted to show.
\end{proof}

\section{\texorpdfstring{$\mathbf{\G}$}{G2} conifolds} \label{conemanifoldssec}

In this section we discuss facts about $\G$~cones, and asymptotically conical (AC) and conically singular (CS) $\G$~manifolds that will be needed in the present paper. Any results stated in this section without proof can be found in~\cite[Section 2]{Kdesings}.

\subsection{\texorpdfstring{$\mathbf{\G}$}{G2} cones} \label{g2conessec}

Let $\Sigma^6$ be a compact, connected, smooth $6$-manifold.  An \SUth~structure on $\Sigma$ is described by a Riemannian metric $\gs$, an almost complex structure $\Js$ which is
orthogonal with respect to $\gs$, the associated $2$-form $\os(u,v) = \gs(\Js u,v)$ which is real and of type $(1,1)$ with respect to $\Js$, and a nowhere vanishing complex $(3,0)$-form $\Os $. The
two forms are related by the normalization condition
\begin{equation*}
\vols = \frac{1}{6} \os^3 = \frac{i}{8} \Os \wedge \bar \Os = \frac{1}{4} \real(\Os) \wedge \imag(\Os).
\end{equation*}

A manifold $\Sigma^6$ with $\SUth$ structure is called (strictly) \emph{nearly K\"ahler} if the following equations are satisfied:
\begin{equation} \label{grayeq}
\ds \os = - 3 \, \real(\Os), \qquad \qquad \ds \imag(\Os) = 2 \, \os^2.
\end{equation}
Such manifolds are also called \emph{Gray manifolds}. The Riemannian metric of a Gray manifold is always Einstein with \emph{positive} Einstein constant~\cite{PS1}.

\begin{defn} \label{g2conedefn}
Let $\Sigma^6$ be nearly K\"ahler. Then there exists a torsion-free $\G$~structure $(\phc, \psc, \gc)$ on $C = (0, \infty) \times \Sigma$ defined by
\begin{align*}
\phc & = r^3 \real(\Os) - r^2 dr \wedge \os, \\ \psc & = - r^3 dr \wedge \imag(\Os) - r^ 4\frac{\os^2}{2}, \\ \gc & = dr^2 + r^2 \gs,
\end{align*}
where $r$ is the coordinate on $(0, \infty)$. The space $C$ is a $\G$~\emph{cone}, and $\Sigma$ is called the \emph{link} of the cone. We choose the orientation on $C$ so that $\volc = r^{6} dr \wedge \vols$ is the volume form on $C$.
\end{defn}

It is known that for a Riemannian cone $C$ with holonomy contained in $\G$, the holonomy is either trivial, in which case $\Sigma$ is the standard round sphere $S^6$ and $C$ is the Euclidean $\R^7$, or else the holonomy is exactly equal to $\G$, in which case the link $\Sigma$ is nearly K\"ahler, but not equal to the round $S^6$. (See B\"ar~\cite{Bar} for more details.) We reiterate that for us, a $\G$~cone will always have holonomy exactly $\G$, thus we will exclude the case where the link is the round $S^6$.

\begin{rmk} \label{Obatarmk}
A theorem of Obata~\cite{Obata} states that on a compact Einstein $6$-manifold $\Sigma$ with positive scalar curvature $R$, the first nonzero eigenvalue of the Laplacian on functions is not less than $\frac{R}{5}$, with equality {if and only if} $\Sigma$ is isometric to the round $S^6$. The Einstein metric on the link $\Sigma$ of our $\G$~cones has been scaled so that $R = 30$ (see~\cite{MNS, MS}) and we always exclude the case of $\Sigma = S^6$, so if $\laps h = \mu h$ for some nonconstant $h \in C^{\infty}(\Sigma)$, then we must have $\mu > 6$. We will use this result repeatedly in what follows.
\end{rmk}

The next two results relate symmetries of $\Sigma$ to symmetries of $C$, and will be useful later.
\begin{prop} \label{Killing-cone-prop}
Let $C$ be a Riemannian cone, and let $X = r^{\lambda + 1} h \ddr + r^{\lambda} Y$ be a vector field on $C$, where $h$ is a function on $\Sigma$, and $Y$ is a vector field on $\Sigma$. Then $X$ is a Killing field for $\gc$ if and only if
\begin{itemize}
\item \emph{either} $\lambda = 0$, $h = 0$, and $Y$ is a Killing field for $\gs$ (so in this case $X = Y$),
\item \emph{or} $\lambda = -1$, $Y = \ds h$, and $\mathcal L_Y \gs = - 2 h \gs$. Moreover, in this case $\laps h = nh$, where $n = \dim \Sigma$.
\end{itemize}
Finally, if $C$ is a \emph{\G~cone}, then only the first case can occur.
\end{prop}
\begin{proof}
It is elementary to compute that
\begin{align*}
\mathcal L_X dr \, & = \, (\lambda + 1) r^{\lambda} h \, dr + r^{\lambda + 1} \ds h, \\
\mathcal L_{r^{\lambda + 1} h \ddr} \gs \, & = \, 0, \\
\mathcal L_{r^{\lambda} Y} \gs \, & = \, r^{\lambda} \gs + \lambda r^{\lambda - 1} ( dr \otimes Y + Y \otimes dr ).
\end{align*}
Using these results and $\gc = dr \otimes dr + r^2 \gs$, a short calculation gives
\begin{equation*}
\begin{aligned}
\mathcal L_X \gc \, & = \, 2 (\lambda + 1) r^{\lambda} h \, dr \otimes dr + r^{\lambda + 1} ( dr \otimes \ds h + \ds h \otimes dr ) \\
& \qquad {} + \lambda r^{\lambda + 1} ( dr \otimes Y + Y \otimes dr ) + 2 h r^{\lambda + 2} \gs + r^{\lambda + 2} \mathcal L_Y \gs.
\end{aligned}
\end{equation*}
Thus we deduce that $\mathcal L_X \gc = 0$ if and only if the following three equations are all satisfied:
\begin{equation*}
(\lambda + 1) h \, = \, 0, \qquad \qquad \ds h + \lambda Y \, = \, 0, \qquad \qquad 2 h \gs + \mathcal L_Y \gs \, = \, 0.
\end{equation*}
If $h = 0$, then the third equation says that $Y$ is a Killing field for $\gs$, and the second equation says $\lambda Y = 0$. Then, either $\lambda = 0$, which yields the first case, or $Y = 0$ in which case $X = 0$ so the value of $\lambda$ is undetermined and can be taken to be zero. On the other hand, if $h \neq 0$, then we must have $\lambda = -1$, which forces $Y = \ds h$ (where we have identified vector fields and $1$-forms on $\Sigma$ using the metric $\gs$), and $\mathcal L_Y \gs = - 2 h \gs$, yielding the second case. It follows easily from a computation in local coordinates that if $Y = \nabs h$ and $\mathcal L_Y \gs = - 2 h \gs$, then $\laps h = n h$.

Now suppose that $C$ is a $\G$~cone. Then $n = 6$, and the second case implies $\laps h = 6 h$. By Remark~\ref{Obatarmk}, we must have $h = 0$, which excludes the second case.
\end{proof}

\begin{prop} \label{NK-Killing-prop}
Let $\Sigma^6$ be nearly K\"ahler (with $\Sigma^6 \neq S^6$) and let $Y$ be a \emph{nonzero} Killing vector field on $\Sigma$, so $\mathcal L_Y \gs = 0$. Then the vector field $X = r^{\lambda - 1} Y$ is a symmetry of the $\G$~cone structure $\phc$ of Definition~\ref{g2conedefn}, in the sense that $\mathcal L_{r^{\lambda - 1} Y} \phc = 0$, if and only if $\lambda = 1$.
\end{prop}
\begin{proof}
In~\cite[Theorem 4.1]{MNS} it is shown that, for a nearly K\"ahler $6$-manifold that is not the round $S^6$, the metric determines the remaining objects $J$, $\omega$, and $\Omega$. In particular, since $\mathcal L_Y \gs = 0$, it follows from~\eqref{grayeq} that
\begin{equation} \label{graysymmetryeq}
\begin{aligned}
\mathcal L_Y \omega \, & = \, \ds (Y \hk \omega) + Y \hk (\ds \omega) \, = \, \ds (Y \hk \omega) - 3 Y \hk \real (\Omega) \, = \, 0, \\
\mathcal L_Y \real(\Omega) \, & = \, \ds (Y \hk \real(\Omega)) \, = \, 0.
\end{aligned}
\end{equation}
Now, using Definition~\ref{g2conedefn} and~\eqref{graysymmetryeq}, since $\dc \phc = 0$, we find that
\begin{align*}
\mathcal L_{r^{\lambda-1} Y} \phc \, & = \, \dc (r^{\lambda-1} Y \hk \phc) \\
& = \, \dc \big( r^{\lambda+1} dr \wedge (Y \hk \omega) + r^{\lambda+2} Y \hk \real(\Omega) \big) \\
& = \, r^{\lambda+1} dr \wedge \big( (\lambda+2) Y \hk \real(\Omega) - \ds (Y \hk \omega) \big) + 
r^{\lambda+2} \ds \big( Y \hk \real(\Omega) \bigr) \\
& = \, (\lambda-1) r^{\lambda+1} dr \wedge \bigl( Y \hk \real(\Omega) \bigr).
\end{align*}
It is easy to check that $Y \hk \real(\Omega) = 0$ if and only if $Y = 0$. Therefore, $\mathcal L_{r^{\lambda - 1} Y} \phc = 0$ if and only if $\lambda = 1$ as claimed.
\end{proof}

\begin{rmk} \label{NK-known-rmk}
There are five known simply connected, compact, nearly K\"ahler manifolds (other than the round $S^6$), and thus five known $\G$~cones with simply connected links. The three homogeneous examples are discussed in detail in B\"ar~\cite{Bar}, but we summarize them here. They are all obtained by taking the bi-invariant metric on a compact Lie group $G$ and descending this to the normal metric on $G/H$ for an appropriate Lie subgroup $H$. In particular, all these examples are homogeneous spaces, and there is a proof by Butruille~\cite{But} that these examples are the only compact, simply connected, \emph{homogeneous} nearly K\"ahler  $6$-manifolds.  The three homogeneous examples (as smooth manifolds) are: $\C \PR^3 \cong \mathrm{Sp}(2) / (\mathrm{Sp}(1) \times \mathrm{U}(1))$, the flag manifold $F_{1,2} \cong \SUth/T^2$, and $S^3 \times S^3 \cong (S^3 \times S^3 \times S^3) / S^3$ where we embed $S^3$ into $S^3 \times S^3 \times S^3$ as the diagonal subgroup. We note for later use some of the Betti numbers for these examples:
\begin{equation} \label{NKBettinumberseq}
\begin{aligned}
b^2 (\C \PR^3) & = 1, \qquad & b^3 (\C \PR^3) & = 0, \\
b^2 (F_{1,2}) & = 2, \qquad & b^3 (F_{1,2}) & = 0, \\
b^2 (S^3 \times S^3) & = 0, \qquad & b^3 (S^3 \times S^3) & = 2.
\end{aligned}
\end{equation}
In~\cite{PS1, PS2} Podest\`a--Spiro obtain some classification results about compact examples of cohomogeneity one. Furthermore, it is demonstrated in~\cite{FH} by Foscolo--Haskins that on $S^6$ and on $S^3\times S^3$ there is a cohomogeneity one nearly K\"ahler structure which is not homogeneous. Locally homogeneous examples that are finite quotients of $S^3 \times S^3$ are described in~\cite{CV} by Cort\'es--V\'asquez.
\end{rmk}

For any $t > 0$, we have a \emph{dilation} map $\mathbf{t} : C \to C$ defined by
\begin{equation*}
\mathbf{t} (\z) = \z, \qquad \qquad \mathbf{t} (r, \sigma) = (t r, \sigma).
\end{equation*}
It is easy to see that
\begin{equation*}
\begin{aligned}
\mathbf{t}^* (\phc) & = t^3 \, \phc, \qquad & \mathbf{t}^* (\psc) & = t^4 \, \psc, & \\ \mathbf{t}^* (\gc) & = t^2 \, \gc, \qquad & \mathbf{t}^* (\volc) & = t^7 \, \volc, &
\end{aligned}
\end{equation*}
and hence we say that the conical $\G$~structure is \emph{dilation-equivariant}. Since $\mathbf{t}^*g_C=t^2g_C$, we have
\begin{equation*}
|\mathbf{t}^*(\gamma)(r, \sigma)|_{\gc(r, \sigma)} \, = \, t^k | \gamma(tr, \sigma) |_{\gc(tr, \sigma)}
\end{equation*}
whenever $\gamma$ is a contravariant tensor of degree $k$.

Let $\alpha$ be a $(k-1)$-form on $\Sigma$ and $\beta$ be a $k$-form on $\Sigma$. Then we have
\begin{equation*}
|r^{k-1}dr \wedge \alpha + r^k \beta|^2_{\gc} \, = \, |\alpha|^2_{\gs} + |\beta|^2_{\gs}.
\end{equation*}
For this reason, we will always write a $k$-form on $C$ as $\gamma = r^{k-1}dr \wedge \alpha + r^k \beta$ for some $\alpha$ and $\beta$, which are forms on $\Sigma$ possibly depending on the parameter $r$. Note that if $\alpha$ and $\beta$ were independent of $r$, then $\gamma$ would be dilation-equivariant, as defined above.

\begin{defn} \label{homoformsdefn}
We say that a smooth $k$-form $\gamma$ on $C$ is \emph{homogeneous of order} $\lambda$ if
\begin{equation*}
\gamma = r^{\lambda} \left( r^{k-1}dr \wedge \alpha + r^k \beta \right)
\end{equation*}
where $\alpha$ and $\beta$ are forms on $\Sigma$, independent of $r$. Then
we see that
\begin{equation*}
| \gamma (tr, \sigma)|_{\gc(tr, \sigma)} = | t^{\lambda + k} \gamma (r, \sigma) |_{\gc(tr, \sigma)} =  t^{\lambda + k} t^{-k} | \gamma (r, \sigma) |_{\gc(r, \sigma)},
\end{equation*}
which we can write more concisely as
\begin{equation*}
\mathbf{t}^* |\gamma|_{\gc} = t^{\lambda} |\gamma|_{\gc},
\end{equation*}
so the function $|\gamma|_{\gc}$ on $C$ is homogeneous of order $\lambda$ in the variable $r$ in the usual sense.
\end{defn}

Let $\sts$, $\nabs$, $\ds$, $\dss$, and $\laps$ denote the Hodge star, Levi-Civita connection, exterior derivative, coderivative, and Hodge Laplacian on $\Sigma$, respectively. Similarly $\stc$, $\nabc$, $\dc$, $\dsc$, and $\lapc$ will denote the corresponding operators on the cone $C$.

For a homogenous $k$-form $\gamma = r^{\lambda} \left( r^{k-1}dr \wedge \alpha + r^k \beta \right)$ of order $\lambda$, it is trivial to calculate that:
\begin{equation} \label{homoddseq}
\begin{aligned}
\dc \gamma = \, & \, r^{\lambda + k - 1} dr \wedge ( (\lambda + k) \beta - \ds \alpha) + r^{\lambda + k} \ds \beta, \\ \dsc \gamma = \, & \, r^{\lambda + k - 3} dr \wedge ( -\dss \alpha) + r^{\lambda + k - 2} ( -(\lambda - k + 7) \alpha + \dss \beta ),
\end{aligned}
\end{equation}
\begin{equation} \label{homolapeq}
\begin{aligned} 
\lapc \gamma & \, = \, r^{\lambda + k - 3} dr \wedge \left( \laps \alpha  - (\lambda + k - 2)(\lambda - k + 7) \alpha - 2 \dss \beta \right) \\ & \qquad {} + r^{\lambda + k - 2} \left( \laps \beta  - (\lambda + k)(\lambda - k + 5) \beta - 2 \ds \alpha \right).
\end{aligned}
\end{equation}

The next lemma is a special case of~\cite[Lemma 2.12]{Kdesings}.
\begin{lemma} \label{exactformslemma}
Let $\gamma$ be a smooth \emph{closed} $3$-form on $C$. Suppose that either
\begin{align*}
\text{i) \, } | \gamma |_{\gc} & = O(r^{\lambda}) \text{ on } (0, \e) \times \Sigma, {\text \, for \, \, } \lambda > -3 \text{ or } \\ \text{ii) \, } | \gamma |_{\gc} & = O(r^{\lambda}) \text{ on } (R, \infty) \times \Sigma, {\text \, for \, \, } \lambda < -3.
\end{align*}
for some small $\e$ or some large $R$. Then for each case respectively we have that
\begin{align*}
\text{i) \, } \gamma & = d\zeta \text{ for some $2$-form } \zeta \text{ on } (0, \e) \times \Sigma, \text{ or } \\ \text{ii) \, } \gamma & = d\zeta \text{ for some $2$-form } \zeta \text{ on } (R, \infty) \times \Sigma.
\end{align*}
\end{lemma}

We will need to consider the possible order $\lambda$ of a homogeneous $k$-form $\gamma_k$ on a cone $C$ which is in the kernel of $\lapc$, or of a mixed degree form $\gamma = \sum_{k=0}^7 \gamma_k$ which is in the kernel of $\dc + \dsc$.
\begin{prop} \label{homolapexcludeprop}
Let $\gamma$ be a homogeneous $k$-form of order $\lambda$ which is harmonic on the cone: $\lapc \gamma = 0$. Then we have:
\begin{align} \label{k07lapexcludeeq}
& \text{For } k = 0, 7, \qquad \qquad \gamma = 0 \, \, \text{if } \, \lambda \in (-5,0), & \\ \label{k16lapexcludeeq} & \text{For } k = 1, 6, \qquad \qquad \gamma = 0 \, \, \text{if } \, \lambda \in (-4,-1), & \\ \label{k25lapexcludeeq} & \text{For } k = 2, 5, \qquad \qquad \gamma = 0 \, \, \text{if } \, \lambda \in (-3,-2). &
\end{align}
\end{prop}
\begin{prop} \label{homoddstarexcludeprop}
Suppose that $\gamma$ is a homogeneous $k$-form of order $\lambda$ which is closed and coclosed: $\dc \gamma = 0$ and $\dsc \gamma = 0$. Then we have:
\begin{align} \label{k07ddstarexcludeeq}
& \text{For } k = 0, 7, \qquad \qquad \gamma = 0 \, \, \text{if } \, \lambda \in (-7,0), & \\ \label{k16ddstarexcludeeq} & \text{For } k = 1, 6, \qquad \qquad \gamma = 0 \, \, \text{if } \, \lambda \in (-6,-1), & \\ \label{k25ddstarexcludeeq} & \text{For } k = 2, 5, \qquad \qquad \gamma = 0 \, \, \text{if } \, \lambda \in (-5,-2), & \\ \label{k34ddstarexcludeeq} & \text{For } k = 3, 4, \qquad \qquad \gamma = 0 \, \, \text{if } \, \lambda \in (-4,-3). &
\end{align}
\end{prop}

Propositions~\ref{homolapexcludeprop} and~\ref{homoddstarexcludeprop} use only the fact that we have a $7$-dimensional Riemannian cone. However, the fact that the link $\Sigma^6$ is a compact Einstein manifold of positive scalar curvature allows us to slightly extend the results of these propositions, for functions and $1$-forms. This is the content of the next several results. These extended ranges of excluded rates for harmonic functions and $1$-forms are used, for example, in Theorem~\ref{infinitesimalslicethm} to establish that theorem in the AC case all the way to $\nu \leq -1$, rather than just $\nu < -2$.

\begin{prop} \label{excludeextensionprop}
Let $f$ be a harmonic function on a $\G$~cone, homogeneous of order $\lambda$. Then
\begin{equation*}
f \, = \begin{cases} \, 0 & \text{if $\lambda \in [-6, 1] \setminus \{-5, 0\}$}, \\
\, K & \text{if $\lambda = 0$}, \\
\, K r^{-5} & \text{if $\lambda = -5$}, \end{cases}
\end{equation*}
where $K$ is a constant. Let $\omega$ be a harmonic $1$-form on a $\G$~cone, homogeneous of order $\lambda$. Then
\begin{equation*}
\omega \, = \, 0, \qquad \text{if $\lambda \in [-5,0]$}.
\end{equation*}
\end{prop}
\begin{proof}
As in Definition~\ref{homoformsdefn}, we can write $f = r^{\lambda} \beta$ for some function $\beta$ on $\Sigma$, independent of $r$. From~\eqref{homolapeq} we have $\lapc f = 0$ if and only if $\laps \beta = \lambda(\lambda + 5) \beta$. Applying the result of Obata from Remark~\ref{Obatarmk} we can say that if $f$ is nonzero and $\mu = \lambda(\lambda + 5) \neq 0$, then $\mu > 6$. Thus if $\lambda \neq -5,0$, then for $f$ to be nonzero we must have $\lambda(\lambda + 5) > 6$ which easily implies that $\lambda > 1$ or $\lambda < -6$. This proves the result for functions.

Now consider a $1$-form $\omega = r^{\lambda} (dr \wedge \alpha + r \beta)$, where $\alpha$ is a function on $\Sigma$ and $\beta$ is a $1$-form on $\Sigma$. From~\eqref{homolapeq} we find that $\lapc \omega = 0$ if and only if
\begin{align} \label{excludeextensiontempeq}
\laps \alpha \, & = \, (\lambda - 1)(\lambda + 6) \alpha + 2 \dss \beta, \\  \label{excludeextensiontempeq2}
\laps \beta \, & = \, (\lambda + 1)(\lambda + 4) \beta + 2 \ds \alpha.
\end{align}
Using~\eqref{excludeextensiontempeq} and~\eqref{excludeextensiontempeq2}, some easy computation yields
\begin{align} \label{excludeextensiontempeq3}
\laps (\dss \beta + (\lambda - 1) \alpha) \, & = \, (\lambda + 1)(\lambda + 6) (\dss \beta + (\lambda - 1) \alpha), \\ \label{excludeextensiontempeq4}
\laps (\dss \beta - (\lambda + 6) \alpha) \, & = \, (\lambda - 1)(\lambda + 4) (\dss \beta - (\lambda + 6) \alpha).
\end{align}
Consider therefore the two functions $h = \dss \beta + (\lambda - 1) \alpha$ and $h' = \dss \beta - (\lambda + 6) \alpha$. We have $\laps h = (\lambda + 1) (\lambda + 6) h$ and $\laps h' = (\lambda - 1) (\lambda + 4) h'$. Thus, by the Obata result of Remark~\ref{Obatarmk}, we know that $h$ is constant if $(\lambda + 1) (\lambda + 6) \leq 6$, which corresponds to $\lambda \in [-7, 0]$, and $h'$ is constant if $(\lambda - 1) (\lambda + 4) \leq 6$, which corresponds to $\lambda \in [-5, 2]$. Hence, for $\lambda \in [-5, 0]$, we conclude that both $h$ and $h'$ are constant. Thus, except possibly when $\lambda = -\frac{5}{2}$, we find that $\alpha$ is also constant because $h - h' = (2 \lambda + 5) \alpha$. But if $\lambda = - \frac{5}{2}$, then $(\lambda+1)(\lambda+6)<0$ and $(\lambda-1)(\lambda+4)<0$, so we have $h = h' = \dss \beta - \frac{7}{2} \alpha = 0$. Substituting $\lambda = -\frac{5}{2}$ and $\dss \beta = \frac{7}{2} \alpha$ into~\eqref{excludeextensiontempeq} yields $\laps \alpha = -\frac{21}{4} \alpha$ and thus $\alpha = 0$ in this case.

Since $\ds \alpha = 0$ for $\lambda \in [-5, 0]$, equation~\eqref{excludeextensiontempeq2} becomes
\begin{equation*}
\laps \beta \, = \, (\lambda + 1)(\lambda + 4) \beta,
\end{equation*}
and $(\lambda +1)(\lambda + 4) \leq 4$. Recall from~\cite{MNS, MS} that for a Gray manifold, we have $\mathrm{Ric}_{\scriptscriptstyle{\Sigma}} = 5 \gs$. Hence the Bochner formula gives
\begin{equation*}
\langle \laps \beta, \beta \rangle \, = \, \langle \nabs^* \nabs \beta, \beta \rangle + \mathrm{Ric}_{\scriptscriptstyle{\Sigma}} (\beta, \beta) \, = \, \langle \nabs^* \nabs \beta, \beta \rangle + 5 | \beta |^2.
\end{equation*}
Integrating the above equation over $\Sigma$, we find that if $\laps \beta = \mu \beta$, then we must have $\mu \geq 5$ for nonzero $\beta$. Thus if $\lambda \in [-5, 0]$, we must have $\beta = 0$.  Equation~\eqref{excludeextensiontempeq} with $\beta = 0$ then gives $\laps \alpha = (\lambda-1)(\lambda+6) \alpha$, so if $\lambda \in [-5,0]$ we obtain $\alpha = 0$. Therefore $\omega = 0$ for all $\lambda \in [-5, 0]$.
\end{proof}

In fact, we can say a little bit more about homogeneous harmonic $1$-forms on a $\G$~cone. We will need an additional tool, given by the following result.

\begin{lemma} \label{BochnerKillinglemma}
Suppose that $X$ is a \emph{coclosed} $1$-form on $\Sigma$ such that $\laps X = \mu X$.
\begin{itemize}
\item if $\mu < 10$ then $X = 0$,
\item if $\mu = 10$, then $X$ is metric dual to a Killing field.
\end{itemize}
\end{lemma}
\begin{proof}
Consider the operation $\dive_{\scriptscriptstyle{\Sigma}} : S^2 (T^* \Sigma) \to \Omega^1(\Sigma)$ defined in~\eqref{divedefneq}. It is easy to check that the formal adjoint $\dive^*_{\scriptscriptstyle{\Sigma}} : \Omega^1(\Sigma) \to S^2 (T^* \Sigma)$ of this map is given in local coordinates by
\begin{equation*}
(\dive_{\scriptscriptstyle{\Sigma}}^* X)_{ij} \, = \, - \frac{1}{2} ( (\nabla_{\scriptscriptstyle{\Sigma}})_i X_j + (\nabla_{\scriptscriptstyle{\Sigma}})_j X_i ) \, = \, - \frac{1}{2} (\mathcal L_X \gs)_{ij}.
\end{equation*}
Using this map, together with the Bochner formula on $1$-forms, a short computation in local coordinates yields the identity
\begin{equation*}
\laps X \, = \, 2 \dive_{\scriptscriptstyle{\Sigma}} \dive_{\scriptscriptstyle{\Sigma}}^* X + 2 \mathrm{Ric}_{\scriptscriptstyle{\Sigma}}(X) - \ds \dss X.
\end{equation*}
Since $\mathrm{Ric}_{\scriptscriptstyle{\Sigma}} = \frac{R}{6} \gs = 5\gs$, the second term above is just $10 X$. Suppose now that $X$ is a \emph{coclosed} $1$-form such that $\laps X = \mu X$. Taking the inner product on both sides with $X$ and integrating over $\Sigma$, we find that
\begin{equation*}
\mu || X ||^2 \, = \, 2 ||\dive_{\scriptscriptstyle{\Sigma}}^* X ||^2 + 10 || X ||^2
\end{equation*}
which immediately implies the result.
\end{proof}

\begin{prop} \label{excludeextensionrefinedprop}
The sets of homogeneous harmonic $1$-forms on a $\G$~cone $C$ of rate $\lambda \in (-6, -5)$ and those of rate $\lambda \in (0,1)$ are both in one-to-one correspondence with the set of scalar eigenfunctions of $\laps$ with eigenvalues in $(6, 14)$. Moreover, the homogeneous harmonic $1$-forms on $C$ of rate $\lambda \in (0,1)$ are \emph{exact and coclosed}. Explicitly, those with rate $\lambda \in (0,1)$ are of the form
\begin{equation*}
\frac{1}{\lambda + 1} \dc (r^{\lambda + 1} \alpha)
\end{equation*}
where $\alpha$ is a function on $\Sigma$ satisfying $\laps \alpha = (\lambda + 1)(\lambda + 6) \alpha$. Finally, the homogeneous harmonic $1$-forms with rate $\lambda = -6$ or $\lambda = 1$ are given by
\begin{equation*}
r^{\lambda} dr \wedge \alpha + \frac{1}{2} r^{\lambda + 1} \ds \alpha - \frac{1}{2} r^{\lambda + 1} Y
\end{equation*}
where $\alpha = K + f$ for some constant $K$ and a function $f$ on $\Sigma$ such that $\laps f = 14 f$, and $Y$ is the dual $1$-form to a Killing vector field on $\Sigma$. In particular, in the $\lambda = 1$ case this simplifies to
\begin{equation*}
\dc \left( \frac{r^2}{2} \alpha \right) - \frac{r^2}{2} Y.
\end{equation*}
\end{prop}
\begin{proof}
As in the proof of Proposition~\ref{excludeextensionprop}, we write $\omega = r^{\lambda} (dr \wedge \alpha + r \beta)$, where $\alpha$ is a function on $\Sigma$ and $\beta$ is a $1$-form on $\Sigma$ satisfying both~\eqref{excludeextensiontempeq} and~\eqref{excludeextensiontempeq2}. Let $c$ be a constant, and consider the $1$-form $\rho = \ds \alpha + c \beta$ on the link $\Sigma$. We seek a value of $c$ so that $\laps (\dss \rho) = \mu (\dss \rho)$ for some $\mu$. A straightforward computation reveals exactly two values of $c$ that work, namely $c = \lambda + 4$ and $c = - (\lambda + 1)$. In these two cases we have
\begin{align*}
& \text{ when $\rho = \ds \alpha + (\lambda + 4) \beta$, } & \laps (\dss \rho) \, & = \, (\lambda + 1)(\lambda + 6) (\dss \rho), \\
& \text{ when $\rho = \ds \alpha - (\lambda + 1) \beta$, } & \laps (\dss \rho) \, & = \, (\lambda - 1)(\lambda + 4) (\dss \rho).
\end{align*}
Thus, since $\laps (\dss \rho) = \mu (\dss \rho)$ implies $\dss \rho = 0$ if $\mu \leq 0$ (as $\dss\rho$ is coexact), we deduce that
\begin{equation} \label{lapsrhotempeq1}
\begin{aligned}
& \text{ when $\lambda \in [-6, -1]$, } & \rho \, & = \, \ds \alpha + (\lambda + 4) \beta \, \text{ is \emph{coclosed} on $\Sigma$}, \\
& \text{ when $\lambda \in [-4, 1]$, } & \rho \, & = \, \ds \alpha - (\lambda + 1) \beta \, \text{ is \emph{coclosed} on $\Sigma$}.
\end{aligned}
\end{equation}
In each case above we therefore have $\dss \rho = \dss \ds \alpha + c \dss \beta = \laps \alpha + c \dss \beta = 0$. If we assume that both $c$ and $c+2$ are nonzero, then using $\dss \beta = - \frac{1}{c} \laps \alpha$ we can compute from~\eqref{excludeextensiontempeq} and~\eqref{excludeextensiontempeq2} that
\begin{align*}
\laps (\ds \alpha) \, & = \, \frac{c}{c+2} (\lambda - 1) (\lambda + 6) (\ds \alpha), \\
\laps \beta \, & = \, (\lambda + 1)(\lambda + 4) \beta + 2 \ds \alpha,
\end{align*}
from which it follows that
\begin{equation*}
\laps \rho \, = \, \laps( \ds \alpha + c \beta ) \, = \, \left( \frac{c}{c+2} (\lambda - 1) (\lambda + 6) + 2 c \right) \ds \alpha + (\lambda + 1)(\lambda + 4) c \beta.
\end{equation*}
Somewhat remarkably, in \emph{both} of the cases $c = \lambda + 4$ and $c = -(\lambda + 1)$, the above expression reduces to
\begin{equation} \label{lapsrhotempeq2}
\laps \rho \, = \, (\lambda + 1)(\lambda + 4) \rho.
\end{equation}
Hence, by Lemma~\ref{BochnerKillinglemma}, if $(\lambda + 1)(\lambda + 4) < 10$, then $\rho = \ds \alpha + c \beta = 0$ in the two cases above. This inequality is equivalent to $\lambda \in (-6, 1)$. Note that the derivation of~\eqref{lapsrhotempeq2} breaks down in either case when $c = 0$ or $c + 2 = 0$. These correspond to $\lambda = -6, -4$ in the range $[-6,-1]$ and $\lambda = -1, 1$ in the range $[-4,1]$. However, we already know from Proposition~\ref{excludeextensionprop} that $\omega = 0$ if $\lambda \in [-5, 0]$, so the problems at rates $-4$ and $-1$ are irrelevant. Thus, the new information we have gained so far is that
\begin{align*}
& \text{ if $\lambda \in (-6, -5)$, } & & \text{then } \, \ds \alpha + (\lambda + 4) \beta \, = \, 0, \\
& \text{ if $\lambda \in (0, 1)$, } & & \text{then } \, \ds \alpha - (\lambda + 1) \beta \, = \, 0.
\end{align*}
Thus in both of the above cases $\beta$ is uniquely determined from $\alpha$. Moreover, substituting $\beta = - \frac{1}{c} \ds \alpha$ into~\eqref{excludeextensiontempeq} in each case yields
\begin{align*}
& \text{ if $\lambda \in (-6, -5)$, } & & \text{then } \, \laps \alpha \, = \, (\lambda - 1) (\lambda + 4) \alpha, \\
& \text{ if $\lambda \in (0, 1)$, } & & \text{then } \, \laps \alpha \, = \, (\lambda + 1) (\lambda + 6) \alpha.
\end{align*}
We thus deduce that the homogeneous harmonic $1$-forms of rate $\lambda \in (-6, -5)$ or of rate $\lambda \in (0,1)$ are in one-to-one correspondence with scalar eigenfunctions of $\laps$ with eigenvalues in $(6, 14)$.

Now, if $\lambda \in (0,1)$ then
\begin{equation*}
\omega = r^{\lambda} \alpha dr + \frac{1}{\lambda+1} r^{\lambda+1} \ds \alpha = \frac{1}{\lambda+1} \dc (r^{\lambda+1} \alpha),
\end{equation*}
which is closed (as it is exact) and coclosed, as claimed.

Finally, we consider the cases $\lambda = -6$ and $\lambda = 1$. In both cases, from equation~\eqref{lapsrhotempeq1} we deduce that $\rho = \ds \alpha - 2 \beta$ is coclosed. The derivation above of equation~\eqref{lapsrhotempeq2} does not work here, but we can argue directly as follows. Again, in both cases, equations~\eqref{excludeextensiontempeq3} and~\eqref{excludeextensiontempeq4} give  $\laps (\dss \beta) = 14 (\dss \beta)$ and $\laps ( \dss \beta - 7 \alpha ) = 0$, so in particular $\dss \beta - 7 \alpha$ is constant, and thus $2 \ds \dss \beta = 14 \ds \alpha$. Now~\eqref{excludeextensiontempeq} and~\eqref{excludeextensiontempeq2} give $\laps \alpha = 2 \dss \beta$ and $\laps \beta = 10 \beta + 2 \ds \alpha$, and thus
\begin{align*}
\laps \rho \, & = \, \laps \ds \alpha - 2 \laps \beta \, = \, 2 \ds \dss \beta - 20 \beta - 4 \ds \alpha \\
& = \, 10 \ds \alpha - 20 \beta \, = \, 10 \rho
\end{align*}
as before. So by Lemma~\ref{BochnerKillinglemma}, the vector field metric dual to $\rho$ is a Killing field. Then $\beta = \frac{1}{2} (\ds \alpha - \rho)$, so $\omega = r^{\lambda} ( dr \wedge \alpha + r \beta) = r^{\lambda} dr \wedge \alpha + \frac{1}{2} r^{\lambda + 1} (\ds \alpha - \rho)$. Moreover, $\laps (\dss \beta) = 14 (\dss \beta)$ and $\laps \alpha = 2 \dss \beta$ then yield $\laps (\laps \alpha) = 14 \laps \alpha$. Hence $\alpha = K + f$ where $\laps f = 14 f$.
\end{proof}

Next, we derive a similar result for the ``modified Laplacian'' $\mlapc = \dc \dsc + \frac{2}{3} \dsc \dc$ that will be needed later. For technical reasons we can only deal with the interval $[-5,1]$ instead of the full $[-6,1]$ but this will suffice for our purposes.

\begin{prop} \label{excludeextensionprop2}
Let $\omega$ be a $1$-form on a $\G$~cone, homogeneous of order $\lambda$, satisfying
\begin{equation*}
\mlapc \omega \, = \, \dc \dsc \omega + \frac{2}{3} \dsc \dc \omega \, = \, 0.
\end{equation*}
Then
\begin{equation*}
\omega \, = \, 0,\quad \text{if $\lambda \in [-5, 0]$.}
\end{equation*}
Moreover, for $\lambda \in (0,1)$, there is a one-to-one correspondence  between  homogeneous $1$-forms of order $\lambda$ in $\ker {\mlapc}$ and scalar eigenfunctions of $\laps$ with eigenvalue in $(6,14)$, and the homogeneous $1$-forms of order $\lambda \in (0,1)$ in $\ker {\mlapc}$ are all \emph{exact and coclosed}. Finally, the homogeneous  $1$-forms with rate $\lambda = 1$ in the kernel of $\mlapc$ are given by
\begin{equation*}
\dc \left( \frac{r^2}{2} \alpha \right) - \frac{r^2}{2} Y
\end{equation*}
where $\alpha = K + f$ for some constant $K$ and a function $f$ on $\Sigma$ such that $\laps f = 14 f$, and $Y$ is the dual $1$-form to a Killing vector field on $\Sigma$. 
\end{prop}
\begin{proof}
The proof is analogous to the proofs of Propositions~\ref{excludeextensionprop} and~\ref{excludeextensionrefinedprop}. First, it is easy to check that if $\dc \dsc \omega + \frac{2}{3} \dsc \dc \omega = 0$, where $\omega = r^{\lambda} (dr \wedge \alpha + r \beta)$, then
\begin{align} \label{excludeextension2tempeq}
\mlaps \alpha \, & = \, \frac{2}{3} \laps \alpha \, = \, (\lambda - 1)(\lambda + 6) \alpha - \frac{1}{3}(\lambda - 5) \, \dss \beta, \\  \label{excludeextension2tempeq2}
\mlaps \beta \, & = \, \ds \dss \beta + \frac{2}{3} \dss \ds \beta \, = \, \frac{2}{3} (\lambda + 1)(\lambda + 4) \beta + \frac{1}{3}(\lambda + 10) \, \ds \alpha.
\end{align}
Using~\eqref{excludeextension2tempeq} and~\eqref{excludeextension2tempeq2}, some computation yields
\begin{align} \label{excludeextension2tempeq3}
\laps \bigl( (\lambda - 5) \dss \beta - (\lambda - 1) (\lambda + 10) \alpha \bigr) \, & = \, (\lambda + 1)(\lambda + 6) \bigl( (\lambda - 5) \dss \beta - (\lambda - 1) (\lambda + 10) \alpha \bigr), \\
\label{excludeextension2tempeq4}
\laps (\dss \beta - (\lambda + 6) \alpha) \, & = \, (\lambda - 1)(\lambda + 4) (\dss \beta - (\lambda + 6) \alpha).
\end{align}
Note that~\eqref{excludeextension2tempeq4} is identical to~\eqref{excludeextensiontempeq4} but~\eqref{excludeextension2tempeq3} is quite different from~\eqref{excludeextensiontempeq3}. But the method of proof of Proposition~\ref{excludeextensionprop} is still valid, with $h' = \dss \beta - (\lambda + 6) \alpha$ as before but now with
\begin{equation*}
h = (\lambda - 5) \dss \beta - (\lambda - 1) (\lambda + 10) \alpha.
\end{equation*}
As before, for $\lambda \in [-5, 0]$, we find that both $h$ and $h'$ are constant, hence so is $\alpha$, except possibly when $\lambda = - \frac{5}{2}$, because
\begin{equation*}
h - (\lambda - 5)h' = -4 (2 \lambda + 5) \alpha.
\end{equation*}
As before, if $\lambda = - \frac{5}{2}$, then both $h = 0$ and $h' = 0$, and it is easy to check that both conditions are equivalent to $\dss \beta - \frac{7}{2} \alpha = 0$. Substituting $\lambda = -\frac{5}{2}$ and $\dss \beta = \frac{7}{2} \alpha$ into~\eqref{excludeextension2tempeq} yields $\laps \alpha = -\frac{21}{4} \alpha$ and thus $\alpha = 0$ in this case. Because $\ds \alpha = 0$ for all $\lambda \in [-5, 0]$, equation~\eqref{excludeextension2tempeq2} now becomes $\laps \beta = \frac{2}{3}(\lambda + 1)(\lambda + 4) \beta$, and $(\lambda +1)(\lambda + 4) \leq 4$. Exactly as in the proof of Proposition~\ref{excludeextensionprop}, we deduce that $\beta = 0$. Now equation~\eqref{excludeextension2tempeq} with $\beta = 0$ then gives $\laps \alpha = \frac{3}{2}(\lambda-1)(\lambda+6) \alpha$, so if $\lambda \in [-5,0]$ we obtain $\alpha = 0$. Therefore $\omega = 0$ for all $\lambda \in [-5, 0]$.

For the second part of the proposition, we proceed as in the proof of Proposition~\ref{excludeextensionrefinedprop}. Let $c_1$ and $c_2$ be constants, and consider the $1$-form $\rho = c_1 \ds \alpha + c_2 \beta$ on the link $\Sigma$. We seek values of $c_1$ and $c_2$ so that $\mlaps (\dss \rho) = \mu (\dss \rho)$ for some $\mu$, where $\mlaps = \dss \ds + \frac{2}{3} \dss \ds$ is the ``modified Laplacian'' on $\Sigma$. A tedious computation reveals exactly two solutions, namely
\begin{align*}
& \text{ when $\rho = (\lambda + 10) \ds \alpha - (\lambda + 4) (\lambda - 5) \beta$, } & \laps (\dss \rho) \, & = \, (\lambda^2 + 7 \lambda - 6) (\dss \rho), \\
& \text{ when $\rho = \ds \alpha - (\lambda + 1) \beta$, } & \laps (\dss \rho) \, & = \, (\lambda - 1)(\lambda + 4) (\dss \rho).
\end{align*}
We will only make use of the second equation above, which is identical to the analogous expression in Proposition~\ref{excludeextensionrefinedprop}. As before, we deduce that
\begin{equation} \label{lapsrho2tempeq1}
\text{ when $\lambda \in [-4, 1]$, } \qquad  \rho \, = \, \ds \alpha - (\lambda + 1) \beta \, \text{ is \emph{coclosed} on $\Sigma$}.
\end{equation}
In this range of rates we therefore have $\dss \rho = \dss \ds \alpha - (\lambda + 1) \dss \beta = \laps \alpha - (\lambda + 1) \dss \beta = 0$. As in Proposition~\ref{excludeextensionrefinedprop}, if we assume that both $c = -(\lambda + 1)$ and $2c - (\lambda - 5) = -3 (\lambda - 1)$ are nonzero, then using $\ds \beta = - \frac{1}{c} \laps \alpha$ we can compute from~\eqref{excludeextension2tempeq} and~\eqref{excludeextension2tempeq2} that
\begin{align*}
\laps (\ds \alpha) \, & = \, \frac{3c}{2c - (\lambda - 5)} (\lambda - 1) (\lambda + 6) (\ds \alpha), \\
\laps \beta \, & = \, (\lambda + 1)(\lambda + 4) \beta + \frac{1}{2} (\lambda + 10) (\ds \alpha) +  \frac{\frac{3}{2}c}{2c - (\lambda - 5)} (\lambda - 1) (\lambda + 6) (\ds \alpha),
\end{align*}
from which it follows that
\begin{equation} \label{lapsrho2tempeq2}
\laps \rho \, = \, (\lambda + 1)(\lambda + 4) \rho
\end{equation}
exactly as in~\eqref{lapsrhotempeq2} in the proof of Proposition~\ref{excludeextensionrefinedprop}. Hence, again by Lemma~\ref{BochnerKillinglemma}, if $(\lambda + 1)(\lambda + 4) < 10$, then $\rho = \ds \alpha + c \beta = 0$. This inequality is equivalent to $\lambda \in (-6, 1)$. Note that the derivation of~\eqref{lapsrho2tempeq2} breaks down when $\lambda = -1, 1$ in the range $[-4,1]$. However, we already established in the current proof that $\omega = 0$ if $\lambda \in [-5, 0]$, so the problem at rate $-1$ is irrelevant. Thus, the new information we have gained is that
\begin{align*}
& \text{ if $\lambda \in (0, 1)$, } & & \text{then } \, \ds \alpha - (\lambda + 1) \beta \, = \, 0.
\end{align*}
Hence $\beta$ is uniquely determined from $\alpha$. Moreover, substituting $\beta = \frac{1}{\lambda + 1} \ds \alpha$ into~\eqref{excludeextension2tempeq} yields
\begin{align*}
& \text{ if $\lambda \in (0, 1)$, } & & \text{then } \, \laps \alpha \, = \, (\lambda + 1) (\lambda + 6) \alpha
\end{align*}
exactly as in Proposition~\ref{excludeextensionrefinedprop}. The argument from the proof of Proposition~\ref{excludeextensionrefinedprop} now applies directly to show that $\omega$ is exact and coclosed when $\lambda \in (0,1)$, and to deduce that the homogeneous $1$-forms
of rate $\lambda \in (0,1)$ in the kernel of $\mlapc$ are in one-to-one correspondence with scalar eigenfunctions of $\laps$ with eigenvalues in $(6, 14)$.

Finally, the result for $\lambda = 1$ also follows in almost exactly the same way as the argument in the proof of Proposition~\ref{excludeextensionrefinedprop}. First, equation~\eqref{lapsrho2tempeq1} is identical to~\eqref{lapsrhotempeq1}, so $\rho = \ds \alpha - 2 \beta$ is again coclosed. Also, when $\lambda = 1$ the equations~\eqref{excludeextension2tempeq} and~\eqref{excludeextension2tempeq4} yield $\mlaps \alpha = \frac{4}{3} \dss \beta$ and $2 \ds \dss \beta = 14 \ds \alpha$. Now from $\mlaps \ds = \frac{3}{2} \ds \mlaps$ we get that $\mlaps \ds \alpha = 2 \ds \dss \beta$. Using all these equations together with equation~\eqref{excludeextension2tempeq2} we find
\begin{align*}
\mlaps \rho \, & = \, \mlaps \ds \alpha - 2 \mlaps \beta \, = \, 2 \ds \dss \beta - 2 \mlaps \beta \\
& = \, 14 \ds \alpha - 2 \left( \frac{20}{3} \beta + \frac{11}{3} \ds \alpha \right) \\
& = \, \frac{20}{3} \left( \ds \alpha - 2 \beta \right) \, = \, \frac{20}{3} \rho.
\end{align*}
But since $\dss \rho = 0$, we conclude that $\laps \rho = \dss \ds \rho = \frac{3}{2} \mlaps \rho = 10 \rho$, as before. So by Lemma~\ref{BochnerKillinglemma} again the vector field metric dual to $\rho$ is a Killing field. Moreover, equation~\eqref{excludeextension2tempeq2} gives $\laps (\dss \beta) = 14 (\dss \beta)$, and hence exactly as in the proof of Proposition~\ref{excludeextensionprop} we conclude that $\omega = r ( dr \wedge \alpha + r \beta) = d ( \frac{r^2}{2} \alpha) - \frac{r^2}{2} \rho$ where $\alpha = K + f$ with $\laps f = 14 f$.
\end{proof}

\begin{rmk} \label{excludeextensionprop2rmk}
The above result is crucial in the proof of Theorem~\ref{infinitesimalslicethm} because, by Proposition~\ref{specialellipticprop}, the operator $d d^* + \frac{2}{3} d^* d$ can be identified with $\pi_7 d^* d : \Omega^2_7 \to \Omega^2_7$, which is the operator that is related to our gauge-fixing condition.
\end{rmk}

The next lemma is similar to~\cite[Proposition 2.22]{Kdesings}. It is more general in that it is valid for any rate $\lambda$, but less general in that it is stated only for forms of pure degree. The method of proof, however, is different and follows the discussion immediately preceding~\cite[Proposition 5.6]{L2}. This result is needed to compute the dimension of the moduli space in Section~\ref{dimensionsec} because the dimension is computed using Theorem~\ref{indexchangethm}, and the spaces $\mathcal K(\lambda)$ a priori could involve $\log$ terms.

\begin{lemma} \label{nologslemma}
Let $m \geq 0$, and let $\gamma = \sum_{l=0}^m (\log r)^l \gamma_l$ be a $k$-form in the kernel of $d + \dsc$, where each $\gamma_l$ is homogeneous of order $\lambda$, and $\gamma_m \neq 0$. Then necessarily $m = 0$. That is, $\gamma = \gamma_0$ has no $\log$ terms.
\end{lemma}
\begin{proof}
Each $\gamma_l$ is homogeneous of order $\lambda$, so it can be written as
\begin{equation} \label{nologshomoeq}
\gamma_l \, = \, r^{k - 1 + \lambda} dr \wedge \alpha_l + r^{k + \lambda} \beta_l
\end{equation}
where $\alpha_l$ and $\beta_l$ are $(k-1)$-forms and $k$-forms on $\Sigma$, respectively, independent of $r$. For any $k$-form $\gamma_l$ on $C$, it is easy to check that
\begin{equation*}
(d + \dsc) ( (\log r)^l \gamma_l) = (\log r)^l (\dc + \dsc) \gamma_l + \frac{l}{r} (\log r)^{l-1} (dr \wedge \gamma_l) - \frac{l}{r} (\log r)^{l-1} \left( \ddr \hk \gamma_l \right).
\end{equation*}
Using this identity, we see that
\begin{align*}
(d + \dsc) \left( \sum_{l=0}^m (\log r)^l \gamma_l \right) \, = \, & (\log r)^m (\dc + \dsc) \gamma_m \\ & \quad {} + \sum_{l=0}^{m-1} (\log r)^l \left( (d + \dsc) \gamma_l + \frac{(l +1)}{r} dr \wedge \gamma_{l+1} - \frac{(l+1)}{r} \ddr \hk \gamma_{l+1} \right).
\end{align*}
The above expression must vanish as a polynomial in $\log r$. Setting the coefficient of $(\log r)^m$ equal to zero, and decomposing into forms of pure degree, we obtain $d \gamma_m = 0$ and $\dsc \gamma_m = 0$, which from~\eqref{nologshomoeq} can be simplified to
\begin{equation} \label{nologsmeq}
\begin{aligned}
d \alpha_m & = (\lambda + k) \beta_m, & \qquad d^*_\Sigma \alpha_m & = 0, \\
d \beta_m & = 0, & \qquad d^*_\Sigma \beta_m & = (\lambda + 7 - k) \alpha_m.
\end{aligned}
\end{equation}
Similarly the coefficient of $(\log r)^{m-1}$ gives $d \gamma_{m-1} + \frac{m}{r} dr \wedge \gamma_m = 0$ and $\dsc \gamma_{m-1} - \frac{m}{r} \ddr \hk \gamma_m = 0$, which simplify to
\begin{equation} \label{nologsmminusoneeq}
\begin{aligned}
d \alpha_{m-1} & = (\lambda + k) \beta_{m-1} + m \beta_m, & \qquad d^*_\Sigma \alpha_{m-1} & = 0, \\
d \beta_{m-1} & = 0, & \qquad d^*_\Sigma \beta_{m-1} & = (\lambda + 7 - k) \alpha_{m-1} + m \alpha_m.
\end{aligned}
\end{equation}
Using the systems of equations~\eqref{nologsmeq} and~\eqref{nologsmminusoneeq} on $\Sigma$ and taking $L^2$ inner products, we find
\begin{align*}
m || \alpha_m ||^2 + m || \beta_m ||^2 \, & = \, \langle \langle m \alpha_m , \alpha_m \rangle \rangle + \langle \langle m \beta_m , \beta_m \rangle \rangle \\ & = \, \langle \langle d^*_\Sigma \beta_{m-1} - (\lambda + 7 - k) \alpha_{m-1} , \alpha_m \rangle \rangle \\ & \qquad \qquad {} + \langle \langle d \alpha_{m-1} - (\lambda + k) \beta_{m-1}, \beta_m \rangle \rangle \\ & = \, \langle \langle \beta_{m-1}, d \alpha_m \rangle \rangle - (\lambda + 7 - k) \langle \langle \alpha_{m-1}, \alpha_m \rangle \rangle \\ & \qquad \qquad {} + \langle \langle \alpha_{m-1}, d^*_\Sigma \beta_m \rangle \rangle - (\lambda + k) \langle \langle \beta_{m-1}, \beta_m \rangle \rangle \\ & = \, 0.
\end{align*}
Since $\gamma_m \neq 0$, we conclude that $m=0$.
\end{proof}

\begin{rmk} \label{nologsrmk}
A generalization of Lemma~\ref{nologslemma} to mixed degree forms is possible, using the same techniques, and in any dimension. We do not state it because we will not have occasion to use it. This means in particular, that in the published version of~\cite{Kdesings}, the last sentence in the proof of~\cite[Proposition 2.23]{Kdesings} is incorrect. That is, there are never any ``$\log$ terms'' for the operator $d + \dsc$.
\end{rmk}

The next proposition is useful for analyzing the \emph{critical rates} of the operator $d + \dsc$ in Section~\ref{conifoldfredholmsec}.

\begin{prop} \label{ddstarhomokernelprop}
Let $\gamma = \sum_{k=0}^7 \gamma_{k}$ be a mixed degree form on the cone, homogeneous of order $\lambda$, and suppose that $(d + \dsc) \gamma = 0$.
\begin{itemize}
\item If $\lambda = -3$, then $\gamma = \beta + dr \wedge \alpha$, where $\beta$ and $\alpha$ are \emph{both} harmonic $3$-forms on $\Sigma$.
\item If $\lambda = -4$, then $\gamma = r^{-2} dr \wedge \alpha + \beta + (r^{-2} \sigma - r^{-1} dr \wedge d\sigma) + (dr \wedge \mu - r^{-1} \dss \mu)$, where $\alpha$ is a harmonic $2$-form on $\Sigma$, $\beta$ is a harmonic $4$-form on $\Sigma$, $\sigma$ is a \emph{coexact} $2$-form on $\Sigma$ satisfying $\laps \sigma = 2 \sigma$, and $\mu$ is an \emph{exact} $4$-form on $\Sigma$ satisfying $\laps \mu = 2 \mu$.
\item If $\lambda = -2$, then $\gamma = \alpha + r^2 dr \wedge \beta + (dr \wedge \sigma + r d\sigma) + (r dr \wedge \dss \mu + r^2 \mu)$, where $\alpha$ is a harmonic $2$-form on $\Sigma$, $\beta$ is a harmonic $4$-form on $\Sigma$, $\sigma$ is a \emph{coexact} $2$-form on $\Sigma$ satisfying $\laps \sigma = 2 \sigma$, and $\mu$ is an \emph{exact} $4$-form on $\Sigma$ satisfying $\laps \mu = 2 \mu$.
\end{itemize}
\end{prop}
\begin{proof}
The even-degree case of the first statement is exactly Proposition~2.21 in~\cite{Kdesings}. The odd-degree case, and the second and third statements, are proved in essentially the same way.
\end{proof}
\begin{rmk} \label{ddstarhomokernelrmk}
Proposition~\ref{ddstarhomokernelprop} is used several times in Section~\ref{topologicalconifoldsec} to explicitly describe the change in the space of closed and coclosed $3$-forms on a $\G$~conifold at rates $-3$ and $-4$. In fact we will mainly need this proposition for $3$-forms. If $\gamma$ is a closed and coclosed $3$-form on $C$, homogeneous of order $\lambda$, then Proposition~\ref{ddstarhomokernelprop} says that: when $\lambda = -3$, then $\gamma = \beta$ is a harmonic $3$-form on $\Sigma$; and when $\lambda = -4$, then $\gamma = r^{-2} dr \wedge \alpha$ where $\alpha$ is a harmonic $2$-form on $\Sigma$, because in this case $\mu = 0$ implies $d^* \mu = 0$.
\end{rmk}

Finally, we need to consider the excluded range of orders of homogeneity for elements in the kernel of the modified Dirac operator $\mdiracc$ defined in equation~\eqref{modifieddiracdefneq}.

\begin{prop} \label{modifieddiracexcludeprop}
Let $\mdiracc : \Omega^0_1 \oplus \Omega^1_7 \to \Omega^3_1 \oplus \Omega^3_7$ be the modified Dirac operator on a $\G$~cone. Let $s = (f, X) \in \Omega^0_1 \oplus \Omega^1_7$ be homogeneous of order $\lambda$. Then $\mdiracc (s) = 0$ precisely when
\begin{align*}
s & = \begin{cases}
(0, K r^{-6} dr) & \text{if } \, \lambda = -6, \\
(0, 0) & \text{if } \, \lambda \in (-6, 0), \\
(K, 0) & \text{if } \, \lambda = 0, \\
(0, X) & \text{if } \, \lambda \in (0,1), \text{ where $\lapc X = 0$}, \\
(0, d(r^2 h) + r^2 Y) & \text{if } \, \lambda = 1, \text{ where $\laps h = 14 h$ and $Y$ is a Killing field on $\Sigma$.}
\end{cases}\end{align*}
Here $K$ is a constant.
\end{prop}
\begin{proof}
Suppose that $\mdiracc(s) = 0$ in $\Omega^3_1 \oplus \Omega^3_7$. Corollary~\ref{mdirackernelcor} tells us that $\lapc f = 0$ and $\lapc X = 0$. Hence, if $\lambda \in [-5,0]$ then Proposition~\ref{excludeextensionprop} shows that $X = 0$. Thus by equation~\eqref{modifieddiracdefneq}, in the range $[-5,0]$ the condition $\mdiracc(s) = 0$ implies $\dc f = 0$, so $f = K$ is constant and $K = 0$ unless $\lambda = 0$.  

Now suppose that $\lambda \in [-6,-5) \cup (0,1]$. Since $\lapc f = 0$ we have that $f = 0$ by Proposition~\ref{excludeextensionprop}. This establishes the value of $f$ in all cases in the statement of the propostion. It remains to establish the value of $X$ when $\lambda \in [-6,-5) \cup (0,1]$. For this range, since $f = 0$, the condition $\mdiracc (0, X) = (0,0)$ is equivalent to $\dsc X = 0$ and $\curl_{\scriptscriptstyle{C}} (X) = 0$.

If $\lambda \in (0,1)$, then by Proposition~\ref{excludeextensionrefinedprop} we know that $\lapc X = 0$ if and only if $\dc X = 0$ and $\dsc X = 0$. Since $\dc X = 0$ implies $\curl_{\scriptscriptstyle{C}} (X) = 0$, we have established the result for $\lambda \in (0,1)$.

We can see the result for $\lambda \in (-6,-5)$ by the symmetry of the situation at hand under $\lambda \mapsto -5 - \lambda$. Explicitly, if $\lambda \in (-6,-5)$ then because $\lapc X = 0$, we saw in the proof of Proposition~\ref{excludeextensionrefinedprop} that 
\begin{equation*}
X = r^{\lambda} \alpha dr - \frac{1}{\lambda+4} r^{\lambda+1} \ds \alpha,
\end{equation*}   
where $\alpha$ is a function on $\Sigma$ with $\laps \alpha = (\lambda-1)(\lambda+4) \alpha$. But then using~\eqref{homoddseq} and $\dsc  X = 0$ we find that 
\begin{equation*}
(\lambda+6) \alpha = - \frac{1}{\lambda+4} \dss \ds \alpha = - (\lambda-1) \alpha,
\end{equation*}
which is not possible in this range of rates unless $\alpha = 0$, and thus $X = 0$.  

Finally we consider the cases $\lambda = -6$ and $\lambda = 1$. We have shown so far that $f = 0$ and that $\mdirac(0, X) = 0$ if and only if $\dsc X = 0$ and $\curl_{\scriptscriptstyle{C}} (X) = 0$. Recall also that if $dX = 0$, then certainly $\curl_{\scriptscriptstyle{C}} (X) = 0$. Now, since $\lapc X = 0$, Proposition~\ref{excludeextensionrefinedprop} shows that in these two cases we can write
\begin{equation} \label{modifieddiracexcludetempeq}
\begin{aligned}
X \, & = \, K r^{\lambda} dr + \left( r^{\lambda} h dr + \frac{1}{2} r^{\lambda + 1} \ds h \right) + \frac{1}{2} r^{\lambda + 1} Y \\
& = \, X_1 + X_2 + X_3,
\end{aligned}
\end{equation}
where $K$ is a constant, $h$ is a function on $\Sigma$ such that $\laps h = 14 h$, and $Y$ is the dual $1$-form to a Killing vector field on $\Sigma$. We will consider each of the three terms in~\eqref{modifieddiracexcludetempeq} separately. It is easy to check using~\eqref{homoddseq} and the fact that $\dss Y = 0$ since $Y$ is Killing that
\begin{equation*}
\text{for $\lambda = -6$,} \quad \left\{ \begin{array}{lll} \dc X_1 = 0, & & \dsc X_1 = 0, \\[0.2em] \dc X_2 = - \frac{7}{2} r^{-6} dr \wedge dh, & & \dsc X_2 = \frac{7}{2} r^{-7} h, \\[0.2em] \dc X_3 = -5 r^{-6} dr \wedge Y + r^{-5} \ds Y, & & \dsc X_3 = 0. \end{array} \right.
\end{equation*}
Since $X_1$ is closed (hence curl-free) and coclosed, it is always in the kernel of $\mdirac$. For the remaining piece $X_2 + X_3$ to be in $\ker \mdirac$, we see from above that we require $h = 0$, so $X_2 = 0$. From Proposition~\ref{Killing.prop}, we find that since $\dsc X_3 = 0$, then $\curl_{\scriptscriptstyle{C}} (X_3) = 0$ if and only if the flow of $X_3$ preserves $\phc$, which never happens for $Y \neq 0$ by Proposition~\ref{NK-Killing-prop}. Therefore $X_3 = 0$ as well and we have established the result for $\lambda = -6$.

On the other hand, for $\lambda = 1$ the term $X_3 = \frac{1}{2} r^2 Y$ is always in the kernel of $\mdirac$ by Propositions~\ref{NK-Killing-prop} and~\ref{Killing.prop}. Then, using~\eqref{homoddseq} again we find that
\begin{equation*}
\text{for $\lambda = 1$,} \quad \left\{ \begin{array}{lll} \dc X_1 = 0, & & \dsc X_1 = - 7 K, \\[0.2em] \dc X_2 = 0, & & \dsc X_2 = 0. \end{array} \right.
\end{equation*}
Thus $X_2$ also lies in $\ker \mdirac$ and $X_1$ never does, unless $K = 0$, establishing the result for $\lambda = 1$.
\end{proof}

\subsection{Asymptotically conical (AC) \texorpdfstring{$\mathbf{\G}$}{G2} manifolds} \label{ACsec}

Let $M$ be a noncompact, connected smooth $7$-dimensional manifold.
\begin{defn} \label{ACdefn}
The manifold $M$ is called an \emph{asymptotically conical} $\G$~manifold with cone $C$ and \emph{rate} $\nu < 0$ if all of the following holds:
\begin{itemize} \setlength\itemsep{-0.5mm}
\item The manifold $M$ is a $\G$~manifold with torsion-free $\G$~structure $\phm$ and complete metric $\gm$.
\item There is a $\G$~cone $(C, \phc, \gc)$ with \emph{connected} link $\Sigma$.
\item There is a compact subset $L \subset M$, an $R > 1$, and a map $h : (R, \infty) \times \Sigma \to M$ that is a \emph{diffeomorphism} of $(R, \infty) \times \Sigma$ onto $M \backslash L$.
\item The pullback $h^* (\phm)$ is a torsion-free $\G$~structure on the subset $(R, \infty) \times \Sigma$ of $C$. We require that this approach the torsion-free $\G$~structure $\phc$ in a $C^{\infty}$ sense, with rate $\nu < 0$. This means that
\begin{equation} \label{ACdefneq}
| \nabc^j ( h^* (\phm) - \phc ) |_{\gc} \, = \, O (r^{\nu - j})
\qquad \forall j \geq 0
\end{equation}
in $(R,\infty) \times \Sigma$. Note that all norms and derivatives are computed using the cone metric $\gc$. It follows immediately from~\eqref{ACdefneq} and Taylor's theorem that the metric on $M$ is asymptotic to the cone metric at the same rate:
\begin{equation*}
| \nabc^j ( h^* (\gm) - \gc ) |_{\gc} \, = \, O (r^{\nu - j}).
\qquad \forall j \geq 0
\end{equation*}
\end{itemize}
\end{defn}
It is clear that an AC $\G$~manifold of rate $\nu_0$ is also an AC $\G$~manifold for all $\nu > \nu_0$.
\begin{rmk} \label{ACdefnrmk}
The link $\Sigma$ of an AC $\G$~manifold $M$ must be connected because $M$ can have only one end. This follows from the Cheeger--Gromoll splitting theorem, which says that a complete noncompact Ricci-flat manifold with more than one end isometrically splits into a Riemannian product, and thus the holonomy would be reducible.
\end{rmk}

\begin{ex} \label{BSexamples}
There are three known examples of asymptotically conical $\G$~manifolds, whose asymptotic cones have links given by the three homogeneous nearly K\"ahler manifolds (we have excluded the round $S^6$). They are all total spaces of vector bundles over a compact base. These manifolds were discovered by Bryant--Salamon~\cite{BS} and were the first examples of complete $\G$~manifolds. Specifically, they are described in the following list, where the metric on the base manifold is the one induced from the Bryant--Salamon metric by restriction.

\begin{itemize}
\item $\Lambda^2_-(S^4)$, the bundle of anti-self-dual $2$-forms over the $4$-sphere. This is a nontrivial rank $3$ vector bundle over the standard round $S^4$. This AC $\G$~manifold is asymptotic to the cone over the nearly K\"ahler $\C \PR^3$, with rate $\nu = -4$.
\item $\Lambda^2_-(\C \PR^2)$, the bundle of anti-self dual $2$-forms over the complex projective plane. This is a nontrivial rank $3$ vector bundle over the standard Fubini-Study $\C \PR^2$. This AC $\G$~manifold is asymptotic to the cone over the nearly K\"ahler flag manifold $F_{1,2}$, also with rate $\nu = -4$.
\item $\spi (S^3)$, the spinor bundle of the $3$-sphere. This is a trivial rank $4$ vector bundle over the standard round $S^3$, hence is topologically $S^3 \times \R^4$. This AC $\G$~manifold is asymptotic to the cone over the homogeneous nearly K\"ahler $S^3 \times S^3$, with rate $\nu = -3$.
\end{itemize}
\end{ex}

\begin{rmk} \label{ACexamplesrmk}
Explicit formulas for these asymptotically conical $\G$~structures, as well as the fact that their rates are $-4$, $-4$, and $-3$, respectively,  can be found in Bryant--Salamon~\cite{BS}, and also in Atiyah--Witten~\cite{AW}. We will not have need for these explicit formulas.
\end{rmk}

\subsection{Conically singular (CS) \texorpdfstring{$\mathbf{\G}$}{G2} manifolds} \label{CSsec}

Let $\overline M$ be a compact, connected topological space, and let $x_1, \ldots, x_n$ be a finite set of isolated points in $\overline M$. Assume that $M = \overline M \backslash \{x_1, \ldots, x_n \}$ is a smooth noncompact $7$-dimensional manifold that we will call the \emph{smooth part} of $\overline M$ and $\{ x_1, \ldots, x_n\}$ will be called the \emph{singular points} of $\overline M$.

\begin{defn} \label{CSdefn}
The space $\overline M$ is called a \emph{$\G$~manifold with isolated conical singularities}, with cones $C_1, \ldots, C_n$ at $x_1, \ldots, x_n$ and \emph{rates} $\nu_1, \ldots, \nu_n$, where each $\nu_i > 0$, if all of the following holds:
\begin{itemize} \setlength\itemsep{-0.5mm}
\item The smooth part $M$ is a $\G$~manifold with torsion-free $\G$~structure $\phm$ and metric $\gm$.
\item There are $\G$~cones $(C_i, \phci, \gci)$ with links $\Sigma_i$ for $i = 1, \ldots , n$.
\item There is a compact subset $K \subset M$ such that $M \backslash K$ is a union of open sets $S_1, \ldots, S_n$ whose closures $\overline S_1, \ldots, \overline S_n$ in $\overline M$ are all \emph{disjoint} in $\overline M$. There is an $\e \in (0,1)$, and for each $i = 1, \ldots, n$, there is a map $h_i : (0, \e) \times \Sigma_i \to M$ that is a \emph{diffeomorphism} of $(0, \e) \times \Sigma_i$ onto $S_i$.
\item The pullback $h_i^* (\phm)$ is a torsion-free $\G$~structure on the subset $(0, \e) \times \Sigma_i$ of $C_i$. We require that this approach the torsion-free $\G$~structure $\phci$ in a $C^{\infty}$ sense, with rate $\nu_i > 0$. This means that
\begin{equation} \label{CSdefneq}
| \nabci^j \! ( h_i^* (\phm) - \phci ) |_{\gci} \, = \, O (r^{\nu_i - j})
\qquad \forall j \geq 0
\end{equation}
in $(0,\e) \times \Sigma_i$. Note that all norms and derivatives are computed using the cone metric $\gci$. It follows immediately from~\eqref{CSdefneq} and Taylor's theorem that the metric on $M$ is asymptotic to the cone metric at the same rate:
\begin{equation*}
| \nabci^j \! ( h_i^* (\gm) - \gci ) |_{\gci} \, = \, O (r^{\nu_i - j})
\qquad \forall j \geq 0
\end{equation*}
\end{itemize}
It is clear that a CS $\G$~manifold of rate $\nu_0$ is also a CS $\G$~manifold for all $\nu < \nu_0$. We also note that $\overline M$ is the closure of $M$ in $\overline M$. We will often abbreviate the phrase ``compact $\G$~manifold with isolated conical singularities'' as conically singular or CS $\G$~manifold.
\end{defn}

There are at present still no examples of conically singular $\G$~manifolds, although they are expected to exist in abundance. The main theorem in~\cite{Kdesings} can be interpreted as evidence for the likelihood of their existence, in the sense that they should arise as ``boundary points'' of the moduli space of smooth compact $\G$~manifolds. Moreover, the discussion in Section~\ref{resolvedmodulisec} of the present paper, which is a corollary of our main theorem, can also be interpreted as saying that CS $\G$~manifolds should in fact make up a large part of the boundary of the moduli space of smooth compact $\G$~manifolds. The first author, in collaboration with Dominic Joyce, has a new construction of compact $\G$~manifolds~\cite{JK} that may be generalizable to produce the first examples of compact $\G$~manifolds with isolated conical singularities, which would all be modeled on the cone over the nearly K\"ahler $\C \PR^3$.

\section{Analysis on \texorpdfstring{$\mathbf{\G}$}{G2} conifolds} \label{analsec}

In this section we collect a plethora of analytic results, some general and some specific to $\G$~conifolds. We begin in Sections~\ref{conifoldSobolevsec} and~\ref{conifoldfredholmsec} by summarizing the essential aspects of the Lockhart--McOwen theory for AC and CS manifolds that we will require. This theory originally appeared in Lockhart--McOwen~\cite{LM} and Lockhart~\cite{Lock}. A very detailed exposition can also be found in Marshall~\cite{M}. Then in Section~\ref{Hodgesec} we use this theory to establish Hodge-theoretic results for weighted Sobolev spaces of forms. These are all combined in Section~\ref{specialindexchangesec} to establish a special index change theorem for an operator $\dd$ that plays the key role in our deformation theory. In Section~\ref{topologicalconifoldsec} we consider some topological results on $\G$~conifolds. These are important ingredients in computing the (virtual) dimension of the moduli spaces later. Finally, in Sections~\ref{parallelsec},~\ref{gaugefixingsec}, and~\ref{particularG2analsec} we discuss parallel tensors, gauge-fixing conditions, and some vanishing results particular to $\G$~conifolds that will be needed to prove our main theorem in Section~\ref{conifoldmodulisec}.

\subsection{Weighted Banach spaces on conifolds} \label{conifoldSobolevsec}

The essential idea is as follows. By introducing appropriately \emph{weighted} Banach spaces of sections of vector bundles over an AC manifold or over the smooth, noncompact part of a CS manifold, one generically obtains a nice Fredholm theory for an elliptic operator $P : V \to W$ such as the Laplacian or the Dirac operator. Basically, as long as we stay away from certain ``critical rates'', which form a discrete set, these operators will be Fredholm and we can write $W = \im (P) \oplus C$ for some finite-dimensional complement $C$ which is isomorphic to $\ker(P^*)$. The precise details are explained below.

We will mostly use this theory for weighted Sobolev spaces. However, it applies equally well to weighted H\"older spaces, and we will require the relations between these spaces (the \emph{Sobolev embedding theorems}) in order to deal with some regularity issues, in particular to ensure that the sections are at least twice continuously differentiable.

Throughout this section, we use $M$ to denote a \emph{$\G$~conifold}, which is either an asymptotically conical (AC) $\G$~manifold, as in Definition~\ref{ACdefn}, or the smooth part of a conically singular (CS) $\G$~manifold, as in Definition~\ref{CSdefn}. Many, but not all, of the results are valid for any \emph{Riemannian conifold}, although the results are always stated in the particular dimension $7$ for convenience.

The analytic results for AC manifolds hold equally well for CS manifolds with minor differences. The most significant difference is that all the inequalities involving rates must be reversed, since the noncompact ends correspond to $r \to 0$ instead of $r \to \infty$. Also, on a CS manifold we can have $n$ ends as opposed to just one.

In order to be able to define sensible ``weighted'' Banach spaces on $M$, we need the concept of a radius function.

\begin{defn} \label{radiusfunctiondefn}
A \emph{radius function} $\varrho$ is a smooth function on $M$ satisfying the following conditions.
\begin{itemize}
\item {\bf AC case:} On the compact subset $L$ of $M$, we define $\varrho \equiv 1$. If $x = h(r, p)$ for some $r \in (2R, \infty)$ and $p \in \Sigma$, then set $\varrho(x) = r$. In the region $h( (R, 2R) \times \Sigma)$, the function $\varrho$ is defined by interpolating smoothly between its definition near infinity and its definition in the compact subset $L$, in a decreasing fashion.
\item {\bf CS case:} On the compact subset $K = M \backslash \sqcup_{i=1}^n S_i$, we define $\varrho \equiv 1$. If $x = h_i(r, p)$ for some $r \in (0, \frac{1}{2}\e)$ and $p \in \Sigma_i$, then set $\varrho(x) = r$. In the regions $h_i( (\frac{1}{2}\e, \e) \times \Sigma_i)$, the function $\varrho$ is defined by interpolating smoothly between its definitions near the singularities and its definition in the compact subset $K$, in an increasing fashion.
\end{itemize}
\end{defn}

We can now define the weighted Sobolev spaces on $M$. Let $E$ be a vector bundle over $M$ with a fibrewise metric. In all instances in this paper, $E$ will either be the bundle $\Lambda^k(T^* M)$ of $k$-forms on $M$, or the space of all forms $\Lambda^{\bullet} (T^*M) = \bigoplus_{k=0}^7 \Lambda^k (T^*M)$ on $M$, or the spinor bundle $\spi(M)$ over $M$, described in Section~\ref{spinorsec}. The fibrewise metric in all these cases is naturally induced from the Riemannian metric on $M$, and the Levi-Civita connection $\nabm$ naturally induces a connection on $E$ which we continue to denote by $\nabm$.

We want to define the \emph{weighted Sobolev space} of sections of $E$ with rate $\lambda$. In the AC case, we let $\lambda \in \R$. In the CS case, we let $\lambda = (\lambda_1, \ldots, \lambda_n) \in \R^n$. We can add such $n$-tuples and multiply them by real numbers using the vector space structure of $\R^n$. We also define $\lambda + j = (\lambda_1 + j, \ldots, \lambda_n + j)$ for any $j \in \R$, and we say that $\lambda > \lambda'$ if $\lambda_i > \lambda_i'$ for all $i = 1, \ldots, n$. Finally we define $\varrho^{\lambda}$ to equal $\varrho^{\lambda_i}$ on $h_i( (0, \e) \times \Sigma_i)$ and to equal $1$ on the compact subset $K$. Then $\varrho^{\lambda}$ is a smooth function on $M$ which equals
$r^{\lambda_i}$ on the neighbourhood $h_i( (0, \frac{1}{2}\e) \times \Sigma_i)$ of the singular point $x_i$.

\begin{defn} \label{conifoldSobolevdefn}
Let $l \geq 0$, $ p > 1$, and let $\lambda$ be as above. We define the \emph{weighted Sobolev space} $L^p_{l, \lambda} (E)$ of sections of $E$ over $M$ as follows. Consider the space $C^{\infty}_{\mathrm{cs}}(E)$ of smooth compactly supported sections of $E$. For such sections the quantity
\begin{equation} \label{conifoldSobolevdefneq}
{||\gamma||}_{L^p_{l,\lambda}} \, = \, {\left( \sum_{j=0}^l \int_{M'} {| \varrho^{- \lambda + j} \nabm^j \gamma|}^p_{\gm} \varrho^{-7} \volm \right)}^{\frac{1}{p}}
\end{equation}
is clearly finite, and is a norm. We define the Banach space $L^p_{l, \lambda} (E)$ to be the completion of $C^{\infty}_{\mathrm{cs}}(E)$ with respect to this norm.
\end{defn}

\begin{rmk} \label{conifoldSobolevdefnrmk}
We make a few important remarks about this definition.
\begin{enumerate}[(a)]
\item As a topological vector space, $L^p_{l, \lambda} (E)$ is independent of the choice of radius function $\varrho$, and any two such choices lead to equivalent norms.
\item We clearly have $L^p_{l, \lambda} (E) \subseteq L^p_{l, \lambda'} (E)$ if $\lambda \leq \lambda'$ in the AC case or if $\lambda \geq \lambda'$ in the CS case.
\item An element $\gamma$ in $L^p_{l, \lambda} (E)$ can be intuitively thought of as a section of $E$ that is $l$ times weakly differentiable such that near each end, the tensor $\nabm^j \gamma$ is growing at most like $r^{\lambda - j}$.
\item The space $L^2_{l, \lambda} (E)$ is a \emph{Hilbert space}, with inner product coming from the polarization of the norm in~\eqref{conifoldSobolevdefneq}. Because of the factor $\varrho^{-7}$ in~\eqref{conifoldSobolevdefneq}, we have
\begin{equation*}
L^2_{0, - \frac{7}{2}} (E) \, = \, L^2 (E),
\end{equation*}
where $L^2(E)$ is the usual space of $L^2$ sections of $E$. Here and henceforth it is understood that in the CS case $\frac{7}{2}$ denotes the `constant' $n$-tuple $(\frac{7}{2}, \ldots, \frac{7}{2})$.
\end{enumerate}
\end{rmk}

\begin{rmk} \label{pismostly2rmk}
We will almost always just take $p=2$ in this paper. The only time we will need to consider $p \neq 2$ is in Lemma~\ref{Qgoodlemma}, which uses the general Sobolev embedding Theorem~\ref{Sobolevembeddingthm} below.
\end{rmk}

The following proposition about dual spaces is easy to see from Definition~\ref{conifoldSobolevdefn}.
\begin{prop} \label{dualspaceprop}
There is a Banach space isomorphism
\begin{equation*}
{\left( L^2_{0, \lambda} (E) \right)}^* \, \cong \, L^2_{0, -\lambda - 7} (E),
\end{equation*}
given by the $L^2$ inner product pairing.
\end{prop}

We will likewise have need of the analogous weighted \emph{H\"older spaces}. Their definition is a bit more involved. See, for example~\cite{Lth, M} for the precise definition. However, all that we will require from the weighted H\"older spaces is that elements in them have some degree of differentiability with control on their growth rate on the ends, and that these spaces are related to the weighted Sobolev spaces by the \emph{embedding theorems}, which we will state precisely. The embedding theorems are used implicitly in the sketch proof of Theorem~\ref{kernelthm} below to explain why elements in the kernel of a uniformly elliptic operator are in fact $C^{\infty}$.

Let $m \geq 0$ and $\alpha \in (0,1)$. Then the weighted H\"older space $C^{m, \alpha}_{\lambda}(E)$ is a Banach space of sections of $E$, whose elements have $m$ \emph{continuous derivatives}.
\begin{thm}[Weighted Sobolev embedding theorem] \label{Sobolevembeddingthm}
Let $l,m \geq 0$ and let $\alpha \in (0,1)$.
\begin{itemize}
\item If $l \geq m$, $l - \frac{7}{p} \geq m - \frac{7}{q}$, and $p \leq q$, then there is a \emph{continuous embedding}
\begin{equation*}
L^p_{l, \lambda} (E) \hookrightarrow L^q_{m, \lambda} (E).
\end{equation*}
\item If $l - \frac{7}{2} \geq m + \alpha$, then there is a \emph{continuous embedding}
\begin{equation*}
L^p_{l, \lambda} (E) \hookrightarrow C^{m, \alpha}_{\lambda}(E).
\end{equation*}
\end{itemize}
\end{thm}
\begin{proof}
See Marshall~\cite[Theorem 4.17]{M} for a proof. We have only stated some special cases, which are all that we will require.
\end{proof}
\begin{cor} \label{C2cor}
If $l \geq 6$, then any section $\gamma \in L^2_{l, \lambda}(E)$ is twice continuously differentiable.
\end{cor}
\begin{proof}
This follows immediately from Theorem~\ref{Sobolevembeddingthm} by taking $m = 2$.
\end{proof}
We will \emph{always} assume that $l \geq 6$ without explicit mention, so that in particular any second order differential operators on such sections are unambiguously defined.

\subsection{Fredholm and elliptic operators on conifolds} \label{conifoldfredholmsec}

Many standard facts will be stated without proof in this section. The reader can consult~\cite{LM, Lock, M} for details. To make some of the equations easier to read, we will often use the following shorthand notation:
\begin{align*}
\Omega^{\bullet}_{l,\lambda} \, & = \, L^2_{l, \lambda}(\Lambda^{\bullet}(T^* M)), \\
\Omega^{k}_{l,\lambda} \, & = \, L^2_{l, \lambda}(\Lambda^{k}(T^* M)), \qquad 0 \leq k \leq 7, \\
\spi_{l, \lambda} \, & = \, L^2_{l, \lambda}(\spi(M)).
\end{align*}
We will be interested in the following three differential operators:
\begin{align} \label{ddstarmapdefneq}
(d + \dsm)_{l + 1, \lambda} \, & : \, \Omega^{\bullet}_{l + 1, \lambda} \, \to \, \Omega^{\bullet}_{l, \lambda - 1}, \\
\label{lapmapdefneq}
(\lapm)_{l + 2, \lambda} \, & : \, \Omega^{k}_{l + 2, \lambda} \, \to \, \Omega^{k}_{l, \lambda - 2}, \\\label{diracmapdefneq}
(\diracm)_{l + 1, \lambda} \, & : \, \spi_{l + 1, \lambda} \, \to \, \spi_{l, \lambda - 1}.
\end{align}
They are defined by extending the operators $d + \dsm$, $\lapm$, and $\diracm$ from smooth compactly supported sections to the Sobolev spaces. Note that the Laplacian $\lapm$ preserves the degree $k$ of forms, so strictly speaking we should include the degree $k$ as an extra label on the left hand side of~\eqref{lapmapdefneq}, but we will not do this, to avoid the proliferation of notation. We will let $r$ denote the order of the differential operator, which is $1$,  $2$, and $1$, respectively, in these three cases. Using the symbol $P$ to denote one of these operators, and $E$ to denote the vector bundle on which it acts, the above three operators can all be written as
\begin{equation} \label{Pmapdefneq}
P_{l + r, \lambda} \, : \, L^2_{l + r, \lambda} (E) \, \to \,  L^2_{l, \lambda - r} (E).
\end{equation}
In fact, we will also be interested in the modified Dirac operator $\mdiracm$ defined in Section~\ref{spinorsec}, as well as in the \emph{restriction} of $d + \dsm$ to the space $\Omega^k_{l+1, \lambda}$ of $k$-forms, which we will denote simply by $(\ddm)^k_{l+1 , \lambda}$. That is,
\begin{equation} \label{dddefneq}
\ddm^k \, = \, \rest{(d + \dsm)}{\Omega^k}.
\end{equation}
It is a standard fact that the operators $d + \dsm$, $\lapm$, and $\diracm$ are \emph{elliptic}, and in Proposition~\ref{modifieddiracprop} we proved that $\mdiracm$ is also elliptic. In fact these operators are also all \emph{uniformly elliptic} in the sense that near infinity, they approach the elliptic operators $d + \dsc$,  $\lapc$, $\diracc$, and $\mdiracc$ on the cone $C$, respectively. See Marshall~\cite[Chapter 4]{M} for the precise definition of uniform ellipticity in this context. We note here that the operator $\ddm^k$, being the restriction of $d + \dsm$ to the space of $k$-forms, is \emph{not} elliptic, but for suitable rates $\lambda$ and suitably redefined codomain, it will be Fredholm. This is discussed in Section~\ref{specialindexchangesec}. The following result is an \emph{elliptic regularity} statement for uniformly elliptic operators.

\begin{thm} \label{ellipticregthm}
Let $P$ be a uniformly elliptic operator. Suppose that $\gamma$ and $\upsilon$ are both locally integrable sections of $E$, and that $\gamma$ is a weak solution of the equation $P(\gamma) = \upsilon$. If $\gamma \in L^2_{0, \lambda} (E)$ and $\upsilon \in L^2_{l, \lambda - r} (E)$, then $\gamma \in L^2_{l + r, \lambda} (E)$, and $\gamma$ is a strong solution of $P(\gamma) = \upsilon$. Furthermore, we have
\begin{equation} \label{ellipticregeq}
{||\gamma||}_{L^2_{l+r,\lambda}} \, \leq \, C \left( {||P(\gamma)||}_{L^2_{l,\lambda -r}} + {||\gamma||}_{L^2_{0,\lambda}} \right)
\end{equation}
for some constant $C > 0$ independent of $\gamma$. That is, $\gamma$ has at least $r$ more derivatives worth of regularity than $\upsilon = P(\gamma)$.
\end{thm}

We will need to use some results about the kernels and indices of linear operators. To this end we must first define the \emph{critical rates} for these operators, which depend on the geometry of the links of the cones on each end.
\begin{defn} \label{criticalratesdefn}
Let $C$ be a $\G$~cone. Let $\Pc$ be one of the operators $d + \dsc$,  $\lapc$, $\diracc$, $\mdiracc$, or $\ddc^k$ acting on sections of some vector bundle $E$ over $C$. The set $\mathcal{D}_{\Pc}$ of \emph{critical rates} of the operator $\Pc$ on sections of $E$ is defined as follows:
\begin{equation} \label{criticalratesdefneq}
\mathcal{D}_{\Pc} = \left\{ \lambda \in \R; \, \, \exists \text{ a nonzero section $\gamma$ of $E$, homogeneous of order $\lambda$, with $\Pc(\gamma) = 0$} \right\}.
\end{equation}
(In the general theory of elliptic operators on $C$ one has to allow complex values for the critical rates, but since the operators we consider  
here are formally self-adjoint, or the restrictions thereof, the critical rates are necessarily real in this setting.)

The definition of `homogeneous of order $\lambda$' for a $k$-form on a cone was given in Definition~\ref{homoformsdefn}. If $\gamma$ is a mixed degree form or a spinor in $\spi(C)$ (which from Section~\ref{spinorsec} consists of a function and a $1$-form), homogeneous means that each graded component is homogeneous. The set $\mathcal D_{\Pc}$ is a countable, discrete subset of $\R$, and has finite intersection with any closed bounded interval of $\R$.

In the AC case a rate $\lambda \in \R$ is a \emph{critical rate} of $P$ if it is a critical rate of the corresponding operator $\Pc$ on its asymptotic cone $C$. In the CS case we have $n$ ends which are modeled on cones, and a rate $\lambda = (\lambda_1, \cdots, \lambda_n) \in \R^n$ will be a \emph{critical rate} for $P$ if any of its components $\lambda_i$ lie in the corresponding critical set $\mathcal D_{\Pci}$ for the cone $C_i$. We say that the ``interval'' $[\lambda, \lambda']$ does not contain any critical rates for $P$ on $M$ if each interval $[\lambda_i, \lambda'_i]$ contains no critical rates for $P$ on the cone $C_i$.
\end{defn}

\begin{thm} \label{kernelthm}
The \emph{kernel} of $P_{l + r, \lambda}$ is independent of $l$. Hence we can denote it unambigiously as $\ker(P)_{\lambda}$. This kernel is also invariant as we change the rate $\lambda$, as long as we do not hit any critical rates. That is, if the interval $[\lambda, \lambda']$ is contained in the complement of $\mathcal D_{P}$, then
\begin{equation*}
\ker (P)_{\lambda'}  \, = \, \ker (P)_{\lambda} .
\end{equation*}
\end{thm}

\begin{proof}
The invariance of the kernel in the absence of critical rates is explained in~\cite{LM, Lock}. We present a sketch of the proof of the independence on $l$. The Sobolev embedding Theorem~\ref{Sobolevembeddingthm} says that for large enough $l$, we can embed the Sobolev space $L^2_{l, \lambda}(E)$ into an appropriate \emph{H\"older space} $C^{m, \alpha}_{\lambda}(E)$ having $m$ continuous derivatives. It follows from this theorem and the elliptic regularity of Theorem~\ref{ellipticregthm} that elements in the kernel of $P$ are smooth, and the independence of the kernels on $l$ follows from this.
\end{proof}

Recall that a linear map between Banach spaces is called \emph{Fredholm} if it has closed image, finite-dimensional kernel, and finite-dimensional cokernel. The main significance of the critical rates $\mathcal D_{P}$ is that they are related to the rates $\lambda$ for which the operator $P_{l + r, \lambda}$ of equation~\eqref{Pmapdefneq} is Fredholm, by the following theorem.

\begin{thm} \label{fredholmthm}
The map $P_{l + r, \lambda} \, : \, L^2_{l + r, \lambda} (E) \, \to \, L^2_{l, \lambda - r} (E)$ is Fredholm if and only if $\lambda \notin \mathcal D_{P}$, where the set of critical rates $\mathcal D_{P}$ is as given in Definition~\ref{criticalratesdefn}.
\end{thm}

Now consider the formal adjoint of the map
\begin{equation*}
P_{l + r, \lambda} \, : \, L^2_{l + r, \lambda} (E) \, \to \, L^2_{l, \lambda - r} (E).
\end{equation*} 
By Proposition~\ref{dualspaceprop}, the formal adjoint is a map
\begin{equation} \label{adjointmapeq}
P^*_{m+r, - 7 + r - \lambda} \, \, : \, \, L^2_{m+r, - 7 + r - \lambda} (E) \to L^2_{m , - 7 - \lambda} (E),
\end{equation}
where $l, m \geq 0$.
\begin{rmk} \label{sloppyrmk}
Here we are being slightly sloppy, in the following sense. Technically, we really have ${(L^2_{l, \lambda})}^* = L^2_{-l, - \lambda - 7}$, but we would like to avoid having to consider the meaning of $L^2_{l, \lambda}$ for $l < 0$. Fortunately, we will only ever be interested in the \emph{kernel} of the formal adjoint $P^*$ on spaces of the form $L^2_{m+r, \lambda}$, which by Theorem~\ref{kernelthm} is independent of $m$, so it is safe to assume that $m \geq 0$.
\end{rmk}

The next result is the version of the `Fredholm Alternative' for conifolds.
\begin{thm} \label{fredholmalternativethm}
Suppose that $\lambda$ is not in $\mathcal D_{P}$, so that by Theorem~\ref{fredholmthm}, the map 
\begin{equation*}
P_{l + r, \lambda} \, : \, L^2_{l + r, \lambda} (E) \, \to \, L^2_{l, \lambda - r} (E)
\end{equation*}
is Fredholm, and also uniformly elliptic. Then:
\begin{enumerate}[(a)]
\item We can choose a finite-dimensional subspace $W_{\lambda - r}$ of $L^2_{l, \lambda - r} (E)$ such that
\begin{equation} \label{fredholmalternativeeq}
L^2_{l, \lambda - r} (E) \, = \, P ( L^2_{l + r, \lambda} (E) ) \oplus W_{\lambda - r},
\end{equation}
such that
\begin{equation} \label{cokernelisomorphismeq}
W_{\lambda - r} \, \cong \, \ker (P^*)_{- 7 + r - \lambda}.
\end{equation}
\item Furthermore, if $\ker (P^*)_{- 7 + r - \lambda}$ lies in $L^2_{l, \lambda - r} (E)$, then we can take
\begin{equation*}
W_{\lambda - r} \, = \, \ker (P^*)_{- 7 + r - \lambda}.
\end{equation*}
By Remark~\ref{conifoldSobolevdefnrmk} (b), this happens whenever $\lambda \geq -\frac{7}{2} + r$ in the AC case, and whenever $\lambda \leq - \frac{7}{2} + r$ in the CS case.
\end{enumerate}
\end{thm}
\begin{rmk} \label{annihilatorrmk}
Equation~\eqref{cokernelisomorphismeq} is a consequence of general Banach space theory and Proposition~\ref{dualspaceprop}, since whenever a subspace $W$ of a Banach space $V$ is closed, any direct complement of it will be isomorphic to its annihilator in the dual space.
\end{rmk}
\begin{rmk} \label{fredholmrmk}
Because of Remark~\ref{conifoldSobolevdefnrmk} (b) and equation~\eqref{cokernelisomorphismeq} we see that $\ker(P)_{\lambda}$ and $\coker(P)_{\lambda}$ are always finite-dimensional, even if $\lambda$ is a critical rate. Thus the failure of $P$ to be Fredholm at a critical rate is due only to $\im (P)$ not being closed.
\end{rmk}

As we will be using both the Fredholm theory of $\lapm$ and that of $d + \dsm$, we will need to know how elements in the kernels of these two operators are related. In particular, a $k$-form which is closed and coclosed is always harmonic, but the converse will only be true for certain rates. Also, when a mixed degree form $\gamma = \sum_{k=0}^7 \gamma_k$ is closed and coclosed, it will not always be the case that each graded component $\gamma_k$ is independently closed and coclosed. Before presenting this result, we give a proof of ``integration by parts'' for weighted spaces.
\begin{lemma} \label{IBPlemma}
Let $\alpha \in \Omega^{k-1}_{l, \lambda}$ and $\beta \in \Omega^k_{m, \mu}$. If $\lambda + \mu < -6$ (AC) or $\lambda + \mu > - 6$ (CS), then we have
\begin{equation*}
\langle \langle d \alpha, \beta \rangle \rangle \, = \, \langle \langle \alpha, \dsm \beta \rangle \rangle.
\end{equation*}
\end{lemma}
\begin{proof}
We give the proof in the AC case. The CS case is identical except that there are $n$ ends instead of just one, and $\varrho \to 0$ on each end instead of $\varrho \to \infty$. Let $M_R = \{ x \in M; \varrho(x) \leq R \}$, and observe that $\partial (M_R) = \{R \} \times \Sigma$. Hence, by Stokes's Theorem and the fact that $d\alpha \wedge \stm \beta - \alpha \wedge \stm \dsm \beta = d (\alpha \wedge \stm \beta)$, we find
\begin{equation*}
\int_{M_R} \langle d \alpha, \beta \rangle \, \volm - \int_{M_R} \langle \alpha, \dsm \beta \rangle \, \volm \, = \, \int_{\{R \} \times \Sigma} (\alpha \wedge \stm \beta).
\end{equation*}
The proof will be complete if we can establish that the integral on the right hand side above goes to zero as $R \to \infty$. But since $|\alpha| \leq C R^{\lambda}$ and  $|\beta| \leq C R^{\mu}$ on the end, we have
\begin{equation*}
\left|  \int_{\{R \} \times \Sigma} (\alpha \wedge \stm \beta) \right| \, \leq \,  \int_{\{R \} \times \Sigma} |\alpha \wedge \stm \beta| \vols \, \leq \, C R^{\lambda + \mu + 6}.
\end{equation*}
This goes to zero as $R \to \infty$ since $\lambda + \mu < - 6$.
\end{proof}
 
\begin{prop} \label{gradedprop}
Let $\gamma = \sum_{k=0}^7 \gamma_k \in \Omega^{\bullet}_{l+1, \lambda}$, where each $\gamma_k \in \Omega^k_{l+1, \lambda}$, and suppose $(d + \dsm) \gamma = 0$. If $\lambda < -\frac{5}{2}$ (AC) or $\lambda > -\frac{5}{2}$ (CS), then in fact $(d + \dsm) \gamma_k = 0$ for each $k$.
\end{prop}
\begin{proof}
Decomposing the equation $(d + \dsm) \gamma = 0$ into graded components, we have
\begin{equation*}
d\gamma_{k-1} \, = \, - \dsm \gamma_{k+1},
\end{equation*}
where $\gamma_{-1} = \gamma_8 = 0$. Since $\gamma_k \in \Omega^{\bullet}_{l, \lambda}$ and $d \gamma_k \in \Omega^{\bullet}_{l - 1, \lambda - 1}$, the hypothesis on $\lambda$ and Lemma~\ref{IBPlemma} then give
\begin{equation*}
{|| d \gamma_k ||}^2 = \langle \langle d\gamma_k, d\gamma_k \rangle \rangle = - \langle \langle d\gamma_k, \dsm \gamma_{k+2} \rangle \rangle = - \langle \langle \gamma_k, \dsm \dsm \gamma_{k+2} \rangle \rangle = 0,
\end{equation*}
and hence $d\gamma_k = \dsm \gamma_k = 0$ for all $k$.
\end{proof}

\begin{cor} \label{L2cor}
Suppose that $\gamma \in \Omega^k_{l + 2, \lambda}$ and that $\lapm \gamma = 0$. Then we have:
\begin{equation} \label{k07lapcceq}
\text{For } k = 0, 7, \, \text{ if $\lambda < 0$ (AC) or $\lambda > -5$ (CS)}, \, \text{then $\dsm \gamma = 0$ and $d\gamma = 0$}.
\end{equation}
\begin{equation} \label{k16lapcceq}
\text{For } k = 1, 6, \, \text{ if $\lambda < -1$ (AC) or $\lambda > -4$ (CS)}, \, \text{then $\dsm \gamma = 0$ and $d\gamma = 0$}.
\end{equation}
\begin{equation} \label{k25lapcceq}
\text{For } k = 2, 5, \, \text{ if $\lambda < -2$ (AC) or $\lambda > -3$ (CS)}, \, \text{then $\dsm \gamma = 0$ and $d\gamma = 0$}.
\end{equation}
\end{cor}
\begin{proof}
In the AC setting, by Proposition~\ref{homolapexcludeprop} we see that in all three cases, as we decrease $\lambda$ there are no critical rates until at the earliest $\lambda = -3$. So by Theorem~\ref{kernelthm}, in all three cases we can say that $\gamma$ actually lies in $\ker (\lapm)_{\mu}$ for some $\mu < - \frac{5}{2}$. In particular we conclude that $d \gamma$ and $\dsm \gamma$ are both in $\Omega^{\bullet}_{l + 1, \mu -1}$. Then using Lemma~\ref{IBPlemma} we find
\begin{equation*}
0 \, = \, \langle \langle \lapm \gamma, \gamma \rangle \rangle \, = \, \langle \langle d \dsm \gamma + \dsm d \gamma , \gamma \rangle \rangle \, = \, || d \gamma ||^2 + || \dsm \gamma ||^2,
\end{equation*}
so $\dsm \gamma = 0$ and $d\gamma = 0$.

In the CS setting, by Proposition~\ref{homolapexcludeprop} we see that in all three cases, as we increase each $\lambda_i$ there are no critical rates until at the earliest $\lambda_i = -2$. Therefore by Theorem~\ref{kernelthm}, in all three cases we can say that $\omega_k$ actually lies in $\ker (\lapm)_{\mu}$ for some $\mu > -\frac{5}{2}$. The claims now follow just as in the AC case.
\end{proof}

Recall that the \emph{index} of a Fredholm operator $P$ is given by $\ind(P) = \dim(\ker P) - \dim(\ker P^*)$. In order to compute the dimension of the moduli space, we will need to understand how the index of $P$ changes as we cross a critical rate. To this end we require the following definition. Let $\Pc$ denote the operator corresponding to $P$ on the cone.

\begin{defn} \label{homowithlogsdefn}
Let $C$ be a $\G$~cone. For $\lambda \in \R$, we define the space $\mathcal K (\lambda)_{\Pc}$ to be
\begin{equation} \label{homowithlogseq}
\mathcal K(\lambda)_{\Pc} \, = \, \left\{ \begin{array}{l} \gamma = \sum_{j=0}^m (\log r)^j \gamma_j ; \, \text{ such that } \Pc \gamma = 0, \text{where} \\ \text{each } \gamma_j \text{ is a section of $E$ that is homogeneous of order $\lambda$} \end{array} \right\}.
\end{equation}
That is, $\mathcal K(\lambda)_{\Pc}$ consists of the sections of $E$ over $C$ in the kernel of $\Pc$, that are polynomials in $\log r$ whose coefficients are sections of $E$ over $C$ that are homogeneous of order $\lambda$. These spaces are all finite-dimensional. This follows from the ellipticity of $\Pc$ and is discussed in~\cite{LM}.
\end{defn}
The importance of the $\mathcal K(\lambda)_{\Pc}$ spaces is that their dimensions tell us how the index of $P$ changes when we cross a critical rate. The following crucial result appears in general in Lockhart--McOwen~\cite[\S 8]{LM}, and can also be found explicitly for AC manifolds in~\cite[\S 6.3.2]{Lth}.

\begin{thm} \label{indexchangethm}
Let $\nu < \mu$ be two \emph{noncritical rates} for $P$. By Theorem~\ref{fredholmthm}, the maps
\begin{equation*}
P_{l + r, \nu} \, : \, L^2_{l + r, \nu} (E) \, \to \, L^2_{l, \nu - r} (E)
\end{equation*}
and
\begin{equation*}
P_{l + r, \mu} \, : \, L^2_{l + r, \mu} (E) \, \to \, L^2_{l, \mu - r} (E)
\end{equation*}
are both Fredholm. The difference in their indices is given by
\begin{align} \label{ACindexchangeeq}
\ind (P_{l + r, \mu}) - \ind (P_{l + r, \nu}) \, & = \, \sum_{\lambda \in \mathcal D_{\Pc} \cap (\nu , \mu)} \dim \mathcal K(\lambda)_{\Pc}. & \text{(AC)} \\ \label{CSindexchangeeq}
\ind (P_{l + r, \mu}) - \ind (P_{l + r, \nu}) \, & = \, - \sum_{i=1}^n \, \, \sum_{\lambda \in \mathcal D_{\Pci} \cap (\nu_i , \mu_i)} \dim \mathcal K(\lambda)_{\Pci}. & \text{(CS)}
\end{align}
That is, the index of $P$ jumps precisely by the dimension of the space $\mathcal K(\lambda)_{\Pc}$ as we cross each critical rate in the interval $(\nu, \mu)$. This sum is finite because the set $\mathcal D_{\Pc}$ of critical rates has only finitely many points in any bounded interval.
\end{thm}

Let $\mathcal K(\lambda)_{\Pci}$ be as in Definition~\ref{homowithlogsdefn}. The following result can be deduced from the Lockhart--McOwen theory, and appeared in a less general form in~\cite[Proposition 4.27]{Kdesings}.

\begin{prop} \label{asymptoticexpansionprop}
Let $(M, \ph)$ be a $\G$~conifold of rate $\nu$. Suppose that $\beta_1, \beta_2$ are two noncritical rates for $P$, and that $\beta_1 > \beta_2$ (AC) or $\beta_1 < \beta_2$ (CS). Suppose there exists a single critical rate for $P$ between $\beta_1$ and $\beta_2$. This means that for at least one end of the manifold indexed by $i \in \{1, \ldots, n \}$, there is a critical rate $\lambda_0$ for $P_i$ on the cone $C_i$ between $\beta_1$ and $\beta_2$. Let
\begin{equation*}
\mathcal{F}_{\beta_1} \, = \, \{ \gamma \in L^2_{l+r,\beta_1}(E) \, : \, P \gamma \in L^2_{l,\beta_2-r} (F) \}.
\end{equation*}
(In particular, if $\gamma \in \mathcal{F}_{\beta_1}$ then $P \gamma$ decays faster than expected.)  Then there are linear maps
\begin{equation*}
\upsilon : \mathcal{F}_{\beta_1} \to \mathcal{K}(\lambda_0)_{\Pci} \qquad \text{and} \qquad
\vartheta : \mathcal{K}(\lambda_0)_{\Pci} \to L^2_{l+r,\lambda_0+\nu} (\rest{E}{\text{$i^{\text{th}}$ end of $M$}})
\end{equation*}
such that, on the $i^{\text{th}}$ end of $M$, we have
\begin{equation} \label{asymptoticexpansioneqrev}
\gamma - h_i^{-1} \big( \upsilon(\gamma) \big) - \vartheta \big( \upsilon(\gamma) \big) \, \in \, L^2_{l+r,\beta_2} ( \rest{E}{\text{$i^{\text{th}}$ end of $M$}})
\end{equation}
for all $\gamma \in \mathcal{F}_{\beta_1}$.
\end{prop}

Proposition \ref{asymptoticexpansionprop} yields an immediate useful corollary describing how the kernel changes when crossing a critical rate.

\begin{cor} \label{asymptoticexpansioncor}
Consider the setup of Proposition~\ref{asymptoticexpansionprop}. 
There exists a linear map
\begin{equation*}
\eta : \mathcal{K}(\lambda_0)_{\Pci} \to L^2_{l+r,\lambda_0+\nu} ( \rest{E}{\text{$i^{\text{th}}$ end of $M$}} )
\end{equation*}
such that for all $\gamma_1 \in \ker (P)_{\beta_1}$, there exists $\gamma_2 \in \ker(P)_{\beta_2}$ such that, on the $i^{\text{th}}$ end of $M$, we have
\begin{equation} \label{asymptoticexpansioneq}
\gamma_1 - h_i^{-1} \big( \upsilon(\gamma_1) \big) - \eta \big(\upsilon(\gamma_1)\big) \, = \, \gamma_2 \, \in \, \ker(P)_{\beta_2}.
\end{equation}
Note that the term $\gamma_2$, which is in the kernel of $P$ with noncritical rate $\beta_2$, decays \emph{faster} on the end.
\end{cor}
\begin{proof}
If $\gamma_1 \in \ker(P)_{\beta_1}$ then on the $i^{\text{th}}$ end of $M$ we have
\begin{equation*}
\tilde{\gamma}_2 \, = \, \gamma_1 - h_i^{-1} \big( \upsilon(\gamma_1) \big) - \vartheta\big( \upsilon(\gamma_1) \big) \, \in \, L^2_{l+r,\beta_2} ( \rest{E}{\text{$i^{\text{th}}$ end of $M$}} )
\end{equation*}
for linear maps $\upsilon$ and $\vartheta$ as described in Proposition \ref{asymptoticexpansionprop}. The issue is that $P \tilde{\gamma}_2$ is not necessarily zero. We know that $\ker(P)_{\beta_2}$ is finite-dimensional, so 
\begin{equation*}
\mathcal{J}_{\beta_2} \, =\, \{ \rest{\gamma_2}{\text{$i^{\text{th}}$ end of $M$}} \, : \, \gamma_2 \in \ker(P)_{\beta_2} \} \, \subseteq \, L^2_{l+r, \beta_2} ( \rest{E}{\text{$i^{\text{th}}$ end of $M$}} )
\end{equation*}
is a finite-dimensional subspace and therefore has a direct complement $\mathcal{J}'_{\beta_2}$.  Let $\pi$ and $\pi'$ be the projection maps onto $\mathcal{J}_{\beta_2}$ and $\mathcal{J}'_{\beta_2}$, respectively. Then there is $\gamma_2 \in \ker(P)_{\beta_2}$ such that on the $i^{\text{th}}$ end of $M$ we have
\begin{align*}
\gamma_2 & \, = \, \pi (\tilde{\gamma}_2) \, = \, \tilde{\gamma}_2 - \pi' (\tilde{\gamma}_2) \\
& \, = \, \gamma_1 - h_i^{-1} \big( \upsilon(\gamma_1) \big) - \big( \vartheta \big( \upsilon(\gamma_1) \big) + \pi' (\tilde{\gamma}_2) \big) \\
& \, = \, \gamma_1 - h_i^{-1} \big( \upsilon(\gamma_1) \big) - \eta' (\gamma_1),
\end{align*}
where
\begin{equation*}
\eta' (\gamma_1) \, = \, \vartheta \big( \upsilon(\gamma_1) \big) + \pi' \big( \gamma_1 - h_i^{-1} \big( \upsilon(\gamma_1) \big) - \vartheta\big( \upsilon(\gamma_1) \big) \big).
\end{equation*}
We see that $\eta' (\gamma_1)$ is linear in $\gamma_1$ since $\upsilon$, $\vartheta$, and $\pi'$ are all linear maps. Moreover, if $\upsilon (\gamma_1) = 0$, we have $P \tilde{\gamma}_2 = P\gamma_1 = 0$, so $\eta'(\gamma_1) = 0$. Hence, $\eta'$ only depends on $\upsilon (\gamma_1)$, and thus we can define $\eta \big(\upsilon(\gamma_1)\big) = \eta' (\gamma_1)$ and the proof is complete.
\end{proof}

\begin{rmk} \label{kernelchangermk}
Recall Theorem~\ref{kernelthm} says that the kernel will only change as we cross a critical rate. The essential content of Corollary~\ref{asymptoticexpansioncor} is that, when the kernel does indeed change as we cross a critical rate $\lambda$, any section which is added or removed from the kernel must be asymptotic at the $i^{\text{th}}$ end to an element of $\mathcal K(\lambda)_{\Pci}$.
\end{rmk}

We can now deduce a further corollary to Proposition \ref{asymptoticexpansionprop} which will be crucial for understanding the change in index as we cross a critical rate.

\begin{cor} \label{asymptoticexpansioncor2}
Consider the setup of Proposition~\ref{asymptoticexpansionprop} and suppose further that $| \beta_2 - \lambda_0 | < |\nu|$. Let $\chi_i$ be a smooth cutoff function on $M$ which is $1$ on the $i^{\text{th}}$ end and $0$ on all other ends, so that $\chi_i \mathcal{K}(\lambda_0)_{\Pci}$ can be viewed as a subspace of $L^2_{l+r,\beta_1}(E)$. Define the map $\tilde P$ to be the restriction of $P$ to the subspace $L^2_{l+r, \beta_2} (E) + \chi_i \mathcal{K}(\lambda_0)_{\Pci}$ of $L^2_{l+r, \beta_1} (E)$. Then the two linear maps
\begin{align}
P & \, : \, L^2_{l+r, \beta_1} (E) \to L^2_{l, \beta_1-r} (F), \\
\tilde P & \, : \, L^2_{l+r, \beta_2} (E) + \chi_i \mathcal{K}(\lambda_0)_{\Pci} \to L^2_{l,\beta_2-r} (F)
\end{align}
satisfy $\ker(P) = \ker(\tilde P)$ and $\coker(P) \cong \coker(\tilde P)$.
\end{cor}
\begin{proof}
We know that if $\upsilon \in \mathcal{K}(\lambda_0)_{\Pci}$, then we can bound the rate of $P (\chi_i \upsilon)$ by $\lambda_0 - r$. However, since $P_{C_i} (\upsilon) = 0$ and $P$ is asymptotic to $P_{C_i}$, we see that $P (\chi_i \upsilon)$ has a bound on its rate by $\lambda_0 + \nu - r$. The assumption $|\beta_2 - \lambda_0| < |\nu|$ then ensures that $P (\chi_i \upsilon) \in L^2_{l, \beta_2-r} (F)$, so $\tilde P$ indeed maps into the space claimed. The assumption $|\beta_2 - \lambda_0| < |\nu|$ also guarantees that $L^2_{l+r,\lambda_0+\nu} \subseteq L^2_{l+r,\beta_2}$. The fact that $\ker(P) = \ker(\tilde P)$ then follows immediately from Corollary~\ref{asymptoticexpansioncor}.

Consider the natural linear map $\pi : \coker(\tilde P) \to \coker(P)$. Let $B$ denote the closure in $L^2_{l,\beta_2-r} (F)$ of the subspace $P( L^2_{l+r, \beta_2} (E) + \chi_i \mathcal{K}(\lambda_0)_{\Pci} )$. Since the restriction of $\tilde P$ to $L^2_{l+r, \beta_2} (E)$ is Fredholm, the quotient space $L^2_{l,\beta_2-r} (F) / P(L^2_{l+r, \beta_2} (E))$ is finite-dimensional. It follows that the quotient $L^2_{l,\beta_2-r} (F) / B$ is also finite-dimensional. Because $L^2_{l,\beta_2-r} (F)$ is a dense subspace of $L^2_{l,\beta_1-r} (F)$, we deduce that $L^2_{l,\beta_2-r} (F) / B$ is dense in $L^2_{l,\beta_1-r} (F) / B$, and since the first space is finite-dimensional, the two spaces are equal. Hence $\coker(\tilde P) = L^2_{l,\beta_2-r} (F) / B = L^2_{l,\beta_1-r} (F) / B$. It is now evident that $\pi : L^2_{l,\beta_1-r} (F) / B \to L^2_{l,\beta_1-r} (F) / P(L^2_{l+r, \beta_1} (E)) = \coker(P)$ is surjective. Suppose $[\xi] \in \ker \pi$. Then $\xi = P(\gamma)$ for some $\gamma \in L^2_{l+r,\beta_1} (E)$. This means that $\gamma \in \mathcal{F}_{\beta_1}$, in the notation of Proposition~\ref{asymptoticexpansionprop}, and thus $\gamma \in L^2_{l+r,\beta_2} (E) + \chi_i \mathcal{K}(\lambda_0)_{\Pci}$. We deduce that $[\xi] = 0$ and hence $\pi$ is injective and thus an isomorphism.
\end{proof}

\begin{rmk} \label{indexchangeformularmk}
The reader should observe that the index change formula in Theorem~\ref{indexchangethm} actually follows from Corollary~\ref{asymptoticexpansioncor2}.
\end{rmk}

\begin{defn} \label{Hdefn}
Let $0 \leq k \leq 7$. We define the space $\mathcal H^k_{\lambda}$ to be the space of closed and coclosed $k$-forms of rate $\lambda$ on the ends. Explicitly, we have
\begin{equation*}
\mathcal H^k_{\lambda} \, = \, \{ \gamma \in \Omega^k_{l, \lambda}; \, d \gamma = 0, \, \dsm \gamma = 0 \}.
\end{equation*}
This definition makes sense for any $l \geq 0$, since $\mathcal H^k_{\lambda}$ is a subspace of $\ker(d + \dsm)_{\lambda}$ where $d + \dsm$ is acting on $\Omega^{\bullet}_{l, \lambda}$, and thus by Theorem~\ref{kernelthm} and Remark~\ref{fredholmrmk} the space $\mathcal H^k_{\lambda}$ is independent of $l$ and finite-dimensional.
\end{defn}
Given a form $\gamma \in \Omega^k$, its \emph{pure-type} components are the components in the decompositions~\eqref{lambda3eq} and~\eqref{lambda2eq} into $\G$~representations.
\begin{lemma} \label{Hpuretypelemma}
Suppose that $\lambda < - \frac{5}{2}$ (AC) or $\lambda > -\frac{5}{2}$ (CS). Then the pure-type components of an element $\gamma$ in $\mathcal H^k_{\lambda}$ are each closed and coclosed.
\end{lemma}
\begin{proof}
An element $\gamma$ of $\mathcal H^k_{\lambda}$ is closed and coclosed, hence harmonic. By Remark~\ref{laplaciansrmk} the projections onto the pure-type forms commute with the Laplacian, and thus the pure-type components of $\gamma$ are each harmonic. Then it follows from Corollary~\ref{L2cor} that the pure-type components of $\gamma$ are each closed and coclosed.
\end{proof}

For any $0 \leq k \leq 7$, the space $\mathcal H^k_{\lambda}$ is always a subspace of $\ker(d + \dsm)_{\lambda}$. Consider varying the rate $\lambda$ in the direction in which the space $\mathcal H^k_{\lambda}$ potentially gets larger. If new elements are added to $\mathcal H^k_{\lambda}$, then new elements are added to $\ker(d + \dsm)_{\lambda}$, and thus this can only happen when we cross a rate $\lambda_0$ which is critical for $d + \dsm$. The next lemma says that new elements are added to $\mathcal H^k_{\lambda}$ only when there exist closed and coclosed $k$-forms of rate $\lambda_0$ on the asymptotic cones.

\begin{lemma} \label{kernelchangelemma}
Let $\lambda_0$ be a critical rate for $d + \dsm$. For $\e > 0$ define
\begin{align*}
\lambda_+ & = \begin{cases} \lambda_0 + \e \quad \text{ (AC)}, \\ \lambda_0 - \e \quad \text{ (CS)}, \end{cases} &
\lambda_- & = \begin{cases} \lambda_0 - \e \quad \text{ (AC)}, \\ \lambda_0 + \e \quad \text{ (CS)}. \end{cases}
\end{align*}
Thus in either case $\lambda_+$ is a \emph{slower} rate of decay and $\lambda_-$ is a \emph{faster} rate of decay on the ends. Further choose $\e$ small enough so that $2 \e = | \lambda_+ - \lambda_- | < | \nu |$, where $\nu$ is the rate of the $\G$~conifold, and such that the interval $(\lambda_0 - \e, \lambda_0 + \e)$ contains no other critical rates for $d + \dsm$. Let $\gamma_+ \in \mathcal H^k_{\lambda_+}$. From equation~\eqref{asymptoticexpansioneq} we know that there exist elements $\upsilon \in \mathcal K(\lambda_0)_{d +\dsci}$, $\eta \in L^2_{l+r,\lambda_0+\nu}  ( \rest{E}{\text{$i^{\text{th}}$ end of $M$}} )$, and $\gamma_- \in \ker(d+d^*_M)_{\lambda_-}$ such that, on the $i^{\text{th}}$ end, 
\begin{equation*}
\gamma_+ - (h_i^{-1})^* (\upsilon) - \eta \, = \, \gamma_- \, \in \,  \ker(d+d^*_M)_{\lambda_-}.
\end{equation*}
Then the form $\upsilon$ is a closed and coclosed $k$-form on the cone $C_i$.
\end{lemma}
\begin{proof}
We give the proof in the AC case. The CS case is identical with the inequalities reversed. Since there is only one end, we drop the index $i$. Let $\gamma_+ \in \mathcal H^k_{\lambda_+}$. Let $\upsilon_m$ and $\eta_m$ denote the degree $m$ components of $\upsilon$ and $\eta$, for $m \neq k$. Since $\gamma$ is a pure degree $k$-form, we find that
\begin{equation*}
(h^{-1})^* (\upsilon_m) + \eta_m \, \text { is \emph{at most} $O(\varrho^{\lambda_-})$},
\end{equation*}
and $\eta_m$ is \emph{at most} $O(\varrho^{\lambda_+ + \nu})$. Our hypothesis that $\lambda_+ + \nu < \lambda_-$ then allows us to conclude that $\upsilon_m$ is \emph{at most} $O(\varrho^{\lambda_-})$. However, because the form $\upsilon$ lies in $\mathcal K(\lambda_0)_{d + \dsc}$, we know in fact that if $\upsilon_m \neq 0$ then it is \emph{at least} $O(\varrho^{\lambda_0})$, and $\lambda_0 > \lambda_-$. Thus we must have $\upsilon_m = 0$ for all $m \neq k$. Since $\upsilon \in \mathcal K(\lambda_0)_{d + \dsc}$, we conclude that $\upsilon$ is a closed and coclosed $k$-form on the cone $C$.
\end{proof}

\begin{cor} \label{kernelchangecor}
Let $0 \leq k \leq 7$. Let $\lambda, \mu$ be two noncritical rates for $d + \dsm$. If there are no closed and coclosed homogeneous $k$-forms on the asymptotic cones of $M$ of any rates between $\lambda$ and $\mu$, then $\mathcal H^k_{\lambda} = \mathcal H^k_{\mu}$. In particular, we have
\begin{equation} \label{kernel34eq}
\mathcal H^k_{\lambda} \, = \, \mathcal H^k_{\mu} \qquad \text{ if $\lambda, \mu \in (-4, -3)$}.
\end{equation}
\end{cor}
\begin{proof}
We observe first that Lemma~\ref{kernelchangelemma}, together with Corollary~\ref{asymptoticexpansioncor} and Lemma~\ref{nologslemma}, says that the space $\mathcal H^k_{\lambda}$ will only change when we cross a rate $\lambda_0$ for which there exists a homogeneous closed and coclosed $k$-form of rate $\lambda_0$ on some asymptotic cone $C_i$ of $M$. This proves the first statement. Equation~\eqref{kernel34eq} now follows from Proposition~\ref{homoddstarexcludeprop}, which says that there are no nontrivial homogeneous closed and coclosed $k$-forms of any rate in $(-4,-3)$ for any $\G$~cone $C$.
\end{proof}

\subsection{Hodge theoretic results for \texorpdfstring{$k$}{k}-forms} \label{Hodgesec}

In this section we derive Hodge theoretic results for $\G$~conifolds. Several of these results are used later in Section~\ref{conifoldmodulisec}. While we do not need the full strength of all of these results in the present paper, we nevertheless attempt to give a comprehensive treatment for the sake of future applications.

In order to avoid the proliferation of too much notation, from now on all the $M$ subscripts will be dropped. It will be understood that the Hodge star operator $\st$, the covariant derivative $\nabla$, the Hodge Laplacian $\Delta$, the coderivative $d^*$, the projection maps $\pi_k$, and the maps $\Lp$ and $\Qp$ defined in Lemma~\ref{quadlemma} will all be taken with respect to a fixed $\G$~structure $\ph = \phm$ on $M$. Furthermore, we often have to deal with sections of a vector bundle $E$ that are \emph{smooth}, but have particular growth on the ends. Therefore we use the following notation:
\begin{equation*}
C^{\infty}_{\lambda}(E) \, = \, \{ \gamma \in C^{\infty}(E); \, | \nabla^j s | = O(\varrho^{\lambda - j}) \, \forall j \geq 0 \}.
\end{equation*}

In this section we use Fredholm theory of the operator $d + d^*$ to determine Hodge theoretic results for the spaces of $k$-forms with specified rates of decay on the ends. These results are used repeatedly throughout the sequel. Recall from Definition~\ref{Hdefn} that $\mathcal H^k_{\lambda}$ is the space of closed and coclosed $k$-forms of rate $\lambda$ on the ends.

\begin{lemma} \label{Hodgeclosedlemma}
Suppose $\lambda+1$ is a noncritical rate for $d+d^*$. Then
\begin{equation*} 
d (\Omega^{k-1}_{l+1, \lambda+1}) + d^* (\Omega^{k+1}_{l+1, \lambda+1}) \, \subseteq \,  \Omega^k_{l, \lambda}
\end{equation*}
is a \emph{closed} subspace of \emph{finite codimension}.
\end{lemma}
\begin{proof}
Since $\lambda+1$ is noncritical, by Theorem~\ref{fredholmthm} we know that
\begin{equation*}
(d+d^*) (\Omega^\bullet_{l+1, \lambda+1}) \, \subseteq \, \Omega^\bullet_{l, \lambda}
\end{equation*}
is a closed subspace of finite codimension. Since
\begin{equation*}
(d+d^*) (\Omega^\bullet_{l+1, \lambda+1}) \, \subseteq \, d (\Omega^\bullet_{l+1, \lambda+1}) + d^* (\Omega^\bullet_{l+1, \lambda+1}),
\end{equation*}
we see that the latter space must also have finite codimension in $\Omega^{\bullet}_{l, \lambda}$.  Therefore,
\begin{equation*}
d (\Omega^{k-1}_{l+1, \lambda+1}) + d^* (\Omega^{k+1}_{l+1, \lambda+1}) \, \subseteq \, \Omega^k_{l, \lambda}
\end{equation*}
has finite codimension. However, the former space is the image of the continuous map $(\alpha,\beta) \mapsto d \alpha + d^* \beta$ from the Banach space $\Omega^{k-1}_{l+1, \lambda+1} \oplus \Omega^{k+1}_{l+1, \lambda+1}$ to the Banach space $\Omega^k_{l, \lambda}$ and hence 
by~\cite[Chapter XV, Corollary 1.8]{LangRFA} it is closed.
\end{proof}

We will see in Section~\ref{conifoldmodulisec} that several of the ingredients in the proof of our main theorem have two distinct flavours, depending on which of the following two situations we are considering:
\begin{itemize}
\item In the $L^2$ setting (when $\nu \leq -\frac{7}{2}$ in the AC case or when $\nu \geq -\frac{7}{2}$ in the CS case), many of the analytic arguments are simple, but this is precisely the regime in which \emph{obstructions} occur.
\item In the complementary regime (when $\nu > -\frac{7}{2}$ in the AC case or when $\nu < -\frac{7}{2}$ in the CS case), we are \emph{not} in $L^2$, and because of this most of the analytic arguments are more delicate. For example, we need the surjectivity of the Dirac operator to prove our infinitesimal slice theorem. However, this regime has the nice feature of having an \emph{unobstructed} deformation theory.
\end{itemize}
The Hodge theory results we establish are slightly different for these two settings, so we state and prove them separately.

We begin with the $L^2$ setting. Recall that $\Omega^k_{l, \lambda} = L^2_{l, \lambda}(\Lambda^k (T^* M))$ is a Hilbert space.
\begin{prop} \label{HodgedecompositionpropL2}
Suppose $\lambda + 1$ is noncritical for $d + d^*$. Let $0 \leq k \leq 7$. In the $L^2$ setting (when $\lambda \leq -\frac{7}{2}$ for the AC case or when $\lambda \geq - \frac{7}{2}$ for the CS case), there exists a decomposition
\begin{equation} \label{HodgeL2eq}
\Omega^{k}_{l, \lambda} \, = \, d( \Omega^{k-1}_{l+1, \lambda+1}) \oplus d^*( \Omega^{k+1}_{l+1, \lambda+1} ) \oplus \mathcal H^k_{\lambda} \oplus W^k_{l, \lambda}.
\end{equation}
Here $W^k_{l, \lambda}$ is a finite-dimensional space. Moreover, the spaces $d (\Omega^{k-1}_{l+1, \lambda+1})$, $d^* (\Omega^{k+1}_{l+1, \lambda+1})$, and $\mathcal H^k_{\lambda}$ are $L^2$-orthogonal to each other, and  $W^k_{l, \lambda}$ can be defined as the $L^2_{l,\lambda}$ orthogonal complement of $d (\Omega^{k-1}_{l+1, \lambda+1}) \oplus d^* (\Omega^{k+1}_{l+1, \lambda+1}) \oplus \mathcal H^k_{\lambda}$.
\end{prop}
\begin{proof}
Since we are in $L^2$, integration by parts is valid so if $\gamma \in \mathcal H^k_{\lambda}$ we have that
\begin{equation*}
\langle \langle d \alpha, d^* \beta \rangle \rangle_{L^2} \, = \, \langle \langle d \alpha, \gamma \rangle \rangle_{L^2} \, = \, \langle \langle d^* \beta, \gamma \rangle \rangle_{L^2} \, = \, 0.
\end{equation*}
Thus the spaces $d (\Omega^{k-1}_{l+1, \lambda+1})$, $d^* (\Omega^{k+1}_{l+1, \lambda+1})$, and $\mathcal H^k_{\lambda}$ are $L^2$-orthogonal to each other. By Lemma~\ref{Hodgeclosedlemma} we know that
\begin{equation*}
d (\Omega^{k-1}_{l+1, \lambda+1}) \oplus d^* (\Omega^{k+1}_{l+1, \lambda+1}) \oplus \mathcal H^k_{\lambda}
\end{equation*}
is a closed subspace of $\Omega^k_{l, \lambda}$ of finite codimension. Hence a finite dimensional complement $W^k_{l, \lambda}$ exists. Since $\Omega^k_{l, \lambda}$ is a Hilbert space, we know that the orthogonal complement of a closed subspace with respect to the Hilbert space inner product is a direct complement, so we can uniquely define $W^k_{l, \lambda}$ as claimed.
\end{proof}

\begin{cor} \label{Wkdimensioncor}
Consider the setup of Proposition~\ref{HodgedecompositionpropL2}. The dimension of the space $W^k_{l, \lambda}$ is given by
\begin{equation} \label{Wkdimensioneq}
\dim W^k_{l, \lambda} \, = \, \dim \mathcal H^k_{-7 - \lambda} - \dim \mathcal H^k_{\lambda}.
\end{equation}
In particular, the dimension of $W^k_{l, \lambda}$ is independent of $l \geq 0$.
\end{cor}
\begin{proof}
By Remark~\ref{annihilatorrmk}, since the subspace $d( \Omega^{k-1}_{l+1, \lambda+1}) \oplus d^*( \Omega^{k+1}_{l+1, \lambda+1} )$ of $\Omega^k_{l, \lambda}$ is closed, any direct complement of it will be isomorphic to its annihilator in the dual space. It is trivial to see that this annihilator is the space $\mathcal H^k_{-7-\lambda}$ of closed and coclosed forms of the dual rate $-7 - \lambda$, using the fact that the dual space of $\Omega^k_{l,\lambda}$ is $\Omega^k_{-l, - 7 - \lambda}$ and Remark~\ref{sloppyrmk}. Thus $\dim (\mathcal H^k_{\lambda} \oplus W^k_{l, \lambda}) = \dim \mathcal H^k_{-7-\lambda}$, and hence equation~\eqref{Wkdimensioneq} follows immediately.
\end{proof}

The next result is the analogue to Proposition~\ref{HodgedecompositionpropL2} for the non-$L^2$ setting.
\begin{prop} \label{HodgedecompositionpropnonL2}
Suppose $\lambda + 1$ is noncritical for $d + d^*$. Let $0 \leq k \leq 7$. In the non-$L^2$ setting (when $\lambda > -\frac{7}{2}$ for the AC case or when $\lambda < - \frac{7}{2}$ for the CS case), there exists a decomposition
\begin{equation} \label{HodgenonL2eq}
\Omega^k_{l, \lambda} \, = \, A^k_{l, \lambda} \oplus B^k_{l, \lambda} \oplus \mathcal H^k_{- 7 - \lambda} \, = \, A^k_{l, \lambda} \oplus \mathcal H^k_{\lambda},
\end{equation}
where $B^k_{l, \lambda}$ is the intersection of $\mathcal H^k_{\lambda}$ with the Banach space $d( \Omega^{k-1}_{l+1, \lambda+1}) + d^*( \Omega^{k+1}_{l+1, \lambda+1} )$, and $A^k_{l, \lambda}$ is a topological complement of $B^k_{l, \lambda}$ in $d( \Omega^{k-1}_{l+1, \lambda+1}) + d^*( \Omega^{k+1}_{l+1, \lambda+1} )$, and thus a closed subspace. Moreover, the intersection of $A^k_{l, \lambda}$, the image of $d$, and the image of $d^*$ is zero.
\end{prop}
\begin{proof}
We prove the AC case ($\lambda > -\frac{7}{2}$). The CS case is identical with all inequalities reversed. Note first that by Lemma~\ref{Hodgeclosedlemma}, we have that $d( \Omega^{k-1}_{l+1, \lambda+1}) + d^*( \Omega^{k+1}_{l+1, \lambda+1} )$ is closed in $\Omega^k_{l, \lambda}$. The argument from the proof of Corollary~\ref{Wkdimensioncor} applies here, so any topological complement of $d( \Omega^{k-1}_{l+1, \lambda+1}) + d^*( \Omega^{k+1}_{l+1, \lambda+1} )$ will be isomorphic to $\mathcal H^k_{- 7 - \lambda}$. However, since $\lambda > - \frac{7}{2}$ is equivalent to $- 7 - \lambda < \lambda$, we see that $\mathcal H^k_{- 7 - \lambda}$ actually lies in $\Omega^k_{l, \lambda}$ and thus we can write
\begin{equation*}
\Omega^k_{l, \lambda} \, = \, \left( d( \Omega^{k-1}_{l+1, \lambda+1}) + d^*( \Omega^{k+1}_{l+1, \lambda+1} ) \right) \oplus \mathcal H^k_{- 7 - \lambda}.
\end{equation*}
Let $B^k_{l, \lambda}$ be the intersection of $\mathcal H^k_{\lambda}$ with $d( \Omega^{k-1}_{l+1, \lambda+1}) + d^*( \Omega^{k+1}_{l+1, \lambda+1})$ and let $A^k_{l, \lambda}$ be a topological complement of the finite-dimensional space $B^k_{l, \lambda}$ in the Banach space $d( \Omega^{k-1}_{l+1, \lambda+1}) + d^*( \Omega^{k+1}_{l+1, \lambda+1})$. We thus have
\begin{equation*}
\Omega^k_{l, \lambda} \, = \, A^k_{l, \lambda} \oplus B^k_{l, \lambda} \oplus \mathcal H^k_{-7 - \lambda}
\end{equation*}
with $\mathcal H^k_{\lambda} = B^k_{l, \lambda} \oplus \mathcal H^k_{-7 - \lambda}$. Finally, suppose $\eta = d \alpha = d^* \beta$ in $A^k_{l, \lambda}$. Then $\eta$ is both closed and coclosed, and thus lies in $\mathcal H^k_{\lambda}$. Since $A^k_{l, \lambda} \cap \mathcal H^k_{\lambda} = \{ 0 \}$, we have $\eta = 0$.
\end{proof}

\begin{rmk} \label{Hodgedecompositionoverlaprmk}
In fact, there is an ``overlap region'' $\lambda \in (-4, -3)$ in which \emph{both} decompositions~\eqref{HodgeL2eq} and~\eqref{HodgenonL2eq} agree. We show this in the AC case. The CS case is identical with all inequalities reversed. If $\lambda \in (-4, -\frac{7}{2}]$, then $- 7 - \lambda \in [-\frac{7}{2}, -3)$, and hence Corollary~\ref{kernelchangecor} and Corollary~\ref{Wkdimensioncor} tell us that $W^k_{l, \lambda} = 0$. Thus Proposition~\ref{HodgedecompositionpropL2} then says 
\begin{equation*}
\Omega^{k}_{l, \lambda} \, = \, d( \Omega^{k-1}_{l+1, \lambda+1}) \oplus d^*( \Omega^{k+1}_{l+1, \lambda+1} ) \oplus \mathcal H^k_{\lambda}.
\end{equation*}
Similarly, if $\lambda \in (-\frac{7}{2}, -3)$, then $- 7 - \lambda \in (-4, -\frac{7}{2})$, and hence Corollary~\ref{kernelchangecor} and $\mathcal H^k_{\lambda} = B^k_{l, \lambda} \oplus \mathcal H^k_{-7 - \lambda}$ tell us that $B^k_{l, \lambda} = 0$. Therefore in this case Proposition~\ref{HodgedecompositionpropnonL2} says $A^k_{l, \lambda} = d( \Omega^{k-1}_{l+1, \lambda+1}) \oplus d^*( \Omega^{k+1}_{l+1, \lambda+1} )$, and
\begin{equation*}
\Omega^{k}_{l, \lambda} \, = \, d( \Omega^{k-1}_{l+1, \lambda+1}) \oplus d^*( \Omega^{k+1}_{l+1, \lambda+1} ) \oplus \mathcal H^k_{\lambda}.
\end{equation*}
Since $\mathcal H^k_{\lambda}$ is independent of $\lambda \in (-4,-3)$, we have established that in the interval $(-4,-3)$ the two decompositions~\eqref{HodgeL2eq} and~\eqref{HodgenonL2eq} are identical.
\end{rmk}

We are now ready for the main result of this section, which is a Hodge-type decomposition for $k$-forms on $\G$~conifolds.
\begin{thm} \label{kformsHodgethm}
Let $(M, \ph)$ be a $\G$~conifold of rate $\nu$, and suppose that $\nu + 1$ is a noncritical rate for $d + d^*$. Let $\eta \in \Omega^k_{l, \nu}$.
\begin{itemize}
\item In the $L^2$ setting (when $\nu \leq -\frac{7}{2}$ for the AC case or when $\nu \geq -\frac{7}{2}$ in the CS case), we can express the $k$-form $\eta$, in a \emph{unique way}, as
\begin{equation} \label{kformsHodgeL2eq}
\eta \, = \, d \alpha + d^* \beta + \kappa + \gamma,
\end{equation}
where $\alpha \in \Omega^{k-1}_{l+1, \nu+1}$, $\beta \in \Omega^{k+1}_{l+1, \nu+1}$, $\kappa \in \mathcal H^k_{\nu}$, and $\gamma$ is in $W^k_{l, \nu}$. Moreover, if $d \eta = 0$, then we can express $\eta$ in a unique way as
\begin{equation} \label{kformsHodgeL2eq2}
\eta \, = \, d \alpha + \kappa + \delta
\end{equation}
where $\alpha \in \Omega^{k-1}_{l+1, \nu+1}$, $\kappa \in \mathcal H^k_{\nu}$, and
\begin{equation*}
\delta \in U^k_{\nu} \, = \, \{ d^* \beta + \gamma \, ; \, \beta \in \Omega^{k+1}_{l+1, \nu+1}, \, \gamma \in W^k_{l, \nu}, \, d (d^*\beta + \gamma) = 0 \}.
\end{equation*}
Moreover, $U^k_{\nu}$ is finite-dimensional and $\dim U^k_{\nu} \leq \dim W^k_{l, \nu}$.
\item In the non-$L^2$ setting (when $\nu > -\frac{7}{2}$ for the AC case or when $\nu < -\frac{7}{2}$ in the CS case), we can express the $k$-form $\eta$, in a \emph{unique way}, as
\begin{equation} \label{kformsHodgenonL2eq}
\eta \, = \, d \alpha + d^* \beta + \kappa,
\end{equation}
where $d \alpha + d^* \beta \in A^k_{l, \nu}$ and $\kappa \in \mathcal H^k_{\nu}$. Moreover, if $d \eta = 0$, then we can actually write
\begin{equation} \label{kformsHodgenonL2eq2}
\eta \, = \, d \alpha + \tilde \kappa
\end{equation}
for some $\tilde \kappa \in \mathcal H^k_{\nu}$.
\end{itemize}
Finally, in the interval $\nu \in (-4, -3)$, both cases can be applied and they agree.
\end{thm}
\begin{proof}
Equations~\eqref{kformsHodgeL2eq} and~\eqref{kformsHodgenonL2eq} are immediate from Propositions~\ref{HodgedecompositionpropL2} and~\ref{HodgedecompositionpropnonL2}. Now suppose $d \eta =0$.

In the $L^2$ setting, Proposition~\ref{HodgedecompositionpropL2} says that we can write $\eta$ uniquely as
\begin{equation*}
\eta \, = \, d \alpha + d^* \beta + \kappa + \gamma
\end{equation*}
for $\alpha \in \Omega^{k-1}_{l+1, \nu+1}$, $\beta \in \Omega^{k+1}_{l+1, \nu+1}$, $\kappa \in \mathcal H^k_{\nu}$, and $\gamma \in W^k_{l, \nu}$. Since $d \eta = 0$ and $d\alpha + \kappa$ is closed we see that $d(d^* \beta + \gamma) = 0$. Moreover, if $d d^* \beta = d d^* \beta' = - d \gamma$ then $d^* (\beta - \beta')$ is closed and coclosed and thus lies in $\mathcal H^k_{\nu}$, which implies that $d^* \beta = d^*\beta'$ since $d^* (\Omega^{k+1}_{l+1, \nu+1})$ is $L^2$-orthogonal to $\mathcal H^k_{\nu}$. Thus any $\gamma \in W^k_{l, \nu}$ can be paired with at most one $d^* \beta$ so that $\gamma + d^* \beta$ is closed. We have established~\eqref{kformsHodgeL2eq2}.

In the non-$L^2$ setting corresponding to equation~\eqref{kformsHodgenonL2eq} we get $d(d^* \beta) = 0$, so $\tilde \kappa = d^* \beta + \kappa$ is both closed and coclosed, and thus lies in $\mathcal H^k_{\nu}$, establishing~\eqref{kformsHodgenonL2eq2}.

The fact that both cases agree for $\nu \in (-4,-3)$ is immediate from Remark~\ref{Hodgedecompositionoverlaprmk}.
\end{proof}

\subsection{The obstruction space and a special index-change theorem} \label{specialindexchangesec}

In our study of the moduli space of $\G$~conifolds in Section~\ref{conifoldmodulisec}, we will need to consider the operator
\begin{equation} \label{Pdefneq}
\dd^k_{l,\lambda} \, = \, \rest{(d + d^*)_{l,\lambda}}{\Omega^k_{l, \lambda}} \, : \, \Omega^k_{l, \lambda} \to d(\Omega^k_{l, \lambda}) + d^*(\Omega^k_{l, \lambda}).
\end{equation}
For simplicity we will often use the symbol $\dd^k_{l, \lambda}$ to denote this map, which is just $(d + d^*)_{l, \lambda}$ with domain restricted to $\Omega^k_{l, \lambda}$ and codomain restricted to $d(\Omega^k_{l, \lambda}) + d^*(\Omega^k_{l, \lambda})$. One of the principal results we will need is a refined version of the ``index-change'' formula of Theorem~\ref{indexchangethm} for the operator $\dd^3_{l, \lambda}$ defined in~\eqref{Pdefneq} for $k = 3$. Note that Theorem~\ref{indexchangethm} \emph{does not} directly apply to this operator $\dd^k_{l, \lambda}$, because although (for generic rates $\lambda$) we show in Proposition~\ref{Pindexchangeprop} that it is Fredholm, it is clearly \emph{not} elliptic.

\begin{defn} \label{KPspacedefn}
Let $\lambda = (\lambda_1, \ldots, \lambda_n) \in \R^n$ be an $n$-tuple of rates, with $n=1$ in the AC case as usual. Suppose there exists a nontrivial \emph{closed and coclosed} $k$-form $\upsilon_i$ on the cone $C_i$, homogeneous of order $\lambda_i$ for some $i = 1, \ldots, n$. Then we say $\lambda$ is a \emph{critical rate} for the operator
\begin{equation*}
\dd^k_{l, \lambda} = \rest{(d + d^*)_{l, \lambda}}{\Omega^k_{l, \lambda}} :  \Omega^k_{l, \lambda} \to d(\Omega^k_{l, \lambda}) + d^*(\Omega^k_{l, \lambda})
\end{equation*}
on the conifold $M$. The critical rates for $\dd^k$ are thus a subset of the critical rates for the operator $d + d^* : \Omega^{\bullet}_{l, \lambda} \to \Omega^{\bullet}_{l - 1, \lambda - 1}$. From Lemma~\ref{nologslemma}, we know that there are no $\log r$ terms for the operator $d + d^*$ on the cone, so we can use the notation of Definition~\ref{homowithlogsdefn} to define the space $\mathcal K(\lambda_i)_{\ddci^k}$ to be \emph{exactly} the space of such forms $\upsilon_i$. That is,
\begin{equation} \label{KPdefneq}
\mathcal K(\lambda_i)_{\ddci^k} \, = \, \left\{ \gamma \in \Gamma( \Lambda^k(T^* C_i)) ; \, d \gamma = 0, \, \dsci \gamma = 0, \, \gamma \text{ is homogeneous of order $\lambda_i$} \right\}.
\end{equation}
\end{defn}

\begin{ex} \label{criticalratesexamples}
Consider the operator $\dd^3_{l, \lambda}$ on the AC $\G$~manifolds of Bryant--Salamon discussed in Example~\ref{BSexamples}. By Remark~\ref{ddstarhomokernelrmk}, we see that $\lambda = -3$ is a critical rate for $\dd^3_{l, \lambda}$ if and only if $b^3(\Sigma)$ is nonzero, which by~\eqref{NKBettinumberseq} occurs only for $\spi (S^3)$. Similarly, $\lambda = -4$ is a critical rate for $\dd^3_{l, \lambda}$ if and only if $b^2(\Sigma) = b^4(\Sigma)$ is nonzero, which by~\eqref{NKBettinumberseq} occurs only for $\Lambda^2_-(\C \PR^2)$ and $\Lambda^2_-(S^4)$. Hence, in all three cases the rate $\nu$ of convergence at infinity to the asymptotic cone is a critical rate for $\dd^3_{l, \lambda}$.
\end{ex}

The next lemma shows that elements in the space $\mathcal K(\lambda)_{\ddci^k}$ correspond to solutions to a certain system of eigenvalue equations on the link $\Sigma_i$ of the cone $C_i$.
\begin{lemma} \label{KPspacelemma}
Let $\gamma = r^{\lambda}(r^{k-1} dr \wedge \alpha_{k-1} + r^k \alpha_k)$ be a $k$-form on the cone $C = (0, \infty) \times \Sigma$, homogeneous of order $\lambda$, where $\alpha_{k-1} \in \Omega^{k-1}(\Sigma)$ and $\alpha_k \in \Omega^k(\Sigma)$. Then $(d + \dsc) \gamma = 0$ if and only if
\begin{equation} \label{KPeigenvalueeqs}
\begin{aligned}
\ds \alpha_{k-1} & = (\lambda + k) \alpha_k, & \quad  \ds \alpha_k & = 0, \\
\dss \alpha_{k-1} & = 0, & \quad \dss \alpha_k & = (\lambda - k + 7) \alpha_{k-1}.
\end{aligned}
\end{equation}
\end{lemma}
\begin{proof}
This is immediate from~\eqref{homoddseq}.
\end{proof}

We now proceed to discuss the \emph{obstruction space} $\mathcal O^k_{l, \lambda}$ for our deformation problem. To simplify notation, we define the linear spaces
\begin{equation} \label{Ydefneq}
\begin{aligned}
\mathcal Y \, & = \, d(\Omega^k_{l, \nu}) + d^*(\Omega^k_{l,\nu}), \\
\mathcal Y_0 \, & = \, (d + d^*)(\Omega^k_{l, \nu}).
\end{aligned}
\end{equation}
Clearly we have $\mathcal Y_0 \subseteq \mathcal Y$. We will show that both $\mathcal Y_0$ and $\mathcal Y$ are Banach spaces, and that there exists a finite-dimensional space $\mathcal O^k_{l, \lambda}$ such that $\mathcal Y = \mathcal Y_0 \oplus \mathcal O^k_{l, \lambda}$.

\begin{lemma} \label{noobstructionslemma}
In the AC case when $\lambda > -4$, or in the CS case when $\lambda < -3$, we have $\mathcal Y = \mathcal Y_0$, so we can take $\mathcal O^k_{l, \lambda} = \{ 0 \}$.
\end{lemma}
\begin{proof}
We need to show that $d(\Omega^k_{l, \lambda}) + d^*(\Omega^k_{l, \lambda}) \subseteq (d + d^*) (\Omega^k_{l, \lambda})$. Let $\sigma, \tau \in \Omega^k_{l, \lambda}$. By Remark~\ref{Hodgedecompositionoverlaprmk}, for such $\lambda$ we can apply the decomposition~\eqref{kformsHodgenonL2eq} to $\sigma - \tau$. Hence we can write $\sigma - \tau = \kappa + d \alpha + d^* \beta$ where in particular $ \kappa \in \mathcal H^k_{\lambda}$. But then we find that
\begin{align*}
d \sigma + d^* \tau \, & = \, (d + d^*) \tau + d(\sigma - \tau) \, = \, (d + d^*) \tau + d(d\alpha + d^*\beta + \kappa) \\
& = \, (d + d^*) \tau + d(d^* \beta) \, = \, (d + d^*)( \tau + d^* \beta),
\end{align*}
which is what we wanted to show.
\end{proof}

\begin{lemma} \label{Yspacelemma}
In the $L^2$ setting (when $\lambda \leq -\frac{7}{2}$ in the AC case or when $\lambda \geq -\frac{7}{2}$ in the CS case), there exists a finite-dimensional subspace $\widehat{\mathcal O}^k_{l, \lambda}$ of the space $\Omega^k_{l, \lambda}$ such that the space $\mathcal Y$ is the vector space sum of the subspaces $\mathcal Y_0$ and $d^* (\widehat{\mathcal O}^k_{l, \lambda})$. That is, we have
\begin{equation} \label{Yspacelemmaeq}
\mathcal Y \, = \, \mathcal Y_0 + d^* (\widehat{\mathcal O}^k_{l, \lambda}).
\end{equation}
\end{lemma}
\begin{proof}
Recall the finite-dimensional space $W^k_{l, \lambda}$ in the decomposition of Proposition~\ref{HodgedecompositionpropL2}. Let $(W_{\! \text{c}})^k_{l, \lambda}$ be the subspace of $W^k_{l, \lambda}$ consisting of closed forms, and similarly let $(W_{\! \text{cc}})^k_{l, \lambda}$ be the subspace of coclosed forms. We have $(W_{\! \text{c}})^k_{l, \lambda} \cap (W_{\! \text{cc}})^k_{l, \lambda} = \{ 0 \}$. Define $\widehat{\mathcal O}^k_{l, \lambda}$ to be the $L^2$-orthogonal complement in $W^k_{l, \lambda}$ of the subspace $(W_{\! \text{c}})^k_{l, \lambda} \oplus (W_{\! \text{cc}})^k_{l, \lambda}$. That is,
\begin{equation} \label{obstructionspacedefneq}
W^k_{l, \lambda} \, = \, \big( (W_{\! \text{c}})^k_{l, \lambda} \oplus (W_{\! \text{cc}})^k_{l, \lambda} \big) \oplus \widehat{\mathcal O}^k_{l, \lambda}.
\end{equation}
The second $\oplus$ symbol above is an orthogonal direct sum, but the sum $(W_{\! \text{c}})^k_{l, \lambda} \oplus (W_{\! \text{cc}})^k_{l, \lambda}$ need not be orthogonal. Hence any $\gamma \in W^k_{l, \lambda}$ can be written uniquely as $\gamma = \gamma_{\text{c}} + \gamma_{\text{cc}} + \gamma_o$ where $d \gamma_{\text{c}} = 0$ and $d^* \gamma_{\text{cc}} = 0$ and $\gamma_o \in \widehat{\mathcal O}^k_{l, \lambda}$ is neither closed nor coclosed.

It is clear that $(d + d^*) (\Omega^k_{l, \lambda}) + d^*(\widehat{\mathcal O}^k_{l, \lambda}) \subseteq d(\Omega^k_{l, \lambda}) + d^*(\Omega^k_{l, \lambda})$. We need to show the reverse inclusion. Let $\sigma, \tau \in \Omega^k_{l, \lambda}$. Applying the decomposition~\eqref{kformsHodgeL2eq} to $\sigma - \tau$, we can write $\sigma - \tau = \kappa + d \alpha + d^* \beta + \gamma$ where in particular $ \kappa \in \mathcal H^k_{\lambda}$ and $\gamma$ in $W^k_{l, \lambda}$. But then we find that
\begin{align*}
d \sigma + d^* \tau \, & = \, (d + d^*) \tau + d(\sigma - \tau) \\ & = \, (d + d^*) \tau + d(d\alpha + d^*\beta + \kappa + \gamma) \\
& = \, (d + d^*) \tau + d(d^* \beta + \gamma).
\end{align*}
By~\eqref{obstructionspacedefneq} we can write $\gamma = \gamma_\text{c} + \gamma_{\text{cc}} + \gamma_o$ for some \emph{closed} form $ \gamma_{\text{c}}$, some \emph{coclosed} form $\gamma_{\text{cc}}$, and some form $\gamma_o \in \widehat{\mathcal O}^k_{l, \lambda}$. Therefore we have\begin{align*}
d \sigma + d^* \tau \, & = \, (d + d^*) \tau + d (d^* \beta + \gamma_\text{cc} + \gamma_o) \\
& = \, (d + d^*)(\tau + d^* \beta + \gamma_{\text{cc}} + \gamma_o) + d^*(- \gamma_o) \\
& \in \, (d + d^*) (\Omega^k_{l, \lambda}) + d^* (\widehat{\mathcal O}^k_{l, \lambda})
\end{align*}
which is what we wanted to show.
\end{proof}

\begin{defn} \label{realobstructionspacedefn}
We define the finite-dimensional \emph{obstruction space} $\mathcal O^k_{l, \lambda}$ for rate $\lambda$ to be a direct complement to $\mathcal Y_0 = (d + d^*)(\Omega^k_{l, \lambda})$ in $\mathcal Y = d(\Omega^k_{l, \lambda}) + d^*(\Omega^k_{l,\lambda})$. That is,
\begin{equation} \label{realobstructionspacedefneq}
\mathcal Y \, = \, \mathcal Y_0 \oplus \mathcal O^k_{l, \lambda}.
\end{equation}
In particular, we have that $\mathcal O^k_{l, \lambda}$ is isomorphic to the quotient
\begin{equation}
\mathcal O^k_{l, \lambda} \, \cong \, \bigl( d(\Omega^k_{l, \lambda}) + d^*(\Omega^k_{l,\lambda}) \bigr) \, /  \, (d + d^*)(\Omega^k_{l, \lambda}).
\end{equation}
Recall the operator $\dd^k_{l, \lambda} = \rest{(d + d^*)_{l, \lambda}}{\Omega^k_{l, \lambda}} :  \Omega^k_{l, \lambda} \to d(\Omega^k_{l, \lambda}) + d^*(\Omega^k_{l, \lambda}).$ Thus, the space $\mathcal Y$ is the codomain of $\dd^k_{l, \lambda}$ and the image of $\dd^k_{l, \lambda}$ is
\begin{equation} \label{Y0defneq}
\im (\dd^k_{l, \lambda}) \, = \, (d + d^*)(\Omega^k_{l, \lambda}) \, = \, \mathcal Y_0.
\end{equation}
That is, we have
\begin{equation} \label{obstructionisomcokereq}
\mathcal O^k_{l, \lambda} \, \cong \, \coker \dd^k_{l, \lambda}.
\end{equation}
\end{defn}

\begin{rmk} \label{noobstructionrmk}
By Lemma~\ref{noobstructionslemma}, in the AC case when $\lambda > -4$, or in the CS case when $\lambda < -3$, we have $\mathcal O^k_{l, \nu} = \{ 0 \}$.
\end{rmk}

\begin{cor} \label{obstructionspacecor}
When $\lambda \leq -4$ in the AC case or when $\lambda \geq -3$ in the CS case, a finite-dimensional space $\mathcal O^k_{l, \lambda}$ in $\Omega^k_{l, \lambda}$ satisfying~\eqref{realobstructionspacedefneq} exists and can be chosen to consist of \emph{coexact} forms.

[Note that the $L^2$ setting properly corresponds to $\lambda \leq - \frac{7}{2}$ (AC) or $\lambda \geq - \frac{7}{2}$ (CS). However, Remark~\ref{noobstructionrmk} absorbs the cases $\lambda \in (-4, - \frac{7}{2}]$ (AC) and $\lambda \in [-\frac{7}{2}, -3)$ (CS) into the ``non-$L^2$'' setting, hence the restriction to $\lambda \leq -4$ (AC) and $\lambda \geq -3$ (CS) in this statement.]
\end{cor}
\begin{proof}
We choose $\mathcal O^k_{l, \lambda}$ to be a subspace of  $d^* (\widehat{\mathcal O}^k_{l, \lambda})$ from Lemma~\ref{Yspacelemma} that is a direct complement to $(d + d^*)(\Omega^k_{l, \lambda})$ in $d(\Omega^k_{l, \lambda}) + d^*(\Omega^k_{l,\lambda})$. That is,
\begin{equation*}
d(\Omega^k_{l, \lambda}) + d^*(\Omega^k_{l,\lambda}) \, = \, (d + d^*)(\Omega^k_{l, \lambda}) \oplus \mathcal O^k_{l, \lambda},
\end{equation*}
with $\mathcal O^k_{l, \lambda} \subseteq d^*(\widehat{\mathcal O}^k_{l, \lambda})$. For example, we can choose $\mathcal O^k_{l, \lambda}$ to be the intersection with $d^* (\widehat{\mathcal O}^k_{l, \lambda})$ of the $L^2$-orthogonal complement in $d(\Omega^k_{l, \lambda}) + d^*(\Omega^k_{l,\lambda})$ of $d^* (\widehat{\mathcal O}^k_{l, \lambda}) \cap (d + d^*)(\Omega^k_{l, \lambda})$.
In particular, we have that $\mathcal O^k_{l, \lambda}$ is isomorphic to the quotient
\begin{equation}
\mathcal O^k_{l, \lambda} \, \cong \, \bigl( d(\Omega^k_{l, \lambda}) + d^*(\Omega^k_{l,\lambda}) \bigr) \, /  \, (d + d^*)(\Omega^k_{l, \lambda}).
\end{equation}
From the definition of $\widehat{\mathcal O}^k_{l, \lambda}$ in the proof of Lemma~\ref{Yspacelemma}, we know that $d^*$ is injective on $\widehat{\mathcal O}^k_{l, \lambda}$. It follows that the dimension of $\mathcal O^k_{l, \lambda}$ is no larger than the dimension of $\widehat{\mathcal O}^k_{l, \lambda}$, which is finite.
\end{proof}

\begin{lemma} \label{YBanachspacelemma}
For generic rates $\lambda$, the spaces $\mathcal Y_0$ and $\mathcal Y = \mathcal Y_0 \oplus {\mathcal O}^k_{l, \lambda}$ are both Banach spaces.
\end{lemma}
\begin{proof}
If we assume that $\lambda + 1$ is noncritical for $d + d^*$, the map
\begin{equation*}
(d + d^*)_{l+1, \lambda+1} : \Omega^{\bullet}_{l+1, \lambda+1} \to \Omega^{\bullet}_{l, \lambda}
\end{equation*}
is Fredholm, and thus has closed image. In fact, the Lockhart--McOwen theory~\cite[Section 2]{LM} says that at a noncritical rate, for any $\eta \in \Omega^k_{l, \lambda}$ that is orthogonal (with respect to the $L^2_{l, \lambda}$ inner product) to the kernel of $d + d^*$, we have the estimate
\begin{equation*}
{||\eta||}_{L^2_{l,\lambda}} \, \leq \, C {||(d + d^*) \eta ||}_{L^2_{l - 1,\lambda - 1}}.
\end{equation*}
From this estimate, it is a standard result~\cite[Corollary 2.15]{AA} that $\mathcal Y_0 = (d + d^*)(\Omega^k_{l, \lambda})$ is a \emph{closed} subspace of $\Omega^{\bullet}_{l-1, \lambda-1}$, and thus a Banach space. By equation~\eqref{realobstructionspacedefneq}, since ${\mathcal O}^k_{l, \lambda}$ is finite-dimensional, we deduce that the space $d(\Omega^k_{l, \lambda}) + d^*(\Omega^k_{l,\lambda}) = \mathcal Y_0 \oplus {\mathcal O}^k_{l, \lambda}$ is also a Banach space.
\end{proof}

The next result establishes that generically, $\dd^k_{l, \lambda}$ is Fredholm, and surjective for certain rates.

\begin{lemma} \label{PFredholmlemma}
Let $\lambda$ be a \emph{noncritical rate} for $d + d^*$ on $M$. The map
\begin{equation*}
\dd^k_{l, \lambda} \, : \, \Omega^k_{l, \lambda} \to d(\Omega^k_{l, \lambda}) + d^*(\Omega^k_{l, \lambda})
\end{equation*}
is Fredholm. Moreover, $\dd^k_{l, \lambda}$ is \emph{surjective} if $\lambda > -4$ in the AC case and if $\lambda < -3$ in the CS case.
\end{lemma}
\begin{proof}
For any $\lambda$, we have $\ker \dd^k_{l, \lambda} = \mathcal H^k_{\lambda}$ is finite-dimensional. From~\eqref{obstructionisomcokereq} and Definition~\ref{realobstructionspacedefn} we know that $\coker \dd_{l, \lambda}$ is finite-dimensional. Finally, if $\lambda$ is not critical for $d + d^*$ on $M$, then we proved in Lemma~\ref{YBanachspacelemma} that $\dd^k_{l, \lambda}$ has closed image. Thus $\dd^k_{l, \lambda}$ is Fredholm. The statements about the surjectivity of $\dd^k_{l, \lambda}$ are a reiteration of Lemma~\ref{noobstructionslemma}.
\end{proof}

Next, we determine a particular characterization of $\coker \dd^k_{l, \lambda} \cong \mathcal O^k_{l, \lambda}$.

\begin{prop} \label{Pindexchangeprop}
Consider the setup of Lemma~\ref{PFredholmlemma}.
\begin{enumerate}[(a)]
\item The space $\ker \dd^k_{l, \lambda}$ is a subspace of $\ker (d + d^*)_{l, \lambda}$, and $\coker \dd^k_{l, \lambda}$ is a subspace of $\coker (d + d^*)_{l, \lambda}$, in the following sense: the topological complement $\mathcal O^k_{l, \lambda}$ of $\im (\dd^k_{l, \lambda})$ in $d(\Omega^k_{l, \lambda}) + d^*(\Omega^k_{l, \lambda})$, which is isomorphic to $\coker \dd^k_{l, \lambda}$, is a subspace of the orthogonal complement of $\im (d + d^*)_{l, \lambda}$ in $\Omega^{\bullet}_{l-1, \lambda-1}$, with respect to the Hilbert space inner product.
\item The space $\coker \dd^k_{l, \lambda}$ is isomorphic to the \emph{quotient} of the space $\ker (d + d^*)_{-6 - \lambda} \cap (\Omega^{k-1} \oplus \Omega^{k+1})$ of closed and coclosed forms of degree $k-1$ plus degree $k+1$ of rate $-6 - \lambda$ by the subspace $\mathcal H^{k-1}_{-6 - \lambda} \oplus \mathcal H^{k+1}_{-6 - \lambda}$ of closed and coclosed $(k-1)$-forms plus closed and coclosed $(k+1)$-forms of rate $-6 - \lambda$.
\end{enumerate}
\end{prop}
\begin{proof}
For any $\lambda$, we have $\ker \dd^k_{l, \lambda} = \mathcal H^k_{\lambda}$ is finite-dimensional. We know from Definition~\ref{realobstructionspacedefn} that $\coker \dd_{l, \lambda}$ is finite-dimensional. Finally, if $\lambda$ is not critical for $d + d^*$ on $M$, then we proved in Lemma~\ref{YBanachspacelemma} that $\dd^k_{l, \lambda}$ has closed image. Thus $\dd^k_{l, \lambda}$ is Fredholm. Next, we will prove the statements about the kernel and cokernel of $\dd^k_{l, \lambda}$. The arguments are identical in the CS case (except for the fact that we have $n$ ends instead of just one, and the inequalities are reversed) so we prove just the AC case.

It is clear from the definition of $\dd^k_{l, \lambda}$ that $\ker \dd^k_{l, \lambda}$ is a subspace of $\ker (d + d^*)_{l, \lambda}$. We need to establish the analogous result for $\coker \dd^k_{l, \lambda}$. To simplify notation, in this proof only, we will use $E$ to denote the subspace $(d + d^*)(\Omega^{\bullet}_{l, \lambda})$ of $\Omega^{\bullet}_{l-1, \lambda-1}$, which is \emph{closed} if $\lambda$ is noncritical for $d + d^*$. Also, let $F$ denote the orthogonal complement of $E$ with respect to the Hilbert space inner product on $\Omega^{\bullet}_{l, \lambda}$. Thus we have
\begin{equation*}
\Omega^{\bullet}_{l-1, \lambda-1} \, = \, (d + d^*)(\Omega^{\bullet}_{l, \lambda}) \oplus \coker (d + d^*)_{l, \lambda} \, = \, E \oplus F
\end{equation*}
where in fact by Remark~\ref{annihilatorrmk} we know that
\begin{equation*}
F \, \cong \, \Ann (E) \,= \, \ker (d + d^*)_{- 6 - \lambda}
\end{equation*}
where $\Ann (E)$ denotes the annihilator of $E$ in the dual space. 

Now consider the orthogonal projection $P$ of $E$ onto the closed subspace $\Omega^{k-1}_{l-1, \lambda - 1} \oplus \Omega^{k+1}_{l-1, \lambda - 1}$. We have that
\begin{equation*}
P(E) \, = \, d (\Omega^{k-2}_{l, \lambda}) + (d + d^*)(\Omega^k_{l, \lambda}) + d^* (\Omega^{k+2}_{l, \lambda}) \, = \, E'
\end{equation*}
is closed in the Hilbert space $\Omega^{k-1}_{l-1, \lambda - 1} \oplus \Omega^{k+1}_{l-1, \lambda - 1}$. Thus we can write
\begin{equation*}
\Omega^{k-1}_{l-1, \lambda - 1} \oplus \Omega^{k+1}_{l-1, \lambda - 1} \, = \, E' \oplus F'
\end{equation*}
where we take $F'$ to be the orthogonal complement of $E'$ with respect to the Hilbert space inner product on $\Omega^{k-1}_{l-1, \lambda - 1} \oplus \Omega^{k+1}_{l-1, \lambda - 1}$. By Remark~\ref{annihilatorrmk} we have
\begin{equation} \label{Pindexchangetempeq2}
F' \, \cong \, \Ann (E').
\end{equation}
It is trivial to compute that
\begin{equation} \label{Pindexchangetempeq3}
\Ann (E')  \, = \, \ker (d + d^*)_{-6 - \lambda} \cap (\Omega^{k-1} \oplus \Omega^{k+1}).
\end{equation}
That is, $F'$ is isomorphic to the space of forms of degree $k-1$ plus degree $k+1$ of rate $-6 - \lambda$ in the kernel of $d + d^*$.

From Lemma~\ref{Yspacelemma} and Lemma~\ref{YBanachspacelemma} we have that
\begin{equation*}
d(\Omega^k_{l, \lambda}) + d^*(\Omega^k_{l, \lambda}) \, = \, (d + d^*)(\Omega^k_{l, \lambda}) + d^*(\widehat{\mathcal O}^k_{\lambda}) \, = \, (d + d^*)(\Omega^k_{l, \lambda}) \oplus \mathcal O^k_{\lambda}
\end{equation*}
is closed in $\Omega^{k-1}_{l-1, \lambda - 1} \oplus \Omega^{k+1}_{l-1, \lambda - 1}$, and $\coker \dd^k_{l, \lambda} \cong \mathcal O^k_{\lambda}$. Note that $\mathcal O^k_{\lambda}$ is a subspace of $(k-1)$-forms, and is thus always transverse to $d^*(\Omega^{k+2}_{l, \lambda})$. In addition, it is transverse to $d(\Omega^{k-2}_{l, \lambda})$, because in the $L^2$ setting the images of $d$ and $d^*$ are orthogonal, and in the non-$L^2$ setting we know that $\mathcal O^k_{\lambda} = \{ 0 \}$. These observations tell us that we can also write
\begin{equation*}
\Omega^{k-1}_{l-1, \lambda - 1} \oplus \Omega^{k+1}_{l-1, \lambda - 1} \, = \, E'' \oplus F''
\end{equation*}
where
\begin{align*}
E'' \, & = \, d (\Omega^{k-2}_{l, \lambda}) + d(\Omega^k_{l, \lambda}) + d^*(\Omega^k_{l, \lambda}) + d^* (\Omega^{k+2}_{l, \lambda}) \\
& = \, E'  \oplus \mathcal O^k_{\lambda},
\end{align*}
and $F''$ is the orthogonal complement of $E''$ with respect to the Hilbert space inner product on $\Omega^{k-1}_{l-1, \lambda - 1} \oplus \Omega^{k+1}_{l-1, \lambda - 1}$. We therefore clearly have
\begin{equation} \label{Pindexchangetempeq4}
F' \, = \, \mathcal O^k_{\lambda} \oplus F''
\end{equation}
and we note again that by Remark~\ref{annihilatorrmk} we have
\begin{equation} \label{Pindexchangetempeq5}
F'' \, \cong \, \Ann (E'').
\end{equation}
In this case it is easy to see that
\begin{equation} \label{Pindexchangetempeq6}
\Ann (E'')  \, = \, \mathcal H^{k-1}_{-6 - \lambda} \oplus \mathcal H^{k+1}_{-6 - \lambda}.
\end{equation}
We now observe that equations~\eqref{Pindexchangetempeq2},~\eqref{Pindexchangetempeq3},~\eqref{Pindexchangetempeq4},~\eqref{Pindexchangetempeq5}, and~\eqref{Pindexchangetempeq6} together imply part (b) of the proposition.

Finally, we have that $\coker \dd^k_{l, \lambda} \cong \mathcal O^k_{\lambda}$, which is a subspace of $F'$. But we see that
\begin{align*}
F' \, & = \, \{ \gamma \in \Omega^{k-1}_{l-1, \lambda - 1} \oplus \Omega^{k+1}_{l-1, \lambda - 1}; \, \langle \langle P \alpha, \gamma \rangle \rangle_{\Omega^{\bullet}_{l-1, \lambda-1}} = 0, \, \forall \alpha \in E \} \\ & = \, \{ \gamma \in \Omega^{k-1}_{l-1, \lambda - 1} \oplus \Omega^{k+1}_{l-1, \lambda - 1}; \, \langle \langle \alpha, \gamma \rangle \rangle_{\Omega^{\bullet}_{l-1, \lambda-1}} = 0, \, \forall \alpha \in E \} \\ & = \, \{ \gamma \in \Omega^{\bullet}_{l-1, \lambda - 1} ; \, \langle \langle \alpha, \gamma \rangle \rangle_{\Omega^{\bullet}_{l-1, \lambda-1}} = 0, \, \forall \alpha \in E \} \cap \left( \Omega^{k-1}_{l-1, \lambda - 1} \oplus \Omega^{k+1}_{l-1, \lambda - 1} \right) \\ & = \, F \cap \left( \Omega^{k-1}_{l-1, \lambda - 1} \oplus \Omega^{k+1}_{l-1, \lambda - 1} \right) \\ & \, \subseteq F
\end{align*}
and the proof is complete.
\end{proof}

We pause here to state and prove an important result about homogeneous forms on a cone, namely Theorem~\ref{cones24thm} below, which relates closed and coclosed $k$-forms on $C$, homogeneous of order $\lambda$, to a particular subspace of forms on the cone $C$ of degree $k-1$ plus degree $k+1$, homogeneous of order $-6-\lambda$, in the kernel of $d + d^*$. Before we can state it, we need to define several spaces.

\begin{nota} \label{ABCnotation}
Consider a form $\gamma$ of degree $k-1$ plus degree $k+1$ on the cone, homogeneous of order $-6 - \lambda$. Using Definition~\ref{homoformsdefn} we can write
\begin{equation} \label{cokerPchangethmtempeq1}
\gamma \, = \, r^{-6 - \lambda} (r^{k-2} dr \wedge \beta_{k-2} + r^{k-1} \beta_{k-1} + r^k dr \wedge \beta_k + r^{k+1} \beta_{k+1})
\end{equation}
where each $\beta_m$ is an $m$-form on $\Sigma$. We will write this form as a $4$-tuple $(\beta_{k-2}, \beta_{k-1}, \beta_k, \beta_{k+1})$. From the equations~\eqref{homoddseq}, it follows easily that $\gamma$ is in the kernel of $d + d^*$ if and only if
\begin{equation} \label{cokerPchangethmtempeq2}
\begin{aligned}
\dss \beta_{k-2} & = 0, & \qquad \dss \beta_{k-1} & = - (\lambda + k - 2) \beta_{k-2}, \\
\ds \beta_k & = - (\lambda - k + 5) \beta_{k+1}, & \qquad \ds \beta_{k+1} & = 0, \\
\ds \beta_{k-2} + \dss \beta_k & = - (\lambda - k + 7) \beta_{k-1}, & \qquad \ds \beta_{k-1} + \dss \beta_{k+1} & = - (\lambda + k) \beta_k.
\end{aligned}
\end{equation}
We will denote by $A(\lambda)$ the space of solutions to the system of equations~\eqref{cokerPchangethmtempeq2}. Let $B(\lambda)$ denote the subspace consisting of forms $\gamma \in A(\lambda)$ of degree $k-1$ plus degree $k+1$, homogeneous of order $-6 - \lambda$, such that each pure degree component $\gamma_{k-1} =  r^{-6 - \lambda} (r^{k-2} dr \wedge \beta_{k-2} + r^{k-1} \beta_{k-1})$ and $\gamma_{k+1} =  r^{-6 - \lambda} ( r^k dr \wedge \beta_k + r^{k+1} \beta_{k+1})$ is independently closed and coclosed. Again using equations~\eqref{homoddseq}, we find that $\gamma$ is in $B(\lambda)$ if and only if, in addition to equations~\eqref{cokerPchangethmtempeq2}, we also have
\begin{equation} \label{cokerPchangethmtempeq3}
\ds \beta_{k-1} = 0, \qquad \dss \beta_k = 0.
\end{equation}
Finally, denote $C(\lambda)$ to be the subspace of $A(\lambda)$ consisting of forms of the type~\eqref{cokerPchangethmtempeq1} with $\beta_{k-2} = 0$ and $\beta_{k+1} = 0$. That is, $\gamma$ lies in $C(\lambda)$ if and only if $\gamma = r^{-6 - \lambda} (r^{k-1} \beta_{k-1} + r^k dr \wedge \beta_k)$ with
\begin{equation} \label{cokerPchangethmtempeq4}
\begin{aligned}
\dss \beta_{k-1} & = 0, & \qquad \ds \beta_k & = 0, \\
\dss \beta_k & = - (\lambda - k + 7) \beta_{k-1}, & \qquad \ds \beta_{k-1} & = - (\lambda + k) \beta_k.
\end{aligned}
\end{equation}
\end{nota}

\begin{rmk} \label{conesCrmk}
From Lemma~\ref{KPspacelemma}, we note that $(0, \beta_{k-1}, \beta_k, 0) \in C(\lambda)$ if and only if the homogeneous $k$-form $r^{\lambda} (r^{k-1} dr \wedge \beta_{k-1} - r^k \beta_k)$ is closed and coclosed. That is, the space $C(\lambda)$ is isomorphic to the space of closed and coclosed $k$-forms on the cone, homogeneous of order $\lambda$.
\end{rmk}

\begin{thm} \label{cones24thm}
We have $C(-k) \subseteq B(-k)$ and $C(k-7) \subseteq B(k-7)$. Furthermore, if $\lambda \neq -k$ and $\lambda \neq k-7$, then $A(\lambda) = B(\lambda) \oplus C(\lambda)$, where the direct sum is orthogonal with respect to the $L^2$ inner product on forms on $\Sigma$. That is, for $\lambda \notin \{-k, k-7\}$, the subspace of forms on the cone of degree $k-1$ plus degree $k+1$, homogeneous of order $-6 - \lambda$, in the kernel of $d + d^*$, and $L^2$-orthogonal to those forms which are independently closed and coclosed, is \emph{isomorphic} to the space of closed and coclosed $k$-forms, homogeneous of order $\lambda$.
\end{thm}
\begin{proof}
Suppose $\lambda = -k$, and that $(0, \beta_{k-1}, \beta_k, 0) \in C(-k)$. Then equations~\eqref{cokerPchangethmtempeq4} say that $\beta_{k-1}$ is a closed and coclosed (thus harmonic) $(k-1)$-form on $\Sigma$, which is also coexact. By Hodge theory, we get $\beta_{k-1} = 0$ and hence $\beta_k$ is a harmonic $3$-form on $\Sigma$. But then $(0, 0, \beta_k, 0)$ satisfies the equations~\eqref{cokerPchangethmtempeq2} and~\eqref{cokerPchangethmtempeq3}, and thus lies in $B(-k)$. The proof of $C(k-7) \subseteq B(k-7)$ is similar.

Next we show that if $\lambda \notin \{-k,k-7\}$, the subspaces $B(\lambda)$ and $C(\lambda)$ are $L^2$-orthogonal. Suppose $(\beta_{k-2}, \beta_{k-1}, \beta_k, \beta_{k+1}) \in B(\lambda)$ and $(0, \gamma_{k-1}, \gamma_k, 0) \in C(\lambda)$. Then using equations~\eqref{cokerPchangethmtempeq2} and~\eqref{cokerPchangethmtempeq3} for $(\beta_{k-2}, \beta_{k-1}, \beta_k, \beta_{k+1})$, and equations~\eqref{cokerPchangethmtempeq4} for $(0, \gamma_{k-1}, \gamma_k, 0)$ we compute that
\begin{align*}
\langle \langle \beta_{k-1} , \gamma_{k-1} \rangle \rangle \, & = \, \langle \langle - (\lambda - k + 7)^{-1} ( \ds \beta_{k-2} + \dss \beta_k) , - (\lambda - k + 7)^{-1} \dss \gamma_k \rangle \rangle \\ & = \, (\lambda - k + 7)^{-2} \langle \langle \ds \beta_{k-2}, \dss \gamma_k \rangle \rangle \, = \, 0
\end{align*}
and similarly
\begin{align*}
\langle \langle \beta_k , \gamma_k \rangle \rangle \, & = \, \langle \langle - (\lambda + k)^{-1} ( \ds \beta_{k-1} + \dss \beta_{k+1}) , - (\lambda + k)^{-1} \ds \gamma_{k-1} \rangle \rangle \\ & = \, (\lambda + k)^{-2} \langle \langle \dss \beta_{k+1}, \ds \gamma_{k-1} \rangle \rangle \, = \, 0.
\end{align*}
Thus we indeed have $B(\lambda) \perp C(\lambda)$.

Finally, we will complete the proof by showing that if $(\gamma_{k-2}, \gamma_{k-1}, \gamma_k, \gamma_{k+1}) \in A(\lambda)$ is $L^2$-orthogonal to $B(\lambda)$, then it is in $C(\lambda)$. This would imply that $A(\lambda) = B(\lambda) \oplus C(\lambda)$, as claimed. Define $(\beta_{k-2}, \beta_{k-1}, \beta_k, \beta_{k+1})$ by
\begin{align*}
\beta_{k-2} \, & = \, \gamma_{k-2}, & \qquad \beta_{k-1} \, & = \, - (\lambda - k + 7)^{-1} \ds \gamma_{k-2}, \\
\beta_k \, & = \, - (\lambda + k)^{-1} \dss \gamma_{k+1}, & \qquad \beta_{k+1} \, & = \, \gamma_{k+1}.
\end{align*}
Using the fact that $(\gamma_{k-2}, \gamma_{k-1}, \gamma_k, \gamma_{k+1})$ satisfies equations~\eqref{cokerPchangethmtempeq2}, we find that
\begin{equation*}
\beta_k = \gamma_k +(\lambda + k)^{-1}\ds \gamma_{k-1} \quad\text{and}\quad \beta_{k-1}=\gamma_{k-1}+(\lambda-k+7)^{-1}\dss\gamma_{k}.
\end{equation*}
Hence, $(\beta_{k-2}, \beta_{k-1}, \beta_k, \beta_{k+1})$ satisfies equations~\eqref{cokerPchangethmtempeq2} and~\eqref{cokerPchangethmtempeq3}, so $(\beta_{k-2}, \beta_{k-1}, \beta_k, \beta_{k+1})$ lies in $B(\lambda)$. Our hypothesis is that $(\gamma_{k-2}, \gamma_{k-1}, \gamma_k, \gamma_{k+1})$ is $L^2$-orthogonal to the space $B(\lambda)$. We compute
\begin{equation*}
\begin{aligned}
\langle \langle \gamma_{k-2}, \beta_{k-2} \rangle \rangle & \, = \, || \gamma_{k-2} ||^2, \\
\langle \langle \gamma_{k-1}, \beta_{k-1} \rangle \rangle & \, = \, \langle \langle \gamma_{k-1}, - (\lambda - k + 7)^{-1} \ds \gamma_{k-2} \rangle \rangle \, = \, - (\lambda - k + 7)^{-1} \langle \langle \dss \gamma_{k-1}, \gamma_{k-2} \rangle \rangle \\ & \, = \, - (\lambda - k + 7)^{-1} \langle \langle - (\lambda + k - 2) \gamma_{k-2}, \gamma_{k-2} \rangle \rangle \, = \, \frac{\lambda + k - 2}{\lambda - k + 7} \, || \gamma_{k-2} ||^2, \\
\langle \langle \gamma_k, \beta_k \rangle \rangle & \, = \, \langle \langle \gamma_k, - (\lambda + k)^{-1} \dss \gamma_{k+1} \rangle \rangle \, = \, - (\lambda + k)^{-1} \langle \langle \ds \gamma_k, \gamma_{k+1} \rangle \rangle \\ & \, = \, - (\lambda + k)^{-1} \langle \langle - (\lambda - k + 5) \gamma_{k+1}, \gamma_{k+1} \rangle \rangle \, = \, \frac{\lambda - k + 5}{\lambda + k} \, || \gamma_{k+1} ||^2, \\
\langle \langle \gamma_{k+1}, \beta_{k+1} \rangle \rangle & \, = \, || \gamma_{k+1} ||^2, \\
\end{aligned}
\end{equation*}
and thus we find that
\begin{equation} \label{cokerPchangethmtempeq5}
\sum_{m=k-2}^{k+1} \langle \langle \gamma_m, \beta_m \rangle \rangle \, = \, \left( 1 + \frac{\lambda + k - 2}{\lambda - k + 7} \right) || \gamma_{k-2} ||^2 + \left( 1 + \frac{\lambda - k + 5}{\lambda + k} \right) || \gamma_{k+1} ||^2 \, = \, 0.
\end{equation}
We have $\laps \gamma_{k-2} = \ds \dss \gamma_{k-2} + \dss \ds \gamma_{k-2} = - (\lambda - k + 7) \dss \gamma_{k-1} = (\lambda + k - 2)(\lambda - k + 7) \gamma_{k-2}$. Thus by the nonnegativity of the Hodge Laplacian, we have $\gamma_{k-2} = 0$ if $(\lambda + k - 2)(\lambda - k + 7) < 0$. On the other hand, if $(\lambda + k - 2)(\lambda - k + 7) \geq 0$, then since $\lambda \neq - k + 7$ we must have $1 + \frac{\lambda + k - 2}{\lambda - k + 7} > 0$. Similarly we observe that $\laps \gamma_{k+1} = \ds \dss \gamma_{k+1} + \dss \ds \gamma_{k+1} = - (\lambda + k) \ds \gamma_k = (\lambda + k)(\lambda - k + 5) \gamma_{k+1}$. Thus we have $\gamma_{k+1} = 0$ if $(\lambda + k)(\lambda - k + 5) < 0$. However, if $(\lambda + k)(\lambda - k + 5) \geq 0$, then since $\lambda \neq - k$ we must have $1 + \frac{\lambda - k + 5}{\lambda + k} > 0$.  Thus we conclude that in all cases when $\lambda \neq -k$ and $\lambda \neq k-7$, equation~\eqref{cokerPchangethmtempeq5} tells us that $\gamma_{k-2} = 0$ and $\gamma_{k+1} = 0$. Thus indeed we have $(\gamma_{k-2}, \gamma_{k-1}, \gamma_k, \gamma_{k+1}) \in C(\lambda)$, and the proof is complete.
\end{proof}

\begin{rmk} \label{cones24rmk}
Essentially, Theorem~\ref{cones24thm} and Remark~\ref{conesCrmk} together say that on the cone, if $\lambda \notin \{ -k, k - 7 \}$, then
\begin{equation*}
\mathcal H^{(k-1) + (k+1)}_{-6 - \lambda} \cong \mathcal H^k_{\lambda} \oplus ( \mathcal H^{k-1}_{-6 - \lambda} \oplus \mathcal H^{k+1}_{-6 - \lambda} ),
\end{equation*}
where the notation should be self-explanatory.
\end{rmk}

We now prove the analogue of Proposition~\ref{asymptoticexpansionprop} for $\dd^k$. Note that this is \emph{not} immediate because $\dd^k$ is \emph{not} uniformly elliptic. Recall that the critical rates for $\dd^k$ are a subset of the critical rates for $d + \dsm$ and both sets of rates are discrete subsets of $\R$. Hence, given any critical rate $\lambda_0$ for $\dd^k$, there exists $\epsilon > 0$ so that if $0 < | \beta - \lambda_0 | < \epsilon$ then $\beta$ is \emph{not} a critical rate for $d + \dsm$.

\begin{prop} \label{ddk.asymptoticexpansionprop}
Let $(M,\varphi)$ be a $G_2$~conifold of rate $\nu$. Let $\lambda_0$ be a critical rate for $\dd^k$ on $M$, arising from a critical rate for $\ddci^k$, and let  $\beta_1$, $\beta_2$ be two noncritical rates for $d + \dsm$ on $M$ so that either $\beta_1 > \beta_2$ (AC) or $\beta_1 < \beta_2$ (CS) and so that
$\lambda_0$ is the unique critical rate of $d + \dsm$ in the interval between $\beta_1$ and $\beta_2$. Suppose further that $| \beta_2 - \lambda_0 | < |\nu|$.

Let 
\begin{equation*}
\mathcal{F}_{\beta_1} \, = \, \{ \gamma \in \Omega^k_{l+1, \beta_1} \, : \, (d + \dsm) \gamma \in \Omega^{k-1}_{l,\beta_2-1} \oplus \Omega^{k+1}_{l, \beta_2 - 1} \}.
\end{equation*}
(Thus, if $\gamma \in \mathcal{F}_{\beta_1}$, then $\dd^k_{l+1, \beta_1} \gamma$ decays faster than expected.) Then there are linear maps 
\begin{equation} \label{upsilon.ddk.eq}
\upsilon : \mathcal{F}_{\beta_1} \to \mathcal{K}(\lambda_0)_{\ddci^k} \quad \text{and} \quad \vartheta : \mathcal{K}(\lambda_0)_{\ddci^k} \to \rest{{\Omega^{k}_{l+1, \lambda_0 + \nu}}}{\text{$i^{\text{th}}$ end of $M$}}
\end{equation}
such that, on the $i^{\text{th}}$ end of $M$, we have
\begin{equation} \label{ddk.asymptotic.eq}
\gamma-h_i^{-1} (\upsilon(\gamma)) - \vartheta (\upsilon(\gamma)) \in \rest{{\Omega^{k}_{l+1,\beta_2}}}{\text{$i^{\text{th}}$ end of $M$}}
\end{equation}
for all $\gamma \in \mathcal{F}_{\beta_1}$.
\end{prop}
\begin{proof}
We cannot apply Proposition~\ref{asymptoticexpansionprop} to $\dd^k$, but we can apply it to $d + \dsm$. Therefore, if we let 
\begin{equation*}
\tilde{\mathcal{F}}_{\beta_1} \, = \, \{ \gamma \in \Omega^{\bullet}_{l+1, \beta_1} \, : \, (d + \dsm) \gamma \in \Omega^{\bullet}_{l, \beta_2-1}\},
\end{equation*}
then there exist linear maps
\begin{equation}
\tilde{\upsilon} : \tilde{\mathcal{F}}_{\beta_1} \to \mathcal{K}(\lambda_0)_{d + \dsci} \quad \text{and} \quad \tilde{\vartheta} : \mathcal{K}(\lambda_0)_{d + \dsci} \to \rest{{\Omega^{\bullet}_{l+1,\lambda_0+\nu}}}{\text{$i^{\text{th}}$ end of $M$}}
\end{equation}
such that, on the $i^{\text{th}}$ end of $M$, we have
\begin{equation} \label{d+d*.ithend.eq}
\gamma - h_i^{-1} (\tilde{\upsilon}(\gamma)) - \tilde{\vartheta} (\tilde{\upsilon}(\gamma)) \in \rest{{\Omega^{\bullet}_{l+1,\beta_2}}}{\text{$i^{\text{th}}$ end of $M$}}
\end{equation}
for all $\gamma \in \tilde{\mathcal{F}}_{\beta_1}$.

Clearly, $\mathcal{F}_{\beta_1} \subseteq \tilde{\mathcal{F}}_{\beta_1}$, so we can restrict $\tilde{\upsilon}$ to $\mathcal{F}_{\beta_1}$ and  take $\gamma\in\mathcal{F}_{\beta_1}$ in \eqref{d+d*.ithend.eq}. Observe that $\gamma$ has no components in degrees $l \neq k$ and $\tilde{\vartheta} (\tilde{\upsilon}(\gamma)) \in \Omega^{\bullet}_{l+1,\beta_2}$ on the $i^{\text{th}}$ end since we chose $| \beta_2 - \lambda_0 | < \nu$. Thus the components of $h_i^{-1} (\tilde{\upsilon}(\gamma))$ in degree $l \neq k$ must decay at rate $\beta_2$. However, such components arise from homogeneous forms of rate $\lambda_0$ on the cone and so the only way this can occur is if these components are zero. Hence, $\tilde{\upsilon}(\gamma)$ is a pure degree $k$-form on $C_i$ satisfying $(d + \dsci) \tilde{\upsilon}(\gamma) = 0$. We deduce that the restriction of $\tilde{\upsilon}$ to $\mathcal{F}_{\beta_1}$ yields a linear map $\upsilon$ as in~\eqref{upsilon.ddk.eq}.

We now know that for degree $l \neq k$ we have that $\gamma - h_i^{-1} (\upsilon(\gamma))$ is zero and so trivially lies in $\Omega^{\bullet}_{l+1, \beta_2}$ on the $i^{\text{th}}$ end of $M$. Moreover, recall that $\tilde{\vartheta} (\upsilon(\gamma)) \in \Omega^{\bullet}_{l+1, \beta_2}$ on the $i^{\text{th}}$ end.  To complete the proof, if we let $\pi_k$ denote the projection from $\Omega^{\bullet}$ to $\Omega^k$, then $\vartheta = \pi_k \tilde{\vartheta}$ is a linear map as in~\eqref{upsilon.ddk.eq} so that~\eqref{ddk.asymptotic.eq} is satisfied, as required.
\end{proof}

Given Proposition~\ref{ddk.asymptoticexpansionprop}, the proof of the following analogues of Corollaries~\ref{asymptoticexpansioncor} and~\ref{asymptoticexpansioncor2} for $\dd^k$ carry over verbatim and so we state them without further proof.

\begin{cor} \label{ddk.asymptoticexpansioncor}
Consider the setup of Proposition~\ref{ddk.asymptoticexpansionprop}. 
There exists a linear map
\begin{equation*}
\eta : \mathcal{K}(\lambda_0)_{\ddci^k} \to \rest{\Omega^k_{l+1, \lambda_0 + \nu}}{\text{$i^{\text{th}}$ end of $M$}}
\end{equation*}
such that for all $\gamma_1 \in \ker (\dd^k_{l+1, \beta_1})$, there exists $\gamma_2 \in \ker(\dd^k_{l+1, \beta_2})$ such that, on the $i^{\text{th}}$ end of $M$, we have
\begin{equation} \label{ddk.ker.asymptoticexpansioneq}
\gamma_1 - h_i^{-1} \big( \upsilon(\gamma_1) \big) - \eta \big(\upsilon(\gamma_1)\big) \, = \, \gamma_2 \, \in \, \ker(\dd^k_{l+1,\beta_2}).
\end{equation}
Note that the term $\gamma_2$, which is in the kernel of $\dd^k$ with noncritical rate $\beta_2$, decays \emph{faster} on the end.
\end{cor}

\begin{cor} \label{ddk.asymptoticexpansioncor2}
Consider the setup of Proposition~\ref{ddk.asymptoticexpansionprop}. Let $\chi_i$ be a smooth cutoff function on $M$ which is $1$ on the $i^{\text{th}}$ end and $0$ on all other ends, so that $\chi_i \mathcal{K}(\lambda_0)_{\ddci^k}$ can be viewed as a subspace of $\Omega^k_{l+1, \beta_1}$. Define the map $\tilde{\dd}^k_{l+1, \beta_1}$ to be the restriction of $\dd^k_{l+1, \beta_1}$ to the subspace $\Omega^k_{l+1,\beta_2} + \chi_i \mathcal{K}(\lambda_0)_{\ddci^k}$ of $\Omega^k_{l+1, \beta_1}$. Then the two linear maps
\begin{align}
\dd^k_{l+1,\beta_1} & \, : \, \Omega^k_{l+1,\beta_1} \to d(\Omega^k_{l+1,\beta_1})+d^*(\Omega^k_{l+1,\beta_1}), \label{tilde.ddk.eq.0} \\
\tilde{\dd}^k_{l+1,\beta_1} & \, : \, \Omega^k_{l+1,\beta_2} + \chi_i \mathcal{K}(\lambda_0)_{\ddci^k} \to \big(d(\Omega^k_{l+1,\beta_1})\cap \Omega^{k+1}_{l,\beta_2-1}\big) + \big(d^*(\Omega^k_{l+1,\beta_1}) \cap \Omega^{k-1}_{l,\beta_2-1} \big) \label{tilde.ddk.eq}
\end{align}
satisfy $\ker(\dd^k_{l+1,\beta_1}) = \ker(\tilde{\dd}^k_{l+1,\beta_1})$ and $\coker(\dd^k_{l+1,\beta_1}) \cong \coker(\tilde{\dd}^k_{l+1,\beta_1})$.
\end{cor}

Corollary~\ref{ddk.asymptoticexpansioncor2} is not quite the statement that we require because the right hand side of~\eqref{tilde.ddk.eq} is not simply the right hand side of~\eqref{tilde.ddk.eq.0} with $\beta_1$ replaced with $\beta_2$. Whilst this does \emph{not} necessarily always hold, we now show that it \emph{does hold} when $\lambda_0$ is neither $-k$ nor $k-7$.

\begin{lemma} \label{dd*.beta1.beta2.lem}
Let $\beta_1$, $\beta_2$, and $\lambda_0$ be as in Proposition~\ref{ddk.asymptoticexpansionprop} and let $\chi_i$ and $\tilde{\dd}^k_{l+1,\beta_1}$ be as in Corollary \ref{ddk.asymptoticexpansioncor2}.   If $\lambda_0\notin\{-k,k-7\}$ then
\begin{equation}
\tilde{\dd}^k_{l+1,\beta_1}  :  \Omega^k_{l+1,\beta_2} + \chi_i \mathcal{K}(\lambda_0)_{\ddci^k}\to d (\Omega^k_{l+1, \beta_2}) + d^* (\Omega^k_{l+1, \beta_2}).
\end{equation}
\end{lemma}
\begin{proof}
Let $\gamma_i\in\mathcal{K}(\lambda_0)_{\ddci^k}$.  If $\lambda_0\neq -k$ then we see from Lemma~\ref{KPspacelemma} that
$$\gamma_i = d (r^{\lambda_0 + k} \alpha_i  )$$
for some $\alpha_i\in\Omega^{k-1}(\Sigma_i)$. Observe that
$$\eta_i=\chi_i\gamma_i-d(\chi_i r^{\lambda_0 + k} \alpha_i  )=-d(\chi_i)\wedge  r^{\lambda_0 + k}\alpha_i$$
is smooth and compactly supported with $d\eta_i=d(\chi_i\gamma_i)$. Therefore, for any $\gamma\in \Omega^k_{l+1,\beta_2}$, we have that
$$d(\gamma+\chi_i\gamma_i)=d(\gamma+\eta_i)\in d(\Omega^k_{l+1,\beta_2}).$$
Moreover, we see that if $\lambda\neq k-7$ then Lemma~\ref{KPspacelemma} shows that
$$*_C\gamma_i=d (r^{\lambda_0+7-k}\sigma_i)$$
for some $\sigma_i\in \Omega^{6-k}(\Sigma_i)$.  Recall that $|*_M\gamma_i-*_C\gamma_i|=O(r^{\lambda_0+\nu})$ and $|\beta_2-\lambda_0|<|\nu|$. Therefore, 
$$\zeta_i=\chi_i*_M\gamma_i-d(\chi_ir^{\lambda_0+7-k}\sigma_i)=\chi_i(*_M\gamma_i-*_C\gamma_i)-d(\chi_i)\wedge r^{\lambda_0+7-k}\sigma_i$$
lies in $\Omega^{7-k}_{l+1,\beta_2}$ and satisfies $d\zeta_i=d(\chi_i*_M\gamma_i)$.  We deduce that, for any $\gamma\in \Omega^k_{l+1,\beta_2}$ we have
$$d*_M(\gamma+\chi_i\gamma_i)=d(*_M\gamma+\zeta_i)\in d(\Omega^{7-k}_{l+1,\beta_2}).$$
Taking the Hodge star completes the proof. 
\end{proof}

We can now deduce our index change formula for the operator $\dd^k_{l, \lambda}$, away from the exceptional rates $-k$ and $k-7$.

\begin{thm} \label{Pindexchangethm}
Let $\mu_- < \mu_+$ be two \emph{noncritical rates} for $\dd^k$ on $M$. Suppose that both $-k$ and $k-7$ are \emph{not in} the set $\mathcal D_{\ddc^k} \cap (\mu_- , \mu_+)$ of critical rates between $\mu_-$ and $\mu_+$. Then the difference in the indices of $\dd^k_{l, \mu_-}$ and $\dd^k_{l, \mu_+}$ is given by
\begin{align*}
\ind (\dd^k_{l, \mu_+}) - \ind (\dd^k_{l, \mu_-}) \, & = \, \sum_{\lambda \in \mathcal D_{\ddc^k} \cap (\mu_- , \mu_+)} \dim \mathcal K(\lambda)_{\ddc^k}. & \text{(AC)} \\
\ind (\dd^k_{l, \mu_+}) - \ind (\dd^k_{l, \mu_-}) \, & = \, - \sum_{i=1}^n \, \, \sum_{\lambda \in \mathcal D_{\ddci^k} \cap (\mu_-, \mu_+)} \dim \mathcal K(\lambda)_{\ddci^k}. & \text{(CS)}
\end{align*}
\end{thm}
\begin{proof}
This is now immediate by combining Corollary~\ref{ddk.asymptoticexpansioncor2} and Lemma~\ref{dd*.beta1.beta2.lem}.
\end{proof}

\begin{rmk} \label{kernel.only.rmk}
The changes in the index of $\dd^k_{l, \lambda}$ at $\lambda = -k$ or $\lambda = k-7$ only arise from changes in the kernel. To see this, first notice that Corollary~\ref{ddk.asymptoticexpansioncor} says that $\ker \dd^k_{l, \lambda}$ can only change at a critical rate for $\dd^k$. Now consider $\coker \dd^k_{l, \lambda}$. By Proposition~\ref{Pindexchangeprop}, the cokernel will only change as we cross the rate $\lambda$ if there exist new elements of $\ker(d + d^*)_{- 6 - \lambda}$ that are in $\Omega^{k-1} \oplus \Omega^{k+1}$ but are \emph{transverse} to the subspace $\mathcal H^{k-1}_{- 6 - \lambda} \oplus \mathcal H^{k+1}_{-6 - \lambda}$.  Such elements must be asymptotic at the $i^{\text{th}}$ end to forms of degree $k-1$ plus degree $k+1$ on the cone $C_i$, homogeneous of order $-6 - \lambda$, which are in the kernel of $d + \dsci$, but which are \emph{transverse} to the spaces of closed and coclosed $(k-1)$-forms and closed and coclosed $(k+1)$-forms on $C_i$ homogeneous of order $-6 - \lambda$. In the language of Notation~\ref{ABCnotation}, this corresponds to rates $\lambda$ for which the quotient space $A(\lambda) / B(\lambda)$ is nonzero for some asymptotic cone $C_i$. If $\lambda \notin \{-k, k-7 \}$, then Theorem~\ref{cones24thm} says that this quotient $A(\lambda) / B(\lambda)$ is $C(\lambda)$, and thus by Remark~\ref{conesCrmk} such a $\lambda$ is a critical rate of $\dd^k_{l, \lambda}$. For $\lambda \in \{-k,  k-7 \}$, Theorem~\ref{cones24thm} tells us that $C(\lambda) \subseteq B(\lambda)$, so these rates cannot contribute to changes in $\coker \dd^k_{l, \lambda}$.
\end{rmk}

We will determine the changes in the kernel of $\dd^3_{l, \lambda}$ for $\lambda = -3$ and $\lambda = -4$ in Proposition~\ref{kernelchange34prop}.

\subsection{Topological results for conifolds} \label{topologicalconifoldsec}

The following proposition (in a general setting) appeared originally in Lockhart~\cite[Example 0.16]{Lock}, but a version in the setting of AC/CS manifolds is stated in~\cite[Theorem 6.5.2]{Lth}.
\begin{prop} \label{LMcohomologyprop}
Let $\mathcal H^k_{L^2}$ denote the subspace of $L^2 (\Lambda^k (T^* M))$ consisting of closed and coclosed $k$-forms. Then we have
\begin{align*}
\mathcal H^k_{L^2} \, & \cong \, H_{\mathrm{cs}}^k(M, \R), \text{ for $0 \leq k \leq 3$}, &
\mathcal H^k_{L^2} \, & \cong \, H^k(M, \R), \text{ for $4 \leq k \leq 7$}, & \text{(AC)} \\
\mathcal H^k_{L^2} \, & \cong \, H^k(M, \R), \text{ for $0 \leq k \leq 3$}, &
\mathcal H^k_{L^2} \, & \cong \, H_{\mathrm{cs}}^k(M, \R), \text{ for $4 \leq k \leq 7$}, & \text{(CS)}
\end{align*}
where $H_{\mathrm{cs}}^k(M, \R)$ denotes the degree $k$ compactly supported cohomology group of $M$, and $H^k(M, \R)$ denotes the degree $k$ de Rham cohomology group of $M$. The isomorphism $\mathcal H^k_{L^2} \cong H^k(M, \R)$ is given by the natural map $\alpha \mapsto [\alpha]$. The isomorphism $\mathcal H^k_{L^2} \cong H_{\mathrm{cs}}^k(M, \R)$ is given by the composition of the Hodge star $\st : \mathcal H^k_{L^2} \to \mathcal H^{7-k}_{L^2}$, the natural map $\mathcal H^{7-k}_{L_2} \to H^{7-k}(M, \R)$, and Poincar\'e duality $H^{7 - k}(M, \R) \cong H^k_{\mathrm{cs}}(M, \R)$.
\end{prop}

\begin{cor} \label{LMcohomologycor}
We have that
\begin{equation} \label{dimensionbasecaseeq}
\left. \begin{aligned} \mathcal H^3_{\lambda} \, & \cong \, H^3_{\mathrm{cs}} (M, \R) \quad & \text{(AC)} \\  \mathcal H^3_{\lambda} \, & \cong \, H^3(M, \R) \quad & \text{(CS)} \end{aligned} \quad \right\} \quad \text{for all $\lambda \in (-4,-3)$}.
\end{equation}
\end{cor}
\begin{proof}
This is immediate from Corollary~\ref{kernelchangecor} and Proposition~\ref{LMcohomologyprop}.
\end{proof}

The next proposition comes directly from~\cite[Corollary 7.10]{Lock}. It will be used in the proof of Proposition~\ref{kernelchange34propalt} below.
\begin{prop} \label{LMcohomologyprop2}
The classes in $H^3(M, \R)$ that can be represented by forms in $\mathcal H^3_{L_2} = \mathcal H^3_{-\frac{7}{2}}$ are precisely those classes that lie in the image of the natural inclusion of $H^3_{\mathrm{cs}}(M, \R)$ in $H^3(M, \R)$.
\end{prop}

\begin{defn} \label{Upsilonmapsdefn}
Let $\Sigma = \sqcup_{i=1}^n \Sigma_i$ be the disjoint union of the links of the $n$ asymptotic cones for the $\G$~conifold $M$. Of course, in the AC case we have $n = 1$. Then $H^k(\Sigma, \R) = \bigoplus_{i=1}^n H^k(\Sigma_i, \R)$. Also, by embedding each $\Sigma_i$ in the $i^{\text{th}}$ end of $M$, we get a smooth embedding of $\Sigma$ in $M$, which induces a linear map
\begin{equation*}
\Upsilon^k \, : \, H^k(M, \R) \to H^k(\Sigma, \R).
\end{equation*}
This linear map is most easily described as follows. Let $[\gamma] \in H^k(M, \R)$ be a cohomology class, represented by a closed $k$-form $\gamma$ on $M$. Then the $i^{\text{th}}$ component of the class $\Upsilon^k ([\gamma]) \in \bigoplus_{i=1}^n H^k(\Sigma_i, \R)$ is the class represented by the \emph{restriction} of $\gamma$ to $\Sigma_i$. Note that in general $\Upsilon^k$ is neither injective nor surjective.
\end{defn}
\begin{rmk} \label{Upsilonmapsrmk}
The images of the maps $\Upsilon^k$ for $k = 3,4$ are related to \emph{topological obstructions} to the desingularization of CS $\G$~conifolds, as discussed in~\cite[Section 5]{Kdesings}.
\end{rmk}

From~\cite[$\S$2.4]{JSL1}, any conifold $M$ gives rise to a long exact sequence
\begin{equation} \label{longexactsequenceeq}
\cdots \longrightarrow H_{\mathrm{cs}}^k(M, \R) \overset{\mathcal{I}^k}{\longrightarrow} H^k (M, \R) \overset{\Upsilon^k}{\longrightarrow} \bigoplus_{i=1}^n H^k(\Sigma_i, \R) \overset{\partial^k}{\longrightarrow} H_{\mathrm{cs}}^{k+1}(M, \R) \longrightarrow \cdots
\end{equation}
where $\Upsilon^k: H^k (M, \R) \rightarrow \bigoplus_{i=1}^n H^k(\Sigma_i, \R)$ is the map from Definition~\ref{Upsilonmapsdefn}, and $\mathcal{I}^k : H_{\mathrm{cs}}^k(M, \R) \to H^k(M, \R)$ is the natural map induced from inclusion of the complex of compactly supported forms into the complex of all smooth forms. This is the long exact sequence for cohomology of $M$ relative to its topological boundary $\Sigma$.

Let $b^k = \dim H^k(M, \R)$ and $b^k_{\mathrm{cs}} = \dim H^k_{\mathrm{cs}}(M)$ be the ordinary and compactly supported $k^{\text{th}}$ Betti numbers of $M$, respectively. Note that by Poincar\'e duality we have $H^k (M, \R) \cong H^{7-k}_{\mathrm{cs}}(M, \R)$ and thus $b^k = b^{7-k}_{\mathrm{cs}}$. The next lemma contains results that will be used to compute the virtual dimension of the conifold moduli space in Section~\ref{dimensionsec} and for the applications in Section~\ref{smoothCSsec}.

\begin{lemma} \label{exactsequencelemma}
Let $M$ be a $\G$~conifold. The following equations hold.
\begin{align} \label{imcseq}
b^k - \dim(\im \Upsilon^k) \, & = \, \dim(\im \mathcal{I}^k) \, = \, \dim (\im (H_{\mathrm{cs}}^k \to H^k )) , \\ \label{exactkereq1} \dim(\ker \Upsilon^k) \, & = \, b^k - \dim (\im \Upsilon^k), \\ \label{exactkereq2} \dim (\ker \Upsilon^k) \, & = \, b^k_{\mathrm{cs}} - \dim (\im \Upsilon^{7-k}), \\
\label{exacteq4} \dim \bigl( H^k( \Sigma, \R) \bigr) \, & = \, \dim (\im \Upsilon^k) + \dim (\im \Upsilon^{6-k}).
\end{align}
\end{lemma}
\begin{proof}
Equation~\eqref{exactkereq1} is just the rank-nullity theorem. From~\eqref{exactkereq1} and the exactness of~\eqref{longexactsequenceeq}, we find
\begin{equation*}
b^k \, = \, \dim (\im \Upsilon^k) + \dim (\ker \Upsilon^k) \, = \, \dim (\im \Upsilon^k) + \dim ( \im \mathcal{I}^k)
\end{equation*}
from which we immediately obtain~\eqref{imcseq}.

Using the fact that the Hodge star operator takes compactly supported forms to compactly supported forms, it is easy to see that the \emph{Poincar\'e pairing} between $\ker \Upsilon^k$ and $\ker \Upsilon^{7-k}$ given by $[\alpha], [\beta] \mapsto \int_M (\alpha \wedge \beta)$ is nondegenerate. Hence $\dim (\ker \Upsilon^k) = \dim (\ker \Upsilon^{7-k})$. Thus we have
\begin{align*}
\dim (\ker \Upsilon^k) \, & = \, \dim (\ker \Upsilon^{7-k}) \\ & = \, b^{7-k} - \dim(\im \Upsilon^{7-k}) \\ & = \, b^k_{\mathrm{cs}} - \dim(\im \Upsilon^{7-k})
\end{align*}
which establishes~\eqref{exactkereq2}. For equation~\eqref{exacteq4}, we apply repeatedly the long exact sequence~\eqref{longexactsequenceeq} and rank-nullity, to obtain
\begin{align*}
\dim \bigl( H^k (\Sigma, \R) \bigr) - \dim (\im \Upsilon^k ) \, & = \, \dim \bigl( H^k (\Sigma, \R) \bigr) - \dim (\ker \partial^k ) \\ & = \, \dim (\im \partial^k) \\ & = \dim (\ker \mathcal I^{k+1} ) \\ & = \, b^{k+1}_{\text{cs}} - \dim (\im \mathcal I^{k+1} ) \\ & = \, b^{k+1}_{\text{cs}} - \dim (\ker \Upsilon^{k+1} ) \\ & = \, \dim (\im \Upsilon^{6-k})
\end{align*}
where we have used~\eqref{exactkereq2} in the last step.
\end{proof}

\begin{cor} \label{Upsiloncor}
The Hodge star operator $\sts$ on $\Sigma$ maps $\im \Upsilon^k$ isomorphically onto $(\im \Upsilon^{6-k})^{\perp}$. That is, the harmonic representatives of elements of $\im \Upsilon^k \subseteq H^k(\Sigma, \R)$ are sent by $\sts$ onto harmonic representatives of elements of $(\im \Upsilon^{6-k})^{\perp} \subseteq H^{6-k}(\Sigma, \R)$.
\end{cor}
\begin{proof}
Let $[\sigma] \in \im \Upsilon^k$ have harmonic representative $\sigma \in \Omega^k (\Sigma)$. Then there exists a closed $k$-form $\zeta$ on $M$ such that $\rest{\zeta}{{ \{ R \} \times \Sigma}} = \sigma$. Let $[\tau] \in \im \Upsilon^{6-k}$ have harmonic representative $\tau \in \Omega^{6-k} (\Sigma)$. Then there exists a closed $(6-k)$-form $\eta$ on $M$ such that $\rest{\eta}{{ \{ R \} \times \Sigma}} = \tau$. In the CS case, let $M_{R} = \{ x \in M; \varrho(x) \geq R \}$, and observe that $\partial M_{R} = \{R \} \times \Sigma$. (In the AC case we just reverse the inequality in the definition of $M_{R}$). Now, using Stokes's theorem, we find
\begin{align*}
\langle \langle \sts \sigma, \tau \rangle \rangle_{\Sigma} \, & = \, \pm \int_{\Sigma} \tau \wedge \sigma \, = \, \pm R^6 \int_{ \{ R \} \times \Sigma} \tau \wedge \sigma \\ & = \, \pm R^6 \int_{\partial M_{R}} \zeta \wedge \eta \, = \, \pm R^6 \int_{M_{R}} d( \zeta \wedge \eta) \, = \, 0.
\end{align*}
Thus we have $[\sts \sigma] \in (\im \Upsilon^{6-k})^{\perp}$. Counting dimensions using equation~\eqref{exacteq4} completes the proof.
\end{proof}

Let us denote by $[\gamma]_{\Sigma}$ the cohomology class in $H^k (\Sigma, \R)$ of a closed form $\gamma$ on $\Sigma$, and by $[\eta]_{M}$ the cohomology class in $H^k (M, \R)$ of a closed form $\eta$ on $M$. The next lemma is used in the proof of Proposition~\ref{kernelchange34prop}.

\begin{lemma} \label{kernelchange34preliminarylemma}
Let $M$ be a $\G$~conifold.
\begin{itemize}
\item In the AC case, let $[\gamma]_{\Sigma} \in \im \Upsilon^3 \subseteq H^3 (\Sigma, \R)$. There exists $\eta \in \mathcal H^3_{-3 + \e}$ such that $\Upsilon^3 [\eta]_M = [\gamma]_{\Sigma}$.
\item In the CS case, let $[\gamma]_{\Sigma} \in \im \Upsilon^4 \subseteq H^4 (\Sigma, \R)$. There exists $\eta \in \mathcal H^4_{-4 - \e}$ such that $\Upsilon^4 [\eta]_M = [\gamma]_{\Sigma}$.
\end{itemize}
\end{lemma}
\begin{proof}
Consider the AC case. By~\cite[Corollary 5.9]{M} there exists a \emph{smooth} $3$-form $\zeta$ on $M$ such that $| \zeta | = O(\varrho^{-3})$ on the ends with $d\zeta = 0$ and $\Upsilon^3 [\zeta]_M = [\gamma]_{\Sigma}$. Then $\zeta \in \Omega^3_{l, -3 + \e}$ for any $\e > 0$ and $l \geq 0$. We are in the non-$L^2$ regime of Theorem~\ref{kformsHodgethm} and can therefore apply equation~\eqref{kformsHodgenonL2eq2} to the closed form $\zeta$ to deduce that $\zeta = \eta + d \alpha$ for some $\eta \in \mathcal H^3_{-3 + \e}$. Then $[\zeta]_M = [\eta]_M$ so $\Upsilon^3 [\eta]_M = [\gamma]_{\Sigma}$ as required. In the CS case, we again use~\cite[Corollary 5.9]{M} to obtain a smooth $4$-form $\zeta$ on $M$, with $| \zeta | = O(\varrho^{-4})$ on the ends, such that $d\zeta = 0$ and $\Upsilon^4 [\zeta]_M = [\gamma]_{\Sigma}$. This time, $\zeta \in \Omega^3_{l, -4 - \e}$ for any $\e > 0$ and $l \geq 0$. We can now apply equation~\eqref{kformsHodgenonL2eq2} as before to deduce the result.
\end{proof}

We can now use Lemma~\ref{kernelchange34preliminarylemma} to establish the main result of this section.

\begin{prop} \label{kernelchange34prop}
Let $\lambda_0$ be a critical rate for $d + \dsm$ (understood to be a ``constant'' $n$-tuple in the CS case), and let $\e > 0$ be chosen so that there are no other critical rates in $(\lambda_0 - \e, \lambda_0 + \e)$. Then for $\lambda_0 = -3$ we have
\begin{equation} \label{kernelchange34eq1}
\begin{aligned}
\dim \mathcal H^3_{-3 + \e} - \dim \mathcal H^3_{-3 - \e} \, & = \, \dim (\im \Upsilon^3) & (AC), \\
\dim \mathcal H^3_{-3 + \e} - \dim \mathcal H^3_{-3 - \e} \, & = \, - \dim (\im \Upsilon^3) & (CS),
\end{aligned}
\end{equation}
and for $\lambda_0 = -4$ we have
\begin{equation} \label{kernelchange34eq2}
\begin{aligned}
\dim \mathcal H^3_{-4 + \e} - \dim \mathcal H^3_{-4 - \e} \, & = \, \dim (\im \Upsilon^4) & (AC), \\
\dim \mathcal H^3_{-4 + \e} - \dim \mathcal H^3_{-4 - \e} \, & = \, - \dim (\im \Upsilon^4) & (CS).
\end{aligned}
\end{equation}
\end{prop}
\begin{proof}
Consider first the AC case. By Lemma~\ref{kernelchangelemma}, the space $\mathcal H^3_{\lambda}$ changes by the addition of forms that are asymptotic to closed and coclosed $3$-forms in $\mathcal K(-3)_{d + \dsc}$ as we cross $\lambda_0 = -3$. Also, by Lemma~\ref{nologslemma} and Remark~\ref{ddstarhomokernelrmk}, a $3$-form $\upsilon$ in $\mathcal K(-3)_{d + \dsc}$ must be of the form $\upsilon = \beta$ for some harmonic $3$-form $\beta$ on $\Sigma$. Explicitly, we have that if $\gamma_1 \in \mathcal H_{-3 + \e}$, then \emph{on the end} we have
\begin{equation*}
\gamma_1 \, = \, (h^{-1})^* \beta + \tilde \gamma + \gamma_2, 
\end{equation*}
where $\tilde \gamma + \gamma_2 = O(\varrho^{-3 + \e + \nu}) + O(\varrho^{-3 - \e})$. If $\e$ is sufficiently small so that $-3 + \e + \nu < -3$, then Lemma~\ref{exactformslemma} tells us that the $3$-form component of $\tilde \gamma + \gamma_2$ is exact on the end. Hence, we find that $\Upsilon^3[\gamma_1] = [\beta]$, so a necessary condition for $\beta$ to define a $3$-form on $M$ which adds to $\mathcal H^3_{\lambda}$ as $\lambda$ crosses $\lambda_0 = -3$ is that $[\beta] \in \im \Upsilon^3$. Sufficiency follows from Lemma~\ref{kernelchange34preliminarylemma}. This establishes~\eqref{kernelchange34eq1}.

In exactly the same way, the space $\mathcal H^3_{\lambda}$ changes by the addition of forms that are asymptotic to closed and coclosed $3$-forms in $\mathcal K(-4)_{d + \dsc}$ as we cross $\lambda_0 = -4$. This time, Lemma~\ref{nologslemma} and Remark~\ref{ddstarhomokernelrmk} says that a $3$-form $\upsilon$ in $\mathcal K(-4)_{d + \dsc}$ must be of the form $\upsilon = r^{-2} dr \wedge \alpha$ for some harmonic $2$-form $\alpha$ on $\Sigma$. But $\stm \mathcal H^3_{\lambda} = \mathcal H^4_{\lambda}$, and $\stm \upsilon = \stm (r^{-2} dr \wedge \alpha) = \sts \alpha$, a harmonic $4$-form on $\Sigma$. Thus the previous argument can be repeated to conclude that changes in $\mathcal H^4_{\lambda}$ (and hence to $\mathcal H^3_{\lambda}$) as we cross $\lambda_0 = -4$ correspond to elements of $\im \Upsilon^4$. (Necessity follows exactly as above, but this time sufficiency is easier, since it follows directly from the isomorphism $\mathcal H^3_{-4 + \e} \cong H^4 (M, \R)$ given in Proposition~\ref{LMcohomologyprop}.) This establishes~\eqref{kernelchange34eq2}.

To prove the CS case, the arguments for $\lambda = -3, -4$ are analogous to the $\lambda = -4, -3$ arguments of the AC case, respectively, using the CS statement of Lemma~\ref{kernelchange34preliminarylemma}.
\end{proof}

We pause here to note that we can reinterpret Proposition~\ref{kernelchange34prop} in terms of cohomology, as follows.
\begin{prop} \label{kernelchange34propalt}
Let $\chi^k$ denote the natural map $\chi^k : \mathcal H^k \to H^k (M, \R)$ given by $\chi^k (\eta) = [\eta]_M$.
\begin{itemize}
\item Suppose we are in the AC case.  Then
\begin{equation*}
\mathcal H^3_{-4 - \e} \, \subseteq \, \mathcal H^3_{-4 + \e} \, = \, \mathcal H^3_{-3 - \e} \, \subseteq \, \mathcal H^3_{-3 + \e}
\end{equation*}
and
\begin{enumerate}[({TAC}1)]
\item $\quad \chi^3 : \mathcal H^3_{-3 + \e} \rightarrow H^3(M, \R)$ is surjective,
\item $\quad \ker \Upsilon^3 = \chi^3 ( \mathcal H^3_{-3 - \e} )$,  
\item $\quad \chi^4: \mathcal H^4_{-4 + \e} \rightarrow H^4(M, \R)$ is an isomorphism,
\item $\quad \ker \Upsilon^4 = \chi^4 ( \mathcal H^4_{-4 - \e} )$.
\end{enumerate}
\item Suppose we are in the CS case.  Then
\begin{equation*}
\mathcal H^3_{-4 - \e} \, \supseteq \, \mathcal H^3_{-4 + \e} \, = \, \mathcal H^3_{-3 - \e} \, \supseteq \, \mathcal H^3_{-3 + \e}
\end{equation*}
and
\begin{enumerate}[({TCS}1)]
\item $\quad \chi^4 : \mathcal H^4_{-4 - \e} \rightarrow H^4(M, \R)$ is surjective,
\item $\quad \ker \Upsilon^4 = \chi^4 ( \mathcal H^4_{-4 + \e} )$,  
\item $\quad \chi^3: \mathcal H^3_{-3 - \e} \rightarrow H^3(M, \R)$ is an isomorphism,
\item $\quad \ker \Upsilon^3 = \chi^3 ( \mathcal H^3_{-3 + \e} )$.
\end{enumerate}
\end{itemize}
\end{prop}
\begin{proof}
We prove the AC case. The CS case is essentially the same. The relation $\mathcal H^3_{-4 - \e} \subseteq \mathcal H^3_{-4 + \e} = \mathcal H^3_{-3 - \e} \subseteq \mathcal H^3_{-3 + \e}$ is just Corollary~\ref{kernelchangecor}. Statement (TAC1) is Lemma~\ref{kernelchange34preliminarylemma}. Statement (TAC2) follows from Proposition~\ref{LMcohomologyprop2}, which says $\im \chi^3 = \im (H^3_{\mathrm{cs}}(M, \R) \to H^3(M, \R) )$, and the long exact sequence~\eqref{longexactsequenceeq}. Statement (TAC3) is part of Proposition~\ref{LMcohomologyprop}. Finally, (TAC4) follows from the proof of equation~\eqref{kernelchange34eq2}, statement (TAC3), and the rank-nullity theorem applied to the map $\Upsilon^4 : H^4 (M, \R) \to \oplus_{i=1}^n H^4 (\Sigma_i, \R)$.
\end{proof}

\begin{rmk} \label{kernelchange34propaltrmk}
We see that to go between the AC and CS cases, we effectively switch $3$ with $4$ in the maps and cohomology groups, and $+\e$ with $-\e$.
\end{rmk}
 
The results of this section will be used in Section~\ref{dimensionsec} to compute the virtual dimension of the moduli space, which will have both topological and analytic components.

\subsection{Parallel tensors on \texorpdfstring{$\mathbf{\G}$}{G2} conifolds} \label{parallelsec}

The holonomy $H(\nabla)$ of a connection $\nabla$ on the tangent bundle $TM$ of a connected manifold $M$ is contained in (and often equal to) the subgroup of the general linear group whose action fixes all parallel tensors on $M$. See, for example, Joyce~\cite[Chapter 2]{J4} for more details. In particular, the holonomy of the Levi-Civita connection on an oriented Riemannian manifold $M^n$ is reduced from $\mathrm{SO}(n)$ by each additional parallel tensor. On an irreducible $\G$~manifold (one whose holonomy is exactly $\G$) the only parallel tensors are the metric $g$, the volume form $\vol$, the $\G$~structure $\ph$, and the dual $4$-form $\ps = \st \ph$. Since $\G$~conifolds are all irreducible, they admit no nontrivial parallel $1$-forms.

We now recall the Bochner--Weitzenb\"ock formula, valid for any Riemannian manifold $(M, g)$. Let $X$ be a $1$-form. Then
\begin{equation*}
\langle \Delta X, X \rangle \, = \, \langle \nabla^* \nabla X, X \rangle + \mathrm{Ric}(X, X)
\end{equation*}
where $\Delta$ is the Hodge Laplacian, $\nabla$ is the Levi-Civita covariant derivative, and $\mathrm{Ric}$ is the Ricci tensor of $(M, g)$, with indices raised by the metric to become a symmetric bilinear form on $1$-forms. Since all $\G$~manifolds are Ricci-flat, the last term above vanishes, and we have
\begin{equation} \label{bochnerformulaeq}
\langle \Delta X, X \rangle \, = \, \langle \nabla^* \nabla X, X \rangle.
\end{equation}

\begin{lemma} \label{bochnerlemma}
Let $M$ be a $\G$~conifold. Let $f$ be a harmonic function and let $X$ be a harmonic $1$-form on $M$. If
\begin{equation} \label{bochnerlemmafeq}
f  \, = \, O(\varrho^{\lambda}) \quad \text{ for some $\lambda < 1$ (AC) or $\lambda > - 5$ (CS),}
\end{equation}
then $f$ is constant. If
\begin{equation} \label{bochnerlemmaXeq}
X \, = \, O(\varrho^{\lambda}) \quad \text{ for some $\lambda \leq 0$ (AC) or $\lambda \geq -5$ (CS),}
\end{equation}
then $X = 0$.
\end{lemma}
\begin{proof}
We give the proof in the AC case. First suppose that $X$ is a harmonic $1$-form such that $X =  O(\varrho^{\lambda})$ for some $\lambda<-\frac{5}{2}$. We want to integrate both sides of~\eqref{bochnerformulaeq} over $M$. Note that
\begin{align*}
\langle \nabla^* \nabla X, X \rangle \, & = \, -g^{ij} (\nabla_i \nabla_j X_k) X_m g^{km} \\
& = \, - g^{ij} \nabla_i \bigl( (\nabla_j X_k) X_m g^{km} \bigr) + g^{ij} g^{km} (\nabla_j X_k) (\nabla_i X_m) \\
& = \, d^* Y + | \nabla X |^2
\end{align*}
for the vector field $Y = \langle \nabla X, X \rangle$. Since $X \in L^2_{l, \lambda}$, we have $\nabla X \in L^2_{l-1, \lambda - 1}$ and thus the vector field $Y = \langle \nabla X, X \rangle$ is $O(\varrho^{2 \lambda - 1})$ as $\varrho \to \infty$. Let $M_R = \{ x \in M; \varrho(x) \leq R \}$, and observe that $\partial (M_R) \cong \{R \} \times \Sigma$. Hence, by Stokes's Theorem and the fact that $Y = O(\varrho^{2 \lambda -1})$, for $R$ sufficiently large we have
\begin{equation*}
\left| \int_{M_R} (d^* Y) \, \volm \right| \, = \, \left| \int_{\partial (M_R)} (Y \hk \volm) \right| \,
\leq \, C R^{2 \lambda - 1} \left| \int_{\{R\} \times \Sigma} \vol_{\{R\} \times \Sigma} \right| \, = \, \tilde C R^{2 \lambda + 5}
\end{equation*}
which goes to zero as $R \to \infty$, since $\lambda < - \frac{5}{2}$. Therefore, since $\Delta X = 0$, when we integrate both sides of~\eqref{bochnerformulaeq} over $M$, we obtain
\begin{equation}
0 \, = || \nabla X ||^2_{L_2}. 
\end{equation}
Hence $\nabla X = 0$, so $X$ is a parallel $1$-form. But the hypothesis that $M$ is a $\G$~conifold then implies that $X = 0$. To conclude we observe that the space of harmonic 1-forms on $M$ with decay rate $O(\varrho^\nu)$ does not change for $\nu \in [-5,0]$ by 
Proposition~\ref{excludeextensionprop} and Theorem~\ref{kernelthm}.

The statement about functions follows in the same way by integrating the equation $\langle \Delta f, f \rangle = \langle \nabla^* \nabla f, f \rangle$ and using the excluded rates from Proposition~\ref{excludeextensionprop}.

The CS case is almost identical except that there are $n$ ends instead of just one, and $\varrho \to 0$ on each end instead of $\varrho \to \infty$. One can argue by reversing the appropriate inequalities on the decay and using the lower bound for the excluded rates instead of the upper bound.
\end{proof}

\begin{lemma} \label{bochnerlemma2}
Let $M$ be a $\G$~conifold. Let $X$ be a $1$-form on $M$ that satisfies $d d^* X + \frac{2}{3} d^* d X = 0$, and
\begin{equation*}
X \, = \, O(\varrho^{\lambda}) \quad \text{ for some $\lambda \leq 0$ (AC) or $\lambda \geq -5$ (CS).}
\end{equation*}
Then we have $X = 0$.
\end{lemma}
\begin{proof}
The proof is very similar to the proof of Lemma~\ref{bochnerlemma}, using $d d^* X + \frac{2}{3} d^* d X = \Delta X - \frac{1}{3} d^* d X$ and $| dX | \leq | \nabla X|$. One first proves it for $X = O(\varrho^{\lambda})$ with $\lambda<-\frac{5}{2}$, and the uses the excluded rates on the cones in Proposition~\ref{excludeextensionprop2} to conclude.
\end{proof}

\subsection{A gauge-fixing condition on moduli spaces of \texorpdfstring{$\mathbf{\G}$}{G2} conifolds} \label{gaugefixingsec}

In this section we discuss a gauge-fixing condition on moduli spaces of $\G$~conifolds. At first our discussion is quite general, to motivate the definition of the gauge-fixing condition that we choose.

Let $(M, \ph)$ be a $\G$~manifold, which is \emph{not} necessarily compact. Let $\mathcal T$ be the space of all torsion-free $\G$~structures on $M$. Then the space $\mathcal D$ of diffeomorphisms of $M$ acts on $\mathcal T$ by pullback. If $F \in \mathcal D$, then
\begin{equation*}
F : \ph \mapsto F^* \ph.
\end{equation*}
Consider a smooth curve $F_t = \exp(tX)$ in $\mathcal D$, where $X$ is a smooth vector field on $M$. This path passes through the identity diffeomorphism $F_0 = \mathrm{Id}_M$ at $t=0$. Therefore, the tangent space $T_{\ph} (\mathcal D \cdot \ph)$ at $\ph$ to the orbit $\mathcal D \cdot \ph$ is spanned by elements of the form $\rest{\ddt}{t=0} (F_t^* \ph) = \mathcal L_X \ph = d(X \hk \ph)$. Thus we have
\begin{equation} \label{gaugefixeq1}
T_{\ph} (\mathcal D \cdot \ph) \, = \, d( \Omega^2_7).
\end{equation}
Let $\tilde \ph$ be another torsion-free $\G$~structure on $M$, such that $\tilde \ph = \ph + \eta$ for some smooth $3$-form $\eta$. Since both $\ph$ and $\tilde \ph$ are torsion-free and thus closed, we must have $d \eta = 0$. In order to break the diffeomorphism invariance, we want to consider those new $\G$~structures for which $\tilde \ph - \ph = \eta$ is \emph{transverse} to the infinitesimal diffeomorphisms which, as explained above, are all of the form $\mathcal L_X \ph = d( X \hk \ph)$ for a smooth vector field $X$. Suppose that $\eta$ lies in $L^2 (\Lambda^3 (T^* M))$. Then the condition that $\eta$ is actually $L^2$-orthogonal to $d(\Omega^2_7)$ is that
\begin{equation*}
0 \, = \, \langle \langle d(X \hk \ph) , \eta \rangle \rangle \, = \langle \langle X \hk \ph , d^* \eta \rangle \rangle.
\end{equation*}
Notice that this condition is always implied by the stronger condition that $\pi_7 (d^* \eta) = 0$ \emph{pointwise}. This observation motivates the following definition.

\begin{defn} \label{gaugefixingdefn}
Suppose $\tilde \ph = \ph + \eta$ is another $\G$~structure on $M$, for some closed $3$-form $\eta$. We say that $\tilde \ph$ satisfies the \emph{gauge-fixing condition} (with respect to $\ph$) if
\begin{equation*}
\pi_7 (d^* \eta) \, = \, 0.
\end{equation*}
Here $\pi_7$ and $d^*$ are taken with respect to the $\G$~structure $\ph$.
\end{defn}

We now relate this gauge-fixing condition to a slightly different condition in the conifold case.

\begin{lemma} \label{gaugefixinglemmaA}
Let $M$ be a $\G$~conifold. Let $\zeta$ be a smooth $3$-form such that $d \zeta = 0$ and $\pi_7 (d^* \zeta) = 0$. Let
\begin{equation} \label{gaugefixinglemmaAeq}
\zeta \, = \, \pi_1 \zeta + \pi_7 \zeta + \pi_{27} \zeta \, = \, f \ph + \st (X \wedge \ph) + \pi_{27} \zeta
\end{equation}
for some function $f$ and some $1$-form $X$. Then $\Delta f = 0$ and $\Delta X = 0$. In addition, if $f$ and $X$ satisfy the decay conditions in~\eqref{bochnerlemmafeq} and~\eqref{bochnerlemmaXeq}, respectively, then $f = c$ is constant and $X = 0$, so $\zeta = c \ph + \pi_{27} \zeta$.
\end{lemma}
\begin{proof}
The fact that $\Delta f = 0$ and $\Delta X = 0$ is precisely Corollary~\ref{special27cor2}. The rest of the statement now follows immediately from Lemma~\ref{bochnerlemma}.
\end{proof}

\begin{rmk} \label{gaugefixingrmk}
The gauge-fixing condition of Definition~\ref{gaugefixingdefn} is used by Joyce~\cite{J4} to study the moduli space of \emph{compact} $\G$~manifolds. We have to modify this gauge-fixing condition for the moduli space of $\G$~conifolds, because of complications arising from noncompactness. Specifically, we achieve this in Theorems~\ref{infinitesimalslicethm} and~\ref{onetoonethm}.
\end{rmk}

\begin{rmk} \label{gaugefixingdesingsrmk}
A slightly different gauge-fixing condition for AC $\G$~manifolds was given in~\cite[Definition 3.3]{Kdesings}. That other condition is discussed further in Section~\ref{gaugefixexistssec} of the present paper, where we establish when it can be achieved. Also, Lemma~\ref{Gph.lemma} relates the two gauge-fixing conditions.
\end{rmk}

We end this section with a result that is essentially a \emph{linearized version} of our Theorem~\ref{onetoonethm} that comes much later, because we have now assembled the tools to state and prove this linearized version. It will be crucial in explicitly describing a certain finite-dimensional space later in Theorem~\ref{infinitesimalslicethm}, consisting of forms which are orthogonal to the linearized action by diffeomorphisms but do not satisfy the natural gauge-fixing condition. We will show that this space can be nonzero for geometric reasons, and is not a defect in our analytic approach.

\begin{thm} \label{linearized.onetoonethm}
Let $M$ be a $\G$~conifold. Let $f$ be a function satisfying the decay condition~\eqref{bochnerlemmafeq} and let $X$ be a $1$-form satisfying~\eqref{bochnerlemmaXeq} and $d^* X = 0$. Let $\zeta$ be a $3$-form on $M$ such that $d \zeta = 0$ and such that
\begin{equation} \label{linearized.onetoonezetadecayeq}
\zeta \, = \, O(\varrho^{\lambda}) \quad \text{ for some $\lambda \leq 0$ (AC) or $\lambda > - 5$ (CS),}
\end{equation}
The following two conditions are equivalent.
\begin{enumerate}[(a)]
\item $\quad \pi_1 \zeta = (f + c) \ph, \quad \pi_7 \zeta = \st (X \wedge \ph), \quad \text{and } d (L_{\ph} \zeta) = 0 \quad \text{for some constant } c$,
\item $\quad d \st \zeta \, = \, \frac{7}{3} df \wedge \ps + 2 dX \wedge \ph$.
\end{enumerate}
\end{thm}
\begin{proof}
Suppose that \emph{(a)} holds. We have $\zeta = (f + c) \ph + \st (X \wedge \ph) + \eta$ for some $\eta \in \Omega^3_{27}$. From~\eqref{Lpeq} we find $L_{\ph} \zeta = \frac{4}{3} (f + c) \ps + X \wedge \ph - \st \eta$, and therefore $d (L_{\ph} \zeta) = 0$ implies $\frac{4}{3} df \wedge \ps + dX \wedge \ph - d \st \eta = 0$. Hence $d \st \zeta = df \wedge \ps + dX \wedge \ps + d \st \eta = \frac{7}{3} df \wedge \ps + 2 dX \wedge \ph$, which is \emph{(b)}. Note that the proof that \emph{(a)} implies \emph{(b)} did not need the hypotheses on $f$ and $X$, nor even that $d \zeta = 0$.

Now suppose that $f$ is a function, $X$ is a 1-form and $\zeta$ is a 3-form satisfying~\eqref{bochnerlemmafeq},~\eqref{bochnerlemmaXeq} and~\eqref{linearized.onetoonezetadecayeq}, respectively, together with $d^*X=0$ and $d\zeta=0$, and that (b) holds. We can write $\zeta = \tilde f \ph + \st (\tilde X \wedge \ph) + \tilde \eta$ for some function $\tilde f$, some $1$-form $\tilde X$, and some $\tilde \eta \in \Omega^3_{27}$. Then we have
\begin{equation} \label{linearized.onetoonetempeq}
d \st \zeta \, = \, d \tilde f \wedge \ps + d \tilde X \wedge \ph + d \st \tilde \eta \, = \, \frac{7}{3} df \wedge \ps + 2 dX \wedge \ph.
\end{equation}
We claim that \emph{(a)} will follow if we can show that $\tilde f = f + c$ and $\tilde X = X$. Indeed, if this is the case, then~\eqref{linearized.onetoonetempeq} becomes $d \st \tilde \eta = \frac{4}{3} df \wedge \ps + dX \wedge \ph$ which is precisely $d (L_{\ph} \zeta) = 0$. To show that $\tilde f = f + c$ and $\tilde X = X$, we can use Corollary~\ref{special27cor} and $d \zeta = 0$ to deduce that $d^* \tilde X = 0$ and that $\pi_7 d^* \zeta = ( - \frac{7}{3} d \tilde f + \frac{4}{3} \curl \tilde X) \hk \ph$. But the hypothesis \emph{(b)} says, using~\eqref{o27eq},~\eqref{o214eq}, and~\eqref{curllemmaeq}, that
\begin{align*}
d^* \zeta \, & = \, - \st d \st \zeta \, = \, - \st \big( \frac{7}{3} df \wedge \ps + 2 dX \wedge \ph \big) \\
& = \, - \frac{7}{3} df \hk \ph + 4 \pi_7 dX - 2 \pi_{14} dX \, = \, \bigl( - \frac{7}{3} df + \frac{4}{3} \curl X \bigr) \hk \ph - 2 \pi_{14} dX.
\end{align*}
Hence we find that $- \frac{7}{3} df + \frac{4}{3} \curl X = - \frac{7}{3} d \tilde f + \frac{4}{3} \curl \tilde X$, or equivalently that $d (\tilde f - f) = \frac{4}{7} \curl (\tilde X - X)$. Taking $d^*$ and $\curl$ of this equation, and using the identities in Remark~\ref{curlrmk} and the fact that $d^* (\tilde X - X) = 0$, gives $\Delta (\tilde f - f) = 0$ and $\Delta (\tilde X - X) = 0$. The decay hypothesis on $\zeta$ in~\eqref{linearized.onetoonezetadecayeq} ensure that both $\tilde f$ and $\tilde X$ satisfy the same conditions~\eqref{bochnerlemmafeq} and~\eqref{bochnerlemmaXeq}, respectively, as $f$ and $X$. Thus the differences $\tilde f - f$ and $\tilde X - X$ also have the same decay, and hence by Lemma~\ref{bochnerlemma} we can conclude that $\tilde X - X = 0$ and $\tilde f - f = c$ for some constant $c$, which is what we needed to show.
\end{proof}

\subsection{Further vanishing results for \texorpdfstring{$\mathbf{\G}$}{G2} conifolds} \label{particularG2analsec}

In this section we present further vanishing results that are particular to $\G$~conifolds. The first result is just a special case of~\cite[Proposition 10.3.4]{J4}, with essentially the same proof, except that it has been adapted to the setting of conifolds. Therefore we need to make assumptions that some forms have a certain decay rate on the ends, and for this reason we give the proof for completeness.

\begin{lemma} \label{threezeroeslemma}
Let $(M, \ph)$ be a $\G$~conifold, so in particular $d \ph = d \Theta (\ph) = 0$. Suppose further that $\tilde \ph$ is another \emph{closed} $\G$~structure on $M$ such that
\begin{equation} \label{threezeroeseq}
d(\Theta(\tilde \ph)) \, = \,  \theta \wedge \ps + dX \wedge \ph
\end{equation}
for some $1$-forms $\theta$ and $X$ on $M$. Further assume that
\begin{equation*}
\left. \begin{aligned}
d(\Theta(\tilde \ph)) \, & = \, O(\varrho^{\lambda}) \\ X \, & = \, O(\varrho^{\lambda + 1}) \\ 
\end{aligned} \right\} \quad \text{ for some $\lambda < - \frac{7}{2}$ (AC) or $\lambda > - \frac{7}{2}$ (CS).}
\end{equation*}
Note that this says, in particular, that $d(\Theta(\tilde \ph))$, $\theta$, and $dX$ are in $L^2$. There is a universal constant $\e$ such that if $\tilde \ph$ is within $\e$ of $\ph$ in the $C^0$ norm on $M$, then $\theta = 0$ and $dX = 0$, so  $d(\Theta(\tilde \ph)) = 0$ and thus $\tilde \ph$ is also torsion-free.
\end{lemma}
\begin{proof}
We give the proof in the AC case. The CS case is identical except that there are $n$ ends instead of just one, and $\varrho \to 0$ on each end instead of $\varrho \to \infty$. Let $V$ be a $7$-dimensional vector space, with two $\G$~structures $\ph$ and $\tilde \ph$. It follows from simple linear algebra that if $\tilde \ph$ and $\ph$ are close with respect to the metric $\gp$ induced by $\ph$, then the decompositions $\Lambda^5(V) = \widetilde \Lambda^5_7 \oplus \widetilde \Lambda^5_{14}$ and $\Lambda^5(V) = \Lambda^5_7 \oplus \Lambda^5_{14}$ with respect to $\tilde \ph$ and $\ph$, respectively, will also be close. In particular there exists a universal constant $\e$ such that if $| \widetilde \ph - \ph| < \e$, using the metric $| \cdot |$ from $\ph$, then an element $\xi \in \Lambda^5(V)$ for which $\widetilde \pi_7(\xi) = 0$ will also have $\pi_7 (\xi)$ small enough so that $|\pi_7(\xi)| \leq |\pi_{14}(\xi)|$.

Unless stated otherwise, all our projections and inner products will be taken with respect to the $\G$~structure $\ph$. To simplify notation, we will sometimes write $\zeta = d(\Theta(\tilde \ph))$. We take the decomposition of~\eqref{threezeroeseq} in $\Omega^5 = \Omega^5_7 \oplus \Omega^5_{14}$:
\begin{align} \label{threezeroeseq2}
\zeta_7 & = \, {[d(\Theta(\tilde \ph))]}_7 \, = \,  {[\theta \wedge \ps]}_7 + {[dX \wedge \ph]}_7 = \theta \wedge \ps -2 \st  \pi_7(dX) , \\ \label{threezeroeseq3} \zeta_{14} & = \, {[d(\Theta(\tilde \ph))]}_{14} \, = \, {[\theta \wedge \ps]}_{14} +  {[dX \wedge \ph]}_{14} = 0 + \st \pi_{14}(dX) ,
\end{align}
where we have used the Hodge stars of equations~\eqref{o27eq} and~\eqref{o214eq}. Because $\tilde \ph$ is closed, we know by Remark~\ref{torsionrmk} that $\zeta = d(\Theta(\tilde \ph))$ lies in the space $\widetilde{\Omega}^5_{14}$, where the tilde denotes the decomposition with respect to $\tilde \ph$. Hence, if $| \tilde \ph - \ph |_{C^0} < \e$, by the above remarks we have that $| \zeta_7 | \leq | \zeta_{14} |$. Since we assume that $\zeta = d(\Theta(\tilde \ph))$ is in $L^2$, we can integrate over $M$ to conclude that
\begin{equation} \label{threezeroeseq4}
|| \zeta_7 || \, \leq || \zeta_{14} ||.
\end{equation}
Now consider the $7$-form $dX \wedge dX \wedge \ph$, which is exact since $\ph$ is closed. By equations~\eqref{o27eq} and~\eqref{o214eq}, we have
\begin{equation} \label{threezeroeseqtemp}
d X \wedge d X \wedge \ph \, = \, d X \wedge \left( -2  \st \pi_7 (dX) + \st \pi_{14}(d X) \right) \, = \, \left( -2 | \pi_7 (dX) |^2 + | \pi_{14}(dX)|^2 \right) \vol.
\end{equation}
The integral over $M$ of the right hand side is finite because $dX$ is assumed to be in $L^2$. To compute the integral over $M$ of the left hand side, let $M_R = \{ x \in M; \varrho(x) \leq R \}$, and observe that $\partial (M_R) = \{R \} \times \Sigma$. Hence, by Stokes's Theorem and the hypothesis that $X = O(\varrho^{\lambda + 1})$, for $R$ sufficiently large we have
\begin{equation*}
\left| \int_{M_R} d X \wedge d X \wedge \ph \right| = \left| \int_{\partial (M_R)} X \wedge d X \wedge \ph \right| \leq C R^{2 \lambda + 1} \left| \int_{\{R\} \times \Sigma} \vol_{\{R\} \times \Sigma} \right| = \tilde C R^{2 \lambda + 7}
\end{equation*}
which goes to zero as $R \to \infty$, since $\lambda < - \frac{7}{2}$. Therefore, integrating both sides of~\eqref{threezeroeseqtemp} over $M$, we obtain
\begin{equation} \label{threezeroeseq5}
2 ||  \pi_7 (dX) ||^2 \, = || \, \pi_{14} (dX)||^2. 
\end{equation}

Similarly,  $d(\Theta(\tilde \ph)) \wedge dX$ is an exact $7$-form, and
\begin{equation*}
d(\Theta(\tilde \ph)) \wedge dX \, = \, \zeta \wedge \st \st dX \, = \, \langle \zeta_7, \st \pi_7 (dX) \rangle \vol + \langle \zeta_{14}, \st \pi_{14} (dX) \rangle \vol.
\end{equation*}
Since both $\zeta = d(\Theta(\tilde \ph))$ and $dX$ are in $L^2$, and since $\zeta = O(\varrho^{\lambda})$ and $X = O(\varrho^{\lambda + 1})$, we can argue exactly as before to integrate both sides over $M$ to conclude that
\begin{equation} \label{threezeroeseq6}
\langle \langle \, \zeta_7, \st \pi_7 (dX) \, \rangle \rangle \, = \, -\langle \langle \, \zeta_{14}, \st \pi_{14} (dX) \, \rangle \rangle \, = \, - || \zeta_{14} ||^2,
\end{equation}
using the fact that $\zeta_{14} = \st\pi_{14}(dX)$ from equation~\eqref{threezeroeseq3}.

Now we use the Cauchy--Schwarz inequality, and equations~\eqref{threezeroeseq4},~\eqref{threezeroeseq5}, and~\eqref{threezeroeseq6} to compute
\begin{align*}
|| \zeta_7 || \, || \pi_7 (dX) || & = \, || \zeta_7 || \, || \st \pi_7(dX) || \\ & \geq \, - \langle \langle \, \zeta_7, \st \pi_7 (dX) \, \rangle \rangle \\ & = \, +\langle \langle \, \zeta_{14}, \st \pi_{14} (dX) \, \rangle \rangle \\ & = \, || \zeta_{14} || \, || \st \pi_{14}(dX) || \\ & = \, || \zeta_{14} || \, || \pi_{14} (dX) || \\ & = \, \sqrt{2} \, || \zeta_{14} || \, || \pi_{7}(dX) || \\ & \geq \sqrt{2} \, || \zeta_7 || \, || \pi_7(dX) ||.
\end{align*}
Therefore we have concluded that
\begin{equation*}
|| \zeta_7 || \, || \pi_7 (dX) || \, \geq \,  \sqrt{2} \, || \zeta_7 || \, || \pi_7(dX) ||.
\end{equation*}
We have two cases. If $\pi_7(dX) = 0$, then by~\eqref{threezeroeseq5} we have $\pi_{14} (dX) = 0$, so $\zeta_{14} = 0$ by~\eqref{threezeroeseq3}, and thus $\zeta_7 = 0$ by~\eqref{threezeroeseq4}. If, on the other hand, we have $\zeta_7 = 0$, then~\eqref{threezeroeseq6} forces $\zeta_{14} = 0$, and then~\eqref{threezeroeseq3} and~\eqref{threezeroeseq5} together give $dX = 0$. In either case we also get $\theta \wedge \ps = 0$ from~\eqref{threezeroeseq2}, which implies that $\theta = 0$, since wedge product with $\ps$ is injective on $1$-forms.
\end{proof}

\begin{rmk} \label{threezeroesrmk}
The reason that Lemma~\ref{threezeroeslemma} is true is because of the representation theory of $\G$. Essentially, equations~\eqref{o27eq} and~\eqref{o214eq} and Stokes's theorem force the very powerful restrictions~\eqref{threezeroeseq5} and~\eqref{threezeroeseq6} on the forms $\zeta = d(\Theta(\tilde \ph))$ and $dX$. The remaining ingredients are the $C^0$ proximity of $\tilde \ph$ and $\ph$, together with the facts that $\tilde \ph$ is closed and $\ph$ is torsion-free, which force the $\Omega^5_7$ component of $\zeta$ to be controlled by the $\Omega^5_{14}$ component.
\end{rmk}

The remaining results in this section concern the modified Dirac operator $\mdirac$ defined in Section~\ref{spinorsec}, and the operator $\mlap = d d^* + \frac{2}{3} d^* d$ which appears often in relation to gauge-fixing. The first result, Lemma~\ref{mdiracHodgelemma}, is used to prove one case of the infinitesimal slice theorem in Section~\ref{infinitesimalslicethmsec}. The second result, Lemma~\ref{modifiedlaplemma}, is used in Section~\ref{ACextensionsec} to extend our AC deformation theory to higher rates and in Section~\ref{smoothCSsec} to establish smoothness of the CS moduli space under certain conditions.

Suppose that $\lambda+1$ is a noncritical rate for $\mdirac$. Thus the operator
\begin{equation*}
\mdirac_{l+1, \lambda+1} : L^2_{l+1, \lambda+1}( \Lambda^0_1 \oplus \Lambda^1_7) \to L^2_{l, \lambda}( \Lambda^3_1 \oplus \Lambda^3_7)
\end{equation*}
is Fredholm, and therefore by Theorem~\ref{fredholmalternativethm} we have
\begin{equation} \label{mdiracHodgeeq1}
L^2_{l, \lambda}( \Lambda^3_1 \oplus \Lambda^3_7 ) \, = \, \mdirac \left( L^2_{l+1, \lambda+1}(\Lambda^0_1 \oplus \Lambda^1_7 ) \right) \oplus (\mathcal V_{\ph}')_{l, \lambda},
\end{equation}
where $(\mathcal V_{\ph}')_{l, \lambda}$ is a \emph{finite-dimensional} subspace of $L^2_{l, \lambda}( \Lambda^3_1   \oplus \Lambda^3_7 )$, such that
\begin{equation} \label{mdiracHodgeeq2}
(\mathcal V_{\ph}')_{l, \lambda} \, \cong \, \ker (\mdiracstar)_{- 7 - \lambda}.
\end{equation}

\begin{lemma} \label{mdiracHodgelemma}
Let $(M, \phm)$ be a $\G$~conifold. Consider the map $\mdirac_{l+1, \lambda+1}$.  
\begin{itemize}
\item[(a)] In the AC case, it is injective if $\lambda < -1$, and it is surjective if $\lambda > -7$. 
\item[(b)]In the CS case, it is injective if $\lambda > -1$, it has a one-dimensional kernel if $\lambda \in  (-7, -1]$, it has a one-dimensional cokernel if $\lambda \in [-7, -1)$, and it is surjective if $\lambda < -7$.
\end{itemize}
\end{lemma}
\begin{proof}
Let $ (f , X)$ lie in $\ker (\mdirac)_{\mu}$. Corollary~\ref{mdirackernelcor} tells us that $\Delta f = 0$ and $\Delta X = 0$.

Suppose first that $M$ is AC. By Lemma~\ref{bochnerlemma}, $f$ is constant if $\mu<1$ and it will thus be zero if $\mu<0$. Now $X$ is a harmonic $1$-form on $M$, so by Lemma~\ref{bochnerlemma} if $\mu\leq 0$ we have $X=0$. Applying this for $\mu = \lambda+1$ says that if $\lambda < -1$ then $\mu < 0$ so $f = 0$ and $X = 0$. Hence, if we set $\mu = -7 - \lambda$, we see $\ker (\mdirac)_{- 7 - \lambda} = 0$ for $\lambda > -7$ and thus by~\eqref{mdiracHodgeeq1} and~\eqref{mdiracHodgeeq2} we conclude that $\mdirac_{l+1, \lambda+1}$ is surjective.

Now suppose that $M$ is CS. In this case, Lemma~\ref{bochnerlemma} says that $f$ is constant and $X =0 $ if $\mu > -5$, and thus $(f,X) = (0,0)$ if $\mu > 0$. (Notice that the pair $(K,0)$ when $K$ is a nonzero constant is indeed in $\ker \mdirac_{\mu}$ for $\mu = 0$.) Moreover, Proposition~\ref{modifieddiracexcludeprop} shows that there are no exceptional rates in $(-6,0)$ for $\mdirac$, so we can in fact say that $f$ is constant and $X = 0$ if $\mu \in (-6,0]$. Setting $\mu = \lambda+1$ we see that $\ker(\mdirac)_{\lambda+1}$ is one-dimensional if $\lambda \in (-7,-1]$ and is zero if $\lambda > -1$. For the cokernel we let $\mu = -7 - \lambda$, so that $\mu \in (-6,0]$ if and only if $\lambda \in [-7, -1)$ and $\mu > 0$ if and only if $\lambda < -7$.
\end{proof}

Recall that, by Proposition~\ref{specialellipticprop}, the operator $\pi_7d^*d : \Omega^2_7 \to \Omega^2_7$ is elliptic and, under the identification $\Omega^1 \cong \Omega^2_7$, corresponds to the operator $\mlap = d d^* + \frac{2}{3} d^* d$ on $1$-forms. Suppose that $\lambda + 1$ is a noncritical rate for $\pi_7 d^*d$. Thus the operator
\begin{equation*}
( \pi_7d^*d )_{l+1, \lambda+1} : (\Omega^2_7)_{l+1,\lambda+1} \to ( \Omega^2_7 )_{l-1,\lambda-1}
\end{equation*}
is Fredholm, and hence
\begin{equation*}
(\Omega^2_7)_{l-1,\lambda-1} \, = \, ( \pi_7d^*d ) \left( (\Omega^2_7)_{l+1,\lambda+1} \right) \oplus U_{\lambda-1},
\end{equation*}
where $U_{\lambda-1}$ is a \emph{finite-dimensional} subspace of $( \Omega^2_7 )_{l-1, \lambda-1}$, such that
\begin{equation*}
U_{\lambda-1} \, \cong \, \ker ( \pi_7d^*d )_{- 6 - \lambda}.
\end{equation*}

\begin{lemma} \label{modifiedlaplemma}
Let $(M, \phm)$ be a $\G$~conifold. Consider the map $(\pi_7d^*d)_{l+1, \lambda+1}$. 
\begin{itemize}
\item[(a)] In the AC case it is injective for $\lambda \leq -1$ and surjective for $\lambda \geq -6$. 
\item[(b)] In the CS case it is injective for $\lambda \geq -6$ and surjective for $\lambda \leq -1$.
\end{itemize}
\end{lemma}
\begin{proof}
In the AC case, suppose  $\omega \in \ker (\pi_7d^*d)_{\mu}$. By Proposition~\ref{specialellipticprop} and Lemma~\ref{bochnerlemma2}, we conclude that $\omega = 0$ if $\mu \leq 0$. Thus $( \pi_7d^*d
)_{l+1, \lambda+1}$ is injective when $\lambda \leq -1$ (taking $\mu = \lambda + 1$) and surjective for all $\lambda \geq -6$ (taking $\mu = -6 - \lambda$). In the CS case, suppose $\omega \in \ker (\pi_7d^*d)_{\mu}$. Again by Proposition~\ref{specialellipticprop} and Lemma~\ref{bochnerlemma2}, we conclude that $\omega = 0$ if $\mu \geq - 5$. Applying this to $\mu=\lambda+1$ and $\mu=-6-\lambda$ gives the result.
\end{proof}

\section{The deformation theory of \texorpdfstring{$\mathbf{\G}$}{G2} conifolds} \label{conifoldmodulisec}

In this section we study the deformation theory of $\G$~conifolds, and state and prove our main theorem. Recall that for us, a $\G$~conifold $(M, \phm)$ is either an AC $\G$~manifold of some rate $\nu < 0$ as in Definition~\ref{ACdefn} or a CS $\G$~manifold of some rates $(\nu_1, \ldots, \nu_n) > (0, \ldots, 0)$ as in Definition~\ref{CSdefn}.

\subsection{The \texorpdfstring{$\mathbf{\G}$}{G2} conifold moduli space} \label{conifoldmodulidefnsec}

In this section we define the $\G$~conifold moduli space, and then give some informal arguments to motivate how we will proceed to prove our main theorem.

Let $\mathcal T_{\nu}$ denote the set of all torsion-free conifold $\G$~structures on $M$ which converge at the same rate $\nu$ on the ends to the same $\G$~cones as the original conifold $\G$~structure $\ph$. Explicitly,
\begin{equation*}
\mathcal T_{\nu} \, = \, \{ \tilde \ph \in \Omega^3_+(M) ; \, \tilde \ph - \ph \in C^{\infty}_{\nu}(\Lambda^3 T^*M), \, d\tilde \ph = 0, \, d\Theta(\tilde \ph) = 0 \}.
\end{equation*}

In order to define the \emph{moduli space} of torsion-free conifold $\G$~structures on $M$, we need to take the quotient of $\mathcal T_{\nu}$ by an appropriate equivalence relation. The torsion-free condition is diffeomorphism invariant, but arbitrary diffeomorphisms do not preserve the convergence condition on the ends.  As mentioned in the introduction, we choose to quotient out by those diffeomorphisms which \emph{fix} the $\G$~cones on the ends.

\begin{rmk} \label{dividemorermk}
One could then in principle further divide out by extra diffeomorphisms later, such as those which are asymptotic to automorphisms of the $\G$~cones on the ends. We eventually do something like this in the CS case in Section~\ref{smoothCSsec}. See Definition~\ref{reduced.moduli.defn} and the rest of that section.
\end{rmk}

Thus we are interested in diffeomorphisms isotopic to the identity which are \emph{generated by vector fields that decay to zero on the ends}. For such diffeomorphisms to preserve the rate of convergence at the ends to the asymptotic cones, their infinitesimal generators (vector fields) must be of rate $\nu+1$ on the ends. Specifically, we define $\mathcal D_{\nu + 1}$ to be the group generated by the set
\begin{equation*}
\{ \exp(X); \, X \in \Gamma(TM), \, X \in C^{\infty}_{\nu + 1}(TM) \}.
\end{equation*}
This is a connected component of the identity in the space of all diffeomorphisms of $M$, and hence a subgroup of $\mathrm{Diff}^0(M)$, the diffeomorphisms isotopic to the identity.

It is clear that $\mathcal D_{\nu + 1}$ acts on $\mathcal T_{\nu}$ by pullback. We now define the \emph{$\G$~conifold moduli space} $\mathcal M_{\nu}$ of rate $\nu$ on $M$ to be the quotient space $\mathcal M_{\nu} = \mathcal T_{\nu} / \mathcal D_{\nu + 1}$. This defines $\mathcal M_{\nu}$ as a topological space. We want to describe the structure of $\mathcal M_{\nu}$ more precisely. In the AC case, we will see that for generic rates that lie in a certain range, the space $\mathcal M_{\nu}$ is actually a finite-dimensional smooth manifold. For other rates in the AC case, and in general for the CS case, the deformation theory will be obstructed and we will describe the obstruction spaces explicitly.

The orbit of $\ph$ under $\mathcal{D}_{\nu+1}$ is an (infinite-dimensional) smooth manifold contained in $\mathcal{T}_\nu$.
Consider a smooth curve $F_t = \exp(tX)$ in $\mathcal D_{\nu + 1}$, where $X \in C^{\infty}_{\nu + 1}(TM)$. This path passes through the identity diffeomorphism $F_0 = \mathrm{Id}_M$ at $t=0$. Therefore, the tangent space $T_{\ph} (\mathcal D_{\nu + 1} \cdot \ph)$ at $\ph$ to the orbit $\mathcal D_{\nu + 1} \cdot \ph$ is spanned by elements of the form $\rest{\ddt}{t=0} (F_t^* \ph) = \mathcal L_X \ph = d(X \hk \ph)$. Thus we have
\begin{equation} \label{orbittangentspaceeq}
T_{\ph} (\mathcal D_{\nu + 1} \cdot \ph) \, = \, d\left( C^{\infty}_{\nu + 1}(\Lambda^2_7 (T^*M)) \right).
\end{equation}

To study $\mathcal{M}_{\nu}$ it is useful to understand what its tangent space at the orbit of $\varphi$ would be, were it a smooth manifold near there.  To this end, suppose $\ph_t$ is a smooth path in $\mathcal T_{\nu}$ passing through $\ph$ at $t=0$. Thus $\ph_t$ is a torsion-free $\G$~structure for all $t$, and therefore by Lemma~\ref{quadlemma} we have that the $3$-form $\eta = \rest{\ddt}{t=0}{\ph_t}$ satisfies $d\eta = 0$ and $\frac{4}{3} d^* \pi_1 (\eta) + d^* \pi_7 (\eta) - d^* \pi_{27}(\eta) = 0$, where the projections are taken with respect to the $\G$~structure $\ph = \ph_0$. Hence we have shown that if $\mathcal{T}_{\nu}$ were  a smooth manifold its tangent space at $\ph$ would satisfy
\begin{equation} \label{premodulitangentspaceeq}
T_{\ph} \mathcal T_{\nu} \,  \subseteq 
\, \{ \eta \in C^{\infty}_{\nu}(\Lambda^3T^*M); \, d \eta = 0, \, d(\Lp(\eta)) = 0 \}, 
\end{equation}
where
\begin{equation*}
\Lp(\eta) \, = \, \frac{4}{3} \st \! \pi_1 (\eta) + \st \pi_7 (\eta) - \st \pi_{27}(\eta)
\end{equation*}
is the linearization of $\Theta$ at $\ph$ defined in equation~\eqref{Lpeq}. For the purpose of motivation, let us assume that the subspace inclusion in~\eqref{premodulitangentspaceeq} is actually an equality. Then if $\mathcal M_{\nu} = \mathcal T_{\nu} / \mathcal D_{\nu + 1}$ were indeed a smooth manifold, we would have that
\begin{equation} \label{sliceeq}
T_{[\ph]} \mathcal M_{\nu} \, \oplus \, T_{\ph} (\mathcal D_{\nu + 1} \cdot \ph) \, = \, T_{\ph} \mathcal T_{\nu}.
\end{equation}
Thus, one of our goals will be to use~\eqref{orbittangentspaceeq} and~\eqref{premodulitangentspaceeq} to find a direct complement of $T_{\ph} (\mathcal D_{\nu + 1} \cdot \ph)$ in $ T_{\ph} \mathcal T_{\nu}$. This will tell us what the ``tangent space'' at $[\ph]$ to $\mathcal M_{\nu}$ would have to be. Then we will use the Banach space implicit function theorem to describe the structure of $\mathcal M_{\nu}$. We will prove our main theorem without requiring any assumption about equality in~\eqref{premodulitangentspaceeq}.

The main theorem that we prove in the next section about the $\G$~conifold moduli space is the following.

\begin{thm} \label{mainthm}
Let $(M, \ph)$ be a $\G$~conifold, asymptotic to particular $\G$~cones on the ends, at some rate $\nu$. Let $\mathcal M_{\nu}$ be the moduli space of all torsion-free $\G$~structures on $M$, asymptotic to the same cones on the ends, at the same rate $\nu$, modulo the action of diffeomorphisms that fix the $\G$~cones on the ends, and fix the rate of convergence $\nu$ to those cones. Then for generic rates $\nu$ (more precisely, for all rates except for a finite set of ``critical rates'' for the operator $d + \dsm$ in the sense of Lockhart--McOwen theory), we have
\begin{itemize}
\item In the AC case, if $\nu \in (-4, 0)$, the space $\mathcal M_{\nu}$ is a \emph{smooth manifold} whose dimension consists of topological and analytic contributions, given precisely in Corollary~\ref{vdimfinalcor}.
\item In the AC case if $\nu < -4$, or in the CS case for any $\nu > 0$, the space $\mathcal M_{\nu}$ is in general only a \emph{topological space}, and the deformation theory may be obstructed. The virtual dimension of $\mathcal M_{\nu}$ again consists of topological and analytic contributions, given precisely in Corollary~\ref{vdimfinalcor}.
\end{itemize}
\end{thm}

\subsection{Proof of the main theorem} \label{maintheoremproofsec}

In this section we prove Theorem~\ref{mainthm} on the deformation theory of $\G$~conifolds. The proof of Theorem~\ref{mainthm} is broken up into the following four steps.
\begin{enumerate}[{\it Step 1:}] \setlength\itemsep{-0.3mm}
\item We prove a slice theorem, showing that the space of torsion-free gauge-fixed $\G$~structures with the correct asymptotics on the ends is homeomorphic to the $\G$~conifold moduli space.
\item We demonstrate that the moduli space ${\mathcal M}_{\nu}$ is locally isomorphic to the zero set of a smooth nonlinear map.
\item We use the Banach space implicit function theorem to describe the structure of this zero set, and explain when it is a smooth manifold.
\item We compute the (expected) dimension of the $\G$~conifold moduli ${\mathcal M}_{\nu}$ space in terms of topological and analytic data.
\end{enumerate}

\subsubsection{Step 1: Gauge-fixing and the slice theorem} \label{infinitesimalslicethmsec}

In order to break the diffeomorphism invariance, we need to prove a ``slice theorem'' that establishes a local homeomorphism between: (i) the space of $\G$~structures satisfying a particular condition modulo diffeomorphisms which preserve this condition, and (ii) a space of solutions to a system of differential equations. Ideally, we would like to prove this slice theorem directly for torsion-free $\G$~structures which have prescribed cone-like behaviour on the ends. However, the fact that the torsion-free condition is \emph{nonlinear} makes it difficult to do this directly. Instead, we first prove a slice theorem for the space of \emph{closed} $\G$~structures with prescribed cone-like behaviour on the ends, which is a linear condition, and then in Section~\ref{onetoonesec} we impose the torsion-free condition to describe a smaller subset of this space.

Our approach to the slice theorem for closed $\G$~structures is very similar to that of Nordstr\"om~\cite{Nord}, who considers the asymptotically cylindrical case. A more detailed treatment is in~\cite{NordTh}, to which we will occasionally refer. To begin, we need to find a direct complement of  the tangent space to the space of diffeomorphisms that preserve the cone-like behaviour of the appropriate rate on the ends, within the space of closed $3$-forms with the same decay at the ends. In order to later apply the Banach space implicit function theorem to determine the structure of the moduli space, we will need to consider (weighted) Sobolev spaces of forms, and thus we actually need to establish a ``slice theorem'' for forms in such weighted Sobolev spaces.

In the space $\Omega^3_{l, \nu}$, the analogue of the space of infinitesimal diffeomorphisms defined in equation~\eqref{orbittangentspaceeq} is the space of $3$-forms that are the exterior derivative of a $2$-form of type $\Lambda^2_7$ in the appropriate Sobolev space. Explicitly, we define $\mathcal D_{l + 1, \nu + 1}$ to be the group generated by the set
\begin{equation*}
\{ \exp(X); \, X \in \Gamma(TM), \, X \in L^2_{l+1, \nu+1}(T M) \}.
\end{equation*}
and thus the tangent space to the orbit $\mathcal D_{l + 1, \nu + 1} \cdot \ph$ at $\ph$ is given by
\begin{equation*}
T_{\ph} (\mathcal D_{l + 1, \nu + 1} \cdot \ph) \, = \, d \left( {L^2_{l+1, \nu+1}(\Lambda^2_7 (T^* M)) }\right) \, \subseteq \, \Omega^3_{l, \nu}.
\end{equation*}
Similarly, in $\Omega^3_{l, \nu}$ the tangent space at $\ph$ to the closed $\G$~structures $\mathcal C^+_{l, \nu}$ that are asymptotic to $\ph$ with rate $\nu$ is given by
\begin{equation*}
T_{\ph} \mathcal C^+_{l, \nu} \, = \, \{ \eta \in \Omega^3_{l, \nu} \, ; \, d \eta = 0 \} \, = \, \mathcal C_{l, \nu}.
\end{equation*}
It is clear that $T_{\ph} (\mathcal D_{l + 1, \nu + 1} \cdot \ph)$ is a subspace of $T_{\ph} \mathcal C^+_{l, \nu} = \mathcal C_{l, \nu}$.

\begin{defn} \label{infinitesimalmodulispacedefn}
Given a rate $\nu$, define $(\mathcal G_{\ph})_{l, \nu}$ to be the following subspace of $\Omega^3_{l, \nu}$:
\begin{equation*}
(\mathcal G_{\ph})_{l, \nu} \, = \, \{ \eta \in \Omega^3_{l, \nu} \, ; \, d \eta = 0, \, \pi_7 (d^* \eta) = 0 \}.
\end{equation*}
Thus $(\mathcal G_{\ph})_{l, \nu}$ is a proper subspace of $\mathcal C_{l, \nu}$, corresponding to the ``gauge-fixed'' (with respect to $\ph$) infinitesimal deformations of closed $\G$~structures, given by Definition~\ref{gaugefixingdefn}.
\end{defn}

We can equivalently describe the space $(\mathcal G_{\ph})_{l,\nu}$ for our rates $\nu$ of interest as follows.

\begin{lemma}\label{Gph.lemma}
If $\nu<0$ in the AC case or $\nu>0$ in the CS case, then
 \begin{equation*}
(\mathcal G_{\ph})_{l, \nu} \, = \, \{ \eta \in (\Omega^3_{27})_{l, \nu} \, ; \, d \eta = 0 \}.
\end{equation*}
\end{lemma}

\begin{proof}
Suppose that $\eta \in (\mathcal G_\ph)_{l,\nu}$ and write $\pi_1 \eta = f \ph$ and $\pi_7 \eta = X \hk \ph$. Since $f$ and $X$ have decay $O(\varrho^{\nu})$, by Lemma~\ref{gaugefixinglemmaA} we deduce that $X = 0$ and $f$ is constant. Moreover, $f$ must tend to zero on the ends of $M$, so $f = 0$ also. Hence $\eta \in (\Omega^3_{27})_{l,\nu}$ with $d \eta = 0$. Conversely, by Corollary~\ref{special27cor3}, any closed form $\eta \in (\Omega^3_{27})_{l,\nu}$ will satisfy $\pi_7 (d^* \eta) = 0$, so $\eta \in (\mathcal G_{\ph})_{l,\nu}$.  
\end{proof}

Before we can state the first main result of this subsection, which is an infinitesimal version of our slice theorem, we need to construct a finite-dimensional space $(\mathcal{V}_{\ph})_{l,\nu}$ of $\Omega^3_{l, \nu}$ that will play a crucial role in the rest of the paper. It is defined in terms of the modified Dirac operator $\mdirac$ of~\eqref{modifieddiracdefneq}.

\begin{prop} \label{Vspaceprop}
Consider the map $\pi_1 d : (\Omega^3_{1+7})_{l, \nu} \to \Omega^4_{l-1, \nu - 1}$. This map is continuous, hence its kernel $\mathcal Q_{l, \nu} = \{ \beta \in (\Omega^3_{1+7})_{l, \nu} \, ; \pi_1 d \beta = 0 \}$ is a \emph{closed} subspace of $(\Omega^3_{1+7})_{l, \nu}$. There exists a \emph{finite-dimensional subspace} $(\mathcal V_{\ph})_{l, \nu}$ of $(\Omega^3_{1+7})_{l, \nu}$ such that
\begin{equation} \label{Vspacedefneq}
\mathcal Q_{l, \nu} \, = \, \bigl( \im \mdirac_{l+1, \nu+1} \cap \mathcal Q_{l, \nu} \bigr) \oplus (\mathcal V_{\ph})_{l, \nu}.
\end{equation}
Moreover, if $\nu > -7$ (AC) or $\nu < -7$ (CS), then $(\mathcal V_{\ph})_{l, \nu} = \{ 0 \}$.
\end{prop}
\begin{proof}
From equations~\eqref{mdiracHodgeeq1} and~\eqref{mdiracHodgeeq2}, we know $\im \mdirac_{l+1, \nu+1}$ has finite codimension in $(\Omega^3_{1+7})_{l, \nu}$. 
Thus its intersection with the closed subspace $\mathcal Q_{l, \nu}$ will have finite codimension in $\mathcal Q_{l, \nu}$, establishing~\eqref{Vspacedefneq}. If $\nu > -7$ (AC) or $\nu < -7$ (CS), then from Lemma~\ref{mdiracHodgelemma}, we have that $\im \mdirac_{l+1, \nu + 1} = \Omega^3_{l, \nu}$, and thus $\mathcal Q_{l, \nu} = \im \mdirac_{l+1, \nu+1} \cap \mathcal Q_{l, \nu}$, so $(\mathcal V_{\ph})_{l, \nu} = \{ 0 \}$ in these cases.
\end{proof}

\begin{thm} \label{infinitesimalslicethm}
Suppose that $\nu < 0$ (AC) or $\nu > 0$ (CS).
\begin{enumerate}[{$[$1$\, ]$}]
\item In the $L^2$ setting: AC when $\nu < -\frac{7}{2}$, or CS for any $\nu > 0$. There exists a finite-dimensional subspace $(\mathcal E_{\ph})_{l, \nu}$ of $\mathcal C_{l, \nu}$ such that
\begin{equation} \label{CSinfinitesimalslicethmeq}
\mathcal C_{l, \nu}  \, = \, T_{\ph} (\mathcal D_{l + 1, \nu + 1} \cdot \ph) \oplus (\mathcal G_{\ph})_{l, \nu} \oplus (\mathcal E_{\ph})_{l, \nu}.\\
\end{equation}
Moreover, there is an injective linear map $\pi_{1+7} : (\mathcal E_{\ph})_{l, \nu} \to (\mathcal V_{\ph})_{l, \nu}$, and hence the space $(\mathcal E_{\ph})_{l, \nu}$ is trivial whenever $(\mathcal V_{\ph})_{l, \nu}$ is trivial.
\item In the non-$L^2$ setting: AC when $\nu \in (-4, -0)$, there are two subcases.
\begin{itemize}
\item When $\nu \in (-4, -1]$, we have
\begin{equation} \label{ACinfinitesimalslicethmeq}
\mathcal C_{l, \nu}  \, = \, T_{\ph} (\mathcal D_{l + 1, \nu + 1} \cdot \ph) \oplus (\mathcal G_{\ph})_{l, \nu}.
\end{equation}
\item When $\nu \in (-1, 0)$ there is a closed subspace $(\mathcal G'_{\ph})_{l, \nu}$ of $(\mathcal G_{\ph})_{l, \nu}$, of finite codimension, such that
\begin{equation} \label{ACinfinitesimalslicethmeq2}
\mathcal C_{l, \nu}  \, = \, T_{\ph} (\mathcal D_{l + 1, \nu + 1} \cdot \ph) \oplus (\mathcal G'_{\ph})_{l, \nu}.
\end{equation}
\end{itemize}
\end{enumerate}
\end{thm}
\begin{proof}
Case [1]: Note that $\mdirac(h,Y) = 0$ implies $\Delta h = 0$ and $\Delta Y = 0$ so $\mdirac$ is in fact injective in this setting by Lemma~\ref{bochnerlemma}. Let $\eta \in \mathcal{C}_{l,\nu}$. Since $d \eta = 0$, we have $\pi_1 d (\pi_{1+7} \eta) = - \pi_1 d (\pi_{27} \eta) = 0$ by Proposition~\ref{special3formprop}. Therefore, $\pi_{1+7} \eta$ lies in the space $\mathcal Q_{l, \nu}$ of Proposition~\ref{Vspaceprop}, so by Proposition~\ref{Vspaceprop} and the injectivity of $\mdirac$ for this range of rates, there exists unique $(h,Y) \in (\Omega^0_1 \oplus \Omega^1_7)_{l+1, \nu + 1}$ and $\beta \in (\mathcal V_{\ph})_{l, \nu}$ such that
\begin{equation*}
\eta - \mdirac(h,Y) - \beta \, \in \, (\Omega^3_{27})_{l, \nu}
\end{equation*}
where in particular $\mdirac(h,Y) \in \mathcal Q_{l, \nu}$, so $\pi_1 d \mdirac(h,Y) = 0$. Then equation~\eqref{mdiracusefuleq} says that $\Delta h = 0$. Moreover, $h \to 0$ on the ends of $M$ so the maximum principle implies that $h = 0$.  Thus, because equation~\eqref{modifieddiracdefneq} says $\mdirac(h,Y) = \frac{1}{2} \st (dh \wedge \ph) + \pi_{1+7}( d(Y \hk \ph) )$ and we have that $h = 0$, we conclude that $\pi_{1+7} \bigl( \eta - d(Y \hk \ph) \bigr) \in (\mathcal{V}_{\ph})_{l,\nu}$, so that
\begin{equation*}
\mathcal{C}_{l, \nu} \, = \, d (\Omega^2_7)_{l+1, \nu+1} \oplus \{ \zeta \in \mathcal{C}_{l,\nu} \, ; \ \pi_{1+7} \zeta \in (\mathcal{V}_{\ph})_{l,\nu} \}.
\end{equation*}
By Lemma~\ref{Gph.lemma}, the space $(\mathcal{G}_{\ph})_{l, \nu}$ is precisely the subspace of $\{ \zeta \in \mathcal{C}_{l,\nu} \, ; \ \pi_{1+7} \zeta \in (\mathcal{V}_{\ph})_{l,\nu} \}$ consisting of elements which are of pure type $27$, that is those elements in the kernel of $\pi_{1 + 7}$. Hence, there exists a finite-dimensional space $(\mathcal{E}_{\ph})_{l, \nu}$ such that
\begin{equation} \label{Stempeq}
\{ \zeta \in \mathcal{C}_{l,\nu} \, ; \ \pi_{1+7} \zeta \in (\mathcal{V}_{\ph})_{l,\nu} \} \, = \, (\mathcal{G}_{\ph})_{l, \nu} \oplus (\mathcal{E}_{\ph})_{l, \nu}.
\end{equation}
By construction, the map $\pi_{1+7} : (\mathcal{E}_{\ph})_{l,\nu} \to (\mathcal{V}_{\ph})_{l,\nu}$ is injective. This gives the required decomposition.

Case [2]: Let $\eta \in \mathcal C_{l, \nu}$. We will first show that $\eta$ can be written as the sum of an element of $T_{\ph} (\mathcal D_{l + 1, \nu + 1} \cdot \ph)$ and an element of $(\mathcal G_{\ph})_{l, \nu}$, and then show, if $\nu \leq -1$, that the intersection of these two subspaces is trivial.  Since $\eta$ is closed, Theorem~\ref{kformsHodgethm} implies that we can write $\eta = \kappa + d \alpha$ for some $\kappa \in \mathcal H^3_{\nu}$ and $\alpha\in\Omega^2_{l+1,\nu+1}$. By Lemma~\ref{mdiracHodgelemma}, the modified Dirac operator $\mdirac$ of equation~\eqref{modifieddiracdefneq} is surjective, so there exists a pair $(2 h, Y) \in ( \Omega^0_1 \oplus \Omega^1)_{l+1,\nu+1}$ such that
\begin{equation*}
\pi_{1 + 7} (d \alpha) = \st (dh \wedge \ph) + \pi_{1 + 7} (d (Y \hk \ph)).
\end{equation*}
Since $\st (dh \wedge \ph)$ is pointwise of type $\Lambda^3_7$, the above equation says that
\begin{equation} \label{infinitesimalslicethmeq2}
d \alpha - d(Y \hk \ph) \, = \, \st (dh \wedge \ph) + \zeta_{27}
\end{equation}
for some $\zeta_{27} \in (\Omega^3_{27})_{l,\nu}$. Since $\pi_1 (d \eta_{27}) = 0$, we have $\pi_1 d \st (dh \wedge \ph) = 0$ and hence by~\eqref{temp27eq1}, we conclude that $h$ is harmonic. Because $h = O(\varrho^{\nu+1})$ for $\nu < 0$, Lemma~\ref{bochnerlemma} implies that $dh = 0$. Thus, $d \zeta_{27} = 0$ so $\zeta_{27} \in (\mathcal G_\ph)_{l,\nu}$ by Lemma~\ref{Gph.lemma}.  Since $\kappa$ is closed and coclosed, $\kappa \in (\mathcal G_\ph)_{l,\nu}$ as well, and hence $\eta = \kappa + d \alpha = \kappa + \zeta_{27} + d (Y \hk \ph)$ lies in $(\mathcal G_\ph)_{l,\nu} + T_{\ph} (\mathcal D_{l + 1, \nu + 1} \cdot \ph)$, as required.

To complete the proof of case [2] we need to consider the intersection $T_{\ph} (\mathcal D_{l + 1, \nu + 1} \cdot \ph) \cap (\mathcal G_{\ph})_{l, \nu}$. Let $d(X \hk \ph)$ lie in this intersection. Let $\mu = X \hk \ph$. We have $\pi_7 (d^* d \mu) = 0$. By Proposition~\ref{specialellipticprop}, we have $d d^* X + \frac{2}{3} d^* d X = 0$. But $\mu \in L^2_{l+1, \nu + 1} (\Lambda^2_7 (T^* M))$, so $X \in L^2_{l+1, \nu + 1} (TM)$. Now using Proposition~\ref{excludeextensionprop2} on the excluded range of $1$-forms $X$ satisfying $d d^* X + \frac{2}{3} d^* d X = 0$ and Theorem~\ref{kernelthm} on the invariance of the kernel we conclude that if $\nu + 1 \leq 0$, then in fact we can say that $X$ is actually $O(r^{\nu' +1})$ where $\nu' +1 = -4 + \e < -\frac{7}{2}$. Thus we can use Lemma~\ref{bochnerlemma2} to conclude that $X = 0$ and thus $d (X \hk \ph) = 0$ when $\nu \leq -1$. If $\nu \in (-1, 0)$, then all we have shown is that $X$ is in the kernel of $\pi_7 d^* d : \Omega^2_7 \to \Omega^2_7$. Since this operator is elliptic, there is only a finite-dimensional space $\mathcal J_{\nu}$ of such $1$-forms. Choosing a topological complement $(\mathcal G'_{\ph})_{l, \nu}$ of the finite-dimensional space $\{ d(X \hk \ph); \, X \in \mathcal J_{\nu} \}$ in $(\mathcal G_{\ph})_{l, \nu}$ completes the proof. 
\end{proof}

\begin{rmk} \label{infinitesimalslicethmrmkCS}
In the $L^2$ case we cannot always conclude that the space $(\mathcal E_{\ph})_{l, \nu}$ vanishes. However, in the CS setting, the space $(\mathcal E_{\ph})_{l,\nu}$ can be identified with a subspace of the cokernel of $\mdirac$, so this fact could be used to provide an upper bound for the dimension of $(\mathcal E_{\ph})_{l, \nu}$ in terms of certain eigenvalue equations on the links $\Sigma_i$.
\end{rmk}

\begin{rmk} \label{infinitesimalslicethmrmkAC}
In the AC case when $\nu > - 4$, Theorem~\ref{infinitesimalslicethm} says that the space $(\mathcal G_{\ph})_{l, \nu}$ is only a good infinitesimal slice when $\nu \leq -1$. When $\nu \in (-1, 0)$, not all gauge-fixed infinitesimal deformations are actually transverse to the orbit of the diffeomorphism action. Hence we need to consider a \emph{smaller slice} whose tangent space is $(\mathcal G'_{\ph})_{l, \nu}$. Note that $\nu \leq -1$ is satisfied by all currently known examples. We also notice that the complement of $(\mathcal G'_{\ph})_{l,\nu}$ in $(\mathcal G_{\ph})_{l,\nu}$ is $\ker(\pi_7d^*d)_{\nu+1}$, and since $\pi_7d^*d$ is surjective at these rates by Lemma~\ref{modifiedlaplemma}, we can use Proposition~\ref{excludeextensionprop2} to describe this space via scalar eigenfunctions of $\laps$. 
\end{rmk}

\begin{cor} \label{infinitesimalslicecor} 
Let $M$ be a $\G$~conifold with rate $\nu$. Recall, from Proposition~\ref{Vspaceprop}, that we defined a finite-dimensional subspace $(\mathcal{V}_{\ph})_{l,\nu} \subseteq (\Omega^3_1\oplus\Omega^3_7)_{l,\nu}$, all of whose elements $\beta$ satisfy $\pi_1 d \beta = 0$. Using this space $(\mathcal{V}_{\ph})_{l,\nu}$, define
\begin{equation*}
(\mathcal{S}_{\ph})_{l,\nu} \, = \, \{ \eta \in \mathcal{C}_{l,\nu} \, : \, \pi_{1+7} \eta \in (\mathcal{V}_{\ph})_{l,\nu} \}.
\end{equation*}
Then we have
\begin{equation*}
\mathcal{C}_{l,\nu} \, = \, d \big( (\Omega^2_7)_{l+1, \nu+1} \big) + (\mathcal{S}_{\ph})_{l,\nu},
\end{equation*}
and the intersection of the spaces on the right-hand side is finite-dimensional.  

Moreover, if $M$ is AC with rate $\nu \leq -1$ or if $M$ is CS then 
\begin{equation*}
d \big( (\Omega^2_7)_{l+1, \nu+1} \big) \, \cap \, (\mathcal{S}_{\ph})_{l,\nu} \, = \, \{0\}.
\end{equation*}
Finally, if $M$ is AC with rate $\nu > -7$ then $(\mathcal{V}_{\ph})_{l,\nu} = \{0\}$.
\end{cor}
\begin{proof}
All of these statements follow immediately from properties of $(\mathcal{V}_{\ph})_{l,\nu}$ in Proposition~\ref{Vspaceprop} and from the proof of Theorem~\ref{infinitesimalslicethm}.
\end{proof}

Now let $(\mathcal S'_{\ph})_{l, \nu}$ be a direct complement to $d \bigl( (\Omega^2_7)_{l+1, \nu+1} \bigr) \cap (\mathcal S_{\ph})_{l,\nu}$ in the space $(\mathcal S_{\ph})_{l,\nu}$ defined in Corollary~\ref{infinitesimalslicecor}. Notice that in the AC case when $\nu \leq -1$ and in the CS case we have that $(\mathcal S'_{\ph})_{l, \nu}=(\mathcal S_{\ph})_{l,\nu}$. Moreover, we also have that $(\mathcal S_{\ph})_{l,\nu} = (\mathcal G_{\ph})_{l, \nu}$ in the AC case when $\nu\in(-4,-1]$. If we now let $B_{\e}(0)$ be the $\e$-ball in $C^0 (\Omega^3)$ and define a map $\exp_{\ph}: (\mathcal S'_{\ph})_{l, \nu} \cap B_{\e}(0) \to \mathcal C^+_{l, \nu}$ by affine translation, namely
\begin{equation*}
\exp_{\ph} (\eta) \, = \, \ph + \eta,
\end{equation*}
then it is clear that for $\e$ small the image $\mathcal S^{\e}_{l, \nu}$ of this map is an (infinite-dimensional) smooth submanifold of $\mathcal C^+_{l, \nu}$ whose tangent space at $\ph$ is exactly $T_{\ph} \mathcal S^{\e}_{l, \nu} = (\mathcal S'_{\ph})_{l, \nu}$. We would like to conclude that, near the point $\ph$, the space $\mathcal S^{\e}_{l, \nu}$ contains exactly one representative from each orbit of the action of the diffeomorphisms in $\mathcal D_{l+1, \nu+1}$. More precisely, we want to establish that, near $\ph$,  the space $\mathcal S^{\e}_{l, \nu}$ is homeomorphic to $\mathcal C^+_{l, \nu} / \mathcal D_{l+1, \nu+1}$. In fact, what we can actually prove is that near $\ph$, the space of \emph{torsion-free} elements in $\mathcal S^{\e}_{l, \nu}$ is homeomorphic to our moduli space $\mathcal M_{\nu}$. This result will be sufficient for our purposes. The details of this argument are discussed in Nordstr\"om~\cite[Section 3.1]{NordTh}. Our situation admits several nice features that allow us to use the simplifications that Nordstr\"om explains in~\cite[Section 3.1.3]{NordTh}, which we now briefly discuss. 

We define $\mathcal R^{\e}_{l, \nu} = \mathcal S^{\e}_{l, \nu} \cap \mathcal T$ to be the torsion-free $\G$~structures in $\mathcal S_{l, \nu}$. The content of Corollary~\ref{onetoonesmoothcor} in the next section is that $\mathcal{R}^{\e}_{l, \nu}$ consists of \emph{smooth} elements, so we can drop the subscript $l$ on $\mathcal R^{\e}_{l, \nu} = \mathcal R^{\e}_{\nu}$ and we are able to use \cite[Theorem 3.1.4]{NordTh}. Thus, we can conclude as in~\cite[page 51]{NordTh}, that the space $\mathcal R^{\e}_{\nu}$ is homeomorphic to an open neighbourhood of $[\ph]$ in $\mathcal M_{\nu}$. So the problem of understanding the local structure of $\mathcal M_{\nu}$ reduces to understanding $\mathcal R^{\e}_{\nu}$. We summarize the preceding discussion in the following theorem.

\begin{thm} \label{slicethm}
Let $\mathcal R^{\e}_{l, \nu} = \{ \eta  \in (\mathcal S'_{\ph})_{l, \nu}; \, | \eta  |_{C^0} < \e, d (\Theta(\ph + \eta )) = 0 \}$ be the space parametrizing the torsion-free gauge-fixed $\G$~structures close to $\ph$. Then for $\e$ sufficiently small, any element $\eta \in \mathcal R^{\e}_{l, \nu}$ has $\eta$ smooth and hence $\mathcal R^{\e}_{l, \nu}$ is independent of $l$. Moreover $\mathcal R^{\e}_{l, \nu}$ is homeomorphic to an open neighbourhood of the point $\mathcal D_{\nu + 1} \cdot \ph$ in $\mathcal M_{\nu}$.
\end{thm}

\subsubsection{Step 2: Local one-to-one correspondence with solutions of an elliptic PDE} \label{onetoonesec}

In this section we establish a (local) one-to-one correspondence between (i) gauge-fixed torsion-free $\G$~structures with the same conical asymptotics on the ends as $\ph$ that are sufficiently $C^0$-close to $\ph$; and (ii) solutions of a nonlinear elliptic partial differential equation on $M$.

Let $(M, \ph)$ be a $\G$~conifold of rate $\nu$. Let $\e > 0$ be the constant from Lemma~\ref{threezeroeslemma}. Consider the set of $\G$~structures $\tilde \ph$ that are asymptotic, at the same rate $\nu$, to the same $\G$~cones on the ends, which are $\e$-close to $\ph$ in the $C^0$ norm, such that the difference $\tilde \ph - \ph$ lies in $\Omega^3_{l, \nu}$.

The next result should be compared with~\cite[Theorem 10.3.6]{J4}. It is both a generalization to the conifold setting, and a reformulation in terms of the first order operator $d + d^*$ rather than the Laplacian $\Delta$.  

\begin{thm} \label{onetoonethm}
Consider the subset of $\G$~structures $\tilde \ph$ with $\eta = \tilde \ph - \ph \in \Omega^3_{l, \nu}$, satisfying $| \eta |_{C^0} < \e$,
\begin{equation} \label{onetooneEQ1}
d \eta = 0, \, \quad \, d \Theta(\ph + \eta) = 0, \, \text{ and } \, \pi_{1+7} (\eta) = f \ph + \st (X \wedge \ph) \in  (\mathcal{V}_{\ph})_{l, \nu}.
\end{equation}
In particular, this includes all $\eta \in {\mathcal R}^{\e}_{l, \nu}$ as defined in Theorem~\ref{slicethm}. 

Define the finite-dimensional space 
\begin{equation} \label{mathcalWdefneq}
(\mathcal W_{\ph})_{l-1,\nu-1} \, = \, \left\{ \frac{7}{3} d^* \pi_1 (\beta) + 2 d^* \pi_7 (\beta) \in \Omega^2_{l-1, \nu-1} \, ; \, \, \beta \in (\mathcal{V}_{\ph})_{l,\nu} \right\}
\end{equation}
and consider the following nonlinear  condition on $\eta$:
\begin{equation} \label{onetooneEQ2}
(d + d^*) \eta - d^* \st (\Qp (\eta) ) \, = \, \frac{7}{3} d^* (f \ph) + 2 d^* \st (X \wedge \ph)
 \in (\mathcal W_{\ph})_{l-1, \nu-1}. 
\end{equation}

Then, for any $\G$~conifold, the conditions in~\eqref{onetooneEQ1} imply~\eqref{onetooneEQ2}. Moreover, in the CS case, and in the AC case \emph{if $\nu < -\frac{5}{2}$}, there is a one-to-one correspondence between $3$-forms $\eta \in \Omega^3_{l, \nu}$ with $| \eta |_{C^0} < \e$ that satisfy~\eqref{onetooneEQ1}, and solutions in $\Omega^3_{l,\nu}$ of~\eqref{onetooneEQ2}.
\end{thm}

\begin{rmk} \label{onetoonethmrmk}
Consider the case when $(\mathcal{V}_{\ph})_{l,\nu} = \{0\}$. By Proposition~\ref{Vspaceprop}, a sufficient condition for this is that $\nu > -7$ (AC) or $\nu < -7$ (CS). In those cases where $(\mathcal{V}_{\ph})_{l,\nu} = \{0\}$, the constraint in~\eqref{onetooneEQ1} that $\pi_{1+7} \eta \in (\mathcal{V}_{\ph})_{l,\nu}$ is the condition that $\eta \in (\Omega^3_{27})_{l,\nu}$. By Lemma~\ref{Gph.lemma} we find that this is equivalent to $\eta \in (\mathcal G_{\ph})_{l,\nu}$. Moreover, in such cases equation~\eqref{onetooneEQ2} becomes the usual equation in the study of compact $\G$~manifolds, namely $(d+d^*) \eta = d^* \big ( \st Q_{\ph}(\eta) \big)$.
\end{rmk}

\begin{proof}[Proof of Theorem~\ref{onetoonethm}]
First we establish an identity~\eqref{onetooneeq0} that will be used to prove both directions of this theorem. If we substitute $\eta$ into equation~\eqref{dThetaeq} and simplify using equation~\eqref{Lpeq}, we obtain:
\begin{align*}
\st d(\Theta(\tilde \ph)) & = \, - d^* \st (\Lp(\eta)) - d^* \st (\Qp(\eta)) \\ & = \, - d^* \Bigl( \frac{4}{3} \pi_1 (\eta) + \pi_7 (\eta) - \pi_{27} (\eta) \Bigr) - d^* \st (\Qp(\eta)).
\end{align*}
If we add and subtract $d^* \eta$ we find that
\begin{equation} \label{onetooneeq0}
\st d(\Theta(\tilde \ph)) =  - \frac{7}{3} d^* \pi_1 (\eta) - 2 d^* \pi_7 (\eta) + d^* \eta - d^* \st (\Qp(\eta)).
\end{equation}

Now suppose that~\eqref{onetooneEQ1} holds. Substituting~\eqref{onetooneEQ1} into~\eqref{onetooneeq0} immediately gives~\eqref{onetooneEQ2}, with the same function $f$ and $1$-form $X$, as we wanted to show.

Conversely, suppose that equation~\eqref{onetooneEQ2} holds. We thus have $d \eta = 0$, which is one of the three equations in~\eqref{onetooneEQ1}, and we also have $d^* \eta - d^* \st (\Qp(\eta)) =\frac{7}{3} d^*\pi_1(\beta) + 2 d^* \pi_7 (\beta)$ for some $\beta \in (\mathcal{V}_{\ph})_{l, \nu}$. Substituting this into equation~\eqref{onetooneeq0} and taking the Hodge star, we obtain
\begin{equation} \label{dThetaeq2}
d(\Theta(\tilde \ph)) \, = \, \frac{7}{3} d \st \pi_1 (\eta - \beta) + 2 d \st \pi_7 (\eta - \beta).
\end{equation}
Note that $\pi_1 (\eta - \beta) = f \ph$ and $\pi_7 (\eta - \beta) = \st (X \wedge \ph)$ for some function $f$ and some $1$-form $X$. Therefore, equation~\eqref{dThetaeq2} can be written as
\begin{equation} \label{dThetaeq3}
d(\Theta(\tilde \ph)) \, = \, \frac{7}{3} df \wedge \ps + 2 \, dX \wedge \ph,
\end{equation}
using the fact that $\ph$ is closed and coclosed. Now by the $C^0$-closeness assumption, we can apply Lemma~\ref{threezeroeslemma} to~\eqref{dThetaeq3} to conclude that all three terms of~\eqref{dThetaeq2} vanish, provided we can show that $d(\Theta(\tilde \ph)) = O(\varrho^{\lambda})$ and $X = O(\varrho^{\lambda + 1})$ for some $\lambda < - \frac{7}{2}$ (AC) or $\lambda > -\frac{7}{2}$ (CS).

Since $\tilde \ph - \ph$ is $O(\varrho^{\nu})$, equation~\eqref{dThetaeq} and Lemma~\ref{quadlemma} give us that $d(\Theta(\tilde \ph))$ is $O(\varrho^{\nu - 1}) + O(\varrho^{2 \nu - 1})$. But $\nu < 0$ in the AC case and $\nu > 0$ in the CS case, so in both cases the first term dominates on the ends, and thus $d(\Theta(\tilde \ph)) = O(\varrho^{\nu - 1})$. Certainly in the CS case we have $\nu - 1 > -\frac{7}{2}$. In the AC 
case we need $\nu - 1 < - \frac{7}{2}$, that is $\nu < - \frac{5}{2}$, which is our hypothesis. Meanwhile $X \wedge \ph = \st \pi_7 (\eta - \beta)$ is $O(\varrho^{\nu})$, so $X = O(\varrho^{\nu})$ since $\ph = O(1)$. (Recall it is the \emph{difference} $\ph - \phc$ that is $O(\varrho^{\nu})$. The $\G$~structure $\ph$ is $O(\varrho^0)$ because $\phc$ is.) Therefore $X = O(\varrho^{\lambda + 1})$ for some $\lambda < -\frac{7}{2}$ (AC) or $\lambda > -\frac{7}{2}$ (CS) is equivalent to $\nu < - \frac{5}{2}$ (AC) or $\nu > -\frac{5}{2}$ (CS), which both hold. Thus we can indeed apply Lemma~\ref{threezeroeslemma} to~\eqref{dThetaeq3} to conclude that $df = 0$ and $dX = 0$. Moreover, $f\ph = \pi_1 (\eta - \beta) = O(\varrho^{\nu})$, means that $f$ tends to $0$ on the ends, and thus $f = 0$. 

All that remains to do in order to establish that~\eqref{onetooneEQ2} implies~\eqref{onetooneEQ1} is to show that $\pi_{1+7} (\eta) = \pi_{1+7}(\beta)$. We have shown that $\pi_{1+7}(\eta - \beta) = \st (X \wedge \ph)$ where $dX = 0$. Recall that $d \eta = 0$. Also, by Proposition~\eqref{Vspaceprop} we have $\pi_1 d \beta = 0$, since $\beta \in (\mathcal{V}_{\ph})_{l,\nu}$. Thus $\pi_1 d (\eta - \beta) = 0$, which means that $d^*X = 0$ by Proposition~\ref{special3formprop}, and hence $X$ is closed and coclosed. The decay conditions on $X$ mean we can apply Lemma~\ref{bochnerlemma} to deduce that $X = 0$, and therefore $\pi_{1+7} (\eta - \beta) = 0$, as required.
\end{proof}
\begin{rmk} \label{onetooneratesrmk}
One direction in the proof of Theorem~\ref{onetoonethm} did not require the assumption of $C^0$-closeness. In the AC case, we needed the hypothesis $\nu < -\frac{5}{2}$ for one direction of this theorem to be able to apply the various gauge-fixing results of Section~\ref{gaugefixingsec}. Without this assumption, we do not have a one-to-one correspondence. All we would know is that solutions to~\eqref{onetooneEQ1} give solutions to~\eqref{onetooneEQ2}, but not conversely. \emph{However}, we will nevertheless be able to understand the moduli space in the AC case all the way up to rate $\nu < 0$, using a slightly modified argument. This is done in Corollary~\ref{finalmodulicor}.
\end{rmk}

\begin{cor} \label{onetoonesmoothcor}
After possibly making $\e > 0$ smaller, the space ${\mathcal R}^{\e}_{l, \nu}$ is equal to the set of \emph{smooth} forms $\eta$ with $|\eta|_{C^0}<\e$ 
that satisfy~\eqref{onetooneEQ1}.
\end{cor}
\begin{proof}
Recall from the statement of Theorem~\ref{onetoonethm} that all $\eta \in {\mathcal R}^{\e}_{l, \nu}$ satisfy~\eqref{onetooneEQ1}, and hence by Theorem~\ref{onetoonethm} such $\eta$ also satisfy equation~\eqref{onetooneEQ2}. In particular, $\eta$ is a solution to the equation\begin{equation*}
(d + d^*) \zeta \, = \, d^* \st (\Qp (\zeta_3 ) ) + \frac{7}{3} d^* \pi_1 (\beta) + 2 d^* \pi_7 (\beta),
\end{equation*}
for an unknown $\zeta \in \Omega^{\text{odd}}$, with $\zeta_3$ being the component of $\zeta$ in $\Omega^3$, and $\beta \in (\mathcal{V}_{\ph})_{l,\nu}$ is smooth. Since $Q_{\ph}$ is the quadratic part of the nonlinear map $\Theta$ at $0$, the linearization of the above equation at $0$ is elliptic, so it is a nonlinear elliptic equation. Hence its solutions are smooth~\cite[Theorem 6.8.1]{Morrey}.
\end{proof}

\subsubsection{Step 3: Applying the Banach space implicit function theorem} \label{IFTsec}

In this section we will apply the Banach space implicit function theorem to study the local structure of the moduli space $\mathcal M_{\nu}$ of $\G$~conifolds of rate $\nu$.

For completeness, we explicitly state here the Banach space implicit function theorem that we will use. Its proof can be found, for example, in Lang~\cite[Theorem 6.2.1]{Lang}. The hats on $\widehat{\mathcal U}$ and $\widehat F$ are employed to match notation with the eventual use of this theorem later in this section.

\begin{thm}[Banach space implicit function theorem] \label{IFT}
Let $\mathcal X$ and $\mathcal Y$ be Banach spaces, and let $\widehat{\mathcal U} \subseteq \mathcal X$ be an open neighbourhood of $0$. Let $\widehat F : \widehat{\mathcal U} \to \mathcal Y$ be a $C^k$-map (with $k \geq 1$) such that $\widehat F(0) = 0$. Suppose that the differential $\DFhat : \mathcal X \to \mathcal Y$ is \emph{surjective}, with kernel $\mathcal K$ such that $\mathcal X = \mathcal K \oplus \mathcal Z$ for some \emph{closed} subspace $\mathcal Z$ of $\mathcal X$. Then there exist open sets $\mathcal A \subseteq \mathcal K$ and $\mathcal B \subseteq \mathcal Z$, both containing $0$, with $\mathcal A \times \mathcal B \subseteq \mathcal U$, and a \emph{unique} $C^k$-map $G : \mathcal A \to \mathcal B$ such that
\begin{equation*}
\widehat F^{-1}(0) \cap (\mathcal A \times \mathcal B) \, = \, \{ (x, G(x)); \, x \in \mathcal A \}
\end{equation*}
in $\mathcal X = \mathcal K \oplus \mathcal Z$. That is, the zero set of $\widehat F$ near the origin in $\mathcal X$ is parametrized by a neighbourhood of the origin in the space $\mathcal K$.
\end{thm}

For the remainder of this section, until Corollary~\ref{finalmodulicor}, we assume $\nu < - \frac{5}{2}$ in the AC case so that Theorem~\ref{onetoonethm} gives a one-to-one correspondence. We also let $\e > 0$ be the constant from Lemma~\ref{threezeroeslemma} and Theorem~\ref{onetoonethm} and let $\mathcal U$ denote the open subset of $\Omega^3_{l, \nu}$ consisting of $3$-forms which are within $\e$ of $\ph$ in the $C^0$ norm.

To begin, we define a \emph{nonlinear map}
\begin{equation*}
F \, : \, \mathcal U \subseteq \Omega^3_{l, \nu} \to \Omega^{2}_{l-1, \nu-1} \oplus \Omega^{4}_{l-1, \nu-1}
\end{equation*}
by the rule
\begin{equation} \label{Fdefneq}
F(\eta) \, = \, (d + d^*) \eta - d^* \st \left( \Qp (\eta) \right).
\end{equation}
The motivation for this definition is that, by Theorem~\ref{onetoonethm}, a neighbourhood of $\ph$ in the moduli space $\mathcal{M}_{\nu}$ is homeomorphic to the space of $\eta \in \mathcal{U}$ such that $F(\eta) \in ({\mathcal W}_{\ph})_{l-1, \nu-1}$. Thus, we want to solve $(\pi_{/ \mathcal{W}} \circ  F) (\eta) = 0$, where $\pi_{/ \mathcal{W}}$ is the projection to the quotient of $\Omega^{2}_{l-1, \nu-1}  \oplus \Omega^{4}_{l-1, \nu-1}$ by the finite-dimensional space $({\mathcal W}_{\ph})_{l-1, \nu-1}$.

We now show that the map $F$ is well-defined.
\begin{lemma} \label{Qgoodlemma}
For $\eta \in \mathcal U$, we have $\Qp(\eta) \in \Omega^4_{l,\nu}$, and so $d^* \st (\Qp(\eta)) \in \Omega^2_{l - 1,\nu - 1}$.
\end{lemma}
\begin{proof}
This argument is very similar to~\cite[Proposition 6.4]{JSL2} and~\cite[Proposition 5.7]{LN}, with some minor differences. We present it here for completeness. From Lemma~\ref{quadlemma} we have $Q_{\ph} (0) = 0$ and $| Q_{\ph} (\eta)| \leq C | \eta |^2$ for some positive constant $C$. That is, the smooth function $Q_{\ph}$ is the quadratic term in the second order Taylor polynomial for the smooth function $\Theta$,  as a function of $\eta \in \Lambda^3 (T^*_x M)$, for fixed $x \in M$. More precisely, if we let $x^1, \ldots, x^7$ be local coordinates on $M$, and let $y^1, \ldots, y^{35}$ be local fibre coordinates for a trivialization of the bundle $\Lambda^3 (T^* M)$, then we can regard $Q_{\ph}$ locally as a smooth function
\begin{equation*}
R(x) \, = \, Q_{\ph} (x, y(x))
\end{equation*}
such that, \emph{for fixed $x$} and for $|y| \leq \e$, we have
\begin{equation} \label{Qgoodlemmaeq1}
(\nabla_x)^a (\partial_y)^b Q(x,y) \, = \, O(|y|^{\max(0, 2 - b)}).
\end{equation}
We need to modify~\eqref{Qgoodlemmaeq1} by inserting the appropriate function of $x$ as a multiplier for such an estimate to hold uniformly on $M$. Since $\ph$ is asymptotic to a $\G$~cone at each end, it is clear that the appropriate uniform estimate is
\begin{equation} \label{Qgoodlemmaeq2}
| (\nabla_x)^a (\partial_y)^b Q(x,y) | \, \leq \, C \varrho^{-a} |y|^{\max(0, 2 - b)}, \qquad \forall a, b \geq 0,
\end{equation}
where $\varrho$ is a radius function on $M$. Since we always assume that $l \geq 6$, by Corollary~\ref{C2cor} we have $\eta \in C^{2, \alpha}_{\nu}$, and thus in particular
\begin{equation} \label{Qgoodlemmaeq3}
|\eta| = O(\varrho^{\nu}) \qquad \text{and} \qquad | \nabla \eta | = O( \varrho^{\nu - 1} ).
\end{equation}
Note that we know nothing about $| \nabla^k y |$ for $k > 2$. Now because $\nu < 0$ in the AC case with $\varrho \to \infty$ on the end, and likewise because $\nu > 0$ in the CS case with $\varrho \to 0$ on each end, in either case we find that $\varrho^{\nu}$, and thus $\eta$, is \emph{bounded} on $M$.

To prove that $Q_{\ph} (\eta) \in \Omega^4_{l, \nu}$ we need to show that
\begin{equation*}
\nabla^j R \, \in \, L^2_{0, \nu - j} \qquad \text{for } 0 \leq j \leq l.
\end{equation*}
By the chain rule, we have
\begin{equation} \label{Qgoodlemmaeq4}
| \nabla^j R | \, \leq \, C_j \sum_{\substack{a, b \, \geq \, 0 \\[0.2em] a + b \, \leq \, j}} | (\nabla_x)^a (\partial_y)^b R(x, y(x)) | \cdot \left[ \, \sum_{\substack{ m_1, \ldots, m_b \, \geq \, 1 \\[0.2em] a + m_1 + \cdots + m_b \, = \, j}} \left( \prod_{i = 1}^b | \nabla^{m_i} y(x) | \right) \right]
\end{equation}
for some positive constant $C_j$ that is purely combinatorial and depends only on $j$. To show that $\nabla^j R$ is in $L^2_{0. \nu - j}$, we need to prove that the integral
\begin{equation*}
\int_M | \varrho^{j - \nu} \, \nabla^j R |^2 \varrho^{-7} \volm
\end{equation*}
is finite. From the inequality~\eqref{Qgoodlemmaeq4}, it suffices to prove that each of the integrals
\begin{equation} \label{Qgoodlemmaeq5}
\int_M \varrho^{2 j - 2 \nu} | (\nabla_x)^a (\partial_y)^b R(x, y(x)) |^2 \left( \prod_{i = 1}^b | \nabla^{m_i} y(x) |^2 \right) \varrho^{-7} \volm
\end{equation}
is finite, where $a, b \geq 0$, $m_1, \ldots, m_b \geq 1$, $a + b \leq j$ and $a + m_1 + \cdots + m_b = j$.

Consider first the case $b = 0$. In this case, $a=j$ and the product is empty. Hence, from~\eqref{Qgoodlemmaeq2} and the fact that $| y | = | \eta |$ is bounded on $M$, we have $| (\nabla_x)^j R| \leq C \varrho^{-j} | y |^2 \leq C \varrho^{-j} | y|$. Hence the integral in~\eqref{Qgoodlemmaeq5} is bounded above by
\begin{equation*}
C \int_M \varrho^{2j - 2 \nu} \varrho^{ - 2 j} | \eta |^2 \varrho^{-7} \volm = C \int_M \varrho^{-2 \nu} | \eta |^2 \varrho^{ - 7} \volm
\end{equation*}
which is finite since $\eta \in L^2_{l, \nu} \subseteq L^2_{0, \nu}$.

Next we consider the case $b = 1$. This time, $m_1 \geq 1$ and $a + m_1 = j$. Thus, from~\eqref{Qgoodlemmaeq2} and~\eqref{Qgoodlemmaeq3}, we have $| (\nabla_x)^a (\partial_y) R| \leq C \varrho^{-a} | y | \leq C \varrho^{-a + \nu}$. Hence the integral in~\eqref{Qgoodlemmaeq5} is bounded above by
\begin{equation} \label{Qgoodlemmaeq6}
C \int_M \varrho^{2j - 2 \nu} \varrho^{ - 2 a + 2 \nu} | \nabla^{m_1} \eta |^2 \varrho^{-7} \volm = C \int_M \varrho^{2 m_1} | \nabla^{m_1} \eta |^2 \varrho^{ - 7} \volm.
\end{equation}
However, since $\eta \in L^2_{l, \nu}$, we have $\nabla^{m_1} \eta \in L^2_{l - m_1, \nu - m_1} \subseteq L^2_{0, \nu - m_1}$ and therefore the integral
\begin{equation} \label{Qgoodlemmaeq7}
\int_M \varrho^{-2 \nu + 2 m_1} | \nabla^{m_1} \eta |^2 \varrho^{-7} \volm
\end{equation}
is finite. But in either the AC case or the CS case, the function $\varrho^{-2 \nu} \to \infty$ at the ends, so outside of a compact set the integrand of~\eqref{Qgoodlemmaeq7} dominates the integrand of~\eqref{Qgoodlemmaeq6}. Hence the integrals in~\eqref{Qgoodlemmaeq5} with $b=1$ are indeed finite.

Finally we consider the general case of $b \geq 2$. Now from~\eqref{Qgoodlemmaeq2} we have $| (\nabla_x)^a (\partial_y)^b R| \leq C \varrho^{-a}$. Hence the integral in~\eqref{Qgoodlemmaeq5} is bounded above by
\begin{equation} \label{Qgoodlemmaeq8}
C \int_M \varrho^{2j - 2 \nu} \varrho^{ - 2 a} \left( \prod_{i=1}^b | \nabla^{m_i} \eta |^2 \right) \varrho^{-7} \volm.
\end{equation}
For $i = 1, \dots, b$, define $q_i = \frac{j - a}{m_i}$. Since $b \geq 2$ and $a + m_1 + \cdots + m_b = j$, we have $q_i > 1$, and also $\sum_{i = 1}^b \frac{1}{q_i} = 1$. Observe also that the integrand of~\eqref{Qgoodlemmaeq8} can be written as $\prod_{i=1}^b s_i$, where
\begin{equation*}
s_i \, = \, \varrho^{2 m_i - \frac{2 \nu}{q_i}} | \nabla^{m_i} \eta |^2 \varrho^{-\frac{7}{q_i}}.
\end{equation*}
Now by H\"older's inequality, we have $\int_M (\prod_{i=1}^b s_i ) \volm = || \prod_{i=1}^b s_i ||_1 \leq \prod_{i=1}^b || s_i ||_{q_i}$. Thus we can finish the proof if we can show the finiteness of the integrals
\begin{equation} \label{Qgoodlemmaeq9}
|| s_i ||_{q_i}^{q_i} \, = \, \int_M s_i^{q_i} \volm \, = \, \int_M \varrho^{2 m_i q_i - 2 \nu} | \nabla^{m_i} \eta |^{2 q_i} \varrho^{-7} \volm.\end{equation}
We claim that the above integral is indeed finite, by the Sobolev embedding Theorem~\ref{Sobolevembeddingthm}. To see this, let $p = 2$, and let $q = 2 q_i > 2$ since $q_i > 1$. Let $m = m_i$. We have $l \geq m$ since $m_i \leq j \leq l$. Furthermore, the last remaining inequality we need to use the embedding theorem is
\begin{equation*}
l - \frac{7}{2} \, \geq \, m - \frac{7}{q} \, = \, m_i - \frac{7}{2 q_i} \, = \, m_i \left( 1 - \frac{7}{2(j - a)} \right),
\end{equation*}
which is easy to verify from $2 \leq j - a \leq l$ and $0 < \frac{m_i}{l} \leq 1$. Thus Theorem~\ref{Sobolevembeddingthm} tells us that $L^2_{l, \nu} \subseteq L^{2 q_i}_{m_i, \nu}$, and therefore
\begin{equation} \label{Qgoodlemmaeq10}
\int_M \varrho^{2 m_i q_i - 2 \nu q_i} | \nabla^{m_i} \eta |^{2 q_i} \varrho^{-7} \volm \, = \, \int_M \varrho^{2 \nu(1 - q_i)} \varrho^{2 m_i q_i - 2 \nu} | \nabla^{m_i} \eta |^{2 q_i} \varrho^{-7} \volm
\end{equation}
is finite. But now, just as in the $b=1$ case, since $q_i > 1$, the integrand of~\eqref{Qgoodlemmaeq10} dominates the integrand of~\eqref{Qgoodlemmaeq9} outside of a compact set, and hence the proof is complete.
\end{proof}

We now consider the \emph{linearization} $\DF$ of the map $F$ defined in equation~\eqref{Fdefneq}.
\begin{lemma} \label{linearizationlemma}
The linearization $\DF$ of $F$ at the origin is the map 
\begin{equation} \label{linearizationeq}
\begin{aligned}
\DF \, : \, & \Omega^3_{l, \nu} \to \Omega^{\bullet}_{l-1, \nu-1} \\
& \dot \eta \mapsto (d + d^*) \dot \eta.
\end{aligned}
\end{equation}
\end{lemma}
\begin{proof}
This follows immediately from the definition of $\Qp$ as the quadratic approximation of the nonlinear map $\Theta$ in Lemma~\ref{quadlemma}.
\end{proof}

From Lemma~\ref{linearizationlemma}, the linearization $\DF$ always maps \emph{onto} the space $\mathcal Y_0 = (d + d^*)(\Omega^3_{l, \nu})$ defined in~\eqref{Ydefneq}, which is a Banach space by Lemma~\ref{YBanachspacelemma}. However, to be able to apply the Banach space implicit function theorem to $F$, we would need to know that $F$ maps into $\mathcal Y_0$. If we could show this, we could redefine the codomain of the map $F$ to be $\mathcal Y_0$, surjectivity would then be immediate and we would be able to apply Theorem~\ref{IFT}. However, the problem is that we only know that $F$ maps into $\mathcal Y = d(\Omega^3_{l,\nu})+d^* (\Omega^3_{l, \nu})$, which is in general \emph{not} contained in $\mathcal{Y}_0=(d + d^*)(\Omega^3_{l, \nu})$, and thus $DF|_0$ may not surject onto a Banach space containing the image of $F$. 

We showed in Corollary~\ref{obstructionspacecor} that $\mathcal Y = \mathcal Y_0 \oplus \mathcal O_{\nu}$ for a finite-dimensional space $\mathcal{O}_{\nu}$. Now the image of $F$ is contained in $\mathcal{Y}$, but $DF|_0$ does not map onto $\mathcal{Y}$, so we ``correct'' the map $F$ to a map
\begin{equation*}
\widehat F \, : \,  \mathcal U \oplus \mathcal O_{\nu} \to \mathcal Y= \mathcal Y_0 \oplus \mathcal O_{\nu}
\end{equation*}
by the rule
\begin{equation} \label{Fhatdefneq}
\widehat F(\eta, \xi) \, = \, (d + d^*) \eta - d^* \st \left( \Qp (\eta) \right) + \xi.
\end{equation}

\begin{lemma} \label{imageFhatlemma}
For generic rates $\nu$, the space $\mathcal Y = \mathcal Y_0 \oplus \mathcal O_{\nu}$ is a Banach space, and the map $\widehat F$ defined in equation~\eqref{Fhatdefneq} actually maps into this space.
\end{lemma}
\begin{proof}
The first statement is precisely Lemma~\ref{YBanachspacelemma}. Since $\mathcal Y_0 = (d + d^*)(\Omega^3_{l, \nu})$, to show that the map $\widehat F$ of equation~\eqref{Fhatdefneq} maps into $\mathcal Y$, we need only show that $d^* \st \left( \Qp (\eta) \right)$ lies in $\mathcal Y$.  However, we showed in Lemma~\ref{Qgoodlemma} that the $3$-form $\chi = \st \left( \Qp (\eta) \right)$ lies in $\Omega^3_{l, \nu}$, so the result is now immediate.
\end{proof}

\begin{cor} \label{linearizationFhatsurjectivecor}
The linearization $\DFhat$ of $\widehat F$ at the origin is the map
\begin{equation} \label{linearizationFhateq}
\begin{aligned}
\DFhat \, : \, \Omega^3_{l, \nu} \oplus \mathcal O_{\nu} & \to \mathcal Y \\
(\dot \eta, \dot \xi) & \mapsto (d + d^*) \dot \eta + \dot \xi.
\end{aligned}
\end{equation}
and $\DFhat$ is \emph{surjective} onto $\mathcal Y$.
\end{cor}
\begin{proof}
The first statement follows from Lemma~\ref{linearizationlemma} and equation~\eqref{Fhatdefneq}, while the second statement is immediate from~\eqref{realobstructionspacedefneq}.
\end{proof}

Consider the Banach space $\mathcal X = \Omega^3_{l, \nu} \oplus \mathcal O_{\nu}$. For generic rates $\nu$, we have shown that the differential $\DFhat$ maps surjectively from $\mathcal X$ onto $\mathcal Y$. It is also clear that $\widehat F$ is a $C^{\infty}$ map from $\widehat{\mathcal U} = \mathcal U \times \mathcal O_{\nu}$ to $\mathcal Y$. Finally, note from Corollary~\ref{linearizationFhatsurjectivecor} and Definition~\ref{realobstructionspacedefn} that
\begin{equation} \label{kerDFhateq}
\mathcal K \, = \, \ker \DFhat \, = \, \ker \DF \oplus \{ 0 \} \, = \, \mathcal H^3_{\nu},
\end{equation}
and we have $\mathcal X = \mathcal K \oplus \mathcal Z$ for a closed subspace $\mathcal Z$ by Propositions~\ref{HodgedecompositionpropL2} and~\ref{HodgedecompositionpropnonL2}. 

In the case when $(\mathcal{W}_{\ph})_{l-1,\nu-1}=0$, which is automatic in the AC case for $\nu>-4$, the construction above suffices and we can apply the implicit function theorem. However, in general, we only wish to solve $F(\eta) \in (\mathcal{W}_{\ph})_{l-1, \nu-1}$, rather than $F (\eta) = 0$. Thus we let 
\begin{equation*}
\pi_{/ \mathcal{W}} : \mathcal{Y} \to \mathcal{Y}_{/ \mathcal{W}} = \mathcal{Y} / (\mathcal{W}_{\ph})_{l-1, \nu-1}
\end{equation*}
denote the projection map, so that our problem is to solve $(\pi_{/ \mathcal{W}} \circ F) (\eta) = 0$. We do not know that the linearization $\pi_{/ \mathcal{W}} \circ \rest{DF}{0}$ maps onto $\mathcal{Y}_{/ \mathcal{W}}$, but the existence of $\mathcal{O}_{\nu}$ means that there exists a finite-dimensional subspace $\mathcal{O}_{/ \mathcal{W}} \subseteq \mathcal{Y}_{/ \mathcal{W}}$ such that
\begin{equation*}
\mathcal{Y}_{/ \mathcal{W}} \, = \, \pi_{/ \mathcal{W}}(\mathcal{Y}_0) \oplus \mathcal{O}_{/ \mathcal{W}}.
\end{equation*}
Define the space
\begin{equation*}
\mathcal{X}_{/ \mathcal{W}} \, = \, \Omega^3_{l,\nu} \oplus \mathcal{O}_{/ \mathcal{W}}.
\end{equation*}
Recall from above that $\mathcal{X} = \Omega^3_{l,\nu} \oplus \mathcal{O}_{\nu}$, thus to construct $\mathcal X_{/ \mathcal{W}}$ we just add the subspace $\mathcal{O}_{/ \mathcal W}$ of $\mathcal{O}_{\nu}$ to $\Omega^3_{l, \nu}$ rather than all of $\mathcal{O}_{\nu}$.

We can now define a smooth map between Banach spaces by
\begin{equation} \label{FhatW.defn.eq}
\begin{aligned}
\widehat{F}_{/ \mathcal{W}} \, : \, \mathcal{U}\oplus\mathcal{O}_{/\mathcal{W}}\subseteq\mathcal{X}_{/ \mathcal{W}} \, & \to \, \mathcal{Y}_{/ \mathcal{W}} \\(\eta, \xi) \, & \mapsto \, (\pi_{/ \mathcal{W}} \circ F) (\eta) + \xi,
\end{aligned}
\end{equation}
whose linearization $\rest{D \widehat{F}_{/ \mathcal{W}}}{0}$ at the origin is surjective by Corollary~\ref{linearizationFhatsurjectivecor}. Moreover, the kernel of this linearization is
\begin{equation*}
\mathcal{K}_{/ \mathcal{W}} \,= \,\ker(\pi_{/ \mathcal{W}} \circ \rest{DF}{0} )
\end{equation*}
and $\mathcal{X}_{/ \mathcal{W}} = \mathcal{K}_{/ \mathcal{W}} \oplus \mathcal{Z}_{/ \mathcal{W}}$ for a closed subspace $\mathcal{Z}_{/ \mathcal{W}}$.

Thus, we can apply Theorem~\ref{IFT} to conclude that there exist open sets $\mathcal A \subseteq \mathcal K_{/ \mathcal{W}}$ and $\mathcal B \subseteq \mathcal Z_{/ \mathcal{W}}$, both containing $0$, with $\mathcal A \times \mathcal B \subseteq \widehat{\mathcal U} = \mathcal U \times \mathcal O_{/ \mathcal{W}}$, and a $C^{\infty}$-map $G : \mathcal A \to \mathcal B$ such that
\begin{equation*}
\widehat F_{/ \mathcal{W}}^{-1}(0) \cap (\mathcal A \times \mathcal B) \, = \, \{ (x, G(x)); \, x \in \mathcal A \}
\end{equation*}
in $\mathcal X_{/ \mathcal{W}} = \mathcal K_{/ \mathcal{W}} \oplus \mathcal Z_{/ \mathcal{W}}$. We have therefore established the following result.

\begin{cor} \label{widehatFspacecor}
The set $\widehat F_{/ \mathcal{W}}^{-1}(0)$ is a smooth manifold, diffeomorphic to an open neighbourhood $\mathcal A$ of the origin in $\mathcal K_{/ \mathcal{W}}$. In particular, when $(\mathcal{W}_{\ph})_{l-1, \nu-1} = 0$, then we have that $\mathcal{K}_{/ \mathcal{W}} = \mathcal{H}^3_{\nu}$ and $\dim \widehat F_{/ \mathcal{W}}^{-1}(0) = \dim \mathcal H^3_{\nu}$.
\end{cor}
Notice that $\mathcal Z_{/ \mathcal{W}} = \mathcal Z' \oplus \mathcal O_{/ \mathcal{W}}$ for some closed subspace $\mathcal Z'$ of $\Omega^3_{l, \nu}$. Thus the projection map $\pi_o : \mathcal B \to \mathcal O_{/\mathcal{W}}$ is well-defined and smooth. It is also clear that $(\pi_{/ \mathcal{W}} \circ  F)^{-1}(0)$ is homeomorphic to the subset $(\pi_o \circ G)^{-1}(0)$ of $\mathcal A \subseteq \mathcal K_{/ \mathcal{W}}$. Hence we have shown the following.
\begin{cor} \label{Psimapcor}
The composition $\Psi_{\nu} = \pi_o \circ G$ is a smooth map
\begin{equation*}
\Psi_{\nu} \, : \, \mathcal A \to \mathcal O_{/ \mathcal{W}}
\end{equation*}
from the open subset $\mathcal A$ of the finite-dimensional vector space $\mathcal K_{/ \mathcal{W}}$ to the finite-dimensional vector space $\mathcal O_{/ \mathcal{W}}$, whose zero set $\Psi_{\nu}^{-1}(0)$ is homeomorphic to $(\pi_{/ \mathcal{W}} \circ F)^{-1}(0)$.
\end{cor}

By combining Corollary~\ref{Psimapcor} with Theorem~\ref{onetoonethm} and Corollary~\ref{widehatFspacecor} we have the following result.
\begin{thm} \label{finalmodulithm}
Near $[\ph]$, the moduli space $\mathcal{M}_{\nu}$ is homeomorphic to the zero set $\Psi_{\nu}^{-1}(0)$ of a smooth map $\Psi_{\nu}$ from an open subset $\mathcal A$ of the finite-dimensional vector space $\mathcal K_{/ \mathcal{W}}$ to the finite-dimensional vector space $\mathcal O_{/ \mathcal{W}}$. In particular, in the AC case when $\nu \in (-4, -\frac{5}{2})$, we have $\mathcal O_{/ \mathcal{W}} = \{0\}$ and $\pi_{/ \mathcal{W}}$ is the identity, and thus in this case we conclude that the moduli space $\mathcal M_{\nu}$ is a smooth manifold of dimension $\dim \mathcal H^3_{\nu}$. 
\end{thm}

\begin{rmk} \label{finalmodulitmk}
We could conclude that $\mathcal{M}_{\nu}$ is a smooth manifold if we knew that the map $\Psi_{\nu}$ was the zero map. However, this will not be true in general.
\end{rmk}

We can also extend Theorem~\ref{finalmodulithm} to the AC case for rates $\nu \in [-\frac{5}{2}, 0)$ as follows.
\begin{cor}\label{finalmodulicor}
Let $M$ be an AC $\G$~manifold with generic rate $\nu \in (-4,0)$. Then the moduli space $\mathcal{M}_{\nu}$ is a smooth manifold near $[\ph]$ with dimension $\dim \mathcal{H}^3_{\nu} - \dim\mathcal{H}^1_{\nu+1}$.
\end{cor}
\begin{proof}
Recall from Remark~\ref{onetoonethmrmk} and equation~\eqref{mathcalWdefneq} that $(\mathcal{W}_{\ph})_{l-1, \nu-1} = \{0\}$ in this case. Moreover, from Definition~\ref{realobstructionspacedefn} and we also have that $\mathcal{O}_{\nu} = \{0\}$ since $\nu > -4$. Thus $\widehat F_{/ \mathcal{W}} = F$, so $F^{-1}(0)$ is a smooth manifold near $0$, given as the graph $\Gamma$ of a map $G$ defined on an open set $\mathcal{A}$ in $\mathcal{H}^3_{\nu}$. Notice that if $\eta \in \mathcal{H}^3_{\nu}$, then $\eta \in (\mathcal{G}_{\ph})_{l,\nu}$ since $d^*\eta = 0$ implies $\pi_7d^*\eta = 0$. 

Let $\eta = d (X \hk \ph) \in d(\Omega^2_7)_{l+1, \nu+1} \cap (\mathcal{G}_{\ph})_{l,\nu}$. Then $\mlap X = 0$, since $\mlap = \pi_7 d^* d$. We claim that $dX = 0$ and $d^*X = 0$. To prove the claim, first note that the operator $\dd^1_{l+1, \lambda+1}$ given in~\eqref{Pdefneq} is injective for all $\lambda \leq -1$ by Lemma~\ref{bochnerlemma} and surjective for all $\lambda > -5$ by Lemma~\ref{Yspacelemma} and~\eqref{Ydefneq}. Since the kernel is zero for all $\lambda \leq -1$, and since the cokernel does not change for all $\lambda > -5$, Theorem~\ref{Pindexchangethm} says that all the homogeneous closed and coclosed $1$-forms on the cone $C$ of order $\lambda + 1 < \nu + 1$ satisfy $\lambda+1 \in (0, \nu+1)$, and they define elements of $\ker \dd^1_{\nu+1} = \mathcal{H}^1_{\nu+1}$, and thus elements of $\ker \mlap_{\nu+1}$. Recall that $\ker \mlap_{\lambda+1} = \{0\}$ for $\lambda \leq -1$ by Lemma~\ref{bochnerlemma2} and all homogeneous $1$-forms $Y$ of order $\lambda+1\in (0,1)$ on $C$ solving $\mlapc Y = 0$ are closed and coclosed by Proposition~\ref{excludeextensionprop2}. Thus, we have shown that $\ker \mlap_{\nu+1} = \mathcal{H}^1_{\nu+1}$. We deduce that $X$ is closed and coclosed as claimed. 

Consequently, since $d^* (X \hk \ph) = \st (dX \wedge \ps) = 0$, we have that $\eta = d (X \hk \ph)$ in fact satisfies
\begin{equation*}
d^* \eta = d^* d (X \hk \ph) \, = \, \Delta (X \hk \ph) \, = \, (\Delta X) \hk \ph \, = \, 0.
\end{equation*}
We therefore conclude that $d(\Omega^2_7)_{l+1, \nu+1} \cap (\mathcal{G}_{\ph})_{l, \nu} \subseteq \mathcal{H}^3_{\nu}$. Notice that the map
\begin{equation*}
X \in \mathcal{H}^1_{\nu+1} \, \mapsto \, d (X \hk \ph) \in \mathcal{H}^3_{\nu}
\end{equation*}
is injective since $d (X \hk \ph) = 0$ if and only if $X$ is Killing by Proposition~\ref{Killing.prop} as $dX = 0$ and $d^*X = 0$. Hence, the intersection $d(\Omega^2_7)_{l+1, \nu+1} \cap (\mathcal{G}_{\ph})_{l,\nu}$ is in fact isomorphic to $\mathcal{H}^1_{\nu+1}$.

Now recall from Theorem~\ref{infinitesimalslicethm} that $(\mathcal{G}'_{\ph})_{l,\nu}$ is a direct complement of $d(\Omega^2_7)_{l+1, \nu+1} \cap (\mathcal{G}_{\ph})_{l,\nu}$ in $(\mathcal{G}_{\ph})_{l,\nu}$. If we now define $\mathcal{A}' = \mathcal{A} \cap (\mathcal{G}_{\ph}')_{l,\nu}$ then the graph $\Gamma'$ of $G$ over $\mathcal{A}'$ is a smooth submanifold of $\Gamma$, and thus of $F^{-1}(0)$. Moreover, the tangent space of $\Gamma'$ is isomorphic to the complement of $d(\Omega^2_7)_{l+1, \nu+1} \cap \mathcal{H}^3_{\nu}$ in $\mathcal{H}^3_{\nu}$. Since $T_{\ph} (\mathcal{D}_{\nu+1} \cdot \ph) = d (\Omega^2_7)_{l+1, \nu+1}$, we see that $\Gamma'$ describes all of the solutions to $F(\eta) = 0$ where $\eta \in \mathcal{C}_{l,\nu}$ is near $0$ and gauge-fixed; that is, the AC torsion-free $\G$~structures with rate $\nu$ on $M$ up to the action of $\mathcal{D}_{\nu+1}$. Hence, a neighbourhood of $[\ph]$ in $\mathcal{M}_{\nu}$ can be identified with $\Gamma'$, and we deduce that $\mathcal{M}_{\nu}$ is a smooth near $[\ph]$ and has dimension $\dim \mathcal{H}^3_{\nu} - \dim \mathcal{H}^1_{\nu+1}$.
\end{proof}

\subsubsection{Step 4: Computing the virtual dimension of the moduli space} \label{dimensionsec}

In this section we compute the expected (virtual) dimension of the moduli space $\mathcal{M}_{\nu}$ for rates in both the AC and the CS cases, including exact results for the dimension in the unobstructed setting. From Corollary~\ref{Psimapcor} and the discussion in Section~\ref{IFTsec}, this virtual dimension is
\begin{equation} \label{virtualdimeq}
\vdim \mathcal M_{\nu} \, = \, \ind D (\pi_{/ \mathcal{W}} \circ F) \, = \, \dim \mathcal{K}_{/ \mathcal{W}} - \dim \mathcal{O}_{/ \mathcal{W}},
\end{equation}
where the map $D (\pi_{/ \mathcal{W}} \circ F) = \pi_{/ \mathcal{W}} \circ DF : \Omega^3_{l, \nu} \to \mathcal{Y}_{/ \mathcal{W}}$ is the composition of the two Fredholm maps $DF: \Omega^3_{l, \nu} \to \mathcal Y$ and $\pi_{/ \mathcal{W}} : \mathcal Y \to \mathcal{Y}_{/ \mathcal{W}}$.

\begin{prop} \label{KWOWdimprop}
The difference in dimensions $\dim \mathcal{K}_{/ \mathcal{W}} - \dim \mathcal{O}_{/ \mathcal{W}}$ is given by
\begin{equation} \label{vdimtempeq}
\dim \mathcal{K}_{/ \mathcal{W}} - \dim \mathcal{O}_{/ \mathcal{W}} \, = \, \dim \mathcal{H}^3_{\nu} - \dim\mathcal{O}_{\nu} + \dim (\mathcal{W}_{\ph})_{l-1, \nu-1}.
\end{equation}
\end{prop}
\begin{proof}
Since $DF = \dd^3_{l, \nu}$, we have
\begin{equation*}
\ind DF \, = \, \dim (\ker DF ) - \dim (\coker DF) \, = \, \dim \mathcal H^3_{\nu} - \dim \mathcal O_{\nu}.
\end{equation*}
Moreover, we also have
\begin{equation*}
\ind \pi_{/ \mathcal{W}} \, = \, \dim (\ker \pi_{/ \mathcal{W}}) - \dim (\coker \pi_{/ \mathcal{W}}) \, = \, \dim (\mathcal{W}_{\ph})_{l-1, \nu-1} - 0.
\end{equation*}
Equation~\eqref{vdimtempeq} now follows from the fact that the index of the composition of two Fredholm maps is the sum of the indices of each map.
\end{proof}

\begin{lemma} \label{VWisom.lem}
The map from $(\mathcal{V}_{\ph})_{l,\nu}$ to $(\mathcal{W}_{\ph})_{l-1, \nu-1}$ given by $\beta \mapsto \frac{7}{3} d^* \pi_1 (\beta) + 2 d^* \pi_7 (\beta)$ is an isomorphism.
\end{lemma}
\begin{proof}
From equation~\eqref{mathcalWdefneq} the map is surjective by definition, so it only remains to show that it is injective. Write $\beta = f \ph + \st (X \wedge \ph)$ and suppose that $\frac{7}{3} d^* \pi_1 (\beta) + 2 d^* \pi_7 (\beta) = 0$. This simplifies to $- \frac{7}{3} \st (df \wedge \ps) -2 \st (dX \wedge \ph) = 0$. We see from~\eqref{o27eq} and~\eqref{curllemmaeq} that this implies $-\frac{7}{3} df + \frac{4}{3} \curl  X = 0$. Taking $d^*$ of both sides and using Remark~\ref{curlrmk} gives $\Delta f = 0$ and thus $f = 0$ by the maximum principle. Hence, $d^* \pi_7 (\beta) = 0$, which means $dX \wedge \ph = 0$ and thus $dX = 0$. Moreover, from Proposition~\ref{special3formprop} we have $d^*X = 0$ as $\pi_1 (d \beta) = 0$. We thus conclude that $X = 0$ from Lemma~\ref{bochnerlemma}, completing the proof.
\end{proof}

\begin{cor} \label{vdim.cor}
The virtual dimension $\vdim \mathcal M_{\nu}$ of $\mathcal M_{\nu}$ is given by
\begin{equation} \label{vdimeq}
\vdim \mathcal M_{\nu} \, = \, \dim \mathcal{K}_{/ \mathcal{W}} - \dim \mathcal{O}_{/ \mathcal{W}} \, = \, \dim \mathcal{H}^3_{\nu} - \dim\mathcal{O}_{\nu} + \dim (\mathcal{V}_{\ph})_{l, \nu}.
\end{equation}
\end{cor}
\begin{proof}
This is immediate from~\eqref{virtualdimeq},~\eqref{vdimtempeq}, and Lemma~\ref{VWisom.lem}.
\end{proof}

We have thus shown in~\eqref{vdimeq} that the expected dimension of $\mathcal{M}_{\nu}$ is made up of two contributions: $\dim (\mathcal{V}_{\ph})_{l,\nu}$ and $(\dim \mathcal{H}^3_{\nu} - \dim \mathcal{O}_{\nu})$. From equations~\eqref{Ydefneq},~\eqref{realobstructionspacedefneq}, and ~\eqref{Y0defneq}, the latter contribution is precisely the \emph{index} of the map
\begin{equation*}
\dd^3_{l,\nu} : \Omega^3_{l, \nu} \to \mathcal Y
\end{equation*}
defined in equation~\eqref{Pdefneq}. Thus, we have
\begin{equation} \label{virtualdimeq2}
\vdim {\mathcal M}_{\nu} \, = \, \ind(\dd_{l, \nu})+\dim(\mathcal{V}_{\ph})_{l,\nu}.
\end{equation}

\begin{rmk} \label{simplifynotationrmk}
In equation~\eqref{virtualdimeq2} above and henceforth, to simplify notation, we will always denote $\dd^3_{l,\nu}$ as $\dd_{l,\nu}$ because we will only make use of the map $\dd^k_{l,\nu}$ for $k=3$ from now on.
\end{rmk}

Before moving on to the explicit computation of $\vdim \mathcal M_{\nu}$, we pause to introduce some auxiliary spaces and to derive expressions for the dimensions of both $\mathcal{K}_{/ \mathcal{W}}$ and $\mathcal{O}_{/ \mathcal{W}}$ that will be used later in Section~\ref{smoothCSsec}.

\begin{lemma} \label{Etildelemma}
Recall the spaces $(\mathcal{G}_{\ph})_{l,\nu}$, $(\mathcal{E}_{\ph})_{l,\nu}$, $(\mathcal{V}_{\ph})_{l,\nu}$, and $(\mathcal{S}_{\ph})_{l,\nu}$ from Lemma~\ref{Gph.lemma}, Theorem~\ref{infinitesimalslicethm}, and Corollary~\ref{infinitesimalslicecor}. Consider the cases when the finite-dimensional space $(\mathcal{E}_{\ph})_{l,\nu}$ is not necessarily trivial. This corresponds to $\nu > 0$ (CS) or $\nu < -7$ (AC). We can choose a finite-dimensional space $(\tilde{\mathcal{E}}_{\ph})_{l,\nu}$ such that
\begin{equation} \label{Etildeeq}
(\mathcal{S}_{\ph})_{l,\nu} \cap \ker (d \circ \Lp) \, = \, \mathcal H^3_{\nu} \oplus (\tilde{\mathcal{E}}_{\ph})_{l,\nu}.
\end{equation}
Moreover, the restriction of $\pi_{1+7}$ to $(\tilde{\mathcal{E}}_{\ph})_{l,\nu}$ is injective, and we denote its image by
\begin{equation*}
(\tilde{\mathcal{V}}_{\ph})_{l,\nu} \, = \, \pi_{1+7} (\tilde{\mathcal{E}}_{\ph})_{l,\nu} \, \subseteq \, ({\mathcal{V}}_{\ph})_{l,\nu}.
\end{equation*}
\end{lemma}
\begin{proof}
First, we note that for these rates, $\mathcal H^3_{\nu}\subseteq\Omega^3_{27}$, since there are no nontrivial harmonic functions or $1$-forms of these rates by Lemma~\ref{bochnerlemma}. It now follows from Lemma~\ref{Gph.lemma} and equation~\eqref{Lpeq} that
\begin{equation*}
(\mathcal{G}_{\ph})_{l,\nu} \cap \ker (d \circ \Lp) \, = \, \mathcal H^3_{\nu}.
\end{equation*}
Recall in equation~\eqref{Stempeq} during the proof of Theorem~\ref{infinitesimalslicethm} we established that in these cases
\begin{equation*}
(\mathcal{S}_{\ph})_{l,\nu} \, = \, (\mathcal{G}_{\ph})_{l,\nu} \oplus (\mathcal{E}_{\ph})_{l,\nu}
\end{equation*}
for a finite-dimensional space $(\mathcal{E}_{\ph})_{l,\nu}$. Hence, because $(\mathcal{G}_{\ph})_{l,\nu}$ has finite codimension in $(\mathcal{S}_{\ph})_{l,\nu}$, it follows that $\mathcal H^3_{\nu} = (\mathcal{G}_{\ph})_{l,\nu} \cap \ker (d \circ \Lp)$ has finite codimension in $(\mathcal{S}_{\ph})_{l,\nu} \cap \ker (d \circ \Lp)$. Thus we can choose a finite-dimensional complement $(\tilde{\mathcal{E}}_{\ph})_{l,\nu}$ satisfying~\eqref{Etildeeq}. By construction, we have $(\mathcal{G}_{\ph})_{l,\nu} \cap (\tilde{\mathcal{E}}_{\ph})_{l,\nu} = \{ 0 \}$. The injectivity of the restriction of $\pi_{1+7}$ to $(\tilde{\mathcal{E}}_{\ph})_{l,\nu}$ follows from the fact that the kernel of $\pi_{1+7}$ on $(\mathcal{S}_{\ph})_{l,\nu}$ is equal to $(\mathcal{G}_{\ph})_{l,\nu}$ by definition. The fact that $(\tilde{\mathcal{V}}_{\ph})_{l,\nu} =  \pi_{1+7} (\tilde{\mathcal{E}}_{\ph})_{l,\nu}$ is a subspace of $({\mathcal{V}}_{\ph})_{l,\nu}$ is by definition of the space $(\mathcal{S}_{\ph})_{l,\nu}$ in Corollary~\ref{infinitesimalslicecor}.
\end{proof}

\begin{rmk} \label{Etildetmk}
We could now go back and redefine $(\mathcal{E}_{\ph})_{l,\nu}$ to ensure that it contains $(\tilde{\mathcal{E}}_{\ph})_{l,\nu}$ as a subspace, but this will not be necessary for us. However, this observation justifies our choice of notation, because $(\tilde{\mathcal{V}}_{\ph})_{l,\nu}$ is definitely a subspace of $(\mathcal{V}_{\ph})_{l,\nu}$.
\end{rmk}

\begin{prop} \label{KW.dim.prop}
Recall the notation of Lemma \ref{Etildelemma}.  The dimension of $\mathcal{K}_{/ \mathcal{W}}$ is given by
\begin{equation} \label{KWdimeq}
\dim \mathcal{K}_{/ \mathcal{W}} \, = \, \dim \mathcal{H}^3_{\nu} + \dim(\tilde{\mathcal{E}}_{\ph})_{l,\nu} \, = \, \dim \mathcal{H}^3_{\nu} + \dim (\tilde{\mathcal{V}}_{\ph})_{l,\nu}.
\end{equation}
\end{prop}
\begin{proof}
We have that $\zeta \in \mathcal{K}_{/ \mathcal{W}}$ if and only if $\pi_{/ \mathcal{W}} (d+d^*) \zeta = 0$, which by definition is equivalent to 
\begin{equation} \label{KW.eq}
d \zeta = 0 \quad \text{and} \quad d* \zeta = \frac{7}{3} df \wedge \ps + 2 dX \wedge \ph
\end{equation}
for some function $f$ and $1$-form $X$ such that $\beta = f \ph + \st (X \wedge \ph) \in (\mathcal{V}_{\ph})_{l,\nu}$. We also have $d^*X = 0$ since $\pi_1 (d\beta) = 0$. Since $\nu < 0$ in the AC case and $\nu > 0$ in the CS case, Theorem~\ref{linearized.onetoonethm} shows that solutions to~\eqref{KW.eq} correspond to closed $3$-forms $\zeta \in \Omega^3_{l,\nu}$ with $\pi_{1+7} (\zeta) = \beta$ and $d (\Lp \zeta) = 0$. (Note that the constant $c$ from Theorem~\ref{linearized.onetoonethm} is necessarily zero here because both $\zeta$ and $\beta$ decay to zero on the ends.) Thus we deduce that $\zeta \in \mathcal{K}_{/ \mathcal{W}}$ if and only if $\zeta \in (\mathcal{S}_{\ph})_{l,\nu} \cap \ker (d \circ \Lp)$. The result now follows from Lemma~\ref{Etildelemma}.
\end{proof}

\begin{prop}\label{OW.dim.prop}
The dimension of $\dim \mathcal{O}_{/ \mathcal{W}}$ is given by
\begin{equation} \label{OWdimeq}
\dim \mathcal{O}_{/ \mathcal{W}} \, = \, \dim \mathcal{O}_{\nu} - \dim(\mathcal{V}_{\ph})_{l,\nu} + \dim(\tilde{\mathcal{V}}_{\ph})_{l,\nu}.
\end{equation}
\end{prop}
\begin{proof}
This is immediate from equations~\eqref{vdimeq} and~\eqref{KWdimeq}.
\end{proof}

We now return to the explicit computation of $\vdim \mathcal M_{\nu}$. Because of equation~\eqref{virtualdimeq2}, in order to explicitly compute $\vdim {\mathcal M}_{\nu}$, we need to use our special index change Theorem~\ref{Pindexchangethm} for the operator $\dd_{l, \nu}$. The first step is to compute the index of $\dd_{l, \lambda}$ exactly in some special cases. Recall from Corollary~\ref{LMcohomologycor} and Proposition~\ref{kernelchange34prop} we have, for $\e > 0$ sufficiently small,
\begin{equation} \label{top34dimensionseq}
\left.
\begin{aligned}
\dim \mathcal H^3_{-3 + \e} \, & = \, b^3_{\mathrm{cs}} + \dim (\im \Upsilon^3) \\
\dim \mathcal H^3_{\lambda} \, & = \, b^3_{\mathrm{cs}}, \, \, \lambda \in (-4,-3) \\
\dim \mathcal H^3_{-4 - \e} \, & = \, b^3_{\mathrm{cs}} - \dim (\im \Upsilon^4)
\end{aligned}
\, \, \right\} \, \, \text{ (AC)},
\left. \qquad 
\begin{aligned}
\dim \mathcal H^3_{-3 + \e} \, & = \, b^3 - \dim (\im \Upsilon^3) \\
\dim \mathcal H^3_{\lambda} \, & = \, b^3, \, \, \lambda \in (-4,-3) \\
\dim \mathcal H^3_{-4 - \e} \, & = \, b^3 + \dim (\im \Upsilon^4)
\end{aligned}
\, \, \right\} \, \, \text{ (CS)}, 
\end{equation}
where $b^3 = \dim H^3(M, \R)$ and $b^3_{\mathrm{cs}} = \dim H^3_{\mathrm{cs}}(M)$ are the ordinary and compactly supported third Betti numbers of $M$, respectively. The above equations say that the rates $\lambda = -3$ and $\lambda = -4$ contribute to changes in the \emph{kernel} of $\dd$, not the cokernel. This fact also follows from Theorem~\ref{cones24thm}, as we stated in Remark \ref{kernel.only.rmk}, but the above equations tell us exactly how the kernel (and thus the index) changes at these rates. 

\begin{prop} \label{Pindexexactprop}
The index of $\dd_{l, \lambda}$ is purely topological for a certain range of rates, as follows.
\begin{equation*}
\left.
\begin{aligned}
\ind (\dd_{l, -3 + \e}) \, & = \, b^3_{\mathrm{cs}} + \dim (\im \Upsilon^3) \\
\ind (\dd_{l, \lambda}) \, & = \, b^3_{\mathrm{cs}}, \, \, \lambda \in (-4,-3) \\
\ind (\dd_{l, -4 - \e}) \, & = \, b^3_{\mathrm{cs}} - \dim (\im \Upsilon^4)
\end{aligned}
\, \, \right\} \, \, \text{ (AC)},
\left. \qquad 
\begin{aligned}
\ind (\dd_{l, -3 + \e}) \, & = \, b^3 - \dim (\im \Upsilon^3) \\
\ind (\dd_{l, \lambda}) \, & = \, b^3, \, \, \lambda \in (-4,-3) \\
\ind (\dd_{l, -4 - \e}) \, & = \, b^3 + \dim (\im \Upsilon^4)
\end{aligned}
\, \, \right\} \, \, \text{ (CS)}, 
\end{equation*}
\end{prop}
\begin{proof}
By Lemma~\ref{PFredholmlemma}, the space $\coker(\dd_{l, \lambda}) = \{0 \}$ if $\lambda > -4$ (AC) or $\lambda < -3$ (CS). Moreover, we know that the index cannot change at all in the interval $(-4, -3)$ since by Corollary~\ref{kernelchangecor} there are no critical rates for $\dd$ in this interval. Finally, equation~\eqref{top34dimensionseq} and the discussion following it tells us that the cokernel \emph{does not} change at the rates $-4$ and $-3$, and the change in the kernel at those rates is given by~\eqref{top34dimensionseq}. When these facts are all combined we obtain the statements above.
\end{proof}

We now compute $\dim (\mathcal{V}_{\ph})_{l,\nu}$, by using index change formulas for the Dirac operator, the scalar Laplacian, and the Laplacian on $1$-forms.
\begin{prop} \label{V.dim.prop}
The dimension of $(\mathcal{V}_{\ph})_{l,\nu}$ is given as follows.
\begin{enumerate}[(a)]
\item In the AC case, for generic $\nu < 0$,
\begin{equation*}
\dim (\mathcal{V}_{\ph})_{l,\nu} \, = \, \begin{cases} 0 & \text{if $\nu \in (-7,0)$,} \\
\sum_{\lambda \in (\nu,-7]} \dim \mathcal{K}(\lambda+1)_{\mdiracc} - \sum_{\lambda \in (\nu,-7]}
\dim \mathcal{K}(\lambda+1)_{\lapc} & \text{if $\nu<-7$,}
\end{cases}
\end{equation*}
where $\lapc$ acts on functions on $C$.
\item In the CS case, for $\nu > 0$ sufficiently close to zero,
\begin{equation*}
\dim (\mathcal{V}_{\ph})_{l,\nu} \, = \, 1 + \sum_{i=1}^n \sum_{\lambda \in (-1,0]} \dim \mathcal{K}(\lambda+1)_{\mdiracci},
\end{equation*}
where $\mdiracci$ is the modified Dirac operator on $C_i$.
\end{enumerate}
\end{prop}
\begin{proof}
Recall from Proposition~\ref{Vspaceprop} that $(\mathcal{V}_{\ph})_{l,\nu}$ is a finite-dimensional subspace of $(\Omega^3_1 \oplus \Omega^3_7)_{l, \nu}$ whose elements $\beta$ in particular satisfy $\pi_1 (d \beta) = 0$ for all $\beta \in (\mathcal{V}_{\ph})_{l,\nu}$. Moreover, from the definition of the space $(\mathcal{V}_{\ph}')_{l,\nu}$ in~\eqref{mdiracHodgeeq1} and~\eqref{mdiracHodgeeq2}, it is clear that $(\mathcal{V}_{\ph})_{l,\nu}$ can be taken to be a subspace of $(\mathcal{V}_{\ph}')_{l,\nu}$, which is itself isomorphic to the cokernel of the modified Dirac operator $\mdirac_{l+1, \nu+1}$ acting on $(\Omega^0 \oplus \Omega^1)_{l+1, \nu+1}$. We can identify $\mdirac$ with the usual Dirac operator $\dirac_{l+1,\nu+1}$ given by $\dirac(f,X) = (d^*X, df + \curl X)$, and thus the cokernel of $\mdirac$ with $\ker \dirac_{-7 -\nu}$.

The condition $\pi_1 (d \beta) = 0$ for $\beta \in (\Omega^3_1 \oplus \Omega^3_7)_{l,\nu}$ with $\beta = h \ph + f \st (Y \wedge \ph)$ corresponds to $d^*Y = 0$ 
by Proposition~\ref{special3formprop}. We observe that the coclosed $1$-forms in $\Omega^1_{l,\nu}$ are dual to the exact $1$-forms in $\Omega^1_{-7 -\nu}$. Hence, $(\mathcal{V}_{\ph})_{l,\nu}$ is isomorphic to a direct complement of $(\ker \dirac_{-7-\nu}) \cap d (\Omega^0_{-6 -\nu})$ in $\ker \dirac_{-7 -\nu}$. Since the curl of an exact $1$-form vanishes, we note that $(\ker \dirac_{-7 -\nu}) \cap d (\Omega^0_{-6 -\nu}) \cong d (\ker \Delta_{-6 -\nu})$ where $\Delta$ is the scalar Laplacian. Thus, we have established that
\begin{equation} \label{dimVphtempeq}
\dim (\mathcal{V}_{\ph})_{l,\nu} \, = \, \dim (\ker \dirac_{-7 -\nu}) - \dim \bigl( d (\ker \Delta_{-6 -\nu}) \bigr).
\end{equation}
Consider first the AC case. We already established in Proposition~\ref{Vspaceprop} that $(\mathcal{V}_{\ph})_{l,\nu} = \{0\}$ for $\nu > -7$. Moreover, by Lemma~\ref{mdiracHodgelemma}, the operator $\mdirac_{l+1, \nu+1}$ is injective for $\nu < -1$, so any changes in the index of $\mdirac$ below this rate must add to the cokernel, which is isomorphic to $\ker (\mdirac)_{- 7 - \nu}$. We need to determine which changes in fact add to $(\mathcal{V}_{\ph})_{l,\nu}$. Hence, by~\eqref{dimVphtempeq}, the amount by which the dimension of $(\mathcal{V}_{\ph})_{l,\nu}$ will change as we cross the rate $\lambda$ is thus $\dim \mathcal{K}(\lambda+1)_{\mdiracc}$ minus the change in dimension of $d (\ker \Delta_{-6 -\lambda})$. We claim that this change in dimension is exactly $\dim \mathcal{K}(\lambda+1)_{\lapc}$. To see this, first note by Lemma~\ref{bochnerlemma}, the scalar Laplacian $\Delta_{\lambda + 1}$ is injective for $\lambda + 1 < 0$, since a constant function that is  $O(\varrho^{\lambda +1})$ must vanish if $\lambda + 1 < 0$. Thus any change in the index of $\Delta_{\lambda + 1}$ at rate $\lambda$ adds to the cokernel, which has dimension $\dim (\ker \Delta_{-6 - \lambda})$. Since $\lambda \leq -7$ now, we have $-6 - \lambda \geq 1$, and the constant functions appear in $\ker \Delta_{- 6 - \lambda}$ at rate $0$. Thus it is only the addition of nonconstant functions to $\ker \Delta_{-6 - \lambda}$ as we cross the rate $\lambda$ that can occur. But $d$ is injective on nonconstant functions in $\ker \Delta_{-6 -\nu}$, so the change in the dimension of $d (\ker \Delta_{-6 -\lambda})$ is precisely $\dim \mathcal{K}(\lambda+1)_{\lapc}$. This establishes the result for (a).

Next, consider the CS case. Now Lemma~\ref{mdiracHodgelemma} says that $\mdirac_{l+1, \lambda+1}$ has a $1$-dimensional kernel and cokernel if $\lambda \in (-7,-1)$ and thus in particular its index in this range is zero. Proposition~\ref{modifieddiracexcludeprop} shows that the index change at $\lambda = -1$ corresponds precisely to constant functions on the cones $C_i$, and thus the total index change at  rate $\lambda = -1$ is $n$, the number of singularities. However, we know that $\mdirac_{l+1, \lambda+1}$ is injective for $\lambda > -1$ by Lemma~\ref{mdiracHodgelemma}, and since $\dim \ker \mdirac_{\lambda + 1} = 1$ for $\lambda \in (-7,-1)$, the change in the kernel of the Dirac operator must be $1$ as we cross $\lambda = -1$. Hence, the change in the cokernel of $\mdirac$ as we cross $\lambda = -1$ must be $n - 1$ and thus $\dim \coker \mdirac_{\lambda + 1} = \dim \ker \mdirac_{- 7 - \lambda} = n$ for $\lambda = -1 + \e$ for some $\e$ sufficiently small.

Since $\mdirac_{l+1, \lambda+1}$ is injective for $\lambda > -1$, any further changes in the index for $\lambda \in (-1,0]$ must all add to the cokernel. Thus, for $\nu > 0$ and sufficiently close to $0$ we have that
\begin{equation} \label{dimVphtempeq2}
\dim (\ker \mdirac_{-7-\nu}) \, = \, n + \sum_{i=1}^n \sum_{\lambda \in (-1,0]} \dim \mathcal{K}(\lambda+1)_{\mdiracci}.
\end{equation}
We still need to consider the dimension of $d (\ker \Delta_{-6 -\lambda})$ for rates $\lambda \in (-1,0]$, which corresponds to $-6 - \lambda \in [-6,-5)$. Proposition~\ref{excludeextensionprop} shows that there are no changes in the index of the scalar Laplacian in this range of rates and thus no changes to $d (\ker \Delta_{-6 -\lambda})$. The same proposition also says that the change of the index of $\Delta$ at rate $-5$ is precisely $n$, corresponding to a single harmonic function on each cone $C_i$. Lemma~\ref{bochnerlemma} shows that $\ker \Delta_{-6 -\lambda}$ consists of constant functions if $-6 - \lambda > -5$, and thus $d (\ker \Delta_{-6 -\lambda}) = \{0\}$ in the range $\lambda < -1$. Moreover, Lemma~\ref{bochnerlemma} shows that $\Delta_{\lambda+1}$ is injective if $\lambda + 1 > 0$, which is $-6 - \lambda < -5$, and it has a $1$-dimensional kernel if $\lambda + 1 \in (-5, 0]$, which is $-6 - \lambda \in [-5, 0)$. Thus at $\lambda = -1$, which is $-6 - \lambda = -5$, the dimension of $\ker \Delta_{\lambda+1}$ must change by $1$ and the dimension of $\ker \Delta_{-6-\lambda}$ must therefore change by $n-1$. Note that the new elements that appear in $\ker \Delta_{-6-\lambda}$ are nonconstant functions, on which the map $d$ is injective. Hence
\begin{equation} \label{dimVphtempeq3}
\dim \bigl( d ( \ker \Delta_{-6-\lambda}) \bigr) \, = \, \dim (\ker \Delta_{-6-\lambda}) \, = \, n-1 \qquad \text{for all $\lambda \in (-1,0]$}.
\end{equation}
Combining equations~\eqref{dimVphtempeq},~\eqref{dimVphtempeq2}, and~\eqref{dimVphtempeq3} completes the proof in the CS case.
\end{proof}

We can now combine our results to deduce the following dimension formula.
\begin{cor} \label{vdimfinalcor}
For generic rates $\nu$, the virtual dimension $\vdim {\mathcal M}_{\nu}$ of the  moduli space ${\mathcal M}_{\nu}$  is as follows.
\begin{itemize}
\item In the asymptotically conical (AC) case, we have
\begin{align*}
\vdim {\mathcal M}_{\nu} \, & = \, b^3_{\mathrm{cs}} - \dim (\im \Upsilon^4) - \sum_{\lambda \in  (\nu , -4)} \dim \mathcal K(\lambda)_{\ddc} & & \\
& \qquad {} + \sum_{\lambda \in (\nu, -7]} \dim \mathcal{K}(\lambda+1)_{\mdiracc} - \sum_{\lambda \in (\nu,-7]} \dim \mathcal{K}(\lambda+1)_{\lapc^0}, & & \nu < -7; \\
\vdim {\mathcal M}_{\nu} \, & = \, b^3_{\mathrm{cs}} - \dim (\im \Upsilon^4) - \sum_{\lambda \in  (\nu , -4)} \dim \mathcal K(\lambda)_{\ddc}, & & \nu \in(-7, -4); \\
\dim \mathcal M_{\nu} \, & = \, b^3_{\mathrm{cs}}, & & \nu \in (-4, -3); \\
\dim \mathcal M_{\nu} \, & = \, b^3_{\mathrm{cs}} + \dim (\im \Upsilon^3) + \sum_{\lambda \in  (-3 , \nu)} \dim \mathcal K(\lambda)_{\ddc} & & \\
& \qquad {} - \sum_{\lambda \in (-3,\nu)} \dim \mathcal{K}(\lambda+1)_{\lapc^1} & & \nu \in (-3, 0), 
\end{align*}
where $\lapc^k$ acts on $k$-forms on $C$.

Note that when $\nu \in (-4, 0)$, the deformation problem is unobstructed, the moduli space $\mathcal M_{\nu}$ is a smooth manifold, and the virtual dimension \emph{is} the actual dimension of $\mathcal M_{\nu}$.

\item In the conically singular (CS) case, we have for $\nu > 0$ sufficiently close to $0$ that
\begin{equation} \label{CSvdimeq}
\begin{aligned}
\vdim {\mathcal M}_{\nu} \, & = \, \dim (\im (H_{\mathrm{cs}}^3 \to H^3 )) - \sum_{i=1}^n \, \, \sum_{\lambda \in \mathcal D_{\ddci} \cap (-3 , 0]} \dim \mathcal K(\lambda)_{\ddci} \\
& \qquad {} + 1 + \sum_{i=1}^n \sum_{\lambda \in (-1,0]} \dim \mathcal{K}(\lambda+1)_{\mdiracci},
\end{aligned}
\end{equation}
where $\mdiracci$ is the modified Dirac operator on $C_i$.
\end{itemize}
Note that in general there are both \emph{topological} and \emph{analytic} contributions to the virtual dimension.
\end{cor}
\begin{proof}
We calculate the index of $\dd_{l,\nu}$ from Proposition~\ref{Pindexexactprop} and equation~\eqref{virtualdimeq2}, along with Theorem~\ref{Pindexchangethm}. In the CS case we also use equation~\eqref{imcseq}. We then use the formula for $\dim (\mathcal{V}_{\ph})_{l, \nu}$ in Proposition~\ref{V.dim.prop} to give the result except in the AC case for rates $\nu \in [-\frac{5}{2}, 0)$.  We can deal with this case using Corollary~\ref{finalmodulicor} which shows that the dimension is $\dim \mathcal{H}^3_{\nu} - \dim \mathcal{H}^1_{\nu+1}$, and the factor $\dim \mathcal{H}^1_{\nu+1}$ is equal to the sum of dimensions of the homogeneous closed and coclosed $1$-forms on $C$ of order $\lambda + 1 \in (0, \nu+1)$ by the proof of Corollary~\ref{finalmodulicor}. Since homogeneous harmonic $1$-forms on $C$ of order $\lambda + 1 \in (0, 1)$ are closed and coclosed by Proposition~\ref{excludeextensionprop}, the result follows.
\end{proof}

\begin{rmk} \label{ACoverlaprmk}
Note that in the AC case, the above proof gave one formula for $\dim \mathcal M_{\nu}$ when $\nu \in (-3, - \frac{5}{2})$ and another formula for $\nu \in [- \frac{5}{2}, 0)$. However, both of these formulas can be expressed by the single formula for $\dim \mathcal M_{\nu}$ valid for generic $\nu$ in the interval $(-3, 0)$ that is given in the statement of Corollary~\ref{vdimfinalcor}. This can be verified by applying~\eqref{top34dimensionseq}, Theorem~\ref{Pindexchangethm}, and Lemma~\ref{bochnerlemma}.
\end{rmk}

\begin{rmk} \label{dilationsrmk}
In the AC case, the deformation that corresponds to rescaling at infinity, which is generated by the dilation vector field, is always included in our actual moduli space $\mathcal M_{\nu}$. Note that in particular this observation, together with our dimension formula from Corollary~\ref{vdimfinalcor} implies topological restrictions on AC $\G$~manifolds.
\end{rmk}

\begin{ex} \label{BSexample}
Let us apply Theorem~\ref{mainthm} and Corollary~\ref{vdimfinalcor} to the three known examples of AC $\G$~manifolds, the Bryant--Salamon manifolds of Example~\ref{BSexamples}. The manifolds $\Lambda^2_-(S^4)$ and $\Lambda^2_-(\C\PR^2)$ are of rate $\nu = -4$ and the manifold $\spi(S^3)$ is of rate $\nu = -3$. Since $-4$ and $-3$ are excluded in Theorem~\ref{mainthm}, we can only describe their deformations as AC $\G$~manifolds of rate $\nu + \e$. For the first two examples, we find that the moduli space $\mathcal M_{-4 + \e}$ is a smooth manifold of dimension $b^3_{\mathrm{cs}} = b^4 = 1$. 

For $N = \spi(S^3)$, we find that the moduli space $\mathcal M_{-3 + \e}$ is a smooth manifold of dimension $b^3_{\mathrm{cs}} + \dim (\im \Upsilon^3) = b^4 + \dim (\im \Upsilon^3) = 0 + \dim(\im \Upsilon^3)$. A simple diagram chase in~\eqref{longexactsequenceeq} using the facts that $H^3(\Sigma) = \R^2$, $H^4_{\mathrm{cs}}(N) = \R$, and $H^4(N) = \{0\}$ gives $\dim (\im \Upsilon^3) = 1$.

Notice that the dimension of the moduli space has to be at least one because of dilations, as discussed in Remark~\ref{dilationsrmk} above. Therefore, the Bryant--Salamon manifolds are \emph{locally rigid} as AC $\G$~manifolds of rate $\nu + \e$, modulo the scalings which are always present. In Section~\ref{ACextensionsec}, we show how to extend this result to establish that the Bryant--Salamon manifolds are in fact locally rigid as AC $\G$~manifolds of rate $-\e$ for any small $\e > 0$. Moreover, in Section~\ref{cohomogeneitysec} we push this still further to prove that the Bryant--Salamon manifolds are in fact \emph{globally rigid}.
\end{ex}

In the remainder of this section we make some brief remarks about interpreting these dimension formulas in the CS case.

Consider the $i^{\text{th}}$ end of a CS $\G$~manifold $M$. Let $Y_i$ be a Killing field on $\Sigma_i$, which we could then view as a vector field $Y_i$ on $C_i$ of rate $1$ so that by Proposition~\ref{NK-Killing-prop} we have $\mathcal{L}_{Y_i}\phci = d( Y_i \hk \phci) = 0$. Let $X_i$ be a smooth vector field on $M$ which is equal to the pullback of $Y_i$ on the $i^{\text{th}}$ end of $M$ and is $0$ on the other ends of $M$. Then $\exp( X_i ) \in \mathcal D_1$ is a diffeomorphism of rate $1$ on the ends which is the identity on all of the ends except the $i^{\text{th}}$ one. Now consider $d ( X_i \hk \ph)$. Using the fact that $d(Y_i \hk \phci) = 0$, since $X_i$ is asymptotic to $Y_i$ and $\ph$ is asymptotic to $\phci$ we deduce that $d ( X_i \hk \ph)$ decays on the $i^{\text{th}}$ end with rate $\nu$ and vanishes on the other ends.

We reiterate here that, since $X_i$ is of rate $1$ on the ends, $d(X_i \hk \ph)$ should \emph{a priori} be of rate $0$ on the ends, but in fact it has \emph{faster} decay, being of rate $\nu > 0$ on the ends. Thus $\xi_i = d( X_i \hk \ph) \in \mathcal C_{l, \nu}$ but $\pi_7d^*\xi_i\neq 0$: in fact, $\pi_7d^*$ is injective on the span of the $\xi_i$. We can see this because if $\pi_7 d^* d \xi_i = 0$ then, since $X_i \hk \ph \in (\Omega^2_7)_{l+1,1}$ and  $\pi_7 d^* d$ is in injective on $(\Omega^2_7)_{l+1, \lambda+1}$ for $\lambda \geq -6$ by Lemma~\ref{modifiedlaplemma}, we deduce that $X_i = 0$, a contradiction. Therefore, the $3$-forms $\xi_i$ correspond to linearly independent elements in $(\mathcal E_{\ph})_{l, \nu}$ and thus in $(\mathcal{V}_{\ph})_{l,\nu}$. Each such additional infinitesimal deformation corresponds to a reparametrization of the conical model for the singularity, given by an automorphism of the nearly K\"ahler structure on the link $\Sigma_i$. This fact matches well with the dimension formula for $(\mathcal{V}_{\ph})_{l,\nu}$ in Proposition~\ref{V.dim.prop}(b), since there we see that the space of Killing fields on each $\Sigma_i$ is a subspace of $\mathcal{K}(1)_{\mdiracci}$ and thus contributes to $\dim \mathcal{K}(1)_{\mdiracci}$.

\section{Applications and open problems} \label{applicationssec}

In this section we present several applications of our results, and discuss some open problems. In Section~\ref{grayspectrumsec} we discuss some aspects of the spectrum of the Laplacian on Gray manifolds. These are used in the next three sections. In Section~\ref{ACextensionsec} we first derive an alternative form of the dimension formula for the AC moduli space for generic rates $\nu \in (-3, 0)$, and use this to establish the local rigidity (modulo trivial scalings) of the Bryant--Salamon manifolds as AC $\G$~manifolds of rate $\nu < 0$. In Section~\ref{cohomogeneitysec} we establish that under certain conditions, an AC~$\G$~manifold must be of cohomogeneity one, and this implies that the Bryant--Salamon manifolds are unique as AC~$\G$~manifolds with given asymptotic cone. In Section~\ref{smoothCSsec} we investigate when the CS moduli space is smooth and unobstructed. In Section~\ref{resolvedmodulisec} we relate our main theorem to the desingularization theorem of~\cite{Kdesings}, providing evidence that CS $\G$~manifolds likely make up a large part of the ``boundary'' of the moduli space of smooth compact $\G$~manifolds. In Section~\ref{gaugefixexistssec} we show that a gauge-fixing condition that is needed in~\cite{Kdesings} can always be achieved. Finally we conclude in Section~\ref{openproblemssec} with some open problems for the future.

\subsection{The spectrum of the Laplacian on Gray manifolds} \label{grayspectrumsec}

We first need to discuss some results about the spectrum of the Laplacian on $2$-forms, for compact Gray manifolds, that is, for $6$-dimensional strictly nearly K\"ahler manifolds. These results are required for the applications in Sections~\ref{ACextensionsec},~\ref{cohomogeneitysec}, and~\ref{smoothCSsec}. The current section has a different flavour from the rest of the paper. Readers who are only interested in the use of these results for the applications can just note Proposition~\ref{grayspectrumprop1} and Proposition~\ref{grayspectrumfinalprop}.

Slightly similar calculations can be found (at least implicitly if not explicitly) in \cite{MS}. Note that in~\cite{MS}, the form $\Psi^+ + i \Psi^-$ corresponds to our $-\Omega = - \real(\Omega) - i \imag(\Omega)$.

Let $(\Sigma, J , \omega, \Omega)$ be a Gray manifold, so in particular from~\eqref{grayeq} we have
\begin{equation} \label{grayagaineq}
d \omega \, = \, -3 \, \real(\Omega) \quad \text{and} \quad d \imag(\Omega) \, =\, 2 \, \omega^2.
\end{equation}
From Lemma~\ref{KPspacelemma} we find that the closed and coclosed $3$-forms $\gamma = r^{\lambda} (r^3 \alpha_3 + r^2 dr \wedge \alpha_2)$ on the cone, homogeneous of order $\lambda$, correspond to \emph{coclosed} (in fact coexact, if $\lambda \neq -4$) $2$-forms $\alpha_2$ on the link that satisfy $\laps \alpha_2 = (\lambda + 3)(\lambda + 4) \alpha_2$. From Corollary~\ref{vdimfinalcor}, we are particularly interested in such forms for $\lambda \in (-3, 0]$, in which case $0 < (\lambda + 3)(\lambda + 4) \leq 12$, with equality if and only if $\lambda = 0$. This motivates us to study on $\Sigma$ the pair of equations
\begin{equation} \label{maingrayeq}
\laps \xi = (\lambda+3)(\lambda+4) \xi \quad \text{and} \quad \dss \xi = 0
\end{equation}
for $\lambda \in (-3,0]$, thus $\mu = (\lambda+3)(\lambda+4) \in (0,12]$, and $\xi$ a section of $\Lambda^2 (T^* \Sigma)$.

We can decompose the bundle of real $2$-forms on $\Sigma$ into $\SUth$ representations as follows:
\begin{equation*}
\Lambda^2 (T^* \Sigma) \, = \, \Lambda^{(2,0)+(0,2)} (T^*\Sigma)\oplus \R \langle \omega\rangle \oplus \Lambda^{(1,1)}_0 (T^* \Sigma) ,
\end{equation*}
where
\begin{equation*}
T \Sigma \, \cong \, \Lambda^{(2,0)+(0,2)} T^* \Sigma \quad \text{via} \quad Y \mapsto Y \hk \real (\Omega).
\end{equation*}
We may then write any section $\xi$ of  $\Lambda^2 T^* \Sigma$ as 
\begin{equation} \label{grayeq1}
\xi \, = \, Y \hk \real (\Omega) + f \omega + \gamma
\end{equation}
for a vector field $Y$, a smooth function $f$, and $\gamma$ a section of $\Lambda^{(1,1)}_0 T^* \Sigma$. 

Following~\cite{MS}, define the Hermitian connection $\bar{\nabla}$ by $\bar{\nabla}_X = \nabla_X - \frac{1}{2} A_X$, where $A_X = J(\nabla_X J)$, and define $\bar{\Delta}$, the \emph{Hermitian Laplace operator}, explicitly via the formula $\bar{\Delta} = \bar{\nabla}^* \bar{\nabla} + q(\bar{R})$ where $\bar R$ is the curvature tensor of $\bar \nabla$, and $q(\bar R)$ is the associated curvature operator.  The fundamental formula \cite[equation (17)]{MS} states that for a section $\gamma$ of $\Lambda^{(1,1)}_0 T^* \Sigma$, we have
\begin{equation} \label{fundformulaeq}
(\laps - \bar{\Delta}) \gamma \, = \, -(J \dss \gamma) \hk \real(\Omega).
\end{equation}
Hence, if $\xi$ in~\eqref{grayeq1} has $f = 0$ and $Y = 0$ then $\xi = \gamma \in C^{\infty} (\Lambda^{(1,1)}_0 T^* \Sigma)$ satisfies~\eqref{maingrayeq} if and only if 
\begin{equation} \label{gamma.eq}
\bar{\Delta}\gamma = (\lambda+3)(\lambda+4) \gamma \quad \text{and} \quad \dss \gamma = 0.
\end{equation}
This equation will play a key role, so we define 
\begin{equation} \label{mult.gamma.eq}
m_{\Sigma}(\lambda) \, = \, \dim \{ \gamma \in C^{\infty}(\Lambda^{(1,1)}_0 T^*\Sigma )\, ; \, \gamma \text{ satisfies~\eqref{gamma.eq}} \}.
\end{equation}
We can now refine our description of the homogeneous closed and coclosed $3$-forms on a $\G$~cone.

\begin{prop} \label{grayspectrumprop1}
Let $\eta$ be a closed and coclosed $3$-form on the cone $C$, homogeneous of order $\lambda \in (-3,0]$. Then $\eta$ is of the form $\eta = \dc \beta$, where $\beta = r^{\lambda+3} \xi$, for some $\xi \in C^{\infty}(\Lambda^2T^*\Sigma)$ satisfying~\eqref{maingrayeq} with $\mu = (\lambda+3)(\lambda+4)$. Furthermore, $\beta$ is homogeneous of order $\lambda + 1$ with $\lapc \pi_7(\beta) = 0$ and $\dsc \pi_7 (\beta) = \dsc \pi_{14}(\beta) = 0$.

Finally, the homogeneous closed and coclosed $3$-forms $\eta$ on $C$ of order $\lambda$ are precisely those of the following form:
\begin{enumerate}[(i)]
\item When $\lambda \in (-3, -1]$, then $\eta = \dc (r^{\lambda + 3} \xi)$, where $\xi = \gamma \in C^{\infty}(\Lambda^{(1,1)}_0T^*\Sigma)$ and $\gamma$ satisfies~\eqref{gamma.eq}.
\item When $\lambda \in (-1, 0)$, then
\begin{equation*}
\eta \, = \, \dc \big( \stc (\dc (r^{\lambda+2} f) \wedge \psc) + r^{\lambda+3} \gamma \big),
\end{equation*}
where $(f, \gamma) \in C^{\infty}(\Sigma) \oplus C^{\infty}(\Lambda^{(1,1)}_0 T^*\Sigma)$ with $\laps f = (\lambda+2)(\lambda+7) f$ and $\gamma$ satisfies~\eqref{gamma.eq}.
\item When $\lambda = 0$, then
\begin{equation*}
\eta \, = \, K \phc + \dc \bigl( \stc (\dc (r^2 f) \wedge \psc) + r^3 \gamma \bigr)
\end{equation*}
for some constant $K$, some $f \in C^{\infty}(\Sigma)$ satisfying $\laps f =14 f$, and $\gamma \in C^{\infty}(\Lambda^{(1,1)}_0 T^* \Sigma)$ satisfies~\eqref{gamma.eq} with $\lambda = 0$.
\end{enumerate}
\end{prop}
\begin{rmk} \label{grayspectrumprop1rmk}
Notice that the solutions of~\eqref{gamma.eq} with $\lambda = 0$ describe infinitesimal deformations of the nearly K\"ahler structure on $\Sigma$, modulo scaling, by the work in~\cite{MNS}.
\end{rmk}
\begin{proof}[Proof of Proposition~\ref{grayspectrumprop1}]
Let $\eta$ be a closed and coclosed $3$-form on the cone, homogeneous of order $\lambda \in (-3, 0]$. Then by~\eqref{homoddseq}, we can write
\begin{equation} \label{grayspectrumtempeq0}
\eta = r^{\lambda} (r^3 \ds \xi + r^2 dr \wedge (\lambda + 3) \xi) = \dc (r^{\lambda+3} \xi)
\end{equation}
for a unique $2$-form $\xi$ on $\Sigma$ satisfying~\eqref{maingrayeq}. Let $\beta = r^{\lambda+3} \xi$ so that $\eta = \dc \beta$. Using~\eqref{homoddseq} again, we have $\dsc \beta = r^{\lambda+1} \dss \xi = 0$. Since $\eta = \dc \beta$ and $\dsc \eta = 0$, we deduce that 
\begin{equation*}
\lapc \beta \, = \, \dsc \dc \beta + \dc \dsc \beta \, = \, \dsc \dc \beta \, = \, 0,
\end{equation*}
so $\beta = r^{\lambda+3} \xi$ is a homogeneous harmonic $2$-form of order $\lambda+1$. In particular, by Remark~\ref{laplaciansrmk}, since $\phc$ is torsion-free, $\beta_7 = \pi_7 (\beta)$ satisfies $\lapc \beta_7 = 0$, and thus $\beta_7 = \stc (X \wedge \psc)$ where $X$ is a homogeneous harmonic $1$-form of order $\lambda+1$. Combining Propositions~\ref{excludeextensionprop},~\ref{excludeextensionrefinedprop}, and~\ref{modifieddiracexcludeprop} shows that for $\lambda \in (-3,0]$ any homogeneous harmonic $1$-form $X$ of order $\lambda + 1$ satisfies $\curl_{\scriptstyle{C}} (X) = 0$. Since $\curl_{\scriptstyle{C}}(X) = \stc \bigl( (dX) \wedge \psc \bigr)$, we deduce that $\dsc \st (X \wedge \psc) = 0$. This means that $\dsc \beta_7 = 0$ and thus $\dsc \beta_{14} = 0$ as well since $\dsc \beta = 0$.

Now write $\xi$ in the form~\eqref{grayeq1}. A routine computation gives
\begin{equation} \label{grayspectrumtempeq1}
X \, = \, r^{\lambda + 1} \Bigl( - f dr + \frac{2}{3}r Y \Bigr).
\end{equation}
That is, for any $2$-form $\beta$ on $C$ of the form $\beta = r^{\lambda + 3} \xi$, with $\xi$ of the form~\eqref{grayeq1}, we have that the component $\pi_7 (\beta) = \beta_7$ is independent of $\gamma$, it depends only on $f$ and $Y$.

Therefore, if we let $X$ be given by~\eqref{grayspectrumtempeq1} and define $\beta' = r^{\lambda+3} (Y \hk \real(\Omega) + f \omega)$, we deduce that $\beta_7' = \beta_7$, and thus $\beta'_7$ is coclosed. Moreover, $\dsc \dc \beta'_7 = \lapc \beta'_7 = 0$, and thus $\dc \beta'_7$ is a homogeneous closed and coclosed $3$-form of order $\lambda$. Hence, every harmonic $1$-form $X$ of rate $\lambda + 1 \in (-2, 1]$ determines a homogeneous closed and coclosed $3$-form of rate $\lambda \in (-3, 0]$, namely $\dc \beta'_7$.

Next, we observe that $\dc \beta_7' = \dc \bigl( \stc (X \wedge \psc) \bigr) = \dc (X \hk \phc) = \mathcal L_{X} \phc$, where the vector field $X$ metric dual to the $1$-form $X$ with respect to the cone metric is $X = - r^{\lambda + 1} f \, \ddr + \frac{2}{3} r^{\lambda} Y$. We can thus use Proposition~\ref{Killing-cone-prop} to conclude that $\dc \beta'_7 = 0$ if and only if $\lambda = 0$, $f = 0$, and $Y$ is Killing.

Let $\xi_j = Y \hk \real(\Omega) + f \omega + \gamma_j$ be solutions of~\eqref{maingrayeq} for $j=1, 2$. Then $\gamma = \gamma_1 - \gamma_2$ is a section of $\Lambda^{(1,1)}_0 T^*\Sigma$ which satisfies $\dss \gamma = 0$ and $\laps \gamma = \mu \gamma$ where $\mu = (\lambda+3)(\lambda+4)$. Hence by~\eqref{fundformulaeq}, we find that $\gamma$ satisfies~\eqref{gamma.eq}. Moreover, $\dc (r^{\lambda+3} \gamma)$ is a closed and coclosed $3$-form if $\gamma$ satisfies~\eqref{gamma.eq}.

We have $\eta = \dc (X \hk \phc) + \dc (r^{\lambda + 3} \gamma)$. The statements \emph{(i)}, \emph{(ii)}, and \emph{(iii)}, now follow easily from the description of the homogeneous harmonic $1$-forms of rate $\lambda + 1\in (-2,1]$ in Propositions~\ref{excludeextensionprop} and~\ref{excludeextensionrefinedprop}. The only point to note is that in the case $\lambda + 1 = 1$ we have $\dc *(r dr \wedge \psc) = 3 \phc$.
\end{proof}

It is therefore clear that we should study solutions to~\eqref{gamma.eq}. Specifically, we observe that the calculations in \cite{MS} enable us to determine the eigenvalues (and their multiplicities) in $(0,12]$ for $\bar{\Delta}$ acting on coclosed forms in $\Lambda^{(1,1)}_0 (T^*\Sigma)$ for the three homogeneous Gray manifolds $\Sigma = \C \PR^3$, $S^3 \times S^3$, and $\SUth / T^2$.

\begin{prop} \label{grayspectrumfinalprop}
Let $\Sigma$ be one of the three homogeneous Gray manifolds: $\C \PR^3$, $S^3 \times S^3$, or $\SUth / T^2$. There are no nontrivial coclosed primitive $(1,1)$-forms which are eigenforms of the Hermitian Laplace operator $\bar \Delta$ with eigenvalue in $(0,12)$. Moreover, in the first two cases, we can also exclude the eigenvalue $12$. For the flag manifold $\SUth / T^2$, we get an $8$-dimensional space of such forms with eigenvalue $12$. These forms correspond to infinitesimal deformations of the nearly K\"ahler structure, but it was recently shown by Foscolo~\cite{F} that they \emph{cannot} be integrated to actual deformations.
\end{prop}
\begin{proof}
{\em Case 1:} Let us start with the case of $\C\PR^3$, which follows from the work in~\cite[$\S$5.5]{MS}. Let $E$ denote the usual representation of $\SUtw$ on $\C^2$ and let $\C_l$ be the $\mathrm{U}(1)$ representation on $\C$ given by multiplication by $z^l$ for $z$ in the unit circle in $\C$.  Let $E_{k,l} = \mathrm{Sym}^k (E) \otimes \C_l$ for $k \geq 1$ and $l\in\Z$, and $k \equiv l$ mod 2, which are the irreducible representations of $\mathrm{U}(2)$. From~\cite[Lemma 5.8]{MS} and the discussion before it, we have decompositions
\begin{equation*}
T^* \C \PR^3 \, \cong \, E_{0,-2} \oplus E_{1,1} \oplus E_{0,2} \oplus E_{1,-1} \quad \text{and} \quad \Lambda^{(1,1)}_0 (T^* \C \PR^3) \, \cong \, E_{0,0} \oplus E_{1,3} \oplus E_{1,-3} \oplus E_{2,0}.
\end{equation*} 
If $V_{a,b}$ is an irreducible $\mathrm{SO}(5)$ representation with highest weight $(a,b)$ for $a\geq b\geq 0$, it corresponds to a possible eigenspace of $\bar{\Delta}$ on $\Lambda^{(1,1)}_0 (T^* \C \PR^3)$ with eigenvalue $2 (a(a+3) + b(b+1))$ if there is a homomorphism from $V_{a,b}$ to $\Lambda^{(1,1)}_0 (T^* \C \PR^3)$. So the only possible positive eigenvalue less than $12$ is $8$ for $(a,b) = (1,0)$. However, since $V_{1,0} \cong T^* (\C \PR^3)$ we see that there are no such homomorphisms from $V_{1,0}$ to $\Lambda^{(1,1)}_0 (T^* \C \PR^3)$, and thus the lowest possible positive eigenvalue is $12$. Moreover, the multiplicity of the possible eigenvalue $12$ on coclosed forms in $\Lambda^{(1,1)}_0 (T^* \C \PR^3)$ is shown to be $0$ in~\cite[Theorem 5.10]{MS}, so there are no solutions of~\eqref{maingrayeq} for $\mu \in (0,12]$ on $\C \PR^3$ other than $\xi = c \omega$.

{\em Case 2:} Next we consider the case of $S^3 \times S^3$ as in~\cite[$\S$5.4]{MS}. If $E$ is as above, then by~\cite[Lemma 5.5]{MS},
\begin{equation*}
\Lambda^{(1,1)}_0 (T^* (S^3 \times S^3)) \, \cong \, \mathrm{Sym}^2 (E) \oplus \mathrm{Sym}^4 (E).
\end{equation*}
Then the irreducible representation of $\SUtw \times \SUtw \times \SUtw$ given by
\begin{equation*}
V_{a,b,c} \, = \, \mathrm{Sym}^a (E) \otimes \mathrm{Sym}^b (E) \otimes \mathrm{Sym}^c (E)
\end{equation*}
for $a,b,c \geq 0$ corresponds to a possible eigenspace of $\bar{\Delta}$ of eigenvalue $\frac{3}{2}(a(a+2) + b(b+2) + c(c+2))$ if there is a homomorphism from $V_{a,b,c}$ to $\Lambda^{(1,1)}_0 (T^*(S^3 \times S^3))$. The only possible positive eigenvalues less than $12$ are $\frac{9}{2}$ and $9$ for $(a,b,c) = (1,0,0)$ and for $(a,b,c) = (1,1,0)$, respectively, up to permutation.  There are no homomorphisms from $V_{1,0,0} = \mathrm{Sym}^1 (E)$ to $\mathrm{Sym}^2 (E) \oplus \mathrm{Sym}^4 (E)$, but there \emph{is} one from $V_{1,1,0} = \mathrm{Sym}^1 (E) \otimes \mathrm{Sym}^1 (E) = \mathrm{Sym}^0 (E) \oplus \mathrm{Sym}^2 (E)$ to $\mathrm{Sym}^2 (E) \oplus \mathrm{Sym}^4 (E)$.  
However, the fact that $\mathrm{Sym}^0 (E)$ is a factor in $V_{1,1,0}$ means that $V_{1,1,0}$ also corresponds to an eigenspace for $\Delta$ on functions with eigenvalue $9$, and such functions, by~\cite[Proposition 4.11]{MS}, define elements of $\Lambda^{(1,1)}_0 (T^* (S^3 \times S^3))$ that are eigenforms of $\bar{\Delta}$ with eigenvalue $12$ but which are \emph{not} coclosed, as the eigenvalue is not $6$. We hence deduce that $9$ does not arise as an eigenvalue for $\bar{\Delta}$ acting on coclosed forms in $\Lambda^{(1,1)}_0 (T^* (S^3 \times S^3))$. Moreover, the multiplicity of the possible eigenvalue $12$ on coclosed forms in $\Lambda^{(1,1)}_0 (T^*(S^3 \times S^3))$ is shown to be $0$ in~\cite[Corollary 5.7]{MS}, so again there are no solutions of~\eqref{maingrayeq} for $\mu \in (0,12]$ on $S^3 \times S^3$ other than $\xi = c \omega$.

{\em Case 3:} Finally, we consider the case of $\SUth / T^2$ as in~\cite[$\S$5.6]{MS}. Let $E$ denote the usual representation of $\SUth$ on $\C^3$ and let  
\begin{equation*}
V_{k,l} \, = \, \ker (\mathrm{Sym}^k (E) \otimes \mathrm{Sym}^l (\bar{E}) \rightarrow \mathrm{Sym}^{k-1} (E) \otimes \mathrm{Sym}^{l-1} (\bar{E})) ,\end{equation*}
where the map is the contraction map. If $\e_i$ for $i = 1,2,3$ is the standard basis on $\R^3$ then the weights of $V_{k,l}$ are 
\begin{equation*}
(a-a') \e_1 + (b-b') \e_2 + (c-c') \e_3
\end{equation*}
where $k = a + b + c$, $l = a' + b' + c'$ and $a, b, c, a', b', c' \geq 0$. The representation $V_{k,l}$ corresponds to a possible eigenvalue of $2( k(k + 2) + l(l+2))$ of $\bar{\Delta}$. So we need $(k,l) = (1,0)$ or $(k,l) = (0,1)$ for a positive eigenvalue less than $12$. Now by~\cite[Corollary 5.11]{MS} the possible weights of the $T^2$ representation on $\Lambda^{(1,1)}_0 (T^*(\SUth / T^2))$ are $0$ and $\pm 3 \e_i$ for $i = 1, 2, 3$. However, for eigenvalues less than $12$ we cannot achieve weights $\pm 3 \e_i$ since $k,l \in \{0,1\}$ and we cannot achieve weight $0$ since 
$k \neq l$. Thus there are no eigenvalues in $(0,12)$. Moreover, by the work in~\cite[$\S$5.6 \& $\S$6]{MS}, the space of solutions to~\eqref{maingrayeq} for $\mu \in (0,12]$ is zero \emph{unless} $\mu = 12$, in which case the solutions are of the form $\xi = c \omega + \gamma$, where the $\gamma$'s lie in a space of dimension $8$, isomorphic to $\mathfrak{su}(3)$.
\end{proof}

\subsection{An alternative dimension formula in the AC case} \label{ACextensionsec}

We can now use our refined description of the closed and coclosed $3$-forms on a $\G$~cone in Proposition~\ref{grayspectrumprop1} to give another description of the dimension of the AC moduli space for generic rates $\nu \in (-3,0)$.

\begin{prop}\label{AC.extension.prop}
Let $(M, \ph)$ be an AC $\G$~manifold of generic rate $\nu \in (-3,0)$. Then 
\begin{equation*}
\dim \mathcal{M}_{\nu} = b^3_{\text{\emph{cs}}}(M) + \dim (\im \Upsilon^3) + \sum_{\lambda\in (-3,\nu)} m_{\Sigma}(\lambda),
\end{equation*}
where $m_{\Sigma}(\lambda)$ is given in~\eqref{mult.gamma.eq}.
\end{prop}

\begin{proof}
We know by Corollary~\ref{vdimfinalcor} that 
\begin{equation*}
\dim \mathcal{M}_{\nu} \, = \, b^3_{\mathrm{cs}}(M) + \dim (\im \Upsilon^3) + \sum_{\lambda \in  (-3,  \nu)} \dim \mathcal K(\lambda)_{\ddc} - \sum_{\lambda \in (-3, \nu)} \dim \mathcal{K}(\lambda+1)_{\lapc^1}.
\end{equation*}
Recall that $\mathcal{K}(\lambda)_{\ddc}$ is the space of homogeneous closed and coclosed $3$-forms on $C$ of order $\lambda$ and that $\mathcal{K}(\lambda+1)_{\lapc^1}$ is the space of homogeneous harmonic $1$-forms on $C$ of order $\lambda+1$.

Propositions~\ref{excludeextensionprop} and~\ref{grayspectrumprop1} show that for $\lambda \in (-3,-1]$ we have 
\begin{equation*}
\dim \mathcal{K}(\lambda)_{\ddc} \, = \, m_{\Sigma}(\lambda) \quad \text{and} \quad \dim \mathcal{K}(\lambda+1)_{\lapc^1} \, = \, 0.
\end{equation*}
Moreover, combining Propositions~\ref{grayspectrumprop1} and~\ref{excludeextensionrefinedprop} shows that for $\lambda \in (-1, 0)$, which corresponds to $\lambda + 1 \in (0, 1)$, we have 
\begin{equation*}
\dim \mathcal{K}(\lambda)_{\ddc} \, = \, m_{\Sigma}(\lambda) + \dim \mathcal{K}(\lambda+1)_{\lapc^1}.
\end{equation*}
The result is now immediate.
\end{proof}

Proposition~\ref{AC.extension.prop} allows us to more effectively compute the dimension of the moduli space of AC $\G$ manifolds. In particular, we can show the following result for the Bryant--Salamon examples.

\begin{cor} \label{BSlocallyuniquetozerocor}
The Bryant--Salamon $\G$~conifolds are locally rigid, modulo scalings, as AC $\G$~manifolds with the same asymptotic cones, up to any rate $\nu < 0$.
\end{cor}
\begin{proof}
Proposition~\ref{grayspectrumfinalprop} says that for the three Bryant--Salamon manifolds, there are no solutions to~\eqref{gamma.eq} for any $\lambda \in (-3, 0)$, and so $m_{\Sigma}(\lambda) = 0$ for all $\lambda \in (-3,0)$ in these cases. The conclusion now follows for $\spi(S^3)$ by Example~\ref{BSexample} and Proposition~\ref{AC.extension.prop}. For the cases $\Lambda^2_-(S^4)$ and $\Lambda^2_-(\C\PR^2)$, we have to also use the fact that there are no new deformations as we cross $\lambda = -3$, arising from the term $\dim (\im \Upsilon^3)$ in Proposition~\ref{AC.extension.prop}. But this is immediate since $H^3(\Sigma) = \{0\}$ for both these manifolds, and thus $\Upsilon^3 = 0$.
\end{proof}

\subsection{Cohomogeneity of AC \texorpdfstring{$\mathbf{\G}$}{G2} manifolds} \label{cohomogeneitysec}

In this section we combine our deformation theory with the spectral theory results of Section~\ref{grayspectrumsec} to obtain a strong result about the cohomogeneity of AC $\G$~manifolds under certain conditions.

In the following, we choose $\e > 0$ sufficiently small so that $(-3, -3 + \e)$ contains no rates for homogeneous closed and coclosed $3$-forms on the given asymptotic cone.

\begin{prop} \label{moduli-immersionprop}
Let $(M, \ph)$ be an AC $\G$ manifold with rate $\nu = -3 + \e$. The map from the moduli space $\mathcal M_{\nu}$ to $H^3(M) \times H^4(M)$ given by
\begin{equation*}
\mathcal D_{\nu+1} \cdot \tilde{\ph} \, \mapsto \, ( [ \tilde{\ph} ] , [ \Theta(\tilde{\ph}) ] )
\end{equation*}
is an immersion.
\end{prop}
\begin{proof}
The map is well-defined because we are considering diffeomorphisms isotopic to the identity, so choosing different elements of the orbit will not change the cohomology classes. Hence it suffices to show that its derivative at the orbit of any point $\tilde \ph$ is injective. Since the argument is identical (modulo cumbersome notation) at any point, we show the case $\tilde \ph = \ph$. The tangent space to the moduli space at the orbit of $\ph$ is $\mathcal H^3_{\nu}$ hence the derivative is
\begin{equation*}
\eta \, \mapsto \, ( [ \eta ] , [ L_{\ph}(\eta) ] ) ,
\end{equation*}
which maps $\eta \in \mathcal H^3_{\nu}$ to $H^3(M) \times H^4(M)$.
 
Since $d \eta = d^* \eta = 0$ we know that $\Delta \eta = 0$. Hence $\pi_1( \Delta \eta ) =0 $ and $\pi_7( \Delta \eta ) = 0$ and thus $\pi_1 (\eta)$ and $\pi_7 (\eta)$ are harmonic by 
the torsion-freeness of $\ph$. Arguing exactly as in the proof of Theorem~\ref{onetoonethm}, we find that $\pi_1(\eta) = 0$ and $\pi_7(\eta) = 0$. Therefore
\begin{equation*}
L_{\ph}(\eta) = \stp \left( \frac{4}{3} \pi_1(\eta) + \pi_7(\eta) - \pi_{27}(\eta) \right) \, = \, \stp \left( \frac{7}{3} \pi_1(\eta) + 2 \pi_7(\eta) - \eta \right) \, = \, - \stp \eta.
\end{equation*}
Thus we need to show that the map
\begin{equation} \label{moduli-immersion-tempeq}
\eta \, \mapsto \, ( [\eta] , [- \stp \eta] )
\end{equation}
from $\mathcal H^3_{\nu}$ to $H^3(M) \times H^4(M)$ is injective. Suppose that $\eta$ maps to $(0,0)$ under this map. Since the only exceptional rate of $d + d^*$ in $[ -\frac{7}{2}, \nu]$ is $-3$ we know by Lemma~\ref{kernelchangelemma} that we can write $\eta = \eta_+ + \eta_-$ where $\eta_- \in \mathcal H^3_{-\frac{7}{2}}$ and  $\eta_+$ is asymptotic to a closed and coclosed $3$-form $\gamma$ on the asymptotic cone $C$ of rate $-3$, and $\eta_+$ is nonzero if and only if $\gamma$ is nonzero. Moreover, from equation~\eqref{homoddseq} we know that $\gamma$ is independent of the radial direction on $C$ as it is a closed and coclosed $3$-form on the link $\Sigma$ of $C$.

Now, $\mathcal H^3_{-\frac{7}{2}} = \mathcal H^3_{L^2}$ is the space of $L^2$ closed and coclosed $3$-forms and so by Proposition~\ref{LMcohomologyprop} is isomorphic to $H^3_{\mathrm{cs}}(M)$. Hence $[\eta_-]$ lies in the image of $H^3_{\mathrm{cs}}(M)$ in $H^3(M)$ under inclusion. Therefore, from the long exact sequence~\eqref{longexactsequenceeq}, we find that under the natural map $\Upsilon^3 : H^3(M) \rightarrow H^3(\Sigma)$ we have $\Upsilon^3 ([\eta_-]) = 0$. Furthermore, $\Upsilon^3 ([\eta_+]) = [\gamma]$, since $\eta_+$ is asymptotic to $\gamma$. As we are assuming that $[\eta] = 0$ we find that $0 = \Upsilon^3 ([\eta]) = \Upsilon^3 ([\eta_+] + [\eta_-]) = [\gamma]$ which implies that $\gamma = 0$.  We find therefore that $\eta_+ = 0$ and hence $\eta = \eta_-$.
  
Again from Proposition~\ref{LMcohomologyprop} we know that $\mathcal H^3_{-\frac{7}{2}}$ is isomorphic to $H^4(M)$, which means that $[ \stp \eta_-] = 0$ in $H^4(M)$ if and only if $\stp \eta_- = 0$. We conclude that $\eta = \eta_- = 0$ and thus the map~\eqref{moduli-immersion-tempeq} is indeed injective. 
\end{proof}

\begin{rmk} \label{moduli-immersionrmk}
A similar immersion (but with more structure) exists for the moduli space of smooth compact $\G$~manifolds. See~\cite[Theorem 10.4.5]{J4} for details.
\end{rmk}

We now apply Proposition~\ref{moduli-immersionprop} to ``lift'' an automorphism of the link to an automorphism of the AC~$\G$~manifold, as follows.

\begin{prop} \label{extend.aut.link.prop}
Let $(M, \phm)$ be an AC $\G$~manifold with rate $\nu = -3 + \e$, where $M$ has asymptotic cone $C$ with link $\Sigma$. Suppose that there are no solutions to~\eqref{gamma.eq} for any $\lambda \in (-3, \nu]$. Let $\Fs$ be a diffeomorphism of $\Sigma$ isotopic to and sufficiently close to the identity which preserves the nearly K\"ahler structure on $\Sigma$. Then there exists a diffeomorphism $\Fm$ of $M$ preserving $\phm$ which is asymptotic to $\Fs$ with rate $\nu + 1$.
\end{prop}

\begin{proof}
By Proposition~\ref{AC.extension.prop}, we have that $\mathcal M_{\nu}$ is smooth and equal to $\mathcal M_{-3 + \e}$, so we can apply Proposition~\ref{moduli-immersionprop} to $\mathcal M_{\nu}$. By hypothesis, we can write $\Fs = \exp(\Xs)$ for a Killing field $\Xs$ on $\Sigma$. The Killing field $\Xs$ naturally defines a Killing field $\Xc$ on the cone $C$. Define a smooth increasing cutoff function $\rho : (0, \infty) \rightarrow [0,1]$ such that $\rho(r) = 0$ for $r \in (0, R]$ and $\rho(r) = 1$ for $r \geq R+1$. Using the notation of Definition~\ref{ACdefn}, we can then define a vector field $X$ on $M$ such that $h_*(\rho \Xc) = X$ on $h \bigl( (R, \infty) \times \Sigma \bigr) = M \setminus L$ and which vanishes on the compact subset $L$ of $M$. Finally, we let $F = \exp(X)$.

Now, since $\phm$ is asymptotic to $\phc$ and $F$ is asymptotic to $\Fc = \exp(\Xc)$, we see that $F^* \phm$ is asymptotic to $\Fc^* \phc = \phc$ with rate $\nu$. Moreover, since $F$ is isotopic to the identity we see that $[F^*\phm] = [\phm]$ and $[\Theta(F^*\phm)] = [\Theta(\phm)]$. Using Proposition~\ref{moduli-immersionprop} we find that the orbits of $F^* \phm$ and $\phm$ under $\mathcal{D}_{\nu+1}$ in $\mathcal M_{\nu}$ are equal. Thus there exists $\widetilde F \in \mathcal D_{\nu+1}$ such that $\widetilde F^* (F^* \phm) = \phm$. Since $\widetilde F$ is asymptotic to the identity and $F$ is asymptotic to $\Fs$ we can set $F_M = F \circ \widetilde F$.
\end{proof}

\begin{rmk} \label{conlon.hein.rmk}
A similar but slightly weaker uniqueness result for AC~Calabi-Yau manifolds was obtained by Conlon--Hein in~\cite{CH}. They show that if a biholomorphism of an AC~Calabi-Yau manifold $M$ is asymptotic to an isometry of the cone, then it must be an isometry of $M$. 
\end{rmk}

\begin{cor} \label{cohom.one.cor}
Let $(M, \phm)$ be an AC~$\G$~manifold with rate $\nu < 0$ such that the link $\Sigma$ of the asymptotic cone is one of the three possible \emph{homogeneous} Gray manifolds, namely $\C \PR^3$, $\mathrm{SU}(3) / T^3$, or $S^3 \times S^3$. Then $(M, \phm)$ has cohomogeneity one.

Hence, the Bryant--Salamon $\G$~manifolds $\Lambda^2_- (S^4)$, $\Lambda^2_- (\C  \PR^2)$ and $\spi(S^3)$ are the unique AC $\G$ manifolds of rate $\nu < 0$ asymptotic to the cones on $\C \PR^3$, $\mathrm{SU}(3)/T^2$ and $S^3\times S^3$, respectively.
\end{cor}
\begin{proof}
It follows from Proposition~\ref{grayspectrumfinalprop} that there are no solutions to~\eqref{gamma.eq} for any $\lambda \in (-3, 0)$ for these links. Hence Proposition~\ref{extend.aut.link.prop} applies to $(M, \phm)$. Consequently, any automorphism of the homogeneous nearly K\"ahler manifold can be extended to an automorphism of the AC~$\G$~manifold, so $M$ must have cohomogeneity one.

The cohomogeneity one $\G$~manifolds where the action is by a simple group are classified in \cite[Theorem 9.3]{ClSw}. The uniqueness of $\Lambda^2_- (S^4)$ and $\Lambda^2_- (\C \PR^2)$ amongst such cohomogeneity one $\G$~manifolds follows.

If $M$ is asymptotic to the cone on the homogeneous Gray manifold $S^3\times S^3$, then we know that $M$ has a cohomogeneity one action of $\mathrm{SU}(2)^3$. It then follows from work in~\cite{Brandhuber} that the Bryant--Salamon manifold $\spi(S^3)$ is the unique smooth complete $\G$~manifold with such a cohomogeneity one action.
\end{proof}

\subsection{Smoothness of the CS moduli space for certain cones} \label{smoothCSsec}

In this section we establish that the CS moduli space $\mathcal{M}_{\nu}$ is in fact \emph{smooth} if the singularities are all modeled on $\G$~cones satisfying certain conditions. This includes two of the known $\G$~cones over homogeneous nearly K\"ahler manifolds, and may include the third as well.  We also show that it is more natural to consider a certain \emph{reduced} CS moduli space, due to the potential presence of Killing fields on the nearly K\"ahler links at the singularities, which fits well with Proposition~\ref{extend.aut.link.prop}.  

As mentioned in Remark~\ref{finalmodulitmk}, the moduli space $\mathcal{M}_{\nu}$ will be smooth if $\mathcal{O}_{/ \mathcal{W}} = \{0\}$. It is therefore useful to establish an upper bound for the dimension of $\mathcal{O}_{/ \mathcal{W}}$ in order to determine sufficient conditions under which it vanishes.

On the $i^{\text{th}}$ cone $C_i$, consider the system of equations
\begin{equation} \label{graymaineq3}
\bar \Delta \gamma = (\lambda + 3)(\lambda + 4) \gamma, \qquad \dssi \gamma = 0, \qquad \gamma \text{\, is of type $\Lambda^{(1,1)}_0$},
\end{equation}
and let $m_{\Sigma_i}(\lambda)$ be the dimension of the space of solutions to~\eqref{graymaineq3} for the cone $C_i$.

\begin{prop} \label{OW.dim.prop2}
Let $(M, \ph)$ be a CS $\G$~manifold with rate $\nu > 0$ near $0$. Then we have
\begin{equation*}
\dim \mathcal{O}_{/ \mathcal{W}} \, \leq \, \sum_{i=1}^n \sum_{\lambda \in (-3,0]} m_{\Sigma_i}(\lambda)+n-1.
\end{equation*}
\end{prop}
\begin{proof}
Throughout this proof, $\lapci^1$ denotes the Hodge Laplacian $\lapci$ on $C_i$ restricted to $1$-forms. From Remark~\ref{noobstructionrmk}, we have that $\mathcal{O}_{\lambda} = \{0\}$ for $\lambda<-3$.  Moreover, we observed in Remark~\ref{kernel.only.rmk} that the closed and coclosed forms of rate $-3$ on $C_i$ do not contribute to the cokernel of $\dd$, so we find that that $\mathcal{O}_{\lambda} = \{0\}$ for $\lambda < -3 + \e$ for $\e > 0$ sufficiently small.  Hence, Theorem~\ref{Pindexchangethm} implies we have a surjective map
\begin{equation} \label{surjective.coker.map.eq}
\begin{aligned}
\vartheta : \oplus_{i=1}^n \oplus_{\lambda \in (-3,0]} \mathcal{K}(\lambda)_{\dd_{C_i}} & \to \coker\dd_{l,\nu}, \\
(\gamma_1, \ldots, \gamma_n) & \mapsto \sum_{i=1}^n d^* (\chi_i \gamma_i) \; \text{modulo } \im \dd_{l,\nu},
\end{aligned}
\end{equation}
where $\gamma_i \in \oplus_{\lambda \in (-3,0]} \mathcal{K}(\lambda)_{\dd_{C_i}}$ and $\chi_i$ is a cut-off function which is $1$ on the $i^{\text{th}}$ end.  (In equation~\eqref{surjective.coker.map.eq} we have omitted pullbacks from the cones $C_i$ to the ends for simplicity and shall continue to do so throughout this proof.) We therefore also obtain a surjective map
\begin{equation} \label{surjective.coker.map.eq.2}
\vartheta : \oplus_{i=1}^n \oplus_{\lambda\in (-3,0]} \mathcal{K}(\lambda)_{\dd_{C_i}} \to \coker(\pi_{/\mathcal{W}} \circ \dd_{l,\nu}).
\end{equation}
We want to use the maps~\eqref{surjective.coker.map.eq} and~\eqref{surjective.coker.map.eq.2} to bound the dimension of $\mathcal{O}_{\nu}$ and $\mathcal{O}_{/\mathcal{W}}$, respectively.

Using Proposition~\ref{excludeextensionrefinedprop}, Theorem~\ref{Pindexchangethm}, and Proposition~\ref{grayspectrumprop1} we see from \eqref{surjective.coker.map.eq} that
\begin{equation} \label{smoothCStempeq}
\begin{aligned}
\dim \mathcal{O}_{\nu} \, & \leq \, \sum_{i=1}^n \sum_{\lambda \in (-3,0]} \dim \mathcal{K}(\lambda)_{\ddci} \\
 & = \, \sum_{i=1}^n \sum_{\lambda \in (-3,0]} m_{\Sigma_i}(\lambda) + \sum_{i=1}^n \sum_{\lambda \in (-1,0]} \dim \mathcal{K}(\lambda+1)_{\lapci^1}-\sum_{i=1}^n \dim(\mathrm{Kill} \,\Sigma_i).
\end{aligned}
\end{equation}
We deduce that forms in $\mathcal{O}_{\nu}$ correspond either to coclosed primitive $(1,1)$-forms on $\Sigma_i$ which are eigenforms for $\bar{\Delta}$ with eigenvalue in $(0, 12]$ or to homogeneous harmonic $1$-forms of order $\lambda + 1 \in (0,1]$ not arising from Killing fields on $\Sigma_i$.

Observing that there are no harmonic functions on $C_i$ of rate $\lambda + 1 \in (0,1]$ by Proposition~\ref{excludeextensionprop}, or by comparing the results of Propositions~\ref{excludeextensionrefinedprop} and~\ref{modifieddiracexcludeprop}, the space $\mathcal{K}(\lambda + 1)_{\mdiracci}$ is a subspace of $\mathcal{K}(\lambda + 1)_{\lapci^1}$. It therefore follows from Proposition~\ref{V.dim.prop} that there is a one-to-one correspondence between a subspace of the homogeneous harmonic $1$-forms on $C_i$ of order $\lambda + 1 \in (0,1]$ and a codimension 1 subspace of $(\mathcal{V}_{\ph})_{l,\nu}$. This correspondence can be understood as follows. If we are given, for $i=1,\ldots,n$, a $1$-form $X_i$ on $C_i$ of order $\lambda + 1$ in $\ker \mdiracci$, which is thus harmonic, we then have a $1$-form $X = \sum_{i=1}^n\chi_i X_i$ on $M$. If there is no element of $\ker \mdirac$ on $M$ asymptotic to $X$ on the ends, then $\pi_{1+7} d*( X \wedge \ps )$ defines an element  in $(\mathcal{V}_{\ph})_{l,\nu}$, in the sense that we can choose the codimension $1$ subspace of $(\mathcal{V}_{\ph})_{l,\nu}$ to consist of such elements. The remaining form generating $(\mathcal{V}_{\ph})_{l,\nu}$ can be taken to be $\pi_{1+7}\zeta$ where 
\begin{equation} \label{phm.obs.eq} 
\zeta = \ph_M - \sum_{i=1}^n \frac{1}{3} d (\chi_i \st ( rdr \wedge \psci ) ),
\end{equation}
since $\ph_M$ is asymptotic to $\ph_{C_i} = \frac{1}{3} d (\st ( rdr \wedge \psci ) )$ on each end.

Proposition~\ref{grayspectrumprop1} shows that a homogeneous harmonic $1$-form $X_i$ on $C_i$ of order $\lambda + 1 \in (0,1]$ defines a homogeneous closed and coclosed $3$-form $\dci *(X_i \wedge \psci)$ of order $\lambda \in (-1,0]$ on $C_i$, which must then in turn define a form which either subtracts from the kernel or adds to the cokernel of the operator $\dd$ as the rate crosses $\lambda$ by Theorem~\ref{Pindexchangethm}, \emph{if it is non-zero}. Note that, by Propositions~\ref{Killing-cone-prop} and~\ref{NK-Killing-prop}, we have $\dci*(X_i\wedge\psci) = 0$ if and only if $X_i$ is dual to a Killing field on $\Sigma_i$, which means that $X_i$ does not affect the change of the kernel or cokernel of $\dd$ if and only if $X_i$ is dual to a Killing field on $\Sigma_i$.  

Suppose that
\begin{equation*}
\begin{aligned}
\gamma \, = \, (\gamma_1,\ldots,\gamma_n) & \in \oplus_{i=1}^n \Big\{ \dci \st (X_i \wedge \psci) \, ; \, X_i \in \oplus_{\lambda \in (-1,0]} \mathcal{K}(\lambda+1)_{\mdiracci} \Big\} \\
& \qquad \subseteq \oplus_{i=1}^n \oplus_{\lambda\in (-3,0]} \mathcal{K}(\lambda)_{\dd_{C_i}}
\end{aligned}
\end{equation*}
and let $\zeta = d \st (X \wedge \ps)$ be essentially the pullback of $\gamma$ to $M$, where $X = \sum_{i=1}^n \chi_i X_i$ as above. Observe that $d^* \zeta = \vartheta(\gamma)$ modulo the image of $\dd$. Recall that by Lemma~\ref{VWisom.lem}, we have an isomorphism $P : (\mathcal{V}_{\ph})_{l, \nu} \to (\mathcal{W}_{\ph})_{l-1, \nu-1}$ given by $P(\beta) = \frac{7}{3} d^* \pi_1 (\beta) + 2 d^* \pi_7 (\beta)$. Proposition~\ref{infinitesimaldiffeosprop} shows that
\begin{equation} \label{W.char.eq}
d^* \zeta \, = \, P (\pi_{1+7} \zeta) \, = \, \frac{7}{3} d^* \pi_1 (\zeta) + 2 d^* \pi_7 (\zeta).
\end{equation}
We have seen above that $\pi_{1+7} \zeta \in (\mathcal{V}_{\ph})_{l, \nu}$ and thus, by~\eqref{mathcalWdefneq}, we see that $d^* \zeta \in (\mathcal{W}_{\ph})_{l-1, \nu-1}$. We deduce that the map in~\eqref{surjective.coker.map.eq.2} applied to $\gamma$, given by
\begin{equation*}
\vartheta : \gamma \mapsto \pi_{/\mathcal{W}} \circ d^* \zeta \qquad \text{modulo } \im\dd_{l,\nu} + (\mathcal{W}_{\ph})_{l-1, \nu-1},
\end{equation*}
is trivial for all choices of $\gamma$.

From these considerations, since we already accounted for the fact that Killing fields on $\Sigma_i$ do not affect the index of $\dd$ and we know that $\ph_M$ itself must subtract from the kernel of $\dd$ as we cross rate $0$, we deduce that 
\begin{align*}
\dim \mathcal{O}_{/ \mathcal{W}} & \leq \, \sum_{i=1}^n \sum_{\lambda \in (-3,0]} m_{\Sigma_i}(\lambda) + \sum_{i=1}^n \sum_{\lambda \in (-1,0]} \dim \mathcal{K}(\lambda + 1)_{\lapci^1} - 1 - \sum_{i=1}^n \sum_{\lambda \in (-1,0]} \dim \mathcal{K}(\lambda+1)_{\mdiracci}.
\end{align*}
On the other hand, Propositions~\ref{excludeextensionrefinedprop} and~\ref{modifieddiracexcludeprop} show that 
\begin{equation} \label{smoothCStempeq2}
\dim \mathcal{K}(\lambda+1)_{\lapci^1} - \dim \mathcal{K}(\lambda+1)_{\mdiracci} \, = \, \left\{ \begin{array}{rl} 0 & \lambda \in (-1,0), \\ 1 & \lambda = 0. \end{array} \right.
\end{equation}
The result then follows.
\end{proof}

Proposition~\ref{OW.dim.prop2} then motivates the following definition.
\begin{defn} \label{CSgoodsingsdefn}
Let $M$ be a CS $\G$~manifold. We say that $M$ has \emph{good singularities} if for each link $\Sigma_i$ of the corresponding asymptotic cone $C_i$, the system of equations~\eqref{graymaineq3} has no nontrivial solutions for $\lambda \in (-3, 0]$. Specifically, $M$ has good singularities if  $m_{\Sigma_i}(\lambda) = 0$ for all $\lambda \in (-3,0]$.
\end{defn}

\begin{thm} \label{smoothCSthm}
Let $M$ be a CS $\G$~manifold with good singularities, and choose $\nu > 0$ close enough to zero so that there are no other critical rates between $0$ and $\nu$. Then the moduli space $\mathcal{M}_{\nu}$ is \emph{smooth} with dimension
\begin{equation} \label{smoothCSdimeq}
\dim\mathcal{M}_{\nu} \, = \, \dim (\im (H_{\mathrm{cs}}^3 \to H^3)) + \sum_{i=1}^n \dim ( \mathrm{Kill} \, \Sigma_i),
\end{equation}
where $\mathrm{Kill} \, \Sigma_i$ is the space of Killing fields on $\Sigma_i$.
\end{thm}

\begin{rmk} \label{smoothCSrmk}
The formula~\eqref{smoothCSdimeq} for $\dim \mathcal{M}_{\nu}$ in the case of \emph{good singularities} differs from the expected dimension in equation~\eqref{CSvdimeq}. Specifically, combining equations~\eqref{CSvdimeq},~\eqref{smoothCStempeq},~\eqref{smoothCStempeq2} and \eqref{smoothCSdimeq} in the case of good singularities gives
\begin{equation*}
\vdim \mathcal M_{\nu} \, = \, \dim (\im (H_{\mathrm{cs}}^3 \to H^3 )) + \sum_{i=1}^n \dim ( \mathrm{Kill} \, \Sigma_i) - (n-1) \,=\,  \dim\mathcal{M}_{\nu}-(n-1).
\end{equation*}
The point is that, in the case of good singularities, we show in the proof of Theorem~\ref{smoothCSthm} below that the obstruction space $\mathcal{O}_{/ \mathcal{W}}$ has dimension  $ n-1$ and that these obstructions are ineffective.  Hence $\vdim\mathcal{M}_{\nu}$ differs from $\dim\mathcal{M}_{\nu}$ precisely by $\dim\mathcal{O}_{/\mathcal{W}}$.
\end{rmk}

\begin{proof}[Proof of Theorem~\ref{smoothCSthm}]
Recall that the hypothesis about good singularities means that $m_{\Sigma_i}(\lambda) = 0$ for all $\lambda \in (-3,0]$, and hence $\dim \mathcal{O}_{/ \mathcal{W}} \leq n-1$ by Proposition~\ref{OW.dim.prop2}. We will first argue that in this case, $\dim \mathcal{O}_{/ \mathcal{W}} = n - 1$ and that the elements of $\mathcal{O}_{/ \mathcal{W}}$, viewed as $2$-forms, lie in $\Omega^2_7$.

We see from Proposition~\ref{excludeextensionrefinedprop} and the proof of Proposition~\ref{OW.dim.prop2} that these $n-1$ forms which either subtract from the kernel of $\dd_{l,\lambda}$ acting on $3$-forms or add to $\mathcal{O}_{\lambda}$ as $\lambda$ crosses $0$ arise from $\dim \mathcal{K}(1)_{\Delta_{C_i}}$, specifically from the $1$-forms $r dr = d \bigl( \frac{r^2}{2} \bigr)$ on the $C_i$'s.  The corresponding 3-forms are simply $\phci$ on $C_i$.

Notice that $\phm$ itself is $O(1)$ and thus lies in $\ker \dd_{-\e}$ for any $\e > 0$ but not in $\ker \dd_{\nu}$. Moreover, $\phm$ is asymptotic to $\phci = \frac{1}{3} d \bigl( \st (r dr \wedge \psci) \bigr)$ at each singularity. Thus, if we choose the same $1$-form $rdr$ for \emph{every} $C_i$, this corresponds to $\phm$ which subtracts from the kernel. We now want to show that the other $n - 1$ forms one obtains from choosing non-identical multiples of $r dr$ on the $C_i$'s lead to elements of $\mathcal{O}_{\nu}$.

Suppose that $\eta$ is a closed and coclosed 3-form on $M$ which is of order $0$ and asymptotic to $c_i\phci$ on the $i$th end, where $c_i$ are constants. We know that $\pi_1\eta$ is harmonic, since $\eta$ is harmonic, and hence by Lemma~\ref{bochnerlemma} we have that $\pi_1\eta=c$ constant. We deduce that  $c_i=c$ for all $i$, which means that the forms  $c_i\phci$ can only define an element of the kernel of $\dd$ if all of the $c_i$ are equal.

Hence we find that the $(n-1)$-dimensional space of $1$-forms one obtains by taking non-identical multiples of $r dr$ on each $C_i$ defines an $(n-1)$-dimensional subspace $\mathcal N$ of $\mathcal{O}_{\nu}$. Explicitly, we may write
\begin{equation*}
\mathcal{N} \, = \, \left\{ \sum_{i=1}^n c_i d^* d \big( \chi_i \st (r dr \wedge \psci) \big) \, ; \,c_i \in \mathbb{R} , \, \sum_{i=1}^n c_i = 0 \right\}\subseteq\Omega^2_7,
\end{equation*}
where $\chi_i$ is a cut-off function which is $1$ on the $i^{\text{th}}$ end, and we omit the pullbacks from the $C_i$ to the ends to simplify the presentation. We now want to show that $\mathcal{N}$ is transverse to $\im\dd_{l,\nu}+(\mathcal{W}_{\ph})_{l-1,\nu-1}$. In the proof of Proposition~\ref{OW.dim.prop2}, we observed that we could take $(\mathcal{V}_{\ph})_{l,\nu}$ to consist of $\pi_{1+7}\zeta$ where $\zeta$ lies in the span of the set consisting of the $d( \st (X \wedge \ps) )$ for certain $X \in \Omega^1_{l+1, \nu+1}$ together with the form in equation~\eqref{phm.obs.eq}. We see by Proposition~\ref{infinitesimaldiffeosprop} that all such $\zeta$ lie in the kernel of $d \circ L_{\ph}$ and hence $\zeta$ satisfies~\eqref{W.char.eq}, from which it follows that the $d^* \zeta$ can be taken to span $(\mathcal{W}_{\ph})_{l-1, \nu-1}$ by~\eqref{mathcalWdefneq}. Now suppose that
\begin{equation*}
\sum_{i=1}^n c_i d^* d \big( \chi_i \st (r dr \wedge \psci) \big) \in \mathcal{N} \cap (\im \dd_{l,\nu} + (\mathcal{W}_{\ph})_{l-1, \nu-1}).
\end{equation*}
Then there exists $\eta \in \Omega^3_{l, \nu}$ satisfying $d \eta = 0$ and $d^* \zeta \in (\mathcal{W}_{\ph})_{l-1, \nu-1}$ so that
\begin{equation*}
\sum_{i=1}^n c_i d^* d \big( \chi_i \st (r dr \wedge \psci) \big) \, = \, d^* \eta + d^* \zeta.
\end{equation*}
Therefore, since both $\eta$ and $\zeta$ are closed, we see that
\begin{equation*}
\gamma \, = \, \sum_{i=1}^n c_i d \big( \chi_i \st (r dr \wedge \psci) \big) - \eta - \zeta
\end{equation*}
is a closed and coclosed $3$-form so that $\pi_1\gamma$ is asymptotic to $3 c_i \phci$ on the $i^{\text{th}}$ end. As above we deduce that $c_i = c$ is constant for all $i$, which forces $c_i = 0$ for all $i$. In conclusion, $\mathcal N \subseteq \Omega^2_7$ is an $(n-1)$-dimensional subspace of $\mathcal{O}_{/ \mathcal{W}}$, so $\dim \mathcal{O}_{/ \mathcal{W}} = n-1$ and $\mathcal{O}_{/ \mathcal{W}}=\mathcal{N}$.

We now want to argue that these obstructions in $\mathcal{O}_{/ \mathcal{W}}$ are actually ineffective in this setting. Consider the map $F : \mathcal{U} \to \mathcal{Y}$ given in~\eqref{Fdefneq}. We claim that if 
\begin{equation} \label{obs.ineffective.eq}
F(\eta) + d^* \zeta + \xi \, = \, 0
\end{equation}
for some $d^* \zeta \in (\mathcal{W}_{\ph})_{l-1, \nu-1}$, then $\xi \in \mathcal{N} = \mathcal{O}_{/ \mathcal{W}}$, then $\xi =0$. To establish this claim we follow a similar approach to the proof of Theorem~\ref{onetoonethm}. Since $d^*\zeta$ and $\xi$ are 2-forms we have from~\eqref{obs.ineffective.eq} and~\eqref{Fdefneq} that $d\eta = 0$. We see from~\eqref{onetooneeq0} that
\begin{equation*}
\st d (\Theta(\ph + \eta)) \, = \, -\frac{7}{3} d^* \pi_1(\eta) - 2 d^* \pi_7(\eta) + F(\eta).
\end{equation*}
Using~\eqref{obs.ineffective.eq} and the fact that $\zeta$ satisfies~\eqref{W.char.eq}, we deduce that
\begin{equation*}
\st d (\Theta(\ph + \eta)) \, = \, -\frac{7}{3} d^*\pi_1(\zeta + \eta) - 2 d^* \pi_7(\zeta + \eta) - \xi.
\end{equation*}
Since $\mathcal{N} \subseteq \Omega^2_7$ and consists of coexact forms, we may write $\st \xi = dh \wedge \ps$ for some function $h$ with $dh \in \Omega^1_{l-1, \nu-1}$ which tends to constants on each end \emph{which are not all equal, unless $h = 0$} (by definition of the space $\mathcal{N}$). We also know that $\pi_1(\eta + \zeta) = f \ph$ for some $f \in \Omega^0_{l, \nu}$ and $\pi_7(\eta + \zeta) = \st (X \wedge \ph)$ for some $X \in \Omega^1_{l, \nu}$. Therefore, we have
\begin{equation*}
d (\Theta(\ph + \eta)) \, = \, -\left( \frac{7}{3}df + dh \right) \wedge \ps - 2 dX \wedge\ph
\end{equation*}
and thus we may apply Lemma~\ref{threezeroeslemma} (since $\nu>0$) to deduce that $\frac{7}{3}df + dh = 0$ (and that $dX=0$). Thus, $\frac{7}{3}f + h = c$ is constant, but $f$ tends to $0$ on each end whereas $h$ tends to different constants on at least two ends, if it is nonzero. So we must have $f = h = 0$. We conclude that $\xi = 0$ in~\eqref{obs.ineffective.eq} as desired. We have therefore shown that $(\pi_{/ \mathcal{W}} \circ F)(\eta) + \xi = 0$ for $\eta \in \mathcal{U}$ and $\xi \in \mathcal{O}_{/ \mathcal{W}} \subseteq \Omega^2_7$ if and only if $\xi = 0$ and $(\pi_{/ \mathcal{W}} \circ F) (\eta) = 0$.

Recall that the set of $\eta$ such that $(\pi_{/ \mathcal{W}} \circ F) (\eta) = 0$ describes a neighbourhood of $[\ph]$ in $\mathcal{M}_{\nu}$ by Theorem~\ref{finalmodulithm}. Our discussion above thus shows that
\begin{equation*}
\hat{F}_{/ \mathcal{W}}^{-1}(0) \, = \, \{ \eta + \xi \, ; \, (\pi_{/ \mathcal{W}} \circ F)(\eta) + \xi = 0 \} \, = \, \{ \eta \, ; \, (\pi_{/ \mathcal{W}} \circ F)(\eta) = 0 \},
\end{equation*}
describes $\mathcal{M}_{\nu}$ near $[\ph]$. Moreover, Corollary~\ref{widehatFspacecor} gives us that $\hat{F}_{/ \mathcal{W}}^{-1}(0)$ is smooth and has dimension given by $\dim \mathcal{K}_{/ \mathcal{W}}$.

We deduce that $\mathcal{M}_{\nu}$ is smooth near $[\ph]$ and has dimension equal to $\dim \mathcal{K}_{/ \mathcal{W}}$, which by Corollary~\ref{vdimfinalcor} and the fact that in this case $\dim \mathcal{O}_{/ \mathcal{W}} = n - 1$, is
\begin{equation*}
\dim (\im (H_{\mathrm{cs}}^3 \to H^3 )) - \sum_{i=1}^n \sum_{\lambda \in (-3 , 0]} \dim \mathcal K(\lambda)_{\ddci} + 1 + \sum_{i=1}^n \sum_{\lambda \in (-1,0]} \dim \mathcal{K}(\lambda + 1)_{\mdiracci} + (n - 1).
\end{equation*}
Propositions~\ref{modifieddiracexcludeprop} and~\ref{grayspectrumprop1}  show that for $\lambda \in (-3,-1]$ we have
\begin{equation*}
\dim \mathcal{K}(\lambda)_{\ddci} \, = \, m_{\Sigma_i}(\lambda) \quad \text{and} \quad \dim \mathcal{K}(\lambda+1)_{\Delta_{\mdiracci}} \, = \, 0,
\end{equation*}
and for $\lambda \in (-1,0)$ we have 
\begin{equation*}
\dim \mathcal{K}(\lambda)_{\ddci} \, = \, m_{\Sigma_i}(\lambda) + \dim \mathcal{K}(\lambda+1)_{\Delta_{C_i}}=m_{\Sigma_i}(\lambda) + \dim \mathcal{K}(\lambda+1)_{\mdiracci}.
\end{equation*}
Finally, for $\lambda = 0$, Propositions~\ref{excludeextensionrefinedprop},~\ref{modifieddiracexcludeprop}, and~\ref{grayspectrumprop1} show that
\begin{equation*}
\dim \mathcal{K}(0)_{\ddci} \, = \,  m_{\Sigma_i}(0) + \dim \mathcal{K}(1)_{\Delta_{C_i}} - \dim \mathrm{Kill} \, \Sigma_i \quad \text{and} \quad \dim \mathcal{K}(1)_{\mdiracci} \, = \, \dim \mathcal{K}(1)_{\Delta_{C_i}} - 1.
\end{equation*}
Since we assume that $m_{\Sigma_i}(\lambda) = 0$ for $i = 1,\ldots, n$ and $\lambda \in (-3,0]$, we have that
\begin{equation*}
\sum_{\lambda \in (-3 , 0]} \dim \mathcal K(\lambda)_{\ddci} \, = \, \sum_{\lambda \in (-1,0)} \dim \mathcal{K}(\lambda+1)_{\mdiracci} + \dim \mathcal{K}(1)_{\Delta_{C_i}} - \dim \mathrm{Kill} \, \Sigma_i
\end{equation*}
and
\begin{equation*}
\sum_{\lambda \in (-1,0]} \dim \mathcal{K}(\lambda + 1)_{\mdiracci} \, = \, \sum_{\lambda \in (-1,0)} \dim \mathcal{K}(\lambda + 1)_{\mdiracci} + \dim \mathcal{K}(1)_{\Delta_{C_i}} - 1.
\end{equation*}
Combining these formulas gives the result.
\end{proof}

Theorem~\ref{smoothCSthm} and Proposition~\ref{extend.aut.link.prop} give the indication that we should actually consider a \emph{reduced} moduli space defined as follows.
\begin{defn} \label{reduced.moduli.defn}
Let $(M, \ph)$ be a CS $\G$ manifold. Recall that $\mathcal{T}_{\nu}$ is the space of torsion-free $\G$ structures on $M$ which differ from $\ph$ by a $3$-form in $C^{\infty}_{\nu}$. Let $\breve{\mathcal{D}}_{\nu+1}$ be the subgroup of $\text{Diff}^0(M)$ generated by vector fields which are asymptotic with rate $\nu+1$ on each end to a Killing field on $\Sigma_i$. Since Killing fields on $\Sigma_i$ preserve the $\G$ form on $C_i$ by Proposition~\ref{NK-Killing-prop}, we see that $\breve{\mathcal{D}}_{\nu+1}$ acts on $\mathcal{T}_{\nu}$. The reduced moduli space $\breve{\mathcal{M}}_{\nu}$ is defined to be the quotient $\mathcal{T}_{\nu} / \breve{\mathcal{D}}_{\nu+1}$. Notice that since $\mathcal{D}_{\nu+1} \subseteq \breve{\mathcal{D}}_{\nu+1}$, we can view $\breve{\mathcal{M}}_{\nu}$ as a quotient of $\mathcal{M}_{\nu}$.
\end{defn}

We can now show the following.
\begin{cor} \label{smoothCScor}
Let $M$ be a CS $\G$~manifold with rate $\nu > 0$ near $0$ such that $M$ has good singularities. Then the reduced moduli space $\breve{\mathcal{M}}_{\nu}$ is smooth with dimension $\dim (\im (H_{\mathrm{cs}}^3 \to H^3 ))$.
\end{cor}
\begin{proof}
Recall that we have a diffeomorphism $h_i : (0,\e) \times \Sigma_i \to S_i \subseteq M$, where $S_i$ is an open set in $M$ which is disjoint from the other ends of $M$, satisfying~\eqref{CSdefneq}. Let $X$ be a Killing field on $\Sigma_i$, which then defines a Killing field which we also call $X$ on $C_i$. We then have a vector field $(h_i)_*(X)$ on the $i^{\text{th}}$ end of $M$. If $\chi : M \to [0,1]$ is a smooth function which is $1$ on $h_i (0, \frac{\e}{2} )$ and $0$ on $M \setminus S_i$, then $Y = \chi (h_i)_* X$ is a vector field on $M$ such that $(h_i^{-1})_* Y = O(r)$ as $r \to 0$. Hence $Y \hk \ph \notin (\Omega^2_7)_{\nu+1}$ for any $\nu > 0$ but $d (Y \hk \ph) \in \mathcal{C}_{l,\nu}$ since $d (X \hk \phc) = 0$.

Suppose that $\pi_7 d^*d (Y \hk \ph) = 0$. Then $Y = 0$ by Lemma~\ref{bochnerlemma2}, which is a contradiction. Recalling the definition of $(\mathcal{E}_{\ph})_{l,\nu}$ from the proof of Theorem~\ref{infinitesimalslicethm}, we deduce that $\eta = d (Y \hk \ph)$ defines an element in $(\mathcal{E}_{\ph})_{l,\nu}$, so $\pi_{1+7} \eta \in (\mathcal{V}_{\ph})_{l,\nu}$. Moreover, by Proposition~\ref{infinitesimaldiffeosprop}, we have $d (L_{\ph} \eta) = 0$. Theorem~\ref{linearized.onetoonethm} then shows that
\begin{equation*}
d \eta \, = \, 0 \quad \text{and} \quad d^* \eta \, = \, \frac{7}{3} d^* \pi_1 \eta + 2 d^* \pi_7 \eta.
\end{equation*}
That is, $\eta \in \mathcal{K}_{/ \mathcal{W}}$.

We conclude that each Killing field on $\Sigma_i$ defines an element of $\mathcal{K}_{/ \mathcal{W}}$, and hence the Killing fields on $\Sigma_i$ for $i = 1, \ldots, n$ define a subspace $\mathcal{K}_1$ of $\mathcal{K}_{/ \mathcal{W}}$ of dimension $\sum_{i=1}^n \dim (\mathrm{Kill} \, \Sigma_i)$. Notice that $d (\Omega^2_7)_{\nu+1} \oplus \mathcal{K}_1$ is equal to the tangent space to $\breve{\mathcal{D}}_{\nu+1} \cdot \ph$ at $\ph$. Let $\breve{\mathcal{K}}$ be a direct complement of this subspace in $\mathcal{K}_{/ \mathcal{W}}$, which then has dimension $\dim (\im (H_{\mathrm{cs}}^3 \to H^3 ))$ by the proof of Theorem~\ref{smoothCSthm}. In fact, $\breve{\mathcal{K}}$ can be taken to be $\mathcal{H}^3_{-3+\varepsilon}$.

Recall that $\mathcal{M}_{\nu}$ is identified near $[\ph]$ with a neighbourhood of $0$ in $F_{/ \mathcal{W}}^{-1}(0)$, which is given as a graph $\Gamma$ of a map $G$ over an open subset $\mathcal{V}$ of $\mathcal{K}_{/ \mathcal{W}}$. Hence, the graph $\breve{\Gamma}$ of $G$ over $\mathcal{V} \cap \breve{\mathcal{K}}$ is a smooth submanifold of $F_{/ \mathcal{W}}^{-1}(0)$.  Since $\breve{\mathcal{K}}$ is a space of elements of $\mathcal{K}_{/ \mathcal{W}}$ tranverse to the orbit of $\ph$ under $\breve{\mathcal{D}}_{\nu+1}$, the graph $\breve{\Gamma}$ can be identified with a neighbourhood of $\breve{\mathcal{M}}_{\nu}$ near $[\ph]$, which then gives the result.
\end{proof}

\begin{cor} \label{smoothCSsomeconescor}
Let $M$ be a CS $\G$~manifold of rate $\nu = 0 + \e$, all of whose conical singularities are modeled on $\G$~cones whose links are either $\C \PR^3$ or $S^3 \times S^3$. Then both the moduli space $\mathcal{M}_{\nu}$ and the reduced moduli space $\breve{\mathcal{M}}_{\nu}$ of CS deformations of $M$ with rate $\nu$ are smooth manifolds.
\end{cor}
\begin{proof}
This follows from Theorem~\ref{smoothCSthm}, Corollary~\ref{smoothCScor} and Proposition~\ref{grayspectrumfinalprop}.
\end{proof}

\begin{rmk}
We can define a weaker notion of good singularities, where we allow for nontrivial solutions of~\eqref{graymaineq3} for $\lambda = 0$ as long as they always define \emph{integrable} nearly K\"ahler deformations of the link of the cone at the singularity. In this setting, the arguments above can be extended to show that the moduli space is still smooth. These various notions of good singularities are closely related to the idea of a \emph{stability index} for $\G$~conical singularities, which is potentially a source of further study, as we shall mention in Section~\ref{openproblemssec}.
\end{rmk}

\subsection{Relation with the moduli of resolved CS \texorpdfstring{$\mathbf{\G}$}{G2} manifolds} \label{resolvedmodulisec}

In this section we relate our results to the resolution of singularities construction of~\cite{Kdesings}. Our observations provide evidence that CS $\G$~manifolds likely arise as the ``most common'' form of singular object in any attempt to compactify the moduli space of compact smooth $\G$~manifolds.

Let $M$ be a CS $\G$~manifold with one conical singularity and let $N$ be an AC $\G$~manifold asymptotic to the same $\G$~cone, with link $\Sigma$ at infinity, as $M$ has at its singularity. The main result of~\cite{Kdesings} says that, if a particular necessary topological condition~\cite[Theorem 3.8]{Kdesings} is satisfied, then one can \emph{desingularize} $M$ by gluing in $N$ to obtain a compact \emph{smooth} $\G$~manifold, which we will denote by $X$. When $M$ has a single conical singularity, the topological condition can be expressed using the maps $\Upsilon^k$ of Definition~\ref{Upsilonmapsdefn} as follows:
\begin{equation} \label{topconditioneq}
\Upsilon^3_N (\phn) \, \in \, \im (\Upsilon^3_M), \qquad \qquad \Upsilon^4_N (\psn) \, \in \, \im (\Upsilon^4_M).
\end{equation}
\begin{rmk} \label{UpsilonACrmk}
In~\cite[Definition 2.40]{Kdesings}, the elements $\Upsilon^3_N (\phn)$ and $\Upsilon^4_N (\psn)$ are denoted by $\Phi(N)$ and $\Psi(N)$, respectively.
\end{rmk}

Since $X = M \cup N$, and since $M \cap N$ is homotopy equivalent to $\Sigma$, the Mayer--Vietoris long exact sequence gives
\begin{equation} \label{mayervietoriseq}
\cdots \longrightarrow H^k(X) \longrightarrow H^k(M) \oplus H^k(N) \overset{\delta^k}{\longrightarrow} H^k(\Sigma) \longrightarrow H^{k+1}(X) \longrightarrow \cdots
\end{equation}
where the map $\delta^k : H^k(M) \oplus H^k(N) \to H^k(\Sigma)$ is given by $\delta^k(a) = \Upsilon^k_M(a) - \Upsilon^k_N(a)$, and the $\Upsilon^k$ maps are those of Definition~\ref{Upsilonmapsdefn}.

\begin{lemma} \label{mayervietorislemma}
Let $M$, $N$, $X$, and $\Sigma$ be as above. The following equation holds.
\begin{equation} \label{mvlemmaeq}
\begin{aligned}
b^k(X) \, & = \,  b^k(M) - \dim(\im \Upsilon^k_M) \\ & \qquad {} + b^k_{\mathrm{cs}}(N) - \dim( \im \Upsilon^{7-k}_N) + \dim ( (\im \Upsilon^{7-k}_M) \cap (\im \Upsilon^{7-k}_N) ) \\ & \qquad {} + \dim ( (\im \Upsilon^k_M) \cap (\im \Upsilon^k_N) ).
\end{aligned}
\end{equation}
\end{lemma}
\begin{proof}
We will use the shorthand notation $H^k_A$ for $H^k(A)$. By the rank-nullity theorem we have
\begin{equation} \label{mvtempeq1}
b^k(X) \, = \, \dim(\ker (H^k_X \to H^k_M \oplus H^k_N)) + \dim(\im(H^k_X \to H^k_M \oplus H^k_N)).
\end{equation}
The exactness of~\eqref{mayervietoriseq} gives
\begin{equation} \label{mvtempeq2}
 \dim(\im(H^k_X \to H^k_M \oplus H^k_N)) \, = \, \dim (\ker \delta^k),
\end{equation}
but since $\delta^k = \Upsilon^k_M - \Upsilon^k_N$, it is easy to check that
\begin{equation} \label{mvtempeq3}
\dim(\ker \delta^k) \, = \, \dim (\ker \Upsilon^k_M) + \dim(\ker \Upsilon^k_N) + \dim ( (\im \Upsilon^k_M) \cap (\im \Upsilon^k_N) ).
\end{equation}
Substituting~\eqref{mvtempeq2} and~\eqref{mvtempeq3} into~\eqref{mvtempeq1} gives
\begin{equation} \label{mvtempeq4}
\begin{aligned}
b^k(X) \, & = \, \dim(\ker (H^k_X \to H^k_M \oplus H^k_N)) \\ & \qquad {} +  \dim (\ker \Upsilon^k_M) + \dim(\ker \Upsilon^k_N) + \dim ( (\im \Upsilon^k_M) \cap (\im \Upsilon^k_N) ).
\end{aligned}
\end{equation}

Again using the exactness of~\eqref{mayervietoriseq}, we find
\begin{align*}
\dim (\ker (H^k_X \to H^k_M \oplus H^k_N)) \, & = \, \dim (\im (H^{k-1}_{\Sigma} \to H^k_X)) \\ & = \, b^{k-1}(\Sigma) - \dim (\ker (H^{k-1}_{\Sigma} \to H^k_X)) \\ & = \, b^{k-1}(\Sigma) - \dim (\im \delta^{k-1}) \\ & = \, \dim (\coker \delta^{k-1}).
\end{align*}
But the dimension of this cokernel is equal to the dimension of the kernel of the \emph{formal adjoint}, which is a map $(\delta^{k-1})^* : H^{6-(k-1)}(\Sigma) \to H^{7-(k-1)}_{\mathrm{cs}}(M) \oplus H^{7-(k-1)}_{\mathrm{cs}}(N)$. It can be easily checked from the definitions of all the maps involved that $(\delta^{k-1})^* (a) = (\partial^{7-k}_M (a), \partial^{7-k}_N(a))$, where the maps $\partial^k_{A} : H^k (\Sigma) \to (H^{k+1}_{\mathrm{cs}})(A)$ are the connecting homomorphisms in the long exact sequence~\eqref{longexactsequenceeq}. Thus we have
\begin{equation} \label{mvtempeq5}
\begin{aligned}
\dim (\ker (H^k_X \to H^k_M \oplus H^k_N)) \, & = \, \dim (\ker (\delta^{k-1})^*) \\ & = \, \dim ((\ker \partial^{7-k}_M) \cap (\ker \partial^{7-k}_N)) \\ & = \, \dim ( (\im \Upsilon^{7-k}_M) \cap (\im \Upsilon^{7-k}_N) )
\end{aligned}
\end{equation}
where in the last step above we have used the exactness of~\eqref{longexactsequenceeq} for both $M$ and $N$. Finally, we substitute~\eqref{mvtempeq5} into~\eqref{mvtempeq4} and use~\eqref{exactkereq1} for $M$ and~\eqref{exactkereq2} for $N$ to obtain~\eqref{mvlemmaeq}.
\end{proof}

Specializing Lemma~\ref{mayervietorislemma} to $k=3$ gives
\begin{equation} \label{mvmaineq}
\begin{aligned}
b^3(X) \, & = \,  b^3(M) - \dim(\im \Upsilon^3_M) \\ & \qquad {} + b^3_{\mathrm{cs}}(N) - \dim( \im \Upsilon^4_N) + \dim ( (\im \Upsilon^4_M) \cap (\im \Upsilon^4_N) ) \\ & \qquad {} + \dim ( (\im \Upsilon^3_M) \cap (\im \Upsilon^3_N) ).
\end{aligned}
\end{equation}
Notice that the left hand side gives the dimension of the moduli space of deformations of the \emph{smooth, compact} $\G$~manifold $X$. Also, by Corollary~\ref{smoothCScor} and equation~\eqref{imcseq} the first two terms on the right hand side give the dimension of the reduced moduli space $\breve{\mathcal M}_{\nu}$ of the CS $\G$~manifold $M$, \emph{in the cases when Theorem~\ref{smoothCSthm} applies}. In particular, by Corollary~\ref{smoothCSsomeconescor}, this is true if there is one conical singularity whose link is either $\C \PR^3$ or $S^3 \times S^3$.

Let us consider two particular cases.

{\em Case 1.} Suppose $b^4(N) = b^3_{\mathrm{cs}}(N) = 1$, $b^3(N) = b^4_{\mathrm{cs}}(N) = 0$, and $b^3(\Sigma) = 0$. Further assume that $\Upsilon^4_N (\psn) \neq 0$ in $H^4(\Sigma)$. [In particular, these assumptions all hold for the Bryant--Salamon manifolds $\Lambda^2_-(S^4)$ and $\Lambda^2_-(\C \PR^2)$, with links $\Sigma = \C \PR^3$ and $\Sigma = \SUth / T^2$, respectively.]

With these assumptions, some simple diagram chasing using the exact sequence~\eqref{longexactsequenceeq} gives
\begin{equation*}
\Upsilon^3_M = 0, \quad \Upsilon^3_N = 0, \quad \ker(\Upsilon^4_N) = \{ 0 \}, \quad \dim (\im(\Upsilon^4_N)) = 1.
\end{equation*}
The condition~\eqref{topconditioneq} in this case thus becomes $\im \Upsilon^4_N = ( \im \Upsilon^4_N) \cap ( \im \Upsilon^4_M)$, and~\eqref{mvmaineq} therefore becomes $b^3(X) = b^3(M) + 1$.

{\em Case 2.} Suppose $b^4(N) = b^3_{\mathrm{cs}}(N) = 0$, $b^3(N) = b^4_{\mathrm{cs}}(N) = 1$, and $b^4(\Sigma) = 0$. Further assume that $\Upsilon^3_N (\phn) \neq 0$ in $H^3(\Sigma)$. [In particular, these assumptions all hold for the Bryant--Salamon manifold $\spi (S^3)$, with link $\Sigma = S^3 \times S^3$.]

As before, diagram chasing using~\eqref{longexactsequenceeq} in this case yields
\begin{equation*}
\Upsilon^4_M = 0, \quad \Upsilon^4_N = 0, \quad \ker(\Upsilon^3_N) = \{ 0 \}, \quad \dim (\im(\Upsilon^3_N)) = 1.
\end{equation*}
The condition~\eqref{topconditioneq} in this case therefore becomes $\im \Upsilon^3_N = ( \im \Upsilon^3_N) \cap ( \im \Upsilon^3_M)$, and~\eqref{mvmaineq} thus becomes $b^3(X) = b^3(M) = b^3(M) - \dim( \im \Upsilon^3_M) + 1$.

To summarize: in both cases (which include all the known examples of AC $\G$~manifolds), we find that the dimension of the moduli space of glued compact $\G$~manifolds that are constructed in~\cite{Kdesings} is exactly one dimension higher than the ``reduced moduli space'' $\breve{\mathcal M}_{\nu}$ of the CS $\G$~manifold which has been resolved so, informally, we can view the reduced CS moduli space almost literally as ``the boundary'' of the moduli space of compact $\G$~manifolds, at least locally.

\subsection{Existence of a gauge-fixing diffeomorphism} \label{gaugefixexistssec}

In~\cite{Kdesings}, a gauge-fixing condition was defined for AC $\G$~manifolds that  is slightly different from our Definition~\ref{gaugefixingdefn}.

\begin{defn} \label{alternativegaugefixingdefn}
Consider an AC $\G$~manifold $(M, \ph)$, which comes equipped with a choice of diffeomorphism $h : (R, \infty) \times \Sigma \to M \setminus L$ for some compact subset $L \subset M$ and some $R > 0$, where $\Sigma$ is the link of the asymptotic $\G$~cone. The diffeomorphism $h$ is said to satisfy the gauge-fixing condition if $h^*\ph - \phc$ lies in $\Omega^3_{27}$ with respect to the $\G$~structure $\phc$ on the cone. 
\end{defn}

In~\cite{Kdesings}, it was promised that the present paper would give a proof that such a diffeomorphism $h$ can always be chosen to satisfy the gauge-fixing condition. The existence of such a gauge-fixing diffeomorphism follows from a modification of our slice theorem. To state it, we need to introduce some notation.

Let $(M,\ph)$ be an AC $\G$~manifold so that $\ph$ is asymptotic with rate $\nu \in (-4,0)$ to the torsion-free $\G$~structure $\phc$ on the asymptotic cone $C$. For notational convenience we set 
\begin{equation*}
\bar{\Omega}^k_{l, \nu} \, = \, \{ h^* \eta \, ; \, \eta \in \Omega^k_{l, \nu} \} \subseteq \Omega^k ( (R, \infty) \times \Sigma).
\end{equation*} 
Using the conical torsion-free $\G$~structure $\phc$, we can decompose the $3$-forms on the conical end as $\Omega^3 ( (R, \infty \times \Sigma ) = \bar{\Omega}^3_1 \oplus \bar{\Omega}^3_7 \oplus \bar{\Omega}^3_{27}$ and let $\bar{\pi}_1$, $\bar{\pi}_7$, and $\bar{\pi}_{27}$ be the projections onto the components. (We also have a similar decomposition for the $2$-forms.) Let
\begin{equation*}
(\mathcal{G}_{\phc})_{l,\nu} \, = \, \{ \eta \in (\bar{\Omega}^3_{27})_{l, \nu} \, ; \, d \eta = 0 \}
\end{equation*}
and let
\begin{equation*}
(\mathcal{T}_{\ph})_{l, \nu} \, = \, h^* T_{\ph} ({\mathcal D}_{l+1, \nu+1} \cdot \ph) \, = \, \{ d (X \hk h^* \ph) \, ; \, X \in \bar{\Omega}^1_{l+1, \nu+1} \}.
\end{equation*}
We can now state our modified slice theorem, which we prove later in this section.

\begin{thm} \label{modified.inf.slice.thm.2}
Let $(M, \ph)$ be an AC $\G$~manifold with generic rate $\nu \in (-4,0)$. Use the notation above and let $l \geq 4$. By taking $R$ sufficiently large, on $( R, \infty ) \times \Sigma \cong M \setminus L$ we have
\begin{equation} \label{C.decomp.eq.2}
\bar{\mathcal C}_{l, \nu} \, = \, \{ \eta \in \bar{\Omega}^3_{l, \nu} \, ; \, d \eta = 0 \} \, = \, (\mathcal{T}_{\ph})_{l, \nu} + ({\mathcal G}_{\phc})_{l, \nu}.
\end{equation}
\end{thm} 

Although one would like to deduce this result by modifying the proof of Theorem~\ref{infinitesimalslicethm}, that method runs into difficulties for rates $\nu \geq -3$. We therefore take an alternative approach.

Suppose that $\eta \in \bar{\mathcal C}_{l, \nu}$. In this section only, we will use $\bar{\mathcal J}^1_{l, \nu}$ to denote the coexact $1$-forms in $\bar{\Omega}^1_{l,\nu}$; that is,  in the image of $\dsc$. Observe that since $d \eta = 0$ and $\phc$ is exact, we have that  $\stc (\eta \wedge \phc) \in \bar{\mathcal J}^1_{l, \nu}$. (We also have that the function $\stc ( \eta \wedge \psc)$ is coexact but we shall not need this.)

Now consider the map 
\begin{equation*}
F_C : X \in \bar{\Omega}^1_{l+1, \nu+1} \mapsto \big( \stc ( d (X \hk \phc) \wedge \psc) , \stc ( d ( X \hk \phc) \wedge \phc) \big) \in \bar{\Omega}^0_{l, \nu} \oplus \bar{\Omega}^1_{l, \nu}.
\end{equation*}
The map $F_C$ can be naturally identified with the map $X \mapsto \bar{\pi}_{1+7} d ( X \hk \phc )$ and we can write
\begin{equation*}
F_C(X) \, = \, (-3 \dsc X , 2 \curlc X)
\end{equation*}
by Lemma~\ref{dXhkphlemma}, where $\curlc$ is the curl operator on $C$.  Notice that since $\phc$ is closed, we have
\begin{equation*}
F_C: \bar{\Omega}^1_{l+1, \nu+1} \to \bar{\Omega}^0_{l, \nu} \oplus \bar{\mathcal J}^1_{l, \nu}.
\end{equation*}
Our first key observation is the following, which is that $F_C$ has a right-inverse.

\begin{prop} \label{GC.prop}
For generic rates $\nu$, there is a bounded linear map 
\begin{equation*}
G_C: \bar{\Omega}^0_{l, \nu} \oplus \bar{\mathcal J}^1_{l, \nu} \to \bar{\Omega}^1_{l+1, \nu+1}
\end{equation*}
such that $F_C\circ G_C = \id$.
\end{prop}
\begin{proof}
To begin, we know that for generic rates $\nu$ the Laplacian on the cone 
\begin{equation*}
\lapc^0 : \bar{\Omega}^0_{l+2, \nu+2} \to \bar{\Omega}^0_{l,\nu}
\end{equation*}
is surjective. This can be proved easily by separation of variables and by decomposing a function on the cone using eigenfunctions on the link, which converts the surjectivity question into a problem concerning the solution of a system of linear second order ordinary differential equations in the radius $r$ with regular singular points at $r=0$. Thus, there exists a bounded linear map
\begin{equation*}
G_{\Delta}^0: \bar{\Omega}^0_{l, \nu} \to \bar{\Omega}^0_{l+2, \nu+2}
\end{equation*}
such that $\lapc^0 \circ G_{\Delta}^0 = \id$. Hence, if we let 
\begin{equation*}
G_C (f,0) = -\frac{1}{3} d (G_{\Delta}^0 f),
\end{equation*}
then
\begin{equation*}
F_C \circ G_C(f,0) \, = \, \left( -3 \dsc \left(- \frac{1}{3} d (G_{\Delta}^0 f) \right), 0 \right) \, = \, \big( \lapc^0 (G_{\Delta}^0 f), 0 \big) \, = \, (f,0).
\end{equation*}
 
Next, if $X \in \bar{\mathcal J}^1_{l, \nu}$, then $\stc X \in \bar{\Omega}^6_{l,\nu}$ is exact so it may be written $\stc X = d \xi$ for some $\xi$. We want to show that $\xi$ can be chosen so that $\xi \in \bar{\Omega}^5_{0, \nu+1}$. We can write
\begin{equation*}
\stc X \, = \, dr \wedge \kappa(r) + \omega(r),
\end{equation*}
where $\kappa(r)$ and $\omega(r)$ are \emph{uniquely} determined by $X$. Define the $5$-form $\beta(r)$ by
\begin{equation*}
\beta(r) \, = \, \int_0^r \kappa(s) ds.
\end{equation*}
Then $\beta \in \bar{\Omega}^5_{l+1, \nu+1}$. Moreover, since $d \beta = dr \wedge \kappa(r) + \ds \beta$, we find that $\stc X - d (\beta(r) )$ is exact and has no $dr$ component, so in fact $\stc X - d \beta = d \sigma = \ds \sigma$ for some $5$-form $\sigma$ on the link, independent of $r$. By taking the coexact part of any such $\sigma$, it is determined uniquely. Since $\sigma$ is independent of $r$, it is homogeneous of order $-5$, and thus $\sigma \in \bar{\Omega}^5_{0, -5} \subseteq \bar{\Omega}^5_{0, \nu+1}$, because $-5 < -3 < \nu + 1$.

This particular choice of $\xi = \beta(r) + \sigma$ therefore by construction satisfies $\xi \in \bar{\Omega}^5_{0, \nu+1}$. Hence, we have $\stc \xi \in \bar{\Omega}^2_{0, \nu+1}$, and $X = \dsc ( \stc \xi)$. We also know that 
\begin{equation*}
\lapc^2 : \bar{\Omega}^2_{2, \nu+2} \to \bar{\Omega}^2_{0, \nu}
\end{equation*}
is surjective for generic $\nu$ (using the same argument as described above), so there is a bounded linear map  
\begin{equation*}
G_{\Delta}^2 : \bar{\Omega}^2_{0, \nu} \to \bar{\Omega}^2_{2, \nu+2}
\end{equation*}
such that $\lapc^2 \circ G^2_{\Delta} = \id$. If we then let
\begin{equation*}
G_C (0,X) \, = \, \frac{1}{2} \curlc (\dsc (G_{\Delta}^2 \stc \xi )),
\end{equation*}
then using Remark~\ref{curlrmk} we find that
\begin{align*}
F_C \circ G_C(0,X) \, & = \, \left( 0, 2 \curlc \left( \frac{1}{2} \curlc (\dsc (G_{\Delta}^2 \stc \xi) ) \right) \right) \\
& = \, \left( 0, \dsc d (\dsc (G_{\Delta}^2 \stc \xi)) \right) \\
& = \, \big( 0, \dsc (\lapc^2( G_{\Delta}^2 \stc \xi)) \big) \\
& = \, (0, \dsc (\stc \xi)) \, = \, (0,X).
\end{align*}
Since $F_C$ is essentially the Dirac operator, and $X \in \bar{\mathcal J}^1_{l, \nu}$, elliptic regularity then ensures that $G_C(0,X) \in \bar{\Omega}^1_{l+1,\nu+1}$, so we have defined $G_C$ as required.  
\end{proof}

However, the actual map that we require a right-inverse for is
\begin{equation*}
F : X \in \bar{\Omega}^1_{l+1, \nu+1} \mapsto \big( \stc ( d ( X \hk h^* \ph) \wedge \psc), \stc (d ( X \hk h^* \ph) \wedge \phc) \big) \in \bar{\Omega}^0_{l, \nu} \oplus \bar{\mathcal{J}}^1_{l, \nu}.
\end{equation*}
(Note that $F$ maps into the space claimed because $h^* \ph$ is asymptotic to $\phc$ and $d ( X \hk h^* \ph)$ is closed.)

\begin{prop} \label{G.prop}
For generic rates $\nu$, and by making $R$ larger if necessary, there is a bounded linear map 
\begin{equation*}
G : \bar{\Omega}^0_{l, \nu} \oplus \bar{\mathcal J}^1_{l, \nu} \to \bar{\Omega}^1_{l+1, \nu+1}
\end{equation*}
so that $F \circ G = \id$.
\end{prop}
\begin{proof}
By Proposition~\ref{GC.prop}, we have that
\begin{equation*}
F \, = \, F_C \circ ( \id + G_C \circ ( F - F_C) ).
\end{equation*}
Since $h^* \ph - \phc$ is of order $O(r^{\nu})$, by making $R$ larger if necessary, we can ensure that the operator $G_C \circ (F - F_C)$ has norm strictly less than $1$. Hence,
\begin{equation*}
\id + G_C \circ (F - F_C) : \bar{\Omega}^1_{l+1, \nu+1} \to \bar{\Omega}^1_{l+1, \nu+1}
\end{equation*} 
is a bounded linear invertible map. We therefore let
\begin{equation*}
G \, = \, \big( \id + G_C \circ ( F - F_C ) \big)^{-1} \circ G_C
\end{equation*}
to obtain the required map.
\end{proof}

We now have the ingredients required to prove Theorem~\ref{modified.inf.slice.thm.2}.

\begin{proof}[Proof of Theorem~\ref{modified.inf.slice.thm.2}]
Let $\eta \in \bar{\mathcal{C}}_{l, \nu}$ be as defined in~\eqref{C.decomp.eq.2}. Then, as we already observed, we have
\begin{equation*}
( \stc (\eta \wedge \psc), \stc (\eta \wedge \phc) ) \in \bar{\Omega}^0_{l, \nu} \oplus \bar{\mathcal{J}}^1_{l, \nu}.
\end{equation*}
Let
\begin{equation*}
X \, = \, G( \stc (\eta \wedge \psc), \stc (\eta \wedge \phc)) \in \bar{\Omega}^1_{l+1, \nu+1}
\end{equation*}
where $G$ is given by Proposition~\ref{G.prop}, choosing $R$ larger if necessary so that the proposition applies. By construction, we 
then have that
\begin{equation*}
\bar{\pi}_{1+7} d ( X \hk h^* \ph ) \, = \, \bar{\pi}_{1+7} \eta.
\end{equation*}
Therefore
\begin{equation*}
\gamma \, = \, \eta - d ( X \hk h^* \ph) \in (\mathcal{G}_{\phc})_{l, \nu}
\end{equation*} 
and thus we have obtained the decomposition given in~\eqref{C.decomp.eq.2}.
\end{proof}

Given the modified infinitesimal slice theorem (Theorem~\ref{modified.inf.slice.thm.2}) we can now give the promised gauge-fixing result.

\begin{prop} \label{gaugefixingpromiseprop.2}
Let $(M, \ph)$ be an AC $\G$~manifold with generic rate $\nu \in (-4,0)$ and diffeomorphism $h : (R, \infty) \times \Sigma \to M \setminus L$ as given in Definition~\ref{ACdefn}. Then, after possibly making $R$ larger, there exists a diffeomorphism $f \in \mathcal{D}_{\nu+1}$ such that the diffeomorphism $f\circ h : (R, \infty) \times \Sigma \to M \setminus f(L)$ satisfies the gauge-fixing condition in Definition~\ref{alternativegaugefixingdefn}.
\end{prop}

\begin{rmk}
Since an AC $\G$~manifold which is asymptotic with rate $\nu_0$ is also asymptotic with rate $\nu > \nu_0$, any AC $\G$~manifold is AC with generic rate $\nu \in (-4,0)$.
\end{rmk}

Before giving the proof of Proposition~\ref{gaugefixingpromiseprop.2}, we need a lemma.

\begin{lemma} \label{gaugefixingpromise.lemma}
Let $\eta \in \mathcal{C}_{l, \nu}$. There exists a diffeomorphism $f \in \mathcal{D}_{l+1, \nu+1}$, which is the identity on a compact subset $L'$ of $L \subseteq M$, such that $(f\circ h)^* \eta$ lies in $(\mathcal{G}_{\phc})_{l, \nu}$ whenever $\eta$ is sufficiently small on the end.
\end{lemma}
\begin{proof}
Theorem~\ref{modified.inf.slice.thm.2} implies that given any closed $3$-form $\eta \in \mathcal{C}_{l, \nu}$ there exists $X \in \Omega^1_{l+1, \nu+1}$ such that, on the end, $\eta - d ( X \hk \ph )$ lies in $(h^{-1})^* \bar{\Omega}^3_{27}$. We can choose $X$ to vanish on a large compact subset $L'$ of $L \subseteq M$ and choose $X$ uniquely on the end of $M$ by insisting that $d (X \hk h^* \ph)$ lies in a direct complement $(\mathcal{T}'_{\ph})_{l, \nu}$ of $(\mathcal{T}_{\ph})_{l, \nu} \cap (\mathcal{G}_{\phc})_{l, \nu}$ in $(\mathcal{T}_{\ph})_{l,\nu}$. Such a complement exists because the intersection $(\mathcal{T}_{\ph})_{l, \nu} \cap (\mathcal{G}_{\phc})_{l, \nu}$ is isomorphic to a subset of the kernel of the modified Dirac operator, which thus forms a closed set and consists of smooth forms.  A modification of earlier slice theorem arguments for the AC manifold $M$ imply that there exists a diffeomorphism $f \in \mathcal{D}_{l+1, \nu+1}$, which is the identity on $L'$, such that $(f\circ h)^* \eta$ lies in $(\mathcal{G}_{\phc})_{l, \nu}$ whenever $\eta$ is sufficiently small on the end.
\end{proof}

\begin{proof}[Proof of Proposition~\ref{gaugefixingpromiseprop.2}]
First we claim that there exists a closed and coclosed $3$-form $\zeta \in \mathcal{H}^3_{\nu}$ on $M$ and a $2$-form $\beta$ on $M \setminus L$ such that
\begin{equation} \label{phc.decomp.eq.2}
\ph - (h^{-1})^* \phc \, = \, \zeta + d \beta.
\end{equation}
To see this, first note that as observed in~\cite[Proposition 2.5]{Kdesings}, the form $\phc$ is exact, so 
\begin{equation*}
[h^* \ph - \phc] \, = \, [h^* \ph] \, = \, \Upsilon^3 [\ph] \, \in H^3(\Sigma, \R).
\end{equation*}
If $\nu < -3$ then $\Upsilon^3 [\ph] = 0$ by~\cite[Proposition 2.39]{Kdesings}, so $[h^* \ph - \phc] = 0$ and hence $\ph-(h^{-1})^* \phc$ is exact on $M \setminus L$. Thus equation~\eqref{phc.decomp.eq.2} holds in this case with $\zeta = 0$. If instead $\nu \geq -3$ we can assume $\nu > -3$ as $\nu$ is supposed to be generic and we are always free to increase the rate. Lemma~\ref{kernelchange34preliminarylemma} gives the existence of a closed and coclosed $3$-form $\zeta \in \mathcal{H}^3_{\nu}$ on $M$ such that $\Upsilon^3 [\zeta] = \Upsilon^3 [\ph]$. Thus, $\ph - (h^{-1})^* \phc - \zeta$ is exact on $M \setminus L$ and~\eqref{phc.decomp.eq.2} again holds.

Now let $R' > R$ and let $\chi: M \to [0,1]$ be a smooth cutoff function which is $1$ on $h ( (R', \infty) \times \Sigma)$ and $0$ on $L$. Let $\xi = \zeta + d (\chi \beta)$, which is defined on all of $M$, and so that $\xi$ agrees with $\ph-(h^{-1})^* \phc$ outside a compact set by~\eqref{phc.decomp.eq.2}.  Applying Lemma~\ref{gaugefixingpromise.lemma} to $\eta = \xi$, since we can make $\xi$ small on the end by making $R$ larger, implies the existence of a diffeomorphism $f\in\mathcal{D}_{l+1, \nu+1}$ such that $(f\circ h)^* \xi \in  (\mathcal{G}_{\phc})_{l, \nu}$, which is the identity on $L'$ and is the unique such diffeomorphism on the end close to the identity, up to composition with a given space of diffeomorphisms in $\mathcal{D}_{\nu+1}$ (generated by vector fields in the kernel of the modified Dirac operator). We also have that $f \circ h$ is gauge-fixed and satisfies the asymptotic decay conditions~\eqref{ACdefneq} up to order $l$, that $\ph$ and $\phc$ are torsion-free $\G$~structures, and $\xi$ is smooth. We can thus use either elliptic regularity or the aforementioned uniqueness for $f$ to deduce that $f \in \mathcal{D}_{\nu+1}$ and $ f\circ h$ satisfies~\eqref{ACdefneq} for all orders.
\end{proof}

\subsection{Open problems} \label{openproblemssec}

There remain several interesting and important open problems for future study.
\begin{itemize}
\item It is important to find more examples, especially with little or no symmetry, of Gray manifolds (compact strictly nearly K\"ahler $6$-manifolds). This would provide new examples of $\G$~cones, and hopefully one could construct new AC $\G$~manifolds with these asymptotic cones.  In particular, it is worthwhile investigating whether the $\G$~cones whose links are the cohomogeneity one nearly K\"ahler manifolds in~\cite{FH} or the locally homogeneous nearly K\"ahler manifolds in~\cite{CV} actually arise as asymptotic cones of AC $\G$~manifolds.
\item The work of Moroianu--Nagy--Semmelmann~\cite{MNS} describes in detail the \emph{infinitesimal} deformations of Gray manifolds. More recent work by Foscolo~\cite{F} shows that the deformations are in general obstructed, and in particular that the homogeneous Gray manifolds are all rigid. It is still an interesting question to understand more completely the integrability of such infinitesimal deformations to actual deformations. Understanding this would allow us to consider more general deformations of $\G$~conifolds where we allow the asymptotic cones to also deform.
\item A related question is to better understand the spectrum of the Laplacian on $2$-forms for Gray manifolds. Some work on this already appears in Moroianu--Nagy--Semmelmann~\cite{MNS} and Moroianu--Semmelmann~\cite{MS}. But a more thorough understanding would allow us to conclude whether the results in Sections~\ref{ACextensionsec} and~\ref{smoothCSsec} are more general or are particular to $\G$~conifolds whose links are the known Gray manifolds.
\item One can define a \emph{stability index} for $\G$~cones in a similar way to the stability index for special Lagrangian cones~\cite{JSL2} or for coassociative cones~\cite{L1}. Results about the spectrum of the Laplacian for Gray manifolds would also tell us something about the stability index of $\G$~cones. Knowledge of the stability index tells us more about when the CS deformation theory is \emph{unobstructed}.
\item We need to construct the first examples of CS $\G$~manifolds. As mentioned earlier, the approach in~\cite{JK} of constructing compact \emph{smooth} $\G$~manifolds may possibly be generalizable to construct CS $\G$~manifolds.
\end{itemize}

\end{document}